\documentclass[11pt]{article}

 \usepackage[usenames]{color}
 \definecolor{white}{rgb}{1,1,1}

 \usepackage{amsmath}
 \usepackage{amscd}
 \usepackage{amsfonts}
 \usepackage{amssymb}
 \usepackage{amsthm}
 \usepackage{euscript}
 \usepackage{graphicx}
 \usepackage{epsfig}
 \usepackage{hyperref}
 \usepackage{enumerate}
 \usepackage{tikz}
 \usetikzlibrary{matrix,arrows,decorations.pathmorphing}
 \usepackage[margin=.78in]{geometry}
 \usepackage{mdwlist}

\newenvironment{packed_enum}{
\begin{enumerate}[ \bf i.]
  \setlength{\itemsep}{1pt}
  \setlength{\parskip}{1pt}
  \setlength{\parsep}{0pt}
}{\end{enumerate}}

\newenvironment{packed_enum_b}{
\begin{enumerate}[\,\,\,\,  \sc a.]
  \setlength{\itemsep}{1pt}
  \setlength{\parskip}{1pt}
  \setlength{\parsep}{0pt}
}{\end{enumerate}}

\newenvironment{packed_enum_c}{
\begin{enumerate}[\,\,\,\,   (i)]
  \setlength{\itemsep}{1pt}
  \setlength{\parskip}{1pt}
  \setlength{\parsep}{0pt}
}{\end{enumerate}}

\newtheorem{theorem}{Theorem}[section]
\newtheorem{definition}[theorem]{Definition}
\newtheorem{corollary}[theorem]{Corollary}
\newtheorem{proposition}[theorem]{Proposition}
\newtheorem{lemma}[theorem]{Lemma}

\theoremstyle{definition}
\newtheorem{IP}{Interpolation Problem}

\newtheorem*{discussion}{Discussion}
\newtheorem*{helmdec}{Hodge decomposition}
\newtheorem*{remL1L2}{Remark: $L^1$ vs.~$L^2$}
\newtheorem*{example}{Example}
\newtheorem*{examples}{Examples}
\newtheorem{example_n}{Example}
\newtheorem*{remark}{Remark}

\newtheorem*{notation}{Notation}

\newcommand{\idmatrix}{\mathbb{I}}
\newcommand{\prparx}{\mathrm{Pr}^\parallel(x)}
\newcommand{\prperx}{\mathrm{Pr}^\perp(x)}

\newcommand{\hpar}{\aaa}
\newcommand{\hper}{\bbb}
\newcommand{\rot}{R}
\newcommand{\RRot}{\rho}
\newcommand{\cs}{s}
\newcommand{\ce}{\ell}
\newcommand{\kk}{{\bf k}}
\newcommand{\kkh}{\widehat{{\bf k}}}

\newcommand{\dd}{d}
\newcommand{\Rm}{{\mathbb{R}^m}}
\newcommand{\Rd}{{\mathbb{R}^\dd}}
\newcommand{\Rt}{{\mathbb{R}^3}}
\newcommand{\Rdd}{\mathbb{R}^{\dd\times \dd}}
\newcommand{\Rtt}{\mathbb{R}^{3\times 3}}
\newcommand{\AAAm}{H^\parallel}
\newcommand{\BBBm}{H^\perp}
\newcommand{\AAA}{h^\parallel}
\newcommand{\BBB}{h^\perp}

\newcommand{\aaa}{k^\parallel}
\newcommand{\bbb}{k^\perp}
\newcommand{\ccc}{\tilde{k}}
\newcommand{\Fktilde}{\widetilde{F}}
\newcommand{\Fkperp}{F^\perp}
\newcommand{\dindex}{p}
\newcommand{\kscalar}{k_H}
\newcommand{\kvector}{k_W}
\newcommand{\fxa}{f_{\boldsymbol x,\boldsymbol\alpha}}

\begin{document}
\begin{center}
{\Large\bf Matrix-valued Kernels for Shape Deformation Analysis} 
\par\vspace{.5cm}

\begin{tabular}{ccc}
\large Mario Micheli$^\ast$
&\mbox{ }
& 
\large
Joan Alexis Glaun\`es$^\ddagger$ \\
\normalsize
\tt mariomicheli@gmail.com
&
&\tt alexis.glaunes@mi.parisdescartes.fr\\
\end{tabular}
\par\vspace{.3cm}
$^\ast$ 
Department of Mathematics, 
University of Washington,
Box 354350, 
Seattle, WA 98195, USA
\par\vspace{.1cm}
$^\ddagger$ 
Laboratoire MAP5, Universit\'e Paris Descartes and CNRS, Sorbonne Paris Cit\'e, 
\\ 45 rue des Saints-P\`eres, 75006 Paris, France
%MAP5,
%Universit\'e Paris Descartes,
%45 rue des Saints-P\`eres,
%7\`eme \'etage,
%75006 Paris, France
\end{center}
%\par\vspace{.1cm}
\begin{abstract}
The main purpose of this paper is providing a
systematic study and classification
of non-scalar kernels
for Reproducing Kernel Hilbert Spaces (RKHS),
to be used in the analysis 
of deformation in shape spaces endowed with 
metrics induced by the action of groups of diffeomorphisms.
After providing an introduction to matrix-valued kernels and
their relevant differential properties, 
we explore extensively those, that we call TRI kernels, 
that induce a metric on the corresponding Hilbert spaces of vector fields that is
both translation- and rotation-invariant. 
These are analyzed in an effective manner
in the Fourier domain, where the characterization
of RKHS of curl-free and divergence-free vector fields
is particularly natural. %We also develop a simple
A simple technique for constructing generic 
matrix-valued kernels
from scalar kernels is also developed. 
We accompany the exposition of the theory with 
several examples, and provide
numerical results
that show the dynamics induced
by different choices of TRI kernels
on the manifold of labeled landmark points.
%{\bf Ver.: Aug.~8, 2013}%~\cite{MMM1} 
\par\vspace{.1cm}
\end{abstract}
%
%\vspace{.01cm}
%\begin{center}{\large\it Dedicated to David Mumford on the occasion of his 75th birthday.}\end{center}
%
\section{Introduction}
\label{secIntro}
Recent years have seen the rapid development 
of acquisition techniques for medical data, such as 
magnetic resonance imaging, that allow the identification 
and very detailed visualization of %in great detail  
anatomical structures \em in vivo \em and the quantitative description of their
shape. The key challenge for researchers is to devise 
statistical methods that assess normal and abnormal variations
of such shapes across subjects, with the goal 
of gaining understanding of several pathologies 
%by comparing 
%data with templates (of typically healthy or unhealthy individuals),
and ultimately providing practitioners with
shape-based diagnostic tools. 
The problems of comparing geometrical objects of the same nature, 
establishing correspondences that are anatomically justified, and
quantifying deformation (i.e.~computing \em distances \em between shapes,
so to be able perform, for example, shape classification)
are central in addressing all of the 
above issues. The emerging discipline of 
computational anatomy thus lies at the interface
of  geometry, imaging science, statistics, and numerical analysis.
\par
{\bf Group action approach.} 
It was  Ulf Grenander~\cite{grenander} that introduced the notion of group action
in this field. The main idea is to compare anatomical objects through the estimation
and the analysis of a deformation 
of their entire space where such objects are located. The theoretical model thus consists in 
defining (i)~the mathematical space of the observed objects 
(e.g.~the images, surfaces, curves, point sets---in one word, the {\it shapes\/}), (ii)~an appropriate group of transformations, 
and (iii)~a group action on the objects.
This approach stems from the observation that the visual
comparison of medical images coming from different individuals 
suggests the existence of an underlying deformation of the ambient space.
\par
The formulation of precise registration algorithms 
requires rigorous mathematical modeling of the 
deformations %~$\varphi$ 
of such space. Rigid or affine
transformations alone are insufficient to describe the 
complexity of the observed deformations, and it is in fact appropriate
to consider functional spaces of infinite dimensions.
In the diffeomorphic approach introduced in~\cite{ChristensenRabbittMiller,dupuis} 
the transformations are obtained by
the temporal integration of a family of vector fields;
this was 
partly inspired
by Vladimir Arnold's seminal paper~\cite{arnold:66},
in which it was proven that 
incompressible 
fluid dynamics can be characterized as geodesic 
flow in the group of volume preserving
diffeomorphisms, with respect to the kinetic energy metric
(i.e.~the~$L^2$ norm of the fluid velocity).
 A rigorous mathematical foundation for this new model and 
the construction of invariant metrics for the groups of diffeomorphisms 
was established by Alain Trouv\'e~\cite{trouve}: in this approach 
one chooses a Hilbert space~$V$ of vector fields
whose norm~$\|\cdot\|_V$ defines the cost of infinitesimal deformations, and 
time-dependent diffeomorphisms~$\varphi_t$ are obtained by the integration
of a family of vector fields~$v(t,\cdot):\Rd\rightarrow\Rd$ in~$V$, 
$t\in[0,1]$,
via the %differential
equation~$\partial_t\varphi_t(x)=v(t,\varphi_t(x))$,
with~$\varphi_0(x)=x$. Given two shapes (geometric objects or images),
among all diffeomorphisms %of such type 
that perform the registration, 
the one that is generated by a time-dependent
vector field that minimizes the kinetic energy~$\int_0^1\|v(t,\cdot)\|_V^2dt$
is chosen,
and the square root of the minimal energy is in fact a \em distance \em
between the two shapes.
\par
This framework, often referred to as 
Large Deformation Diffeomorphic Metric Mapping (LDDMM),
has received considerable interest and
is now popular in the fields of 
computational anatomy, morphometry, and shape analysis.
The method may be applied for performing registration and computing distances
in different kinds of ``shape spaces'', such as the manifolds curves~\cite{glaunes:3, michormumford:3}, surfaces~\cite{zhang_s}, 
images~\cite{beg:2,miller:2,miller:1}, vector fields~\cite{cao:2}, diffusion tensor images~\cite{cao:1}, measures~\cite{glaunes.phd,glaunes:2}, 
and labeled feature points, %also known as 
or landmarks~\cite{glaunes:1,joshi,MMM1,miller:1}.
We should note that in this context the formulation of the registration problem 
requires the vector fields to be~$L^1$ or~$L^2$
in the time parameter~$t\,$; these and other notions, that deviate from the classical theory of 
dynamical systems, have been %organized and 
expanded in a book 
by Laurent Younes~\cite{Younes10}.
\par
David Mumford %, to whom this work is dedicated, 
and his 
collaborators have given considerable contributions to 
the understanding of
the differential geometry of shape spaces 
endowed with a Riemannian metric,
including several instances of metrics induced by 
the action of groups of diffeomorphisms.
The knowledge of 
such geometry (and in particular, of curvature) is fundamental as it allows one
to infer, for example, about the uniqueness of geodesics between shapes, 
the existence 
of conjugate points, and the well-posedness of the problem of computing
the intrinsic (or Karcher's) mean of a database of shapes. %To this regard,
See, for example, the papers~\cite{MMM1,MMM2,michormumford:1,michormumford:3,
mumford_geometry_2012,michormumford:4}
and references therein.
\par
{\bf Reproducing Kernel Hilbert Spaces.} In LDDMM
the  Hilbert space of vector fields~$V$ has a reproducing kernel,
whose choice uniquely determines the model in use and the properties of the deformation maps; 
moreover, from a computational point of view, the most demanding operations 
in registration algorithms usually consist in convolutions of 
data points with these kernels or their derivatives. 
%The analysis of two- or three-dimensional medical images
%may be performed either on a uniformly sampled
%version of such images or on a set of
%extracted features points, curves (that may
%correspond, for example, to the sulci or gyri of the cerebral cortex),
%or surfaces. 
For example, 
when dealing with the shape manifold~$\mathcal{L}^N(\Rd)$ 
%is the set 
of~$N$ labeled landmarks in~$\Rd$ (typically, $d=2$ or~3), 
%points 
%the mathematical object (i.e.~the \em shape\em\/) of interest  is given by a set of~$N$ points in~$\Rd$
%(typically, $d=2$ or~3) referred to as \em landmarks\em\/, and the 
the registration constraints are described in terms of the displacements of such points. 
The mathematical object of interest is the reproducing 
kernel~$K:\Rd\times\Rd\rightarrow\Rdd$
of the space~$V$ of deformation fields. 
One can in fact show that
the vector fields that minimize the kinetic energy for such registration constraints 
%(in that it minimizes the energy~$\int_0^1\|v(t,\cdot)\|^2_Vdt$)
are sums of spline functions
centered on the landmark trajectories~$x_a(t)$,
$a=1,\ldots,N$, $t\in[0,1]$, i.e.
\begin{equation}
\label{eqone}
v(t,x)
=
%\sum_{b=1}^NK\big(\,\cdot\,,\varphi_{0t}^v(x_b)\big)\alpha_b(t),
\sum_{a=1}^NK\big(x,x_a(t)\big)\alpha_a(t),
\qquad x\in\Rd,\;t\in[0,1],
\end{equation}
and that the associated cost of the infinitesimal 
deformation for vector fields of this type (i.e.~their 
norm in the Hilbert space~$V$)
is also expressed in terms of the reproducing kernel, as follows:
$$
\|v(t,\cdot)\|^2_V=
\sum_{a,b=1}^{N}
\alpha_a(t)\cdot K\big(x_a(t),x_b(t)\big)\alpha_b(t),
\qquad t\in[0,1];
$$
the vectors~$\alpha_a: [0,1]\rightarrow \Rd$, $a=1,\ldots,N$, are called \em momenta\em\/, in analogy with analytical mechanics~\cite{arnold:1}, and 
they completely parametrize the search space of an optimal 
(i.e.~energy-minimizing) solution. 
\par
%{\bf Reproducing Kernel Hilbert Spaces.} 
In this framework it is therefore natural to consider \em the kernel \em 
as the starting point for modeling
the linear deformation space~$V$ and the group of diffeomorphisms that it generates. The theory of Reproducing Kernel Hilbert Spaces (RKHS) was developed starting in the
in the 1940s, mostly 
by Aronszajn~\cite{aronszajn:43,aronszajn} and Schwartz~\cite{schwartz:64}, who built on previous studies by Bergman~\cite{bergman,bergman:book}, Bochner~\cite{bochner33}, Schoenberg~\cite{schoenberg:38,schoenberg}, and others. 
Such theory is mostly used in complex  
and functional analysis, and in more recent years it has found numerous 
applications and developments 
in statistics and machine learning~\cite{ScholkopfSmola}.
Reproducing kernels are also used as spline functions for performing data interpolation~\cite{wahba}; 
their interest in this framework stems from the fact that functions expressed in terms
of kernels solve the minimal norm interpolation
problem
in the corresponding Hilbert space. 
Kernels are also commonly used to interpolate data with values in the 
Euclidean space~$\Rd$ (again, mostly with $d=2$ or $3$), such as 
vector fields in fluid dynamics~\cite{bonaventura:11} %.~\cite{duchon:77,meinguet:79}.
\par
%While the theory of RKHS developed by Schwartz~\cite{schwartz:64} introduces
%matrix-valued kernels, 
The kernels that are most commonly used in applications
are %{\it scalar\/}, in that they are 
either scalar-valued functions or (in the case of RKHS of~$\Rd$-valued functions) 
they are given by a scalar-valued
function multiplied by the~$d\times d$ identity matrix. However,
employing kernels that are truly non-scalar
allows one to obtain desirable properties of the vector fields
that are not achievable otherwise.
To our knowledge
a
systematic study and classification of such kernels have not yet taken place. 
Here we focus our attention on kernels that induce
translation- and rotation-invariant norms
in the corresponding RKHS, since this a common 
requirement for the interpolation of geometric data.
We shall 
examine in detail the ties between the properties of the kernels and the corresponding
deformation spaces, 
with the goal of providing a large class of kernels that may 
be used to perform shape analysis.
In particular, matrix-valued 
kernels that induce divergence-free vector fields 
%(which cannot be achieved
%with scalar kernels)
are desirable in analyzing volume-preserving transformations;
however, it is the case that
curl-free and divergence-free vector fields
are only achievable with non-scalar reproducing kernels.
%
%To date all LDDMM studies and applications have been done using scalar kernels, and one of the original motivations of the present work is to enlarge the variety of kernels and 
%corresponding deformation maps in the LDDMM framework. One of the main outcomes of this work is the fact that divergence-free and curl-free kernels form the 
%two limit - although valid - cases of TRI kernels, and can be used in LDDMM applications.
%
%
\par
{\bf Additional related work.} 
As we said, it was Laurent Schwartz in his seminal work~\cite{schwartz:64}
that built the foundations of the theory of (non-scalar) %(in fact, linear operator-valued) 
reproducing kernels. Much more recently, 
%As previously mentioned, previous theoretical 
%work already has focused on the 
%mathematical properties of non-scalar kernels: it was %in fact 
%Schwartz in his seminal paper~\cite{schwartz:64}
%who developed the abstract theory of positive kernels. Recently, 
Carmeli \em et al.\em~\cite{carmeli:10} derived regularity results of 
vector-valued RKHS
and
analyzed some of their properties from the point of view of the machine learning community.
De~Vito \em et al.\em~\cite{devito:12} extend 
Mercer's theorem to matrix-valued measurable kernels, whereas
Micchelli and Pontil~\cite{micchelli:05} use such kernels for
learning vector-valued functions.
Cachier and Ayache~\cite{CachierAyache} consider a class of matrix-valued
kernels that generalize thin-plate splines 
for the interpolation of dense and sparse vector fields, precisely for the
purposes of image registration.
The work by Dodu and Rabut~\cite{DoduRabut} is %very much
related to ours in that it introduces a class of 
irrotational and 
divergence-free kernels that minimize Beppo Levi seminorms, for the interpolation of vector fields
in two and three dimensions; our study is restricted
to the case of positive definite kernels but is more general in that 
it characterizes the \em entire \em class
of kernels that induce translation- and 
rotation-invariant norms.
Last, but not least, we should certainly mention 
a very recent paper by Mumford and Michor~\cite{mumfordEuler},
where they study an approximation
to Euler's equation with EPDiff, i.e.~the 
geodesic equation in the group of diffeomorphisms~\cite{mumford_desolneux},
by choosing kernels that are Green's functions of 
a specific class of differential operators. The latter depend on two
positive parameters, and their limit behavior
yields precisely Euler's equation for fluid flow; such kernels, as well as their limit cases,
fall within our study.
\par
{\bf Paper organization.} 
%Besides setting up notation and relevant definitions
%for RKHS of vector-valued functions, 
In Section~\ref{vv_rkhs} we introduce notation and definitions 
for Reproducing Kernel Hilbert Spaces of vector-valued functions, 
state and prove existence and uniqueness theorems for such spaces, 
and investigate relevant differential properties of the corresponding 
matrix-valued kernels.
Section~\ref{secTRI} 
explores the characterization of what we call {\it TRI kernels\/}, i.e.~those that induce
a translation- and rotation-invariant metric on the corresponding
RKHS: this is also performed in the Fourier domain, and
constitutes the core of the present contribution;
in fact, it turns out that the characterization of kernels Hilbert spaces 
of curl-free
or divergence-free vector fields
is described in a more natural manner in the Fourier domain
than it is in the spatial domain.   
Section~\ref{constr} details methods for building TRI 
kernels from the more commonly used scalar kernels, which is of course
crucial for applications; in fact, we show that all matrix-valued 
kernels may be obtained with such constructive procedures.
Section~\ref{secApp} develops 
the equations of the dynamics induced by the LDDMM 
approach on the manifold of landmark points, such as~\eqref{eqone}, 
and presents some numerical
results for specific choices of TRI kernels.
We summarize conclusions and describe directions of 
potential future research developments %of our research 
in Section~\ref{secCon}. We would like to point out that
the length of our paper is justified by our desire to make it
self-contained, exhaustive, and accessible the largest possible 
audience---in the hope to attract more researchers (mathematicians,
engineers, statisticians, and computer scientists) to
%fascinating and 
this interdisciplinary field. % of computational anatomy.   
\section{Reproducing Kernel Hilbert Spaces of vector-valued %\\ 
functions}
\label{vv_rkhs}
In the classical setting, the elements of Reproducing Kernel Hilbert 
Spaces are scalar-valued 
functions~\cite{aronszajn:43,aronszajn,wahba}.
%\cite{schwartz:61-62} 
Such theory can be extended to vector-valued functions~\cite{carmeli:10,hein:2004,schwartz:64}
and
has recently been adapted to
 the study of 
shape and deformation~\cite{Younes10}. Here we provide
a summary of notions and results that are relevant precisely for this purpose,
also extending some known ones to the vector-valued case.
\begin{notation}
If~$S$ is a subset of a vector space then~$\mathrm{span}(S)$ denotes 
%with~$\mathrm{span}(S)$
the set of finite linear combinations of elements of~$S$. 
%If~$U$ is a subspace of an inner product vector space~$V$
%we indicate with~$U^\perp$ the orthogonal complement of~$U$. 
%
If~$(X,\|\cdot\|_X)$, $(Y,\|\cdot\|_Y)$ are two normed spaces with~$X$
continuously embedded in~$Y$ (i.e.~$X\subseteq Y$ and there
there is a constant $C>0$ such that
$\|x\|_Y\leq C\|x\|_X$ for all $x\in X$)
we write~$X\hookrightarrow Y$.
If~$(H,\langle\cdot,\cdot\rangle_H)$ is an inner product %(i.e.~a pre-Hilbert)  
space 
and~$U$
%we denote with~$\langle\cdot,\cdot\rangle_H$
%its inner product and with~$\|\cdot\|_H$ its norm; 
%if~$U$
is a subspace of~$H$  
we indicate its orthogonal complement with~$U^\perp$.
If~$H$ is a Hilbert space its dual space
is written as~$H^\ast$;
if a sequence~$\{u_n\}$ in~$H$ 
converges weakly to some~$u\in H$,
i.e.~$\langle u_n-u,h\rangle_H\rightarrow 0$ for all $h\in H$, we write~$u_n\rightharpoonup u$ in~$H$.
% and 
%the action of any~$\eta\in H^\ast$ on an element~$h\in H$
%is denoted with~$\langle\eta |h\rangle$. 
The  dot product of two vectors $a,b\in\mathbb{R}^d$
is denoted with~$\langle a,b\rangle_{\mathbb{R}^d}$
or~$a\cdot b$. Vectors~$a\in\Rd$
are treated as column vectors, 
and~$^T$ indicates the transpose of a vector or a matrix.

\end{notation}
%its transpose indicated with~$a^T$.
\begin{definition} 
Let~$(H,\langle\cdot,\cdot\rangle_H)$ be a Hilbert space of $\Rd$-valued functions %that are 
defined on a set~$\Omega$.
We call $H$ a \em Reproducing Kernel Hilbert Space (RKHS) \em if 
%and only if 
the evaluation functionals \mbox{$\delta_x^\alpha: %H\rightarrow R:
u\mapsto \alpha \cdot u(x)$} are linear and continuous for all~$x\in \Omega$
and~$\alpha\in\Rd$; that is, %if and only 
if~$\delta_x^\alpha\in H^\ast$. 
\end{definition}
%\par
%\end{itemize}
By the Riesz Representation Theorem~\cite{folland},
in Reproducing Kernel Hilbert Spaces, for all~$x\in\Omega$ and~$\alpha\in\Rd$,
there exists
a unique function~$K_x^\alpha(\cdot)\in H$
such that
\begin{equation}
\label{eq:repprop}
%\langle\delta_x^\alpha|u\rangle=
\big\langle K_x^\alpha,u
\big\rangle_H
=
%\big\langle
\alpha \cdot u(x)
%\big\rangle_{\Rd }
\end{equation}
for all $u\in H$.
Such function is called the
\em representer \em of the evaluation functional~$\delta_x^\alpha$,
and relation~\eqref{eq:repprop} is referred to as
the \em reproducing property \em of~$K_x^\alpha$.
%By~\eqref{eq:repprop} 
The map
%$\Rd \rightarrow 
$\alpha\mapsto K^\alpha_x(\cdot)$
is linear in~$\alpha$,
%; in fact, for all $x\in \Rd $,
%$\alpha,\beta\in \Rd $, and $u\in V$,
%\begin{eqnarray*}
%\big\langle K_x^{\alpha+\beta},u
%\big\rangle_{V}
%& = & 
%\big\langle\alpha+\beta,u(x)
%\big\rangle_{\Rd }
%\;=\;
%\big\langle\alpha,u(x)
%\big\rangle_{\Rd }
%+
%\big\langle\beta,u(x)
%\big\rangle_{\Rd }
%\\
%& = & 
%\big\langle K_x^{\alpha},u
%\big\rangle_{V}
%+
%\big\langle K_x^{\beta},u
%\big\rangle_{V}
%\;=\;
%\big\langle K_x^{\alpha}+K_x^{\beta},u
%\big\rangle_{V};
%\end{eqnarray*}
i.e.~$K_x^{a\alpha+b\beta}=aK_x^\alpha+bK_x^\beta$
for all~$x\in\Omega$, $\alpha,\beta\in\Rd$, and~$a,b\in \mathbb{R}$,
by the uniqueness of the representer.
So for any  pair of points $x,y\in \Omega $
there exists a %~$\dd\times \dd$ 
matrix~$K(y,x)\in \Rdd $
such that 
$K_x^\alpha(y)=K(y,x)\alpha$ for all~$\alpha\in \Rd $;
%(we are treating~$\alpha$ as a column vector);
the matrix-valued 
function $K:\Omega \times\Omega \rightarrow\Rdd $
is called the \em reproducing kernel\em\/, or simply the \em kernel\em\/, of
the space~$H$. 
%(whence the name of reproducing kernel Hilbert spaces).
\par
By the reproducing property~\eqref{eq:repprop} we have,
for any pair of points~$x,y\in \Omega $
and any pair of 
vectors $\alpha,\beta\in \Rd $, that
$
\langle K_x^{\alpha},K_y^{\beta}
\rangle_H
= 
%\langle
\alpha\cdot K_y^{\beta}(x)
%\rangle_{\Rd }
%=
%\langle\alpha,K(x,y)\beta
%\rangle_{\Rd }
=
\alpha\cdot K(x,y)\beta
$
%(where~$\alpha^T$ is the transpose of~$\alpha$) 
but also
$
\langle K_y^{\beta},K_x^{\alpha}
\rangle_H
= 
%\langle
\beta\cdot K_x^{\alpha}(y)
%\rangle_{\Rd }
=
%\langle
\beta\cdot K(y,x)\alpha
%\rangle_{\Rd }
%=
%\beta^T K(y,x)\alpha
=
\alpha\cdot K(y,x)^T\beta
,
$
so that 
%we have 
the symmetry $K(x,y)=K(y,x)^T$ holds for all $x,y\in \Omega $.
\begin{definition}
\label{non_deg}
A Reproducing Kernel Hilbert Space % with kernel~$K$ 
is called~\mbox{\em non-degenerate}
when it has the following property: 
for any~$N\in \mathbb{N}$
and any %\em 
distinct %
%\em 
points $x_1,\ldots, x_N \in\Omega$,
if 
the vectors $\alpha_1,\ldots, \alpha_N\in \Rd $
%and the \em distinct \em points $x_1,\ldots, x_N\in \Omega $ 
are such that 
%$\sum_{a=1}^N\big\langle\delta_{x_a}^{\alpha_a}\big|u\big\rangle=0$ 
$\sum_{a=1}^N\alpha_a\cdot u(x_a)=0$ 
for all~$u\in H$, 
%$\sum_{a=1}^N\big\langle\alpha_a,u(x_a)
%\big\rangle_{\Rd }=0$,
then $\alpha_1=\cdots=\alpha_N=0$.
\end{definition}
Non-degeneracy establishes a certain 
richness of functions in
Reproducing Kernel Hilbert Spaces.
%In fact it can be rewritten as follows: 
%if at least one of the vectors~$\alpha_1,\ldots, \alpha_N\in\Rd$ 
%is non-zero, then for any choice of distinct points~$x_1,\ldots, x_N\in \Omega $
%there exists a function~$u\in H$
%such that
%$\sum_{a=1}^N %\langle
%\alpha_a\cdot u(x_a)
%\rangle_{\Rd }
%\not=0$. 
Rather obviously, we will call a RKHS \em degenerate \em if it is not non-degenarate.
%
%for fixed points~$x_1,\ldots, x_N\in \Omega $,
%if  
%vectors~$\alpha_1,\ldots, \alpha_N$ are chosen in $\Rd $
%of which at least one is non-zero,
%then there exists at least a function~$u\in H$
%such that
%$\sum_{a=1}^N %\langle
%\alpha_a\cdot u(x_a)
%\rangle_{\Rd }
%\not=0$.
%We shall return on this concept later on in the paper.
\begin{definition}
\label{def_posdef}
A generic matrix-valued 
function~$K:\Omega\times\Omega\rightarrow\Rdd$ is \em positive definite \em
if
for abitrary $N\in \mathbb{N}$, 
% ~$M$-tuples of
vectors
$\alpha_1,\ldots, \alpha_N\in \Rd $
and 
points
$x_1,\ldots, x_N\in \Omega $, 
it is the case that
\begin{equation}
\label{eq:Gpos}
\sum_{a,b=1}^N
\alpha_a\cdot
K(x_a,x_b)
\alpha_b
\geq 0.
\end{equation}
Moreover, $K$~is \em strictly \em positive definite if 
for abitrary~$N\in \mathbb{N}$, 
% ~$M$-tuples of
vectors
$\alpha_1,\ldots, \alpha_N\in \Rd $,
and 
%\em 
distinct 
%\em 
points
$x_1,\ldots, x_N\in \Omega$
the above inequality holds, 
with equality
 if~and only if~$\alpha_1=\cdots=\alpha_N=0$.
\end{definition}
\begin{proposition}
\label{Kposdef}
The kernel~$K$ of a RKHS is a positive definite matrix-valued function. Moreover, 
the RKHS is non-degenerate if and only if its kernel is
\em strictly \em positive definite. 
%Let~$K$ be the kernel
%of a RHKS. Then
%For abitrary~$N\in \mathbb{N}$, 
% ~$M$-tuples of
%points
%$x_1,\ldots, x_N\in \Rd $ 
%and 
%vectors
%$\alpha_1,\ldots, \alpha_N\in \Rd $,
%it is the case that
%\begin{equation}
%\label{eq:Kpos}
%\sum_{a,b=1}^N
%\alpha_a^T
%K(x_a,x_b)
%\alpha_b
%\geq 0
%\end{equation}
%with equality holding if and only if~$\alpha_1=\ldots=\alpha_N=0$.
\end{proposition}
\begin{proof}
By the reproducing property~\eqref{eq:repprop} 
of the representer function we have
$$
%\textstyle
%\begin{align*}
0\leq
\Big\|
\sum_{i=a}^N
K_{x_a}^{\alpha_a}
\Big\|^2_H
=
%\Big\langle
\bigg\langle
\sum_{a=1}^N
K_{x_a}^{\alpha_a}
,
\sum_{b=1}^M
K_{x_b}^{\alpha_b}
%\Big\rangle_{\!\! H}
\bigg\rangle_{\!\! H}
%\;
=
%\;
\sum_{a,b=1}^M
\big\langle
K_{x_a}^{\alpha_a}
,
K_{x_b}^{\alpha_b}
\big\rangle_H
%\\& 
%= 
%\sum_{a,b=1}^M
%\big\langle
%\alpha_a
%\cdot
%K_{x_b}^{\alpha_b}
%(x_a)
%\big\rangle_{\Rd }
%\;
%=
%\;
%\sum_{a,b=1}^M
%\big\langle
%\alpha_a
%,
%K
%(x_a,x_b)
%\alpha_b
%\big\rangle_{\Rd }
= 
\sum_{a,b=1}^M
\alpha_a\cdot
K(x_a,x_b)
\alpha_b.
%\end{align*}
$$
Assuming that the RKHS is non-degenerate, 
if for $x_1,\ldots, x_N\in\Omega$ (distinct),
$\alpha_1,\ldots,\alpha_N\in\Rd$
we have 
$\|\sum_aK_{x_a}^{\alpha_a}\|_H^2=\sum_{a,b}\alpha_a
\cdot K(x_a,x_b)\alpha_b=0$,
then
$\sum_{a=1}^N
K_{x_a}^{\alpha_a}
=0$, 
i.e.~$
%\big
\langle
\sum_{a=1}^N
K_{x_a}^{\alpha_a}
,u
%\big
\rangle_{H}=0
$
for all~\mbox{$u\in H$};
this is
equivalent to
$
\sum_{a}%^N
%\langle
\alpha_a
\cdot u(x_a)
%\rangle_{\Rd }
=0
$ for all~$u\in H$,
therefore~$\alpha_1=\cdots=\alpha_N=0$ by Definition~\ref{non_deg}.
On the other hand, assuming that the kernel~$K$
is strictly positive definite, if at least one of the 
vectors~$\alpha_1,\ldots,\alpha_N\in\Rd$
is non-zero take~$u:=\sum_aK_{x_a}^{\alpha_a}$, with distinct 
points~$x_1,\ldots,x_N\in\Omega$; 
it is the case that
$\sum_{a}\alpha_a\cdot u(x_a)=\sum_{a,b}\alpha_a\cdot K(x_a,x_b)\alpha_b>0$
by strict positive definiteness.
\end{proof}
\begin{example}
%It is immediate to build an example of a degenerate 
%RKHS.
Consider a finite dimensional space~$H$ of $\Rd$-valued functions
with a given inner product~$\langle\cdot,\cdot\rangle_H$. 
We may assume without loss of generality 
that $H=\mathrm{span}\{f_1,\ldots,f_M\}$ with
$\langle f_i,f_j\rangle_H=\delta_{ij}$, where~$\delta_{ij}$ is Kronecker's delta.
The reproducing kernel of~$H$ is the matrix-valued 
function 
$
K(x,y)=\sum_{i=1,\ldots,M}%^M
f_i(x)f_i(y)^T
$. It is in fact the case that \em all \em finite-dimensional
functions spaces are degenerate. To see this, fix
arbitrarily~$N\in \mathbb{N}$ and the distinct points~$x_1,\ldots,x_N\in\Omega$.
Define:
$$
\boldsymbol{\alpha}:=
\left[
\!\!
\begin{array}{c}
\alpha_1
\\
\vdots
\\
\alpha_N
\end{array}
\!\!
\right]\in \mathbb{R}^{Nd}
\qquad
\mbox{and}
\qquad
f_i(\boldsymbol{x}):=
\left[
\!\!
\begin{array}{c}
f_i(x_1)
\\
\vdots
\\
f_i(x_N)
\end{array}
\!\!
\right]\in \mathbb{R}^{Nd},\quad i=1,\ldots,M,
$$
all intended as column vectors. Having~$\sum_{i=1}^N\alpha_i\cdot u(x_i)=0$
for all~$u\in H$
is equivalent to~$\boldsymbol{\alpha}$ being 
orthogonal 
to~$U:=\mathrm{span}\{f_1(\boldsymbol{x}),\ldots,f_M(\boldsymbol{x})\}$,
i.e.~$\boldsymbol{\alpha}\in U^\perp$. 
But~$\dim U\leq M$, therefore 
we must have that $\dim U^\perp\geq \max\{Nd-M,0\}$. Whence, if~$N$ 
was chosen large enough,~$\boldsymbol\alpha$ is not necessarily zero.
%(where~$^T$ indicates the transpose of a vector). 
\par
%Also, if there exist
%points~$x_1,\ldots,x_N$ such that~$f_i(x_a)=0$ for all~$i=1,\ldots,M$
%and all~$a=1,\ldots,N$, then we have
%$\sum_{a,b=1}^N\alpha_a\cdot K(x_a,x_b)\alpha_b=0$
%for all (possibly nonzero) $\alpha_1,\ldots,\alpha_N\in\Rd$,
%and the space~$H$ is degenerate. However if, say, all~$d$ 
%components of the functions~$f_1,\ldots,f_M$
%are positive everywhere in~$\Omega$ then the space is non-degenerate.
%if, in addition, $\Omega$ is an open and connected subset of~$\Rd$ and 
%the functions~$f_1,\ldots,f_M$ are not continuous in~$\Omega$,
%then~$H$ is not~$0$-admissible.
\end{example}
We have seen above that Reproducing Kernel Hilbert Spaces
%KHS 
have a unique reproducing kernel.
The next definition of ``kernel'' is justified by
the theorem that immediately follows.
\begin{definition}
\label{def_ker}
A positive definite function~$K:\Omega\times\Omega\rightarrow\Rdd$
such that~\mbox{$K(y,x)=K(x,y)^T$}
for all~$x,y\in\Omega$ is called 
a \em positive kernel\em, or simply a \em kernel\em, of dimension~$d$.
If, in addition,~$K$ is a strictly positive definite function, then we shall call it 
a \em strictly positive kernel \em of dimension~$d$.
\end{definition}
\begin{theorem}
\label{exist}
Given a positive kernel~$K:\Omega\times\Omega\rightarrow\Rdd$
of dimension~$d$, 
%of dimension~$d$
there exists a {\em unique} RKHS 
of~$\Rd$-valued functions defined on~$\Omega$ 
that has~$K$ as its reproducing kernel.
\end{theorem}
\begin{proof}
For arbitrary $x\in\Omega$ and~$\alpha\in\Rd$
we define the function~$K_x^\alpha(\cdot):=K(\cdot,x)\alpha$ and the linear space
$H_0:=\mathrm{span}\{K_x^\alpha\,|\,x\in\Omega,\alpha\in\Rd\}$.
%of finite linear combinations of such functions.
We also define the 
inner product~$\langle K^\alpha_x,K^\beta_y\rangle_{H_0}:=\alpha\cdot
K(x,y)\beta$. In order to extend it to all of~$H_0$
by bilinearity it is sufficient to show that if 
%we have two 
%different representations of~$f\in H_0$, i.e.~if
$$
%\textstyle
f=\sum_{i=1}^{n}K_{x_i}^{\alpha_i}=\sum_{i=1}^{n'}K_{x'_i}^{\alpha'_i}
\qquad\mbox{(i.e.~there are two different representations of~$f\in H_0$)}
$$
then for all $g=K_y^\beta\in H_0$ 
%(with arbitrary $y\in\Omega$ and~$\beta\in\Rd$) 
we have
$\langle f,g\rangle_{H_0}
:=
\sum_{i=1}^{n}\langle K_{x_i}^{\alpha_i},K_y^\beta\rangle_{H_0}
=
\sum_{i=1}^{n'}\langle K_{x'_i}^{\alpha'_i},K_y^\beta\rangle_{H_0}
$. To see this 
take~$m:=n+n'$, $\gamma_i:=\alpha_i$ 
for $1\leq i\leq n$ and
$\gamma_i:=-\alpha'_i$
for $n+1\leq i\leq n+n'$; 
%To show this, 
now observe that since~$\sum_{i=1}^mK_{z_i}^{\gamma_i}=0$ 
we also have %~$\sum_{j=1}^mK_{z_i}^{\gamma_i}(y)=0$ for all~$y\in\Omega$, therefore
$
\sum_{j=1}^m\langle K_{z_j}^{\gamma_j}, K_{y}^{\beta}\rangle_{H_0}=
\sum_{j=1}^m \gamma_j\cdot K(z_j,y)\beta
=
\big( \sum_{j=1}^m K_{z_j}^{\gamma_j} (y)\big)\beta
=0
$.
\par
It is immediate to verify that the inner product~$\langle\cdot,\cdot\rangle_{H_0}$ is symmetric;
also, if~$u\in H_0$
%=\sum_{a}
%K_{x_a}^{\alpha_a}\in H_0$
then it is the case that~$\langle u,u\rangle_{H_0}\geq 0$
%=
%\sum_{a,b} \alpha_a\cdot K(x_a,x_b)\alpha_b\geq 0$
by the positive definiteness of~$K$.
%On the other hand, 
By construction  
$\langle K_x^\alpha,u\rangle_{H_0}=\alpha\cdot u(x)$ 
for all~$u\in H_0$,
so by the Cauchy-Schwarz inequality we have
$|\alpha\cdot u (x)|\leq\|K_x^\alpha\|_{H_0}\|u\|_{H_0}$
for all~$u\in H_0$, $x\in\Omega$ and $\alpha\in\Rd$;
whence if~$\langle u,u\rangle_{H_0}=0$
then~$u=0$. 
Therefore $H_0$ is an inner product space.
\par
Now one can follow the standard procedure
for the completion of metric spaces~\cite{rudin1}
with some modifications. Letting~$\mathcal{C}$ be 
\em the set of Cauchy sequences in~$H_0$\/\em, one first
partitions it into equivalence classes (we call~$\{u_n\},\{v_n\}\in\mathcal{C}$
equivalent and write~$\{u_n\}\sim\{v_n\}$
if~\mbox{$\lim_{n\rightarrow\infty}\|u_n-v_n\|_{H_0}=0$}).
We denote with~$H_e:=\mathcal{C}/\!\!\sim$ the set of equivalence classes,
and with~$[\{u_n\}]$ the generic element of~$H_e$.
Then one defines the inner product between
$\mathcal{U},\mathcal{V}\in H_e$
as~$\langle\mathcal{U},\mathcal{V}\rangle_{H_e}
:=\lim_{n\rightarrow\infty}\langle u_n,v_n\rangle_{H_0}$,
where~$\{u_n\}\in \mathcal{U}$ and~$\{v_n\}\in \mathcal{V}$.
The map~$\varphi:H_0\rightarrow H_e$ given by $\varphi(u)=[\{u,u,u,\ldots\}]$ 
is an isometry, and~$\varphi(H_0)$ is in fact dense in~$H_e$.
Any Cauchy sequence in~$\varphi(H_0)$
converges to some~$\mathcal{U}\in H_e$
(in fact if~$\{\varphi(u_n)\}$ is Cauchy in~$H_e$ then by isometry~$\{u_n\}$ is Cauchy in~$H_0$, 
therefore $\{u_n\}$ belongs to some~$\mathcal{U}\in H_e$; one can prove that
$\varphi(u_n)\rightarrow\mathcal{U}$ in $H_e$). 
Together with the density of~$\varphi(H_0)$ in~$H_e$, this can 
be used to show that~$H_e$
is complete. As we mentioned above, this procedure was rather standard; 
what follows is specific to our setting.
%We can identify~$H_0$ and~$\varphi(H_0)$
%and regard~$H_0$ as embedded in~$H$. 
\par  
It is in fact the case that any Cauchy sequence in~$H_0$ converges pointwise
to some function~$\Omega\rightarrow\Rd$.
%%%
%%%
%%%
To see this, 
if~$\{u_n\}\in\mathcal{C}$ then by the Cauchy-Schwarz inequality
%it is the case that
$
|(u_m(x)-u_n(x))\cdot\alpha|
=|\langle u_m-u_n,K_x^\alpha\rangle_{H_0}|
\leq
%\|f_m-f_n\|_{H_0}
%\|K_x^\alpha\|_{H_0}
%=
\|u_m-u_n\|_{H_0}
\sqrt{\alpha\cdot K(x,x)\alpha}
$
for all~$x\in\Omega$, $\alpha\in\Rd$, and~$m,n\in\mathbb{N}$.
Whence %if $\{u_n\}$ is Cauchy in~$H_0$ then 
$\{u_n(x)\}$ is Cauchy in~$\Rd$ for all~$x\in\Omega$
and~there exists a function~$u:\Omega\rightarrow\Rd$ such that~$u_n\rightarrow u$
pointwise. 
\par
On the other hand
Cauchy sequences in~$H_0$
from distinct equivalence classes converge, pointwise, to distinct functions.
In fact, if 
$\{u_n\},\{v_n\}\in \mathcal{C}$ %are Cauchy sequences in~$H_0$ that
converge 
pointwise to the same~$u:\Omega\rightarrow\Rd$ then
the sequence~$\{h_n\}$, with~$h_n:=u_n-v_n$, is also Cauchy
in~$H_0$ and it converges to~$0$ pointwise. Therefore,
since $\langle h_n,K_x^\alpha\rangle_{H_0}=h_n(x)\cdot\alpha$ 
for all~$x\in\Omega$ and~$\alpha\in\Rd$, we have that 
$h_n\rightharpoonup 0$  in~$H_0$.
By the Cauchy property of~$\{h_n\}$ there exists an integer~$M$ such that
$
\|h_n\|^2_{H_0}-2\langle h_n,h_M\rangle_{H_0}
\leq
\|h_n-h_M\|^2_{H_0}
\leq\varepsilon
$
for all~$n\geq M$.
But we have~$\lim 
\langle h_n,h_M\rangle_{H_0}=0$ by weak convergence, whence
$\|h_n\|_{H_0}^2=\|u_n-v_n\|_{H_0}^2\leq 2\varepsilon$
for sufficiently large~$n$.  So it is the case that $\{u_n\}\sim\{v_n\}$.
%We conclude that Cauchy sequences 
%from distinct equivalence classes converge, pointwise, to distinct functions.
\par
This suggests that we may realize the completion~$H_0$ as follows:
we associate to each equivalence class in~$\mathcal{C}/\!\!\sim$ 
of Cauchy sequences in~$H_0$ the 
function~$u:\Omega\rightarrow\Rd$ to which they converge pointwise, and let~$H$
be the space of such functions.
If $\{u_n\},\{v_n\}\in \mathcal{C}$ and~$u_n\rightarrow u$,~$v_n\rightarrow v$ pointwise,
we say that $u=\lim_n u_n$ in~$H$ and 
set $\langle u,v\rangle_H:=\lim_{n\rightarrow\infty}
\langle u_n,v_n\rangle_{H_0}$, so~$H$ is isometric to~$H_e$. By construction the function space~$H$
is complete, $H_0$ is embedded in~$H$, and $H=\overline{H}_0$. 
Moreover~$K$ is a reproducing kernel for~$H$, since 
if~$\{u_n\}\in\mathcal{C}$ and~$u_n\rightarrow u\in H$ pointwise
then $\langle K_x^\alpha ,u\rangle_H=\lim_n\langle K_x^\alpha ,u_n\rangle_{H_0}
=\lim_n \alpha\cdot u_n(x)=\alpha \cdot u(x)$,
for all~$x\in\Omega$ and~$\alpha\in\Rd$.
\par
Finally, assume that~$H_1$ is 
another RKHS with the same kernel~$K$. So it must be the case
that~$H_0\subseteq H_1$, and the inner 
products~$\langle\cdot,\cdot\rangle_{H}$
and~$\langle\cdot,\cdot\rangle_{H_1}$ 
must coincide on~$H_0$. 
If~$\{u_n\}\in \mathcal{C}$ %be a Cauchy sequence in~$H_0$
is such that~$u_n\rightarrow u\in H$ pointwise,
by the completeness of~$H_1$ there exists~$v\in H_1$
with~$u_n\rightarrow v$ in~$H_1$, and 
$
|(u(x)-v(x))\cdot\alpha|=
%\lim_{n\rightarrow\infty}
\lim_{n}
|(u_n(x)-v(x))\cdot\alpha|
=
%\lim_{n\rightarrow\infty}
\lim_{n}
|\langle u_n-v,K_x^\alpha\rangle_{H_1}|
\leq 
%\lim_{n\rightarrow\infty}
\lim_{n}
\|u_n-v\|_{H_1}\|K_x^\alpha\|_{H_0}=0,
$
for all~$x\in\Omega$ and~$\alpha\in\Rd$, so it is the case that $u=v$.
Therefore~$H\subseteq H_1$.
Also $\langle\cdot,\cdot\rangle_{H_1}=\langle\cdot,\cdot\rangle_{H}$
on~$H$, whence~$H$ is a closed subset of~$H_1$ and~$H_1=H\oplus H^\perp$.
If~$u\in H_1$ is orthogonal to~$H$
then
$\alpha\cdot u(x)=\langle u, K_x^\alpha\rangle_{H_1}=0$
for all~$x\in\Omega$ and~$\alpha\in\Rd$,
whence~$u=0$. In conclusion~$H=H_1$.
\end{proof}
\begin{corollary}
Let~$H$ be a RKHS of\/ $\Rd$-valued functions that are defined on~$\Omega$, with 
reproducing kernel~$K$.
It is the case that~$H=\overline{H}_0$, where
$H_0=\mathrm{span}\{K(\cdot,x)\alpha\,|\,x\in\Omega,\alpha\in\Rd\}$.
\end{corollary}
\begin{corollary}
\label{weak_cor}
Let~$H$ be a RKHS, and~$\{u_n\}$ a sequence in~$H$. If $u_n\rightharpoonup u$
in~$H$ then it converges pointwise to~$u$.
On the other hand, if~$\{u_n\}$ 
is bounded in~$H$ and it converges pointwise to some function~$u$,
then~$u\in H$ and~$u_n\rightharpoonup u$ in H.  
\end{corollary}
\begin{proof}
Let~$x\in\Omega$ and~$\alpha\in \Rd$ be arbitrary.
If~$u_n\rightharpoonup u$ in~$H$ then 
%for all~$x\in\Omega$ and~$\alpha\in\Rd$
%we have~
$\lim_n \alpha\cdot u_n(x)
=\lim_n \langle K_x^\alpha,u_n\rangle_H
=\langle K_x^\alpha,u\rangle_H
=\alpha\cdot u(x)$.
Vice versa, if~$\{u_n\}$ is bounded in~$H$ then it has
a subsequence~$\{u_{\psi(n)}\}$ that converges weakly to some~$w\in H$;
however, 
%for all~$x\in\Omega$ and~$\alpha\in \Rd$ 
%we have
$\alpha\cdot w(x)
=
\langle K_x^\alpha,w\rangle_H
=
\lim_n\langle K_x^\alpha,u_{\psi(n)}\rangle_H
=
\lim_n\alpha\cdot u_{\psi(n)}(x)
=
\alpha\cdot u(x)$, whence~$u=w$ and $u\in H$.
Now, %for all $x\in\Omega$ and $\alpha\in\Rd$ 
we have
$
\lim_n\langle K_x^\alpha,u_n\rangle_H 
=
\lim_n\alpha\cdot u_n(x) 
=
\alpha\cdot u(x)= \langle K_x^\alpha,u\rangle_H,
$
so~$\lim_n\langle u_n,h_0\rangle_H=\langle u,h_0\rangle_H$
for all $h_0\in H_0$. If~$M:=\sup_n\|u_n-u\|_H$ then, since~$H=\overline{H}_0$,
for arbitrary~$\varepsilon>0$ and~$h\in H$
we can choose~$h_0\in H_0$ such that $\|h-h_0\|<\varepsilon/M$,
and we have
$|\langle u_n-u,h\rangle_H|
\leq
|\langle u_n-u,h_0\rangle_H|+\|u_n-u\|_H\|h-h_0\|_H<2\varepsilon$ for sufficiently large~$n$.
%\frac{\varepsilon}{2M}$
%For the same reason any weakly converging subsequence of~$\{u_n\}$
%must converge to~$u$. Now, if~$\{u_n\}$ did \em not \em converge
%weakly to~$u$ there would exist an~$\varepsilon>0$ 
%and a subsequence~$\{u_{\phi(n)}\}$ such that~$|\langle u_{\phi(n)},u\rangle_H$ 
\end{proof}
\begin{proposition}[Kernel of separable RKHS] Let $H$ be a separable RKHS,
$K$ its kernel and~$\{u_i\}_{i\in\mathbb{N}}$
an orthonormal basis of~$H$. Then
$K(x,y)=\lim_{n\rightarrow\infty}
\sum_{i=1}^n u_i(x)\,u_i(y)^T$ pointwise,
for all~$x,y\in\Omega$.
\end{proposition}
\begin{proof}
For all~$y\in\Omega$ and~$\alpha\in\Rd$
we have
$K_y^\alpha=\lim_{n\rightarrow\infty}\sum_{i=1}^n\langle K_y^\alpha,u_i\rangle_Hu_i$
(limit in~$H$ and therefore pointwise), whence
$
K(x,y)\alpha
=
\lim_{n\rightarrow\infty}\sum_{i=1}^n\langle K_y^\alpha,u_i\rangle_H u_i(x)
=
\lim_{n\rightarrow\infty}\sum_{i=1}^n\alpha\cdot u_i(y)\, u_i(x).
$
\end{proof}
We now introduce notions of \em regularity \em 
of $\Rd$-valued functions in RKHS, for which we assume that~$\Omega$
is an open, connected subset of~$\Rm$. 
Denote with~$C_0(\Omega ,\Rd )$  
the space
of continuous functions $u:\Omega \rightarrow\Rd $
that vanish at infinity (i.e.~such that for every~$\varepsilon>0$
the set~$\{x\in\Omega:\|u(x)\|_{\Rd }
\geq\varepsilon\}$
is compact) 
which is Banach with the norm
$
\|u\|_{\infty}
:=
\sup_{x\in \Omega }
\|u(x)\|_{\Rd }.
$
For any integer $\cs\geq 0$ we define
$$
C_0^\cs(\Omega ,\Rd )
:=
\big\{
u
\in
C^\cs(\Omega ,\Rd ):
\partial^\dindex u
\in
C_0(\Omega ,\Rd )
\mbox{ for } |\dindex|\leq \cs
\big\},
$$
where~$C^\cs(\Omega ,\Rd )$
is the space of functions %$\Rd \rightarrow\Rd $ 
that are continuously differentiable~$\cs$
times. (We have used multi-index notation~\cite{evans}
for partial 
derivatives~$\partial^\dindex
=
\partial_1^{\dindex_1}\partial_2^{\dindex_2}
\ldots\partial_m^{\dindex_m}$: i.e.~if $\dindex=(\dindex_1,\ldots,\dindex_m)$ 
then we set $|\dindex|:=\dindex_1+\ldots+\dindex_m$. Also, we have used~$\cs$ instead 
of the more customary~$k$ to avoid confusion with a symbol that we will use
later for kernels).
%for partial derivatives~\cite{evans,folland}).   
The space $C_0^\cs(\Omega,\Rd)$ is Banach with the~$W^{\cs,\infty}$
(Sobolev)
norm 
%\mbox{
$$\|u\|_{\cs,\infty}
:=
\sum_{|\dindex|\leq \cs}\,
\sup_{x\in \Omega }
\|
\partial^\dindex u(x)
\|_{\Rd },$$
%}
and  $C_0^\cs(\Omega,\Rd)\hookrightarrow C_0^{\cs-1}(\Omega,\Rd)$
for any~$\cs\geq 1$.  
Also,~$C(\Omega,\Rd)$ and~$C^{\cs}(\Omega,\Rd)$, respectively,
have the topology  induced by the 
convergence with respect to norms~$\|\cdot\|_{\infty}$ and~$\|\cdot\|_{\cs,\infty}$
on compact sets.
\begin{definition}
\label{def:adm}
Let~$\cs\geq0$ be an integer.
A Hilbert space $H$ %~$(H,\langle\cdot,\cdot\rangle_H)$ 
of $\Rd$-valued functions defined on a set~$\Omega$ %~$\Omega \rightarrow \Rd $
is called 
$\cs$-\em admissible \em %Hilbert space
if 
%the following conditions hold:
%\begin{enumerate}
%\item \label{adm:1} 
$H\hookrightarrow C_0^\cs(\Omega ,\Rd )$. 
%is continuously 
%embedded in the space~$C_0^k(\Omega ,\Rd )$,
%i.e.~there exists a positive constant~$C$ such that 
%$\|u\|_{k,\infty}\leq C\|u\|_H$, for all~$u\in H$.
\end{definition}
\par
Given a $\cs$-admissible space~$H$,
for a fixed point~$x\in \Omega $
and a fixed vector~$\alpha\in \Rd $
the evaluation functional
%\begin{equation}
%\label{eval:funct}
$
\delta_x^\alpha:H\rightarrow \mathbb{R}:
u\mapsto
%\big
%\langle
\alpha\cdot u(x)
%\big
%\rangle_{\Rd }
$
%\end{equation}
is obviosuly linear in~$u$ but also bounded:
in fact, for all $u\in H$,
%\begin{eqnarray*}
$$
|\delta_x^\alpha(u)|
= 
|
\alpha\cdot u(x)
|
\;\leq\;
\|\alpha\|_{\Rd }
\|u(x)\|_{\Rd }
\leq 
\|\alpha\|_{\Rd }
\|u\|_{\cs,\infty}
\;\leq\;
C\|\alpha\|_{\Rd }
\|u\|_{H}\,;
$$
%\end{eqnarray*}
%for all~$u\in H$;
we have used the 
Cauchy-Schwarz inequality, the definition 
of~$\|\cdot\|_{\cs,\infty}$,
and  
Definition~\ref{def:adm}.
Therefore if~$H$ is a $\cs$-admissible Hilbert space then it is also a Reproducing Kernel Hilbert Space,
and admits a reproducing kernel such that~$%K_x^\alpha(\cdot)=
K(\cdot,x)\alpha\in C_0^\cs(\Omega,\Rd)$
for all~$x\in\Omega$, $\alpha\in\Rd$. 
\begin{remark}
By the symmetry~$K(x,y)=K(y,x)^T$ it is the case that if~$K$ is continuous
(respectively, differentiable) in one of the two variables then it is continuous (differentiable)
in the other one too.
\end{remark}
If kernel~$K:\Omega\times\Omega\rightarrow\Rdd$ is smooth enough and
$p=(p_1,\ldots,p_m)$ and $q=(p_1,\ldots,p_m)$ are multi-indices, we indicate
with $\partial_1^p\partial_2^qK(\cdot,\cdot)$ %and~$\partial_2^qK(\cdot,\cdot)$
the partial derivative of~$K$ where $\partial^p$ and $\partial^q$
are taken with respect to the first and second sets of variables of~$K(\cdot,\cdot)$, respectively.
%\clearpage
%
%with respect to the first set of~$m$
%variables and the $q$-partial derivative of~$K$
%with respect to the second set of~$m$ variables, respectively.
\begin{theorem} 
\label{Th_diff}
Let~$H$ be a RKHS with kernel~$K:\Omega\times\Omega\rightarrow\Rdd$, 
and~$\cs\geq0$ be an integer. The following two statements are equivalent:
\begin{packed_enum_b}
\item $H\hookrightarrow C^\cs(\Omega,\Rd)$;
\item the function $\partial_1^p\partial_2^pK$ exists for all multi-indices~$p$ %and $q$
with~$0\leq|p|\leq \cs$, it is  continuous  
in each of the two variables (separately), and it is locally bounded. %for   
%all multi-indices~$p$ such that
%with~$0\leq|p|\leq k$. %and~$0\leq|q|\leq k$
\end{packed_enum_b}
Under the above assumptions, the following also holds:
\begin{packed_enum_b}
\setcounter{enumi}{2}
\item for all~$x\in\Omega$, $\alpha\in\Rd$, and multi-index p such that
$0\leq |p|\leq \cs$, we have 
 $\partial_2^pK(\cdot,x)\alpha\in H$ and
\begin{equation}
\label{rep_diff}
\big\langle
\partial_2^pK(\cdot,x)\alpha,u
\big\rangle_H
=
\alpha\cdot\partial^pu(x),
\qquad\mbox{for all }u\in H.
\end{equation}
\end{packed_enum_b}
If, furthermore, we have~$K(\cdot,x)\alpha\in C_0^\cs(\Omega,\Rd)$
for all~$x\in \Omega$ and~$\alpha\in\Rd$, then $H\hookrightarrow C_0^\cs(\Omega,\Rd)$.
\end{theorem}
\begin{proof}
We shall prove that ({\sc a}\,$\Rightarrow$\,{\sc b}\,\&\,{\sc c}) and 
({\sc b}\,$\Rightarrow$\,{\sc a}\,\&\,{\sc c}); in particular, this will imply that 
({\sc a}\,$\Leftrightarrow$\,{\sc b}).
% and that either condition {\sc a}
%or {\sc b} is sufficient for {\sc c} to hold as well. 
\par
$\bullet$ ({\sc a}$\,\Rightarrow\,${\sc b}\,\&\,{\sc c}).
We will prove by induction on~$|p|$ that
for all multi-indices~$p$ with~$0\leq|p|\leq \cs$, the function~$\partial_2^pK(\cdot,x)\alpha$,
with fixed $x\in\Omega$ and~$\alpha\in\Rd$,
is in~$H$ and~\eqref{rep_diff} holds, and that the function~$\partial_1^p\partial_2^pK$
exists
and is continuous in each of its two variables (its local boundedness 
is proven separately). 
If~$|p|=0$ then~$\partial^p_2K(\cdot,x)\alpha=K(\cdot,x)\alpha\in H$ 
and~\eqref{rep_diff}
is simply the reproducing 
property~\eqref{eq:repprop}; also, since~$H\subset C^\cs(\Omega,\Rd)$ we have that
$\partial^p_1\partial^p_2K(\cdot,x)=K(\cdot,x)$ is continuous in 
its first variable: the continuity in the second one
derives from the symmetry~$K(x,y)=K(y,x)^T$,~$x,y\in\Omega$.
\par
We now fix~$|p|$ and assume that for any multi-index~$r$ 
with~$0\leq|r|\leq|p|$ we have that~$\partial_1^r\partial_2^rK$ exists and
it is continuous in each of its two variables;
also, $\partial^r_2K(\cdot,x)\alpha\in H$ and 
$\langle\partial^r_2K(\cdot,x)\alpha,u\rangle_H=\alpha\cdot\partial^r u(x)$, for 
all~$u\in H$. For fixed~$|p|$ take~$p=(p_1,\ldots,p_m)$ and~$q=p+e_\ell
=(p_1,\ldots,p_\ell+1,\ldots,p_m)$ 
for some arbitrary index $\ell\in\{1,\ldots,m\}$. 
For any~$x\in\Omega$ and~$\alpha\in\Rd$ 
we have that~$K(\cdot,x)\alpha\in H\subset
C^\cs(\Omega,\Rd)$; by the symmetry $K(x,y)=K(y,x)^T$
we also have that~$K(x,\cdot)\alpha$ is in~$C^\cs(\Omega,\Rd)$.
Whence if we take a sequence~$\{\varepsilon_n\}$ 
in~$\mathbb{R}$ with~$|\varepsilon_n|\rightarrow 0$,  the sequence
$\Delta_n(\cdot):=\big(\partial_2^pK(\cdot,x+\varepsilon_n e_\ell)\alpha-
\partial_2^pK(\cdot,x)\alpha\big)/\varepsilon_n\in H$ converges pointwise to
$\partial^q_2K(\cdot,x)\alpha$. For all~$u\in H$
we have
$\langle \Delta_n, u\rangle_H=\alpha\cdot\big(\partial^p 
u(x+\varepsilon_ne_\ell)-\partial^pu(x)\big)/\varepsilon_n$,
which converges to~$\alpha\cdot \partial^q u(x)$ because~$H\subset C^\cs(\Omega,\Rd)$. 
Therefore,
for a fixed~$u\in H$, we have $\sup_{n}|\langle \Delta_n,u\rangle_H|<\infty$,
%which implies that
so by the Uniform Boundedness Principle~\cite{folland} 
the sequence~$\|\Delta_n\|_H$ is bounded. Together with the pointwise 
convergence of~$\Delta_n$, by Corollary~\ref{weak_cor}
this implies  
that:
\begin{packed_enum_c}
\item
$\partial^q_2K(\cdot,x)\alpha\in H$. Since $H\subset C^\cs(\Omega,\Rd)$,
 $\partial_1^q\partial_2^qK(\cdot,x)$ exists and it is continuous
in its first variable; by the symmetry~$K(x,y)=K(y,x)^T$ it is also continuous in the second 
variable.
\item
for all~$u\in H$ 
the sequence~$\langle \Delta_n,u\rangle_H$ converges to 
$\langle \partial^q_2K(\cdot,x)\alpha, u \rangle_H$. But we saw above that
it also converges to~$\alpha\cdot\partial^q u(x)$, whence 
$\langle \partial^q_2K(\cdot,x)\alpha, u \rangle_H=\alpha\cdot\partial^q u(x)$. 
\end{packed_enum_c}
\par
By the arbitrariness of~$\ell\in\{1,
\ldots,m\}$, this concludes the induction argument.
To prove local boundedness of~$\partial_1^p\partial_2^pK$ for an arbitrary multi-index~$p$
with~$0\leq|p|\leq \cs$, fix a compact subset~$D\subset\Omega$. For an arbitrary $\alpha\in\Rd$
consider the maps
$\{\Lambda_x:H\rightarrow \mathbb{R}:u\mapsto\langle
\partial^p_2K(\cdot,x)\alpha,u\rangle_H\,|\,x\in D\}\subset H^\ast$, parameterized by~$x\in  D$.
For a fixed~$u\in H$ we have
$\sup_{x\in D}|\Lambda_x u |
=
\sup_{x\in D}
|\langle \partial^p_2K(\cdot,x)\alpha,u\rangle_H|
=\sup_{x\in D}
|\alpha\cdot \partial^p u(x)|<\infty$, by the continuity of~$\partial^pu$. So
$\sup_{x\in D}\|\Lambda_x\|^2_{H^\ast}=\sup_{x\in D}
\alpha\cdot \partial_1^p\partial_2^pK(x,x)\alpha<\infty$, again
by the Uniform Boundedness Principle; that is, 
the map~$x\mapsto\alpha\cdot \partial_1^p\partial_2^p K(x,x)\alpha$ 
is locally bounded.
Therefore~$\partial_1^p\partial_2^pK$ is locally bounded, since
for all~$x,y\in\Omega$ and~$\alpha,\beta\in\Rd$ we have
%$$
\begin{align*}
\alpha\cdot \partial_1^p\partial_2^p K(x,y)\beta
&=
\big\langle \partial_2^pK(\cdot,x)\alpha,\partial_2^pK(\cdot,y)\beta\big\rangle_H
\leq
\big\|\partial_2^pK(\cdot,x)\alpha\big\|_H
\big\|\partial_2^pK(\cdot,y)\beta\big\|_H
\\
&=
\sqrt{\alpha\cdot \partial_1^p\partial_2^pK(x,x)\alpha}\,
\sqrt{\beta\cdot \partial_1^p\partial_2^pK(y,y)\beta}.
\qquad
%\mbox{This proves ii.}
%\quad%\mbox{for all }
%x,y\in\Omega,\,\alpha,\beta\in\Rd.
%\qedhere
\end{align*}
\par
$\bullet$ ({\sc b}$\,\Rightarrow\,${\sc a}\,\&\,{\sc c}).
Similarly, we will prove by induction on $|p|$
that for all multi-indices~$p$ with $0\leq|p|\leq \cs$
we have that~$H\subset C^{|p|}(\Omega,\Rd)$, that
the function~$\partial_2^pK(\cdot,x)\alpha$,
with fixed $x\in\Omega$ and~$\alpha\in\Rd$,
is in~$H$ and that~\eqref{rep_diff} holds
(the continuity of the inclusion,~$H\hookrightarrow C^\cs(\Omega,\Rd)$, 
will be proven separately).
If~$|p|=0$ then~$\partial^p_2K(\cdot,x)\alpha=K(\cdot,x)\alpha\in H$ 
and~\eqref{rep_diff}
is simply the reproducing 
property~\eqref{eq:repprop}. To prove the continuity of the functions in~$H$,
fix~$x_0\in\Omega\subset \mathbb{R}^m$
and let~$\{x_n\}$ be a sequence in~$\Omega$ with~$x_n\rightarrow x_0$.
Let~$\alpha\in\Rd$ be arbitrary and define~$v_n:=K(\cdot,x_n)\alpha\in H$.
It is the case that (i)~$\|v_n\|^2_H
%=\|K(\cdot,x_n)\alpha\|^2_H
=\alpha\cdot K(x_n,x_n)\alpha$ is bounded, by the local boundedness of~$K$,
and (ii)~$v_n(y)=K(y,x_n)\alpha\rightarrow K(y,x_0)\alpha$, for all~$y\in\Omega$,
by the continuity of~$K$ in each of its variables. Since~$v_n$
is bounded and it converges pointwise, by Corollary~\ref{weak_cor}
it converges weakly to the same function: %. In other words, 
for any~$u\in H$ we have 
%$\langle K_{x_n}^\alpha,v\rangle_H\rightarrow\langle K_{x_0}^\alpha,v\rangle_H$,
$\langle K(\cdot,x_N)\alpha,u\rangle_H\rightarrow\langle K(\cdot,x_0)\alpha,u\rangle_H$,
that is~$\alpha\cdot u(x_n)\rightarrow \alpha\cdot u(x_0)$.
Therefore~$H\subset C(\Omega,\Rd)$ by the arbitrariness of~$u$.
\par
We now fix~$|p|$ and 
assume that~$H\subset C^{|p|}(\Omega,\Rd)$
and that for any multi-index~$r$ with~\mbox{$0\leq|r|\leq|p|$}
we have~$\partial^r_2K(\cdot,x)\alpha\in H$
and~$\langle\partial^rK(\cdot,x)\alpha,u\rangle_H=\alpha\cdot\partial^r u(x)$, for all 
$u\in H$.
So, for fixed~$|p|$, assume that $p=(p_1,\ldots,p_m)$
and~$q=p+e_\ell=(p_1,\ldots,p_\ell+1,\ldots,p_m)$, for some~$\ell\in\{1,\ldots,m\}$.
Take an arbitrary~$u\in H$
and a sequence~$\{\varepsilon_n\}$
in~$\mathbb{R}$, with~$|\varepsilon_n|\rightarrow 0$. By~\eqref{rep_diff}, for 
all~$x\in\Omega$ and~$\alpha\in\Rd$ we have 
\begin{align}
\label{aux_diff}
\psi_n:=
%\nonumber
%\alpha\cdot \partial^qu_0(x)
%& = 
%\lim_{n}
%\lim_{n\rightarrow\infty}
\alpha
\cdot
%\frac{1}{\varepsilon_n}
%\Big(
\frac{\partial^p u(x+\varepsilon_ne_\ell)
-
\partial^pu(x)}{\varepsilon_n}
%\Big)
%\\
%&
=
%\lim_{n\rightarrow\infty}
%\lim_{n}
\Big\langle
%\frac{1}{\varepsilon_n}
%\big(
\frac{\partial^p_2 K(\cdot,x+\varepsilon_ne_\ell)\alpha
-
\partial^p_2 K(\cdot,x)\alpha}{\varepsilon_n}
%\big)
,u
\Big\rangle_{\!\!H}.
\end{align}
For fixed~$x\in\Omega$  and $\alpha\in\Rd$
the sequence of functions~$\Delta_n(\cdot):=
%\frac{1}{\varepsilon_n}
\big(
\partial^p_2 K(\cdot,x+\varepsilon_ne_\ell)\alpha
-
\partial^p_2 K(\cdot,x)\alpha
\big)/\varepsilon_n$
converges pointwise to~$\partial_2^qK(\cdot,x)\alpha$.
%;
%to show weak convergence in~$H$, by Corollary~\ref{weak_cor}
%it suffices to show that the sequence~is bounded. 
%In fact 
Once again by~\eqref{rep_diff}, we have that
$$
\|\Delta_n\|_H^2
=
\frac{1}{\varepsilon_n^2}
\Big\{
\alpha\cdot
\partial_1^p
\partial_2^p
K(x+\varepsilon_ne_\ell,x+\varepsilon_ne_\ell)
\alpha
-2
\alpha\cdot
\partial_1^p
\partial_2^p
K(x,x+\varepsilon_ne_\ell)
\alpha
+
\alpha\cdot
\partial_1^p
\partial_2^p
K(x,x)
\alpha
\Big\},
$$
which converges to~$\alpha\cdot
\partial_1^q
\partial_2^q
K(x,x)
\alpha$ as~$n\rightarrow\infty$, with~$q$ as above. Therefore~$\Delta_n$ is 
bounded in~$H$. By Corollary~\ref{weak_cor} it
converges weakly in~$H$ to~$\partial_2^qK(\cdot,x)\alpha$, 
which must be an element of~$H$.
Whence $\psi_n=\langle\Delta_n,u\rangle_H$ converges,
i.e.~the derivative $\partial^qu(x)$ exists 
and~$\alpha\cdot \partial^qu(x)=\langle \partial_2^qK(\cdot,x)\alpha,u \rangle_H$.
\par 
To prove the continuity of~$\partial^qu$, we proceed as in the case~$|p|=0$. 
Fix~$x_0\in\Omega\subset \mathbb{R}^m$
and let~$\{x_n\}$ be a sequence in~$\Omega$ with~$x_n\rightarrow x_0$.
Let~$\alpha\in\Rd$ be arbitrary and define~$w_n:=\partial_2^qK(\cdot,x_n)\alpha\in H$.
It is the case that (i) $\|w_n\|^2_H
%=\|\partial_2^q K(\cdot,x_n)\alpha\|^2_H
=\alpha\cdot \partial^q_1\partial^q_2K(x_n,x_n)\alpha$ is a bounded sequence, by the local boundedness of~$K$,
and (ii)~$w_n(y)=\partial^q_2K(y,x_n)\alpha\rightarrow \partial^q_1K(y,x_0)
\alpha$, for all~$y\in\Omega$,
by the continuity of~$\partial^q_2K$ in each of its variables.
%i.e.,~$w_n$ converges pointwise to~$\partial^q_2K(\cdot,x_0)\alpha$. 
%Since~$\{u_n\}$
%is bounded and it converges pointwise, 
By Corollary~\ref{weak_cor}, $w_n$ converges weakly to $\partial_2^qK(\cdot,x)\alpha$, 
%, $u_n\rightharpoonup \partial^q_2K(\cdot,x_0)\alpha $ in~$H$;
%in other words, for any
i.e.~for all~$u\in H$ we have 
%$\langle K_{x_n}^\alpha,v\rangle_H\rightarrow\langle K_{x_0}^\alpha,v\rangle_H$,
$\langle  \partial^q_2 K(\cdot,x_N)\alpha,
u\rangle_H\rightarrow\langle  \partial^q_2 K(\cdot,x_0)\alpha,u\rangle_H$,
or~$\alpha\cdot  \partial^q u(x_n)\rightarrow \alpha\cdot  \partial^q u(x_0)$;
whence~$\partial^qu\in C(\Omega,\Rd)$.
By the arbitrariness of~$u$ and the index~$\ell$ we have~$H\subset C^{|p|+1}(\Omega,\Rd)$.
This concludes the induction. 
\par
Consider now the map~$\iota: H\rightarrow C^\cs(\Omega,\Rd):v\mapsto v$,
let~$\{v_n\}$ be a sequence in~$H$ and assume that~$(v_n,v_n)\rightarrow (u,v)$
in~$H\times C^\cs(\Omega,\Rd)$. 
Since~$v_n \rightarrow u$ in~$H$ it converges
to it pointwise (by Corollary~\ref{weak_cor}) and
since~$v_n \rightarrow v$ in~$C^\cs(\Omega,\Rd)$
it also converges to it pointwise:
whence~$u=v$.
Therefore the graph of~$\iota$ is a closed subspace 
of~$H\times C^\cs(\Omega,\Rd)$, and by the Closed Graph Theorem~\cite{folland}
the map~$\iota$ 
is bounded. We conclude that
the inclusion is continuous, i.e.~$H\hookrightarrow C^\cs(\Omega,\Rd)$.
\par
$\bullet$ We have thus proven that ({\sc a}\,$\Leftrightarrow\,${\sc b}). We now 
introduce the further assumption 
that~$K_x^\alpha\in C_0^\cs(\Omega,\Rd)$. 
Let~$\{u_n\}$ be a sequence in~$H_0$ that converges, in the~$H$ norm, to an arbitrary~$u\in H$; therefore~$u_n\rightharpoonup u$ in~$H$. 
For any multi-index~$p$ such that~$0\leq |p|\leq \cs$ we have
$\langle \partial_2^pK(\cdot,x)\alpha,u_n\rangle_H
\rightarrow
\langle \partial_2^pK(\cdot,x)\alpha,u\rangle_H$
i.e.~$\alpha\cdot\partial^pu_n(x)\rightarrow 
\alpha\cdot\partial^pu(x)$. Therefore~$u_n\rightarrow u$
in the topology of~$C^\cs(\Omega,\Rd)$; but~$C_0^\cs(\Omega,\Rd)$
is closed in~$C^\cs(\Omega,\Rd)$, therefore $u\in C_0^\cs(\Omega,\Rd)$.
In conclusion, $H\hookrightarrow C_0^\cs(\Omega,\Rd)$.
\end{proof}
%\clearpage
\begin{example} 
An example of non-degenerate, $\cs$-admissible RKHS that is used in 
applications~\cite{Younes10} is the Sobolev space of 
vector fields~$H:=W^{\ce,2}(\Rd,\Rd)
=H^{\ce}(\Rd,\Rd)$ with the inner product:
\begin{equation}
\label{prod_bess}
\langle u, v\rangle_H
:= \int_{\Rd} 
%\big\langle 
Lu(x)\cdot v(x)
%\big\rangle_{\Rd}
\, dx
%=\|u\|_L^2 
= \int_{\Rd} 
\sum_{j=0}^\ce 
\binom{\ce}{j} \sigma^{2j} \sum_{|\dindex|=j} 
\partial^\dindex u %^\ell 
\cdot
\partial^\dindex v%^\ell 
%\big|^2 
\ dx.
\end{equation}
In~\eqref{prod_bess} the differential operator~$L:=(1-\sigma^2\Delta)^\ce$
(where~$\sigma>0$ is a scaling 
factor, $\ce\geq 0$ is an integer, and~$\Delta$ is the Laplace operator)
is applied to each component of the vector field~$u=(u^1,\ldots,u^d)$;
if $\ce>\cs+\dd/2$ 
then~$H \hookrightarrow
C_0^\cs(\Rd,\Rd)$
by the Sobolev Embedding Theorem~\cite{folland},
i.e.~$H$ is $\cs$-admissible. 
Regarding its reproducing 
kernel, first note that for any~$u\in H$ and $x,\alpha\in\Rd$ we must have
$$
\alpha\cdot u(x)
=
\langle K_x^\alpha,u\rangle_H
=
\int_\Rd
\!\!
K_x^\alpha(y)\cdot Lu(y)
\,
dy
=
\int_\Rd
\big(K(y,x)\alpha\big)^T Lu(y)
\,
dy
=
\alpha\cdot\int_\Rd
K(x,y)\, Lu(y)
\,
dy,
$$ 
therefore~$u(x)
=
\int_\Rd
K(x,y) \,Lu(y)
\,
dy$ for all~$u\in H$.
Since~$L$ is applied to each component of the vector field~$u$
(i.e.~the differential operator does not mix the components of~$u$)
it must be the case that $K(x,y)=G(x,y)\mathbb{I}_d$,
where~$G$ is the \em Green's function \em of~$L$
and $\mathbb{I}_d$ is the~$d\times d$ identity matrix. 
As reported in~\cite{EDM}, since
$L=(1-\sigma^2\Delta)^\ce$ we have that $K(x,y)\alpha=k(\|x-y\|_\Rd)\mathbb{I}_d$,
with
\begin{equation}
\label{k_bess}
k(r)\, =\, C(\sigma,d,\ce)\;
\Big(\frac{r}{\sigma}\Big)^{\ce-\frac{\dd}{2}}\,
K_{\ce-\frac{\dd}{2}}\!\Big(\frac{r}{\sigma}\Big), \qquad r\geq 0,
\end{equation}
where~$C(\sigma,d,\ce):=\big(2^{\ce+\frac{\dd}{2}-1}\pi^{\frac{\dd}{2}}\,\Gamma(\ce)
\,\sigma^\dd\big)^{-1}$
and~$K_\nu$, with~$\nu=\ce-d/2$, is a modified Bessel 
function~\cite{abramowitz} of order~$\nu$ (this should not be confused with the
symbol~$K$ that we use for kernels).
Kernels of the type~$K(x,y)=k(\|x-y\|)\mathbb{I}_d$
with~$k$ given by~\eqref{k_bess} are in fact referred to as \em Bessel kernels\em\/.  
\end{example}
We have just seen an example of a RKHS with a ``scalar'' kernel,
in that~$K$ is given by a scalar-valued function multiplied by 
the identity matrix.
In the next section we shall explore a large class of (non-scalar) 
reproducing kernels, study their properties, and provide some significant examples.
%%%
%%%
%\clearpage
\section{Translation- and rotation-invariant metrics in RKHS}
\label{secTRI}
%\section{Isometry by rigid transformations in RKHS }
In the current and following sections we shall restrict our attention
to RKHS whose kernel induces an inner product that is translation- and rotation-invariant;
we shall call these TRI kernels.
In this case it is natural to assume that~\mbox{$\Omega=\Rd$}, so that the
representer functions~$K_x^\alpha(\cdot)$ are in fact vector fields in~$\Rd$. 
From now on, unless otherwise explicitly specified, we shall assume that~$d\geq 2$.
\subsection{TRI kernels}
\label{gfik}
In~$\Omega=\Rd$ we indicate the generic translation with $\tau: x\mapsto x+t$, for
some fixed~$t\in\Rd$, and the generic rotation 
with~$\RRot:x\mapsto \rot x$, for some fixed~$\rot\in O(\Rd)$ 
(the orthogonal group). 
%A kernel~$K$ is \em translation invariant \em if
%for any translation~$\tau\in \Rd $
%the map
%$v\mapsto v\circ\tau$ is an isometry in~$H$.
%On the other hand, we call kernel \em rotation invariant \em if 
%for any $\rho\in O(\Rd )$ the map
%$v\mapsto \rho v\circ^{-1}$ is an isometry in~$H$.
%If a kernel is translation invariant \em and \em rotation invariant
%we shall simply call it {\it invariant\/}.
\begin{theorem}[translation invariance]
\label{trinv}
Let $H$ be a RKHS  
with kernel~$K:\Omega\times\Omega\rightarrow\Rdd$.
The map \mbox{$u\mapsto u\circ\tau$}
is an isometry in~$H$ for any translation~$\tau$
if and only if
there exists a function $\kk:\Rd \rightarrow \Rdd $ such that $K(x,y)=\kk(x-y)$ for all $x,y\in \Rd.$
%The following statements are equivalent:
%\vspace{-.15cm}
%\begin{packed_enum}
%\item for any translation~$\tau: x\mapsto x+t$, with $t\in\Rd$,
%the map
%$v\mapsto v\circ\tau$ is an isometry in~$H$;
%\item there exists a function $\kk:\Rd \rightarrow \Rdd $ such that $K(x,y)=\kk(x-y)$ for all $x,y\in \Rd .$  
%\end{packed_enum}
\end{theorem}
\begin{proof} %(i.$\,\Rightarrow\,$ii.) 
%We take the generic translation operator~$\tau:x\mapsto x+t$. 
Assume first that~$u\mapsto u\circ\tau$
is an isometry for all translations~$\tau: x\mapsto x+t$.
Let~$f:=K_x^\alpha$ and $g:=K_y^\beta$; their product 
is~$\langle f,g\rangle_H=\alpha\cdot K(x,y)\beta$, while
$\langle f\circ \tau,g\circ \tau\rangle_H
%=\langle K^{\alpha}_{x+t},K^{\beta}_{y+t}\rangle_H
=\langle K(\cdot,x+t)\alpha,K(\cdot,y+t)\beta\rangle_H
=
\alpha\cdot K(x+t,y+t)\beta$.
So it must be the case that
$K(x,y)=K(x+t,y+t)$ for all $x$, $y$ and~$t\in\Rd$;
in particular %(by choosing $t=-y$) 
$K(x,y)=K(x-y,0)=\kk(x-y)$ for all $x,y\in\Rd$,
where $\kk(z):=K(z,0)$, $z\in \Rd$.
%\par
%(ii.$\,\Rightarrow\,$i.) 
\par
Conversely
if $K(x,y)=\kk(x-y)$ for all~$x,y\in\Rd$ 
then it is the case that
$\langle K_{x+t}^\alpha,K_{y+t}^\beta\rangle_H=
\alpha\cdot K(x+t,y+t)\beta
=
\alpha\cdot \kk(x+t-y-t)\beta
=
\alpha\cdot \kk(x-y)\beta
=
\alpha\cdot K(x,y)\beta
=
\langle K_{x}^\alpha,K_{y}^\beta\rangle_H$ for all~$t\in\Omega$.
Letting $H_0=\mathrm{span}\{K_x^\alpha\,|\,x,\alpha\in\Rd\}$
we have proven that  
the linear map~$T:H_0\rightarrow H_0:u\mapsto u \circ \tau$, 
with~$\tau: x\mapsto x+t$,  is an
isometry in~$H_0$ for arbitrary translations~$\tau$. 
This can be extended to an isometry~$\overline{T}$ 
in~$H$ as follows. If
$u\in H$, let~$\{u_n\}$ be a Cauchy sequence such that 
$u_n\rightarrow u$ in~$H$ and therefore pointwise;
$T$~is an isometry whence~$\{Tu_n\}$
is also Cauchy in~$H_0$, therefore it converges to some element of~$H$
which we call~$\overline{T}u$. If~$\{u_n\},\{v_n\}$ are Cauchy in $H_0$
and respectively converge to $u,v\in H$ then
$\langle \overline{T}u,\overline{T}v\rangle_H
=
\lim_n\langle \overline{T}u_n,\overline{T}v_n\rangle_{H_0}
=
\lim_n\langle u_n,v_n\rangle_{H_0}=\langle u,v \rangle_H
$, so~$\overline{T}:H\rightarrow H$ is also an isometry. Finally,
$
\alpha\cdot \overline{T}u(x)
= \langle K_x^\alpha , \overline{T}u\rangle_H
= \lim_n \langle K_x^\alpha , Tu_n\rangle_{H_0}
= \lim_n \alpha\cdot Tu_n(x)
= \lim_n \alpha\cdot u_n(x+t)
= \alpha\cdot u(x+t),$ 
for all~$x,\alpha\in\Rd$,
therefore~$\overline{T}u=u\circ\tau$. This completes the proof.
\end{proof}
%\noindent{
We call a kernel with the properties described in 
Theorem~\ref{trinv} {\it translation invariant}\/; with 
a small abuse of terminology we shall also call~$\kk$ ``kernel''.
We introduce the following matrices in~$\Rdd$: 
\begin{equation}
\label{proj_op}
\mathrm{Pr}_x^\parallel:=
%\frac{xx^T}{\|x\|^2_{\mathbb{R}^d}}
\frac{xx^T}{\|x\|^2}
\qquad
\mbox{and}
\qquad
\mathrm{Pr}_x^\perp:=
%\mathbb{I}_\dd -\frac{xx^T}{\|x\|^2_{\mathbb{R}^d}},
\mathbb{I}_\dd -\frac{xx^T}{\|x\|^2},
\qquad x\in \Rd\setminus\{0\},
\end{equation}
which are, respectively, the projection operators onto vector~$x$
and onto the plane that is perpendicular to~$x$
(we indicate with $\mathbb{I}_\dd $ the $\dd\times \dd$ identity matrix).
In~\eqref{proj_op} we have indicated with~$\|x\|$ 
the Euclidean norm %~$\|x\|^2_{\mathbb{R}^d}$ 
of a point in $x\in\Rd$,
as we shall do, for simplicity, throughout the rest of the paper. 
The following lemma does not require~$\kk$ to be the kernel of 
a RKHS.
%\clearpage
\begin{lemma}
\label{lem_equiv}
Let~$\kk:\Rd\rightarrow\Rdd$ be a generic matrix-valued function. The
following are equivalent:
\begin{packed_enum}
\item for any point~$x\in\Rd$ and rotation~$\rot \in O(\Rd )$ 
it is the case 
that 
$\kk(-x)=\kk(x)^T$ and
\begin{equation}
\label{auxrot}
%\boxed{\,
\kk(\rot x)=\rot\,\kk(x)\rot^{-1};
%}
%\qquad
%\mbox{ for all }z\in\Rd, \;\rot\in O(\Rd).
\end{equation}
%\hspace*{.4cm}{\bf iii.} 
\item there exist a scalar~$k_0\in\mathbb{R}$ such that $\kk(0)=k_0\mathbb{I}_d$ and 
two scalar functions $\hpar,\hper:\mathbb{R}^+\rightarrow \mathbb{R}$ 
such that %$K(x,y)=k(x-y)$ 
for all $x\in \Rd $, $x\not=0$, one can write
\begin{equation}
\label{invker}
\kk(x)
=
\aaa(\|x\|)
\,
\mathrm{Pr}_x^\parallel
%\,\alpha^\parallel(x)
+
\bbb(\|x\|)
\,
\mathrm{Pr}_x^\perp.
%\alpha^\perp(x),
\end{equation}
\end{packed_enum}
\end{lemma}
%%%
\begin{proof}
(i.$\,\Rightarrow\,$ii.)
By choosing~$\rot=-\idmatrix_\dd $ in~\eqref{auxrot}
we have~$\kk(x)=\kk(-x)$ for all~$x\in\Rd$; combining the latter with
$\kk(-x)=\kk(x)^T$ yields the symmetry~$\kk(x)=\kk(x)^T$.
So~$\kk(x)$ can be diagonalized, i.e.~there exist matrices~$U(x)\in O(\Rd)$
and~$\Sigma(x)=\mathrm{diag}(\sigma_1(x),\ldots,\sigma_\dd(x))$
such that~$\kk(x) U(x)=U(x)\Sigma(x)$.
%In particular~$\kk(0) U(0)=U(0)\Sigma(0)$, 
%and 
Since~$\kk(0)\rot=\rot\,\kk(0)$
for all~$\rot\in O(\Rd)$ 
by~\eqref{auxrot},  
we have 
%$U(0)\kk(0)=U(0)\Sigma(0)$, whence~
$\kk(0)=\Sigma(0)$ (diagonal matrix). 
%Let~$\{e_i\}$ be the standard basis in $\Rd$. 
If~$\rot\in O(\Rd)$ is such that~$\rot e_i=e_j$ (for fixed indices~$i\not=j$) 
then
$\sigma_j(0) e_j=\kk(0)e_j=\kk(0)\rot e_i=\rot\,\kk(0)e_i=\rot\sigma_i(0)e_i=\sigma_i(0)e_j$,
i.e.~$\sigma_j(0)=\sigma_i(0)$ (all eigenvalues are the same).
So %all eigenvalues are the same, i.e.~
$\kk(0)=k_0\mathbb{I}_d$,
for some~$k_0\in \mathbb{R}$.
\par
Now fix~$x\in\Rd$, $x\not=0$, and~$\rot\in O(\Rd)$ such that
$\rot x=x$ and 
$\dim\,\{v\in\Rd\,|\,\rot v=v\}=1$. 
%$\dim\mathrm{Eig}(1)=1$ (where $\mathrm{Eig}(1)=\{v\in\Rd\,|\,\rot v=v\}$). 
%is the eigenspace
%relative to eigenvalue~1).
%and $\rot\not=\mathbb{I}_d$ (so~$\rot$ is a non-trivial rotation around~$x$).
By~\eqref{auxrot}
we have $\kk(x)\rot=\rot\,\kk(x)$ for such choice of~$\rot$,
so~$\rot\,\kk(x)x =\kk(x)\rot x=\kk(x)x$, i.e.~both~$\kk(x)x$ and~$x$
are eigenvectors of~$\rot$ with eigenvalue~1. 
%Since $\dim\mathrm{Eig}(1)=1$ 
Whence they
%must be parallel: 
$\kk(x)x=\lambda^\parallel(x)x$
for some scalar~$\lambda^\parallel(x)$.
Therefore~$\lambda^\parallel(x)$ is an eigenvalue of~$\kk(x)$; 
without loss of generality we 
assume that~$\sigma_1(x)=\lambda^\parallel(x)$.
We now claim that~$\sigma_i(x)=\sigma_j(x)$ for $i,j\geq 2$.
Denoting the orthonormal columns of~$U(x)$, i.e.~the eigenvectors
of~$\kk(x)$, with~$u_1(x),\ldots,u_\dd(x)$, for any pair of indices~$i,j\geq2$
with~$i\not=j$ there exists a matrix~$\rot\in O(\Rd)$
such that~$x = \rot x$ and~$u_i(x)=\rot u_j(x)$. By~\eqref{auxrot},
$\kk(\rot x)\rot\, u_j(x)=\rot\,\kk(\rot x)\, u_j(x)$,
whence $\kk(\rot x)\, u_i(x)=\rot\, \kk(\rot x)\, u_j(x)$,
i.e.~$\sigma_i(x)\,u_i(x)=\rot\,\sigma_j(x)\,u_j(x)=\sigma_j(x)\,u_i(x)$;
therefore \mbox{$\sigma_i(x)=\sigma_j(x)$}.
So 
there is a scalar~$\lambda^\perp(x)$
such that~$\kk(x)\beta=\lambda^\perp(x)\beta$ for all vectors~$\beta\perp x$.
\par
We claim that~$\lambda^\parallel$ and~$\lambda^\perp$ 
only depend  on~$\|x\|$. For all~$x\in\Rd$ and~\mbox{$\rot\in O(\Rd)$} we have, by~\eqref{auxrot}, 
that $\kk(\rot x)\rot x=\rot\,\kk(x)x$, whence %by the previous arguments 
$\lambda^\parallel(\rot x)\rot x=\rot\lambda^\parallel(x)x$, 
i.e.~$\lambda^\parallel(\rot x)=\lambda^\parallel(x)$
for all~$x$ and $\rot$, which implies that~$\lambda^\parallel(x)=\aaa(\|x\|)$
for some~$\aaa:\mathbb{R}^+\rightarrow \mathbb{R}$.
Similarly for any~$\alpha\in\Rd$ we have, again by~\eqref{auxrot}, that
$\kk(\rot x)\rot \alpha=\rot\,\kk(x)\alpha$; if~$\alpha\perp x$
then also $\rot\alpha\perp\rot x$, whence
$\lambda^\perp(\rot x)\rot \alpha=\rot\lambda^\perp(x)\alpha$
and in fact $\lambda^\parallel(\rot x)=\lambda^\parallel(x)$
for all~$x$ and~$\rot$.  This again implies that~$\lambda^\perp(x)=\bbb(\|x\|)$
for some~$\bbb:\mathbb{R}^+\rightarrow \mathbb{R}$.
\par
(ii.$\,\Rightarrow\,$i.) 
The implication is obvious in the case~$x=0$. 
When~$x\not=0$, for any~$\rot \in O(\Rd)$ it is the case that~$\mathrm{Pr}_{\rot x}^\parallel
=\rot\,\mathrm{Pr}_x^\parallel\,\rot^{-1}$
and~$\mathrm{Pr}_{\rot x}^\perp=\rot\,\mathrm{Pr}_x^\perp\rot^{-1}$, 
and we conclude immediately.
\end{proof}
%%%
\begin{definition}
\label{def_coe}
When a matrix-valued function~$\kk:\Rd\rightarrow\Rdd$ may be written in the
form~\eqref{invker}, we call~the 
functions~$\aaa,\bbb:\mathbb{R}^+\rightarrow \mathbb{R}$ 
the \em coefficients \em of~$\kk$.
\end{definition}
\begin{corollary} 
If the function~$\kk:\Rd\rightarrow\Rdd$ is such that~\eqref{invker} holds,
then for~$x\not=0$ the eigenvalues of~$\kk(x)$ 
are~$\aaa(\|x\|)$, with multiplicity~$1$ and eigenvector~$x$, and~$\bbb(\|x\|)$,
with multiplicity~$\dd-1$.
\end{corollary}
\begin{theorem}[rotation invariance]
\label{rotinv}
Let $H$ be a RKHS  
with a translation-invariant reproducing kernel, i.e.~$K(x,y)=\kk(x-y)$ for all $x,y\in\Rd$.
The map
$v\mapsto \RRot v\circ\RRot^{-1}$ is an isometry in~$H$
for any rotation~$\RRot:x\mapsto\rot x$, with $\rot \in O(\Rd )$,
if and only if~\eqref{auxrot} holds %$\kk(\rot x)=\rot\kk(x)\rot^{-1}$
for all~$\rot\in O(\Rd )$.
\end{theorem}
\begin{proof}
Assume that 
$v\mapsto \RRot v\circ\RRot^{-1}$ is
 an isometry in~$H$
for any rotation $\RRot:x\mapsto\rot x$, with $\rot\in O(\Rd )$.
For fixed~$x,y,\alpha,\beta\in\Rd$ and~$\rot\in O(\Rd)$,
let~$f:=\rot^{-1}K_{\rot x}^{\rot\alpha}\circ\rot$
and $g:=K_y^\beta\in H$;
then we have that
%\begin{align*}
$$
\langle f,g\rangle_H 
%&
=
\big\langle
\rot^{-1}K(\rot\;\cdot\;,\rot x)\rot\alpha,
K(\,\cdot\,,y)\beta
\big\rangle_H
= \beta^T\rot^{-1}
K(\rot y,\rot x)\rot\alpha,
%\end{align*}
$$
by the reproducing property of the \em second \em factor.
But for all~$\alpha,\beta\in\Rd$  this must be equal to
\begin{align*}
\langle \rot f\circ\rot^{-1},\rot g\circ\rot^{-1}\rangle_H 
&=
\big\langle
K(\,\cdot\,,\rot x)\rot\alpha
,
\rot K(\rot^{-1}\,\cdot\,,y)\beta
\big\rangle_H
\stackrel{(\ast)}{=} 
(\rot\alpha)^T
\rot
K(\rot^{-1}\rot x,y)\beta
\\&=
\alpha^TK(x,y)\beta
=
\beta^TK(x,y)^T\alpha
=
\beta^TK(y,x)\alpha
\end{align*}
in~$(\ast)$ we have used
the reproducing property
of the \em first \em factor.
Whence we have $\rot^{-1}K(\rot y,\rot x)\rot=K(y,x)$
for all~$x,y\in\Rd$ and~$\rot\in O(\Rd)$, and~\eqref{auxrot}
follows from translation invariance (Proposition~\ref{trinv}).
\par
%(ii.$\,\Rightarrow\,$i.)
Now assume that~\eqref{auxrot}
holds for all~$\rot\in O(\Rd)$.
Take 
take~$f:=K_x^\alpha$ and~$g:=K_y^\beta$,
for arbitrary points~$x,y$ and
and vectors~$\alpha,\beta$. 
For any $\rot\in O(\Rd)$ we have
$
%\big
(\rot f\circ\rot^{-1}
%\big
)(\cdot)
=
\rot\,\kk
%\big
(\rot^{-1}\cdot-x
%\big
)\alpha
=
\rot\,\kk
%\big
(\rot^{-1}(\cdot-\rot x)
%\big
)\alpha
=
\kk
%\big
(\cdot-\rot x
%\big
)\rot\alpha
$
by~\eqref{auxrot}, and a similar expression holds for~$\rot g\circ\rot^{-1}$; therefore
\begin{align*}
\big\langle
\rot 
f\circ\rot^{-1}
&,
\rot g
\circ
\rot^{-1}
\big\rangle_H
=
\big\langle
\kk
(\cdot-\rot x
)\rot\alpha
,
\kk
(\cdot-\rot y
)\rot\beta
\big\rangle_H
=
\beta^T\rot^T\kk(\rot(y-x))\rot\alpha
%\stackrel{\mbox{\small by \eqref{auxrot}}}{=}
\\
%(\mbox{by \eqref{auxrot}})
&
\stackrel{(\ast\ast)}{=}
%=
\beta^T\rot^T\rot\,\kk(y-x)\alpha
=
\beta^T\kk(y-x)\alpha
=
\big\langle
\kk
(\,\cdot\,- x
)\alpha
,
\kk
(\,\cdot\,- y
)\beta
\big\rangle_H
=\big\langle f,g\big\rangle_H,
\end{align*}
where we have used~\eqref{auxrot} 
in~$(\ast\ast)$. Whence rotations are isometries 
$H_0=\mathrm{span}\{K_x^\alpha\,|\,x,\alpha\in\Rd\}$.
An argument that is in all similar to the last part of the proof
of Theorem~\ref{trinv} can be employed to prove that 
rotations $\RRot: x\mapsto\rot x$, with~$\rot\in O(\Rd)$, 
are in fact isometries on all of~$H=\overline{H}_0$.
\end{proof}
%\clearpage
\par
Note that by ``rotations'' we intend all the elements of~$O(\Rd)$
and not just the special orthogonal group~$SO(\Rd)$, 
therefore we also include, for example, 
all permutations of the coordinates and reflections.
By Definition~\ref{def_ker}, translation-invariant \em kernels\em~$K(x,y)=\kk(x-y)$ 
have the property that~$\kk(-x)=\kk(x)^T$ for all~$x\in\Rd$, 
whence they may be written in the form~\eqref{invker}.
\begin{definition}[TRI kernels]
\label{def_tri}
The kernels %~$\kk$ 
of 
%Reproducing Kernel Hilbert Spaces
RKHS 
with translation- and 
rotation-invariant inner products, that may therefore be written in the form~\eqref{invker},
are called \em TRI kernels\em\/.
\end{definition}
%\begin{definition}
%We call the kernels~$\kk$ of~RKHS with a translation-  
%and rotation-invariant inner product %, i.e.~of the form~\eqref{invker},
%\em TRI kernels\em.
%, and the corresponding 
%functions~$\aaa,\bbb:\mathbb{R}^+\rightarrow \mathbb{R}$ 
%the \em coefficients \em of the kernel.
%\end{definition}
\par
Not all matrix-valued functions~$\kk$ of the form~\eqref{invker}
are positive definite, i.e.~kernels. In the next section we will characterize
the functions~$(\aaa,\bbb)$ that give rise to TRI kernels, and later in the paper
we shall provide a method to construct such coefficients.
Introducing the auxiliary function
\begin{equation}
\label{def_ktilde}
\ccc(r):=\frac{\hpar(r)-\hper(r)
}{r^2},
\qquad
r>0,
\end{equation}
we may write the generic element of a %matrix-valued 
function~$\kk=(\kk^{ij})_{i,j=1,\ldots,d}$
of the type~\eqref{invker} as follows, for~$x\not=0$:
%a TRI kernel~$\kk(x)$, for $x\not=0$, as
\begin{equation}
\label{kelement}
%\textstyle
\kk^{ij}(x)
=
\big[
\hpar(\|x\|)-\hper(\|x\|)
\big]
\frac{x^ix^j}{\|x\|^2}+\hper(\|x\|)\,\delta^{ij}
=
\ccc(\|x\|)\,
x^ix^j+\hper(\|x\|)\,\delta^{ij},
\quad
i,j=1,\ldots,\dd.
%\qquad
%x\in \Rd,
\end{equation}
%\par
If~$\kk(0)=k_0\mathbb{I}_d $  
we may define~$\aaa$ and~$\bbb$ at zero by setting
$\aaa(0)=\bbb(0)=k_0$; in fact when~$\kk$ is 
continuous we have that~$\lim_{r\rightarrow0^+}\aaa(r)=
\lim_{r\rightarrow0^+}\bbb(r)=k_0$, thus this choice is justified.
%we can extend the coefficients~$\aaa$ and~$\bbb$
%by continuity to the origin  and set~$\aaa(0)=\bbb(0)=k_0$.
\par
The result that follows states
a simple property of TRI kernels, which is in fact an immediate
generalization of a well known property of scalar-valued positive definite  functions. 
\begin{proposition}
\label{kzero}
Let~$H$ be a RKHS with a TRI kernel~$\kk$.
Its coefficients~$\hpar$, $\hper$ and the number~$k_0$,
introduced in Lemma~\ref{lem_equiv}, 
%and their common value at zero $k_0=\hpar(0)=\hper(0)$ 
have the properties:
$k_0\geq 0$, $k_0\geq |\aaa(r)|$ and $k_0\geq |\bbb(r)|$ for all~$r>0$.
If~$H$ is non-degenerate then such inequalities are strict.
\end{proposition}
\begin{proof}
By Proposition~\ref{Kposdef}
when expression~\eqref{invker} 
holds it is the case that  
for all~$N\in \mathbb{N}$, 
% ~$M$-tuples of
points
$x_1,\ldots, x_N\in \Rd $ 
and 
vectors
$\alpha_1,\ldots, \alpha_N\in \Rd $
we must have
$$
%\sum_{a,b=1}^N\alpha_a^T\kk(x_a-x_b)\alpha_b
%= %\substack{a,b=1 \\ a<b}
%2\sum_{\substack{a,b=1 \\ a<b}}^N
2
\!\!\!\!
\sum_{1\leq a<b\leq N}
\!\!\!\!\!
\big\{
\hpar(\|x_a-x_b\|)
\alpha_a^T\mathrm{Pr}^\parallel(x_a-x_b)\alpha_b
+\hper(\|x_a-x_b\|)
\alpha_a^T\mathrm{Pr}^\perp(x_a-x_b)\alpha_b\big\}
+
k_0
\!\sum_{a=1}^{N}\|\alpha_a\|^2\geq0.
$$
For now we shall not assume the non-degeneracy of~$H$.
By choosing~$N=1$ we have~$k_0\|\alpha_1\|^2\geq0$
for all~$\alpha_1\in\Rd$; therefore
$k_0\geq0$. We now fix~$N=2$. 
If we take $\alpha_1=\alpha_2\perp (x_1-x_2)$ 
then the above expression yields
$2[\bbb(\|x_1-x_2\|)+k_0]\|\alpha_1\|^2\geq0$ 
for all~$x_1\not=x_2$ and~$\alpha_1$, 
so that~$k_0\geq-\hper(r)$ for all~$r>0$;
similarly, if we take 
 $\alpha_1=-\alpha_2\perp (x_1-x_2)$ then we get
$2[-\hper(\|x_1-x_2\|)+k_0]\|\alpha_1\|^2\geq0$
for all~$x_1\not=x_2$ and~$\alpha_1$, 
which implies~$k_0\geq\hper(r)$ for all~$r>0$.
Combining the two results yields $k_0\geq|\bbb(r)|$. 
To prove that~$k_0\geq|\aaa(r)|$
one follows an analogous argument, by choosing 
$\alpha_1=\alpha_2$ parallel to $(x_1-x_2)$ first,
and~$\alpha_1=-\alpha_2$ parallel to $(x_1-x_2)$ later.
It is immediate to see that when~$H$ is non-degenerate
the inequalities~($\geq$) become strict~($>$).
\end{proof}
\begin{remark}
So far in this section we have not made any assumption of regularity.
For now we shall limit ourselves to observing that
if~$H$ is a RKHS with a translation-invariant kernel~$\kk$
and we also have~$H\hookrightarrow  C^\cs(\Omega,\Rd)$,
then by
Theorem~\ref{Th_diff} it is the case that~$\kk\in C^{2s}(\Omega,\Rdd)$;
that is, the kernel is ``twice as smooth'' as the functions of the space
that~$H$ is embedded into.
% (we use~$s$ instead of~$k$
%in order 
%not to
%confuse it 
%with the symbols that we use for the coefficients of the TRI kernel). 
Also, since
$\partial_2^qK(\cdot,x)\alpha=(-1)^{|q|}
\partial^q\kk(\cdot-x)\alpha$
and $\partial_1^p\partial_2^qK(\cdot,x)\alpha=(-1)^{|q|}
\partial^{p+q}\kk(\cdot-x)\alpha$, property~\eqref{rep_diff} implies that 
\begin{equation}
\label{diff_ti}
\alpha\cdot\partial^{p+q}\kk(x-y)\beta
=(-1)^{|p|}\big\langle\partial^p
\kk(\cdot-x)\alpha,\partial^q\kk(\cdot-y)\beta
\big\rangle_H,
\end{equation}
for all~$x,y,\alpha,\beta\in\Rd$ and multi-indices~$p$ and~$q$ 
with~$0\leq|p|\leq \cs$ and~$0\leq|q|\leq \cs$.
\end{remark}
\subsection{Characterization of TRI kernels}
We shall now find conditions on~$\aaa$
and~$\bbb$ for a function~$\kk$ for the type~\eqref{invker} 
to be positive definite,
whence a kernel.
The following auxiliary 
result does not require~$\kk$ to be of the type~\eqref{invker}.
\begin{theorem}[Bochner]
\label{bochner2}
Consider a matrix-valued function~$\kk\in L^1(\Rd,\Rdd)$ whose 
Fourier transform~$\widehat{\kk}$ is also in~$L^1(\Rd,\Rdd)$.
Then~$K(x,y):=\kk(x-y)$ is positive definite 
%kernel
%(in the sense of Proposition~\ref{Kposdef})
if and only if~$\widehat{\kk}(\xi)$ is a self-adjoint positive
definite matrix for all~$\xi\in\Rd$.
\end{theorem}
The classical version of the above theorem 
holds for~$d=1$ and states that a function is
positive definite if and only if its Fourier transform is 
nonnegative~\cite{bochner33}. The proof of Theorem~\ref{bochner2}
is reported in Appendix~\ref{appBochner}, and the result can be extended to~$L^2$
functions with the usual density arguments~\cite{stein_weiss}.
\par
In the particular case $\aaa=\bbb=:k$ we have 
the class of kernels that are simply 
referred to as~\em scalar\em\/:
$\kk(x)=k(\|x\|)\mathbb{I}_{\dd}$, $x\in\Rd$. 
In such case positive definiteness (Definition~\ref{def_posdef}) 
is obviously 
equivalent to the following: 
for 
abitrary~$N\in \mathbb{N}$, 
% ~$M$-tuples of
$x_1,\ldots, x_N\in \Rd $ 
and 
$r_1,\ldots, r_N\in \mathbb{R} $,
we have
\begin{equation}
%\label{eq:Kpos}
\sum_{a,b=1}^N
r_ar_b\,k(\|x_a-x_b\|)
\geq 0;
\end{equation}
%with equality holding if and only if~$r_1=\ldots=r_N=0$. 
whence
the scalar-valued 
function~$k(\|\cdot\|)$ must be positive definite; 
to apply Bochner's theorem we must compute 
its Fourier transform. We shall employ the next proposition
%~\cite[\S%
%IV.3]{stein_weiss} 
(see Appendix~\ref{AppA} for a simple proof).
\begin{theorem}
\label{FT}
If $f\in L^1(\Rd)$ is a radial function, i.e.~$f(x)=g(\|x\|)$
for some %function 
$g:\mathbb{R}^+\rightarrow \mathbb{R}$,
then so is
its Fourier transform~$\widehat{f}$.
It is in fact that case that~$\widehat{f}(\xi)=G(\|\xi\|)
$
with
%\begin{eqnarray*}
$$
\displaystyle
G(\varrho)
%&
%=
%&
%\frac{2\pi}{\varrho^{\frac{\dd}{2}-1}}
%\,
%\mathcal{H}_{\frac{\dd}{2}-1}\!
%\big[\,r^{\frac{\dd}{2}-1}g(r)\,\big]\!(2\pi\varrho)
%\\
%&
=
%&
\frac{2\pi}{\varrho^\mu}
\int_0^\infty
r^{\mu+1}g(r)
\,
J_{\mu}(2\pi\varrho r )\,dr\,,
\qquad\varrho>0,
$$
where~$\displaystyle\mu:=\frac{\dd}{2}-1$.
%\end{eqnarray*}
%defined for $\varrho\geq0$.
\end{theorem}
We remind the reader that the 
\em Hankel transform\em~\cite{bracewell,papoulis:2,sneddon}
%\footnote{%The definition above holds for all
%functions~$f\in L^1(\mathbb{R}^+)$ as long as~$\nu\geq-1/2$;
%however its definition can be extended to~$L^2$
%spaces with density arguments.
%The Hankel transform is 
%a particular case of a larger class, called {\it Bessel transforms\/}.
%An extended treatment of the Hankel transform is found in
%Chapter~2 of~\cite{sneddon}; it is otherwise  
%treated in a more elementary 
%fashion in~\cite{bracewell} or~\cite{papoulis:2}.
%References~\cite{ditkin},
%\cite{erdelyi_tables_v2}
%and~\cite{oberhettinger} provide tables of Hankel transforms.} 
of order $\nu$
of a function $f:\mathbb{R}^+\rightarrow \mathbb{R}$
is defined as~$
\mathcal{H}_\nu[f](\varrho)
:=
\int_0^\infty
rf(r)
\,
J_{\nu}
(\varrho r)\,dr,
$
where~$J_\nu$ is the Bessel function of the first 
kind~\cite{abramowitz} of order~$\nu$
(references~\cite{ditkin,erdelyi_tables_v2,oberhettinger}
provide tables of Hankel transforms);
in Appendix~\ref{appBessel} 
we list relevant properties of~$J_\nu$
that we will be using throughout the rest of the paper.
%, 
Here we note that if the function~$f:\Rd\rightarrow \mathbb{R}$
is radial, i.e.~$f(x)=g(\|x\|)$ for some~$g:\mathbb{R}^+\rightarrow\mathbb{R}$,
then
$$
\int_\Rd
|f(x)|^p\,dx=\sigma(\mathbb{S}^{d-1})\int_0^\infty|g(r)|^p\,r^{d-1}dr,
$$
where~$\sigma(\mathbb{S}^n)$ is the surface area of the unit~$n$-sphere; 
therefore $f\in L^p(\Rd)$ if and only if %it is the case that 
\mbox{$g\in L^p(\mathbb{R}^+,r^{d-1})$}, 
i.e.~$g$ is $p^\mathrm{th}$-power integrable with respect to the measure~$r^{d-1}dr$.
\par
An immediate consequence of the above proposition and of Bochner's theorem 
is that a~scalar kernel $\kk(\cdot)=k(\|\cdot\|)\mathbb{I}_\dd\in L^1(\Rd,\Rdd)$ 
%with~$k(\|\cdot\|)\in L^1$
is positive definite if 
and only if the function
\begin{equation}
\label{cond_sc2}
h(\varrho):=\frac{2\pi}{\varrho^\mu}
\int_0^\infty r^{\mu+1} k(r)\,J_{\mu}(2\pi\varrho r)\, dr, %\geq0
\qquad\mbox{defined for }\varrho>0,
\end{equation}
is nonnegative, 
%i.e.~if and only if~$\mathcal{H}_\mu\!\big[r^\mu k(r)\big](2\pi\varrho)\geq0$ 
%for all~$\varrho>0$, 
where we have set~$\mu:=\frac{\dd}{2}-1$. 
Also, we will have that~$\kkh(\xi)=h(\|\xi\|)\mathbb{I}_d$, with $\xi\in\Rd$.
We should note that
if~$\kk$ is in~$L^1(\Rd,\Rdd)$ then the function~$h$ 
is (uniformly) continuous and vanishes at infinity,
by the well-known properties of Fourier transforms~\cite{folland,stein_weiss}.
Also, by applying the Hankel's 
integral formula~\eqref{inv_Htf2} in Appendix~\ref{appBessel} to~\eqref{cond_sc2},
it is immediate to verify that 
\begin{equation}
\label{invhank}
k(r)=\frac{2\pi}{r^\mu}
\int_0^\infty \varrho^{\mu+1} h(\varrho)\,J_{\mu}(2\pi\varrho r)\, dr, %\geq0
\qquad r>0,
\end{equation}
i.e.~the map~$L^1(\mathbb{R}^+,r^{d-1})
\rightarrow
C_0(\mathbb{R}^+):k\mapsto h$ given by formula~\eqref{cond_sc2} is in fact an {\it involution\/}
(it is equal to its inverse, when inversion makes sense);
here we have used the symbol~$C_0(\mathbb{R}^+)$ to indicate 
the set of real-valued continuous functions defined on~$\mathbb{R}^+$
that vanish at~$+\infty$.  

\begin{examples} 
At the end of Section~\ref{vv_rkhs} we already saw the example of
(scalar) Bessel kernels, i.e.~of the type~$\kk(x)=k(\|x\|)
\mathbb{I}_d$
with~$k$ given by~\eqref{k_bess} for the Sobolev space~$H^\ce(\Rd,\Rd)$; in this
case the function~\eqref{cond_sc2} 
is given by~$h(\varrho)=(1+4\sigma^2\pi^2\varrho^2)^{-\ce}$. This can be easily seen
by observing that since~$k(\|\cdot\|)$ is the Green's function 
of the differential operator~$L=(1-\sigma^2\Delta)^\ce$,
then its Fourier transform~$h(\|\cdot\|)$ must be given by~$1/\widehat{L}$;
here~$\widehat{L}(\xi)=(1+4\sigma^2\pi^2\|\xi\|^2)^\ce$, $\xi\in\Rd$, is the Fourier
transform of~$L$, intended as a distribution.
For fixed~$\sigma>0$, other popular examples are given by:
\begin{align}
\label{scGauss}
\mbox{Gaussian kernels:}& & 
k(r)&= \exp\!\Big(\!\!-\frac{1}{2}\frac{r^2}{\sigma^2}\Big),&
h(\varrho)&=\big(2\pi\sigma^2\big)^{\mu+1}\exp\!\big({-2\pi^2\sigma^2\varrho^2}\big);
\\
\label{scCauchy}
\mbox{Cauchy kernels:}& & 
k(r)&=\frac{1}{1+r^2/\sigma^2}, &
h(\varrho)&= 2\pi\sigma^2\Big(\frac{\sigma}{\varrho}\Big)^\mu K_\mu(2\pi\sigma\varrho),
\end{align}
where~$K_\mu$ is a modified Bessel 
function of order~$\mu$; in the above 
expressions~$\mu=\frac{\dd}{2}-1$. 
The functions~$h$ in cases~\eqref{scGauss}
and~\eqref{scCauchy} are computed using~\cite[\S1.5.9]{oberhettinger}
and~\cite[\S1.4.13]{oberhettinger}.
Note that in all cases we have~$h\geq0$.
We refer the reader to~\cite{Younes10} for further 
examples and techniques to build scalar kernels.
\end{examples}
To generalize the above result %to non-scalar TRI kernels 
we first compute the Fourier transform of
matrix-valued functions~$\kk$ of the type~\eqref{invker}, that are not
necessarily kernels.
First note that for all~$\rot\in O(\Rd)$ we have
$$
\kkh(\rot\xi)
=
\int_{\Rd}\kk(x)\, e^{-2\pi ix\cdot \rot\xi}
\,
dx
=
\int_{\Rd}\kk(\rot y)\, e^{-2\pi iy\cdot \xi}
\,
dy
\stackrel{(\ast)}{=}
\int_{\Rd}\rot\,\kk(y)\rot^T\, e^{-2\pi iy\cdot \xi}
\,
dy=\rot\,\kkh(\xi)\rot^T;
$$
in~$(\ast)$ we have used the property~\eqref{auxrot}, which therefore also holds
for~$\kkh$. Whence we may write~$\kkh$ as
\begin{equation}
\label{khat_gen}
\widehat{\kk}(\xi) = \AAA
(\|\xi\|)\,\mathrm{Pr}_\xi^\parallel
+\BBB(\|\xi\|)\,\mathrm{Pr}_\xi^\perp,
\qquad
\xi\in\Rd\setminus\{0\},
\end{equation}
for some functions~$\AAA, \BBB:\mathbb{R}^+\rightarrow \mathbb{R}$,
by Lemma~\ref{lem_equiv}. 
When~$\kk\in L^1(\Rd,\Rdd)$
is of the type~\eqref{invker} we have that the coefficients 
$\aaa,\bbb$ are in $L^1(\mathbb{R}^+,r^{d-1})$
whereas, once again by the properties of Fourier transforms, 
 the functions~$\AAA$ and~$\BBB$ are continuous and vanish at infinity.
%that are provided in terms of~$\aaa,\bbb$ by the following result.
\begin{definition}
We call~the 
functions~$\AAA,\BBB:\mathbb{R}^+\rightarrow \mathbb{R}$ 
in~\eqref{khat_gen}
the \em coefficients \em of~$\kkh$.
\end{definition}
\par
The above definition 
is in line with~Definition~\ref{def_coe}; while this may be a bit redundant, 
we want to associate the symbols~$(\aaa,\bbb)$ with~$\kk$ 
and~$(\AAA,\BBB)$ with its Fourier transform~$\kkh$. 
The coefficients of~$\kkh$ are expressed in terms of those of~$\kk$
%as follows.
by the following result, which generalize
formula~\eqref{cond_sc2}. %, which again does not require~$\kk$ to be a kernel.
\begin{theorem}
\label{prop_FT}
Consider a function~$\kk\in L^1(\Rd,\Rdd)$ of the type~\eqref{invker}.
Its Fourier transform~$\kkh$ is a matrix-valued function
of the form~\eqref{khat_gen},
where~$\AAA,\BBB:\mathbb{R}^+\rightarrow \mathbb{R}$ are the scalar functions:
\begin{subequations}
\begin{align}
\label{Tpar}
\AAA(\varrho)
&
=
\frac{2\pi}{\varrho^\mu}
\int_0^\infty
r^{\mu+1}\,\aaa(r)J_\mu(2\pi\varrho r)\,dr
-
\frac{2\mu+1}{\varrho^{\mu+1}}
\int_0^\infty
r^{\mu}\big(\aaa(r)-\bbb(r)\big)J_{\mu+1}(2\pi\varrho r)\,dr,
\\
\label{Tper}
\BBB(\varrho)
&=
\frac{2\pi}{\varrho^\mu}
\int_0^\infty
r^{\mu+1}\,\bbb(r)J_\mu(2\pi\varrho r)\,dr
+
\frac{1}{\varrho^{\mu+1}}
\int_0^\infty
r^{\mu}\big(\aaa(r)-\bbb(r)\big)J_{\mu+1}(2\pi\varrho r)\,dr,
\end{align}
\end{subequations}
both defined for~$\varrho>0$, where~$\displaystyle
\mu:=\frac{\dd}{2}-1$.
\end{theorem}
\begin{proof} 
Define~$\ccc$ as in~\eqref{def_ktilde}.
%Let $\displaystyle
%\ccc(r):=\frac{\hpar(r)-\hper(r)
%}{r^2}$, $r>0$.
The Fourier transform of the generic element~\eqref{kelement} 
of~$\mathbf{ k}$ is 
\begin{align}
\nonumber
\widehat{\mathbf{ k}}^{j\ell}(\xi)
&
=
\int_{\Rd}
\mathbf{k}^{j\ell}(x)e^{-2\pi i x\cdot\xi}
dx
=
\int_{\Rd}
\ccc(\|x\|)\,x^jx^\ell
\,e^{-2\pi i x\cdot\xi}
dx
+
\delta^{j\ell}
\int_{\Rd}
\bbb(\|x\|)
\,e^{-2\pi i x\cdot\xi}
dx
\\
\label{HH}
&=
-\frac{1}{(2\pi)^2}
\frac{\partial^2}{\partial\xi^\ell\partial\xi^j}
\Fktilde(\xi)+\delta^{j\ell}\Fkperp(\xi),\qquad\xi\in\Rd,
\end{align}
where $\Fktilde,\Fkperp:\Rd\rightarrow \mathbb{C}$ are, respectively,
the Fourier transforms of~$\ccc(\|\cdot\|)$ and~$\bbb(\|\cdot\|)$. 
It is convenient at this point to introduce the 
auxiliary Bessel-type function~$\widetilde{J}_\nu(z):=z^{-\nu}J_\nu(z)$,
which has the property~$\frac{d}{dz}\widetilde{J}_\nu(z)=-z\widetilde{J}_{\nu+1}(z)$
(see~\cite[\S9.1.30]{abramowitz}). By Theorem~\ref{FT} we have 
\begin{align*}
\Fktilde(\xi) 
&=
\frac{2\pi}{\|\xi\|^\mu}
\int_0^\infty
r^{\mu+1}\ccc(r)J_\mu(2\pi r\|\xi\|)\,dr
=
(2\pi)^{\mu+1}
\int_0^\infty
r^{2\mu+1}\ccc(r)\widetilde{J}_\mu(2\pi r\|\xi\|)\,dr,
\quad\mbox{so that}
\\
\frac{\partial\Fktilde}{\partial\xi^j}(\xi)
&=
(2\pi)^{\mu+2}
\frac{\xi^j}{\|\xi\|}
\int_0^\infty
r^{2\mu+2}\ccc(r)\widetilde{J}'_\mu(2\pi r\|\xi\|)\,dr
=
-
(2\pi)^{\mu+3}\xi^j
\int_0^\infty
r^{2\mu+3}\ccc(r)\widetilde{J}_{\mu+1}(2\pi r\|\xi\|)\,dr.
\end{align*}
Therefore the second partial derivatives of~$\Fktilde$
are given by: 
\begin{align*}
\frac{\partial^2\Fktilde}{\partial\xi^\ell\partial\xi^j}(\xi)
&
=
-\delta^{j\ell}
(2\pi)^{\mu+3}
\!\!
\int_0^\infty
\!\!\!\!
r^{2\mu+3}\ccc(r)\widetilde{J}_{\mu+1}(2\pi r\|\xi\|)\,dr
-
(2\pi)^{\mu+4}
\frac{\xi^j\xi^\ell}{\|\xi\|}
\!
\int_0^\infty
\!\!\!\!
r^{2\mu+4}\ccc(r)\widetilde{J}'_{\mu+1}(2\pi r\|\xi\|)\,dr
\\
&=
-\delta^{j\ell}
(2\pi)^{\mu+3}
\!\!
\int_0^\infty
\!\!\!
r^{2\mu+3}\ccc(r)\widetilde{J}_{\mu+1}(2\pi r\|\xi\|)\,dr
+
(2\pi)^{\mu+5}
{\xi^j\xi^\ell}
\!
\int_0^\infty
\!\!\!
r^{2\mu+5}\ccc(r)\widetilde{J}_{\mu+2}(2\pi r\|\xi\|)\,dr,
\end{align*}
which, inserting~$\widetilde{J}_\nu(z)=z^{-\nu}J_\nu(z)$, may be rewritten as
$$
\frac{\partial^2\Fktilde}{\partial\xi^\ell\partial\xi^j}(\xi)
=
-\delta^{j\ell}
\frac{(2\pi)^2}{\|\xi\|^{\mu+1}}
\int_0^\infty
\!\!\!
r^{\mu+2}\,\ccc(r){J}_{\mu+1}(2\pi r\|\xi\|)\,dr
+\frac{(2\pi)^3}{\|\xi\|^\mu}
\frac{\xi^j\xi^\ell}{\|\xi\|^2}
\int_0^\infty
\!\!\!
r^{\mu+3}\,\ccc(r){J}_{\mu+2}(2\pi r\|\xi\|)\,dr
.
$$
\par
Computing~$\Fkperp$ via Proposition~\ref{FT} 
and inserting it into equation~\eqref{HH} finally yields:
\begin{equation}
\label{khatAUX}
\widehat{\mathbf{ k}}^{j\ell}(\xi)
=
C(\|\xi\|) \frac{\xi^j\xi^\ell}{\|\xi\|^2} 
+ 
D(\| \xi \|) \delta^{j\ell}
=
\Big(
C(\|\xi\|) 
+
D(\|\xi\|) 
\Big) \frac{\xi^j\xi^\ell}{\|\xi\|^2}
+ 
D(\| \xi \|) \Big(\delta^{j\ell}- \frac{\xi^j\xi^\ell}{\|\xi\|^2}\Big),
\end{equation}
with 
$\displaystyle C(\varrho):=
-\frac{2\pi}{\varrho^\mu}
\!\!
\int_0^\infty
\!\!\!
r^{\mu+3}\,\ccc(r){J}_{\mu+2}(2\pi r\varrho)\,dr
$,
and $D(\varrho):=\BBB(\varrho)$,
i.e.~expression~\eqref{Tper}. Finally, 
\begin{align*}
C(\varrho)+D(\varrho)
\stackrel{(\ast\ast)}{=}
&
-\frac{2\pi}{\varrho^\mu}
\int_0^\infty
r^{\mu+1}\big(\aaa(r)-\bbb(r)\big)\Big\{\frac{2(\mu+1)}{2\pi r\varrho}
J_{\mu+1}(2\pi r\varrho)
-J_{\mu}(2\pi r\varrho)
\Big\}\,dr
\\&
+
\frac{2\pi}{\varrho^\mu}
\int_0^\infty
r^{\mu+1}\bbb(r)J_\mu(2\pi\varrho r)\,dr
+
\frac{1}{\varrho^{\mu+1}}
\int_0^\infty
r^{\mu}\big(\aaa(r)-\bbb(r)\big)J_{\mu+1}(2\pi\varrho r)\,dr
\\
=
&\;
\frac{2\pi}{\varrho^\mu}
\!\!
\int_0^\infty
\!\!\!
r^{\mu+1}\,\aaa(r)J_\mu(2\pi\varrho r)\,dr
\!
-
\!
\frac{2\mu+1}{\varrho^{\mu+1}}
\!\!
\int_0^\infty
\!\!\!
r^{\mu}\big(\aaa(r)-\bbb(r)\big)J_{\mu+1}(2\pi\varrho r)\,dr=
\AAA(\varrho),
\end{align*}
where in~$(\ast\ast)$ we have used the property~\eqref{Jrec1},
reported in Appendix~\ref{appBessel},
of Bessel functions of the first kind.
Rewriting equation~\eqref{khatAUX} in 
matrix form using definitions~\eqref{proj_op}
completes the proof.
\end{proof}
\begin{corollary} Under the assumptions of Theorem~\ref{prop_FT}, 
by Bochner's theorem we have that~$\kk$ is positive definite,
i.e.~it is a kernel,
if and only if
$\AAA$ and~$\BBB$ are nonnegative functions.
\end{corollary}
As we said above, formulae~\eqref{Tpar} and~\eqref{Tper} generalize~\eqref{cond_sc2}:
in fact the three coincide when~$\aaa=\bbb$, i.e.~for scalar kernels, in which case 
we also have~$\AAA=\BBB$.
By the form of equation~\eqref{khat_gen},
we have 
that, for a fixed~$\xi\in\Rd\setminus\{0\}$, the numbers~$\AAA(\|\xi\|)$ and~$\BBB(\|\xi\|)$ are the eigenvalues of 
the matrix~$\widehat{\kk}(\xi)$; the former has multiplicity~1 and eigenvector~$\xi$,
while the latter  has  
multiplicity equal to~$\dd-1$. Note that while the Fourier transform of
a kernel~$\kk$ may be written in the form~\eqref{khat_gen} it is not a ``kernel'',
in that it is not positive definite itself (but just the Fourier transform of
a matrix-valued positive definite function).
\par
%\par 
The two functions~$\AAA$ and $\BBB$ may be expressed in terms of 
Hankel transforms, as it is the case for
formula~\eqref{cond_sc2} in the scalar case. In fact in the examples that we shall work out  we will write:
%
%We first note
%that, in general, we ma rewrite~\eqref{Tpar} and~\eqref{Tper} in a simplified manner:
%$\AAA(\varrho)=\frac{2\pi}{\varrho^\mu}\AAAm(2\pi\varrho)$ and 
%$\BBB(\varrho)=\frac{2\pi}{\varrho^\mu}\BBBm(2\pi\varrho)$ with:
\begin{equation}
\label{Halt}
%\textstyle
\AAA(\varrho) =\frac{2\pi}{\varrho^\mu}
\,
\AAAm(2\pi\varrho),
\qquad
\mbox{and}
\qquad
\BBB(\varrho) =\frac{2\pi}{\varrho^\mu}
\,
\BBBm(2\pi\varrho),
\end{equation}
where:
\par\vspace{-.5cm}
\begin{subequations}
\begin{align}
\label{AAAf}
%\mbox{where}\qquad
\AAAm(\varrho)
&:=
\int_{0}^\infty 
r^{\mu+1}\aaa(r)J_\mu(\varrho r)\,dr
-
\frac{2\mu+1}{\varrho}
\int_0^\infty 
r^{\mu+2}\ccc(r) J_{\mu+1}(\varrho r)\,dr,
\\
%\mbox{and}\qquad
\label{BBBf}
%\BBB(\varrho) &=
%\frac{2\pi}{\varrho^\mu}\BBBm(2\pi\varrho),
%\mbox{ with: }
%\;
%&
\BBBm(\varrho)
&:=
\int_{0}^\infty 
r^{\mu+1}\bbb(r)J_\mu(\varrho r)\,dr
+
\frac{1}{\varrho}
\int_0^\infty 
r^{\mu+2}\ccc(r) J_{\mu+1}(\varrho r)\,dr.
\end{align}
\end{subequations}
%with $\mu=\frac{d}{2}-1$ as usual. 
%\begin{align*}
%\AAA(\varrho)
%&=
%\frac{2\pi}{\varrho^\mu}\mathcal{H}_{\mu}\!\big[r^\mu\hpar(r)\big](2\pi\varrho)
%-\frac{2\mu+1}{\varrho^{\mu+1}}
%\mathcal{H}_{\mu+1}\!\big[r^{\mu-1}\big(\hpar(r)-\hper(r)\big)\big](2\pi\varrho),
%\\
%\BBB(\varrho)
%&=
%\frac{2\pi}{\varrho^\mu}\mathcal{H}_{\mu}\!\big[r^\mu\hper(r)\big](2\pi\varrho)
%+\frac{1}{\varrho^{\mu+1}}
%\mathcal{H}_{\mu+1}\!\big[r^{\mu-1}\big(\hpar(r)-\hper(r)\big)\big](2\pi\varrho),
%\end{align*}
%for~$\varrho>0$, with $\mu=\frac{d}{2}-1$. 
\par
The following proposition,
proven in Appendix~\ref{appBochner}, ensures that for any 
given TRI kernel~$\kk$ we have that as long as either~$\AAA$
or~$\BBB$ are strictly positive somewhere (i.e.~we are not dealing with 
the trivial case~$\kk=0$) the corresponding RKHS
is in fact \em non-degenerate \em (see Definition~\ref{non_deg}).
%In any case, the direct proof above
%provides an immediate corollary. 
%\begin{corollary} Under the hypotheses of Proposition~\ref{invT}, $\ccc$ is transformed into~$\CCC$ as follows:
%$$
%\CCC(\varrho)=-\frac{2\pi}{\varrho^{\mu+2}}
%\int_0^\infty
%\varrho^{\mu+3}\ccc(r)J_{\mu+2}(2\pi r\varrho)\,d\varrho,
%$$
%with~$\mu=\frac{d}{2}-1$. The above map~$\widetilde{T}:\ccc\mapsto\CCC$ is also an %involution.
%\end{corollary} 
\begin{proposition}
\label{pos_h}
Let~$\kk\in L^1(\Rd,\Rdd)$ 
be a TRI kernel such that~$\kkh\in L^1(\Rd,\Rdd)$. 
Then~$\kk$ is \em strictly \em positive definite
if and only if there exists~$r_0>0$ such that 
either~$\AAA(r_0)>0$ or~$\BBB(r_0)>0$.
\end{proposition}
%\clearpage
%\par\vspace{.2cm}
\noindent{\bf A non-example: Gaussian~$\aaa$ and~$\bbb$.}
One may ask whether we can construct 
a positive definite translation- and rotation-invariant kernel~$\kk$
of the type~\eqref{invker} with~$\aaa(r)=\exp(-c_1r^2)$
and $\bbb(r)=\exp(-c_2r^2)$, where~$c_1>0$ and 
$c_2>0$. The answer is that a kernel of this type is
positive definite if an only if~$c_1=c_2$,
which makes it of the form~$\kk(x)=\exp(-c_1\|x\|)\mathbb{I}_\dd$,
$x\in\Rd$,
i.e.~\em scalar \em and of the type~\eqref{scGauss}.
%
%which makes it scalar, i.e.~of the form $\kk(x)=\exp(-c_1\|x\|)\mathbb{I}_\dd$,
%$x\in\Rd$.
To see this, we compute~$\AAA$,~$\BBB$ and impose their nonnegativity. 
\par
We shall use the following Hankel transforms \cite[\S1.5.9 \& \S1.5.10]{oberhettinger}:
\begin{align}
\label{HtfG}
\int_0^\infty
r^{\nu+1} e^{-c r^2} J_{\nu}(\varrho r)\,dr &
= 
\frac{\varrho^\nu}{(2c)^{\nu+1}}\,
\exp\!\Big(\!-\!\frac{\varrho^2}{4c}\Big),
&
\mbox{ for }&\Re\nu>-1\mbox{ and }\Re c>0,
%\mbox{ (by \cite[\S1.5.9]{oberhettinger})},
\\
%\nonumber
\label{HtfG2}
\int_0^\infty
r^{\nu-1} e^{-c r^2} J_{\nu}(\varrho r)\,dr &
= 
\frac{2^{\nu-1}}{\varrho^\nu}
\,
\gamma\Big(\nu,\frac{\varrho^2}{4c}\Big),
&
\mbox{ for }&\Re\nu>0\mbox{ and }\Re c>0,
%\mbox{ (by \cite[\S1.5.10]{oberhettinger})},
\end{align}
where~$\gamma(\nu,x):=\int_0^x e^{-t}\,t^{\nu-1}dt$, $\Re\nu>0$, is the 
\em 
lower \em incomplete gamma function%
%\em
~\cite[\S6.5.2]{abramowitz}.
Therefore when~$\aaa(r)=\exp(-c_1r^2)$
and~$\bbb(r)=\exp(-c_2r^2)$ 
the expressions~\eqref{AAAf} and~\eqref{BBBf} become
\begin{align*}
\AAAm(\varrho) &= 
\frac{\varrho^\mu}{(2c_1)^{\mu+1}}
\,\exp\!
\Big(\!-\!\frac{\varrho^2}{4c_1}\Big)
-\frac{2^\mu(2\mu+1)}{\varrho^{\mu+2}}
\Big\{
\gamma\Big(\mu+1,\frac{\varrho^2}{4c_1}\Big)
-
\gamma\Big(\mu+1,\frac{\varrho^2}{4c_2}\Big)
\Big\},
\\
\BBBm(\varrho) &= 
\frac{\varrho^\mu}{(2c_2)^{\mu+1}}
\,\exp\!
\Big(\!-\!\frac{\varrho^2}{4c_2}\Big)
+\frac{2^\mu}{\varrho^{\mu+2}}
\Big\{
\gamma\Big(\mu+1,\frac{\varrho^2}{4c_1}\Big)
-
\gamma\Big(\mu+1,\frac{\varrho^2}{4c_2}\Big)
\Big\}.
\end{align*}
Using formulae~\eqref{Halt} and introducing for later 
convenience the \em upper \em incomplete gamma function
$\Gamma(\nu,x):=\Gamma(\nu)-\gamma(\nu,x)=
\int_x^\infty e^{-t}t^{\nu-1}dt$, $\Re\nu>0$
(see \cite[\S6.5.3]{abramowitz}), we finally get:
\begin{subequations}
\begin{align}
\label{GaussA}
\AAA(\varrho) &= 
\frac{\pi^{\mu+1}}{c_1^{\mu+1}}
\,\exp\!\Big(\!-\!\frac{\pi^2\varrho^2}{c_1}\Big)
-
\frac{2\mu+1}{2\pi^{\mu+1}\varrho^{2\mu+2}}
\Big\{
\Gamma\Big(\mu+1,\frac{\pi^2\varrho^2}{c_2}\Big)
-
\Gamma\Big(\mu+1,\frac{\pi^2\varrho^2}{c_1}\Big)
\Big\},
\\
\label{GaussB}
\BBB(\varrho) &= 
\frac{\pi^{\mu+1}}{c_2^{\mu+1}}
\,\exp\!\Big(\!-\!\frac{\pi^2\varrho^2}{c_2}\Big)
+
\frac{1}{2\pi^{\mu+1}\varrho^{2\mu+2}}
\Big\{
\Gamma\Big(\mu+1,\frac{\pi^2\varrho^2}{c_2}\Big)
-
\Gamma\Big(\mu+1,\frac{\pi^2\varrho^2}{c_1}\Big)
\Big\}.
\end{align}
\end{subequations}
\par
We first consider the case~$\dd=2$ (i.e.~$\mu=0$) since it is particularly simple.
In fact~$\Gamma(1,x)=e^{-x}$ and elementary manipulations
lead to the following expressions for~$\AAA$ and~$\BBB$:
%\begin{align*}
$$
\AAA(\varrho)
=F(\varrho;c_1,c_2):=
\frac{1}{2\pi\varrho^2}
\Big[
\Big(\frac{2\pi^2}{c_1}\varrho^2+1
\Big)
%\,e^{-(\pi\varrho)^2/a_1}
\exp\!\Big(\!-\!\frac{\pi^2\varrho^2}{c_1}\Big)
-
\exp\!\Big(\!-\!\frac{\pi^2\varrho^2}{c_2}\Big)
%\,e^{-(\pi\varrho)^2/a_2}
\Big],
\quad
\BBB(\varrho)
=F(\varrho;c_2,c_1)
$$
%\end{align*}
i.e.~each is obtained from the other by exchanging~$c_1$ and~$c_2$.
The condition~$\AAA\geq0$
is equivalent to 
\begin{equation}
\label{auxineq}
\frac{2\pi^2}{c_1}\varrho^2+1\geq 
\exp\Big( 
\frac{c_2-c_1}{c_1c_2}
\pi^2\varrho^2
%\Big(\frac{1}{c_1}-\frac{1}{c_2}\Big)
\Big)
\qquad \mbox{for all }\varrho\geq0.
\end{equation}
This is certainly true if $c_1\geq c_2$ (in which case the graph of the function on the right-hand side of the above inequality is a Gaussian); in fact this is also 
necessary for~\eqref{auxineq} to hold, since if~$c_1<c_2$ 
then such condition
breaks down for large values of~$\varrho$. Similarly, $\BBB\geq0$ if and only if
$c_2\geq c_1$, so the functions~$\AAA$ and~$\BBB$ are both nonnegative if and only if~$c_1=c_2$.
\par
To prove that this is also true in higher dimensions ($d>2$, or~$\mu>0$) we 
rewrite~\eqref{GaussA} as follows:
$$
\AAA(\varrho) = 
\frac{\pi^{\mu+1}}{c_1^{\mu+1}}
\,\exp\!\Big(\!-\!\frac{\pi^2\varrho^2}{c_1}\Big)
-
\frac{2\mu+1}{2}
\pi^{\mu+1}
\int_{1/c_2}^{1/c_1}
\tau^\mu
%\,\exp(
e^{-\tau(\pi\varrho)^2}
d\tau,
$$
which is obviously nonnegative if~$c_1\geq c_2$. The latter condition is in fact 
also necessary for the nonnegativity of~$\AAA$. To see this, 
we 
use
the asymptotic expansion of 
the upper incomplete gamma function~\cite[\S6.5.31]{abramowitz}:
$$
\Gamma(\nu,x)
\sim
x^{\nu-1}e^{-x}
\Big[
1
+\frac{\nu-1}{x}
+\frac{(\nu-1)(\nu-2)}{x^2}
+\cdots
\Big],
\quad
\mbox{ as }
x\rightarrow\infty,
$$
so that for any~$c>0$ it is the case that 
$$
\frac{1}{2\pi^{\mu+1}\varrho^{2\mu+2}}\,
\Gamma\Big(\mu+1,\frac{\pi^2\varrho^2}{c}\Big)
\sim \frac{\pi^{\mu+1}}{2\varrho^2c^\mu}
\,\exp\!\Big(\!-\!\frac{\pi^2\varrho^2}{c}\Big)
\Big[
1+\mu\frac{c}{\pi^2\varrho^2}
+\mu(\mu-1)\frac{c^2}{\pi^4\varrho^4}
+\cdots\Big],
\;\;\;
\mbox{as }
\varrho\rightarrow\infty
$$
(incidentally, note that if the dimension~$d$ is even then~$\mu$ is integer and the asymptotic expansion
on the right-hand side has
a finite number of terms), so~\eqref{GaussA}
can be expanded at infinity as follows:
\begin{align*}
\AAA(\varrho)
\sim
(2\mu+1)
\frac{\pi^{\mu-1}}{2\varrho^2}\frac{1}{c_1^\mu}
\,\exp\!\Big(-&\frac{\pi^2\varrho^2}{c_1}\Big)
\bigg\{
\Big[
\frac{2}{2\mu+1} \frac{\pi^2\varrho^2}{c_1}
+1+\mu\frac{c_1}{\pi^2\varrho^2}+\cdots\Big]
\\
&-\frac{c_1^\mu}{c_2^\mu}
\,\exp\!\Big(\frac{c_2-c_1}{c_1c_2}\pi^2\varrho^2\Big)
\Big[1+\mu\frac{c_2}{\pi^2\varrho^2}+\cdots\Big]\bigg\},
\;\;
\mbox{as }
\varrho\rightarrow\infty.
\end{align*}
%\par\vspace{.5cm}
One can see that if~$c_1<c_2$ then the second exponential function
diverges faster than~$\varrho^2$, implying that~$\AAA$ is
negative for large values of~$\varrho$. Therefore~$\AAA\geq0$ if and only if~$c_1\geq c_2$.
Manipulating~\eqref{GaussB} in a similar manner proves that~$\BBB\geq0$
if and only if~$c_2\geq c_1$. In conclusion, a kernel~$\kk$ with Gaussian~$\aaa$ and~$\bbb$ is positive definite if and only if 
$c_1=c_2$, i.e.~if and only if it is {\em scalar\em}.
\begin{example_n} 
\label{exGauss}
We now try a different route and verify 
for which values of~$(a,b)\in\mathbb{R}^2$
the functions
\begin{equation}
\label{coe_df}
\aaa(r) = b\,e^{-cr^2}
\qquad
\mbox{and}
\qquad 
\bbb(r) = (b-ar^2)\,e^{-cr^2},\quad r>0,
\end{equation}
%\begin{align*}
%\label{goodEx}
%\ccc(r)&= 
%\frac{\aaa(r)-\bbb(r)}{r^2}=a\,e^{-cr^2}
%&
%&\mbox{and}
%&
%\aaa(r) &= b\,e^{-cr^2}
%&
%&\mbox{and}
%&
%\bbb(r) &= (b-ar^2)\,e^{-cr^2},\quad r>0,
%\end{align*}
with~$c>0$, define a \em positive definite\em~$\kk$ of the type~\eqref{invker},
i.e.~a TRI kernel.
In other words we assume that the function~\eqref{def_ktilde} is given by
$\ccc(r)=a\exp(-cr^2)$, so that both~$\ccc$ and~$\aaa$
are Gaussians with the \em same \em variance, but different values at~$r=0$
(on the other hand~$\aaa$ and~$\bbb$ must
have the same value at~0 by Proposition~\ref{kzero}).
%Using the notation introduced in formulae~\eqref{AAAf}
%and~\eqref{BBBf}, 
In this example we have that
\begin{align*}
\AAAm(\varrho)
& =
\frac{\varrho^\mu}{(2c)^{\mu+1}}
\Big[b-(2\mu+1)\frac{a}{2c}\Big]
\exp\!\Big(\!-\!\frac{\varrho^2}{4c}\Big),
\\
\BBBm(\varrho)
& =
\frac{\varrho^\mu}{(2c)^{\mu+1}}
\Big[b-(2\mu+1)\frac{a}{2c}
+
\frac{a}{4c^2}\varrho^2\Big]
\exp\!\Big(\!-\!\frac{\varrho^2}{4c}\Big).
\end{align*}
In order to obtain the above expressions we have used the Hankel transform:
$$
\int_0^\infty
r^{\mu+3} e^{-c r^2} J_{\mu}(\varrho r)\,dr 
= 
\frac{\varrho^\mu}{(2c)^{\mu+3}}
\big[
4c(\mu+1)-\varrho^2
\big]
\exp\!\Big(\!-\!\frac{\varrho^2}{4c}\Big)
,
\quad\mbox{ for }\Re\mu>-2\mbox{ and }\Re c>0,
$$
which is computed by inserting in the left-hand side the  
relation~$J_\mu(z)=2(\mu+1)J_{\mu+1}(z)/z-J_{\mu+2}(z)$ 
(see~\eqref{Jrec1})
%(see~\cite[\S9.1.27]{abramowitz}) 
and applying~\eqref{HtfG}
twice, with~$\nu=\mu+1$ and~$\nu=\mu+2$ respectively.
By~\eqref{Halt} we get: 
\begin{align*}
\AAA(\varrho)
& =
%\frac{\varrho^\mu}{(2c)^{\mu+1}}
\frac{\pi^{\mu+1}}{c^{\mu+1}}
\Big[b-(2\mu+1)\frac{a}{2c}\Big]
\exp\!\Big(\!-\!\frac{\pi^2\varrho^2}{c}\Big),
\\
\BBB(\varrho)
& =
\frac{\pi^{\mu+1}}{c^{\mu+1}}
\Big[b-(2\mu+1)\frac{a}{2c}
+
\frac{a\pi^2}{c^2}\varrho^2\Big]
\exp\!\Big(\!-\!\frac{\pi^2\varrho^2}{c}\Big).
\end{align*}
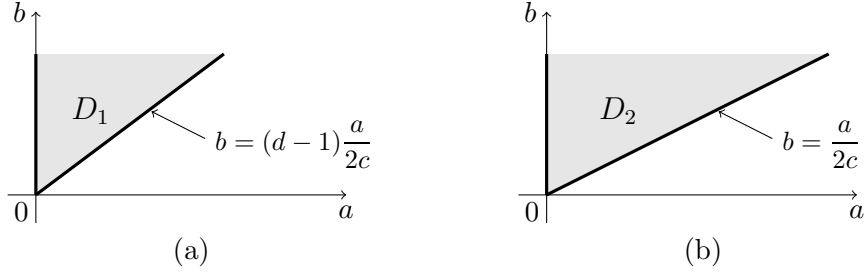
\begin{figure}[t]
\begin{center}
%\begin{picture}(100,65)
%\setlength{\unitlength}{1.2pt}
%\put(11,-12){\includegraphics[height=3.6cm]{Domain}}
%\thinlines
%\put(120,1){\makebox(0,0){\small $a$}} %
%\put(22,70){\makebox(0,0){\small $b$}} %
%\put(24,1){\makebox(0,0){\small $0$}} %
%\put(62,47){\makebox(0,0){\Large$D_1$}} %
%\put(140,40){\makebox(0,0){\scriptsize $\displaystyle b=(d-1)\frac{a}{2c}$}} %
%\end{picture}
\begin{tikzpicture}[scale=.75]
\fill [gray!20] (0,0) -- (3.333,2.5) -- (0,2.5);
\draw [->] (-.5,0) -- (5.5,0) node [below] {$a$};
\draw [->] (0,-.5) -- (0,3.25) node [left] {$b$};
\draw [->] (3,1) -- (2.05,1.48);
\draw (3,.9) node [right] {\small$\displaystyle b=(d-1)\frac{a}{2c}$};
\draw (0.05,-.3) node [left] {$0$};
\draw [very thick] (0,0) -- (3.333,2.5);
\draw [very thick] (0,0) -- (0,2.5);
\draw (.45,1.5) node [right] {\large$D_1$};
\draw (2.25,-1) node [right] {(a)};
\end{tikzpicture}
\hspace{1.5cm}
\begin{tikzpicture}[scale=.75]
\fill [gray!20] (0,0) -- (5,2.5) -- (0,2.5);
\draw [->] (-.5,0) -- (5.5,0) node [below] {$a$};
\draw [->] (0,-.5) -- (0,3.25) node [left] {$b$};
\draw [->] (4,1) -- (3.05,1.48);
\draw (4,.9) node [right] {\small$\displaystyle b=\frac{a}{2c}$};
\draw (0.05,-.3) node [left] {$0$};
\draw [very thick] (0,0) -- (5,2.5);
\draw [very thick] (0,0) -- (0,2.5);
\draw (.75,1.5) node [right] {\large$D_2$};
\draw (2.25,-1) node [right] {(b)};
\end{tikzpicture}
\end{center}
\vspace*{-.5cm}
\caption{Domains~$D_1$ and~$D_2$ of positive definiteness for the kernels in Examples~\ref{exGauss} and~\ref{exGauss2}.}
\label{wedge}
\end{figure}
\par
So~$\AAA\geq0$ if and only if 
%$2cb\geq (2\mu+1)a$,
$b- (2\mu+1)\frac{a}{2c}\geq0$,
and~$\BBB\geq0$
if and only if
%$\varrho^2\geq-2c[2cb-a(2\mu+1)]$
$\frac{a\pi^2}{c^2}\varrho^2\geq-\big[b-(2\mu+1)\frac{a}{2c}\big]$
for all~$\varrho$. Whence~$\AAA$ and~$\BBB$ are \em simultaneously \em
nonnegative, i.e.~$\kk$ 
is positive definite, if and only if the pair~$(a,b)$ is in 
$$
\textstyle
D_1
:=\big\{
(a,b)\in \mathbb{R}^2\,\big|\,
b\geq
(d-1)\frac{a}{2c}
,
a\geq0
\big\}
$$
(since~$2\mu+1=d-1$)
which is the wedge-shaped domain shown in Figure~\ref{wedge}(a). 
Note that the slope of the slanted boundary depends on the dimension~$d$.
The vertical
boundary ($a=0$) corresponds to scalar kernels, while the meaning of the other one 
shall be explored in later sections. 
\par
The graphs~$\aaa$, $\bbb$, $\AAA$ and~$\BBB$ (symmetrized with respect to~$r=0$ 
for clarity) are shown
in Figure~\ref{FigGauss}
for~$d=2$, $a=1.5$ and~$b=c=1$. Note that~$\AAA\geq0$ and~$\BBB\geq0$. The same figure shows
the corresponding vector field~$x\mapsto \kk(x)\alpha$, $x\in\mathbb{R}^2$, 
with~$\kk$ given by~\eqref{invker}
and~$\alpha=e_1\in\mathbb{R}^2$, also shown in the figure
(the importance of vector fields of the type~$x\mapsto \kk(x)\alpha$
is apparent from equation~\eqref{eqone}, on which 
we shall return later in Section~\ref{secApp}).
For~$\alpha=e_1$ the two ``vortices'' of the 2-dimensional 
vector field~$x\mapsto \kk(x)\alpha$
are located at points~$(0,\pm\sqrt{b/a})$; whence for positive definite~$\kk$
their mutual distance is at least~$\sqrt{2/c}$.
\begin{figure}[t]
\begin{center}
\begin{picture}(370,170)
\setlength{\unitlength}{1pt}
\put(3,80){\includegraphics[height=2.25cm]{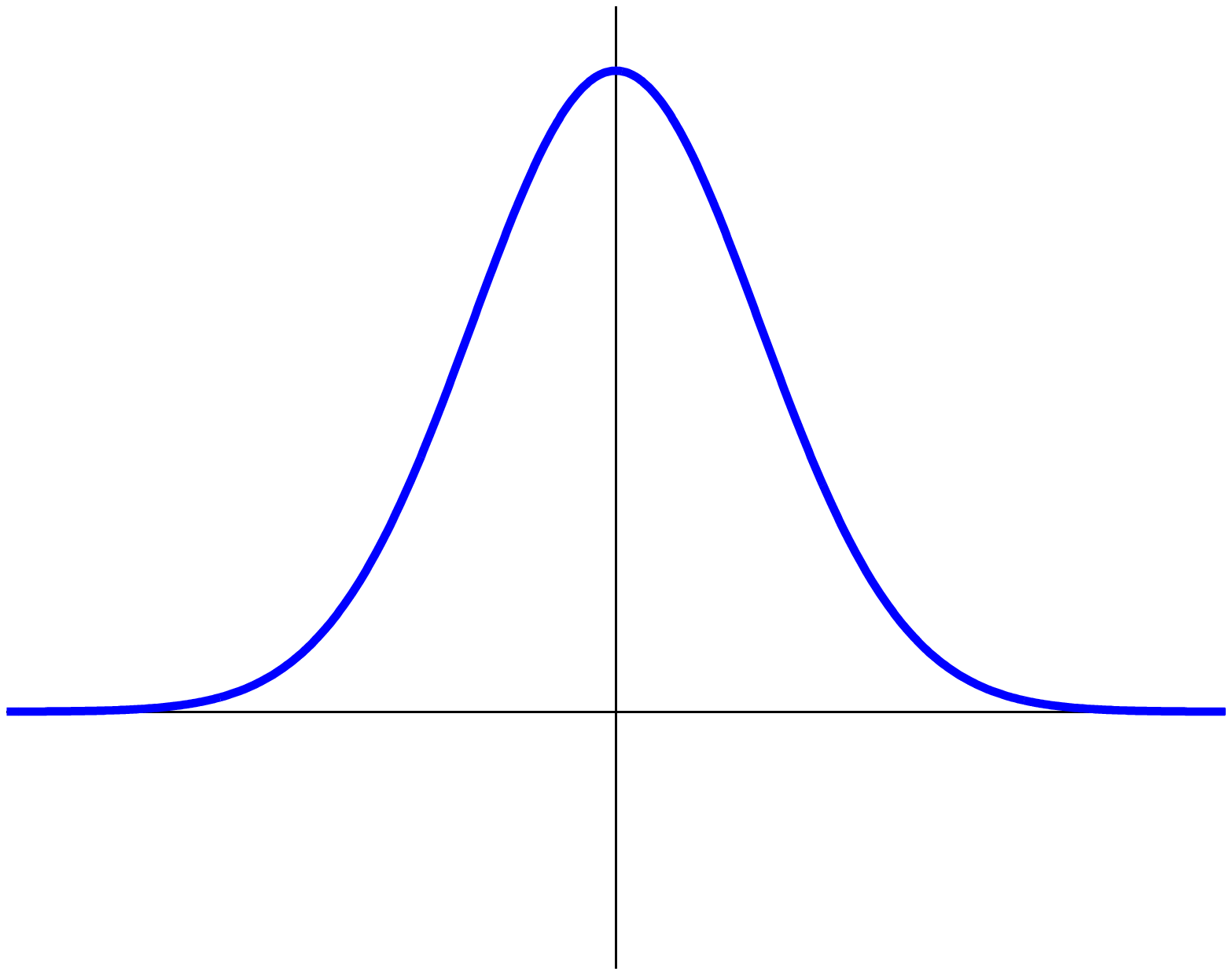}} 
\put(96,80){\includegraphics[height=2.25cm]{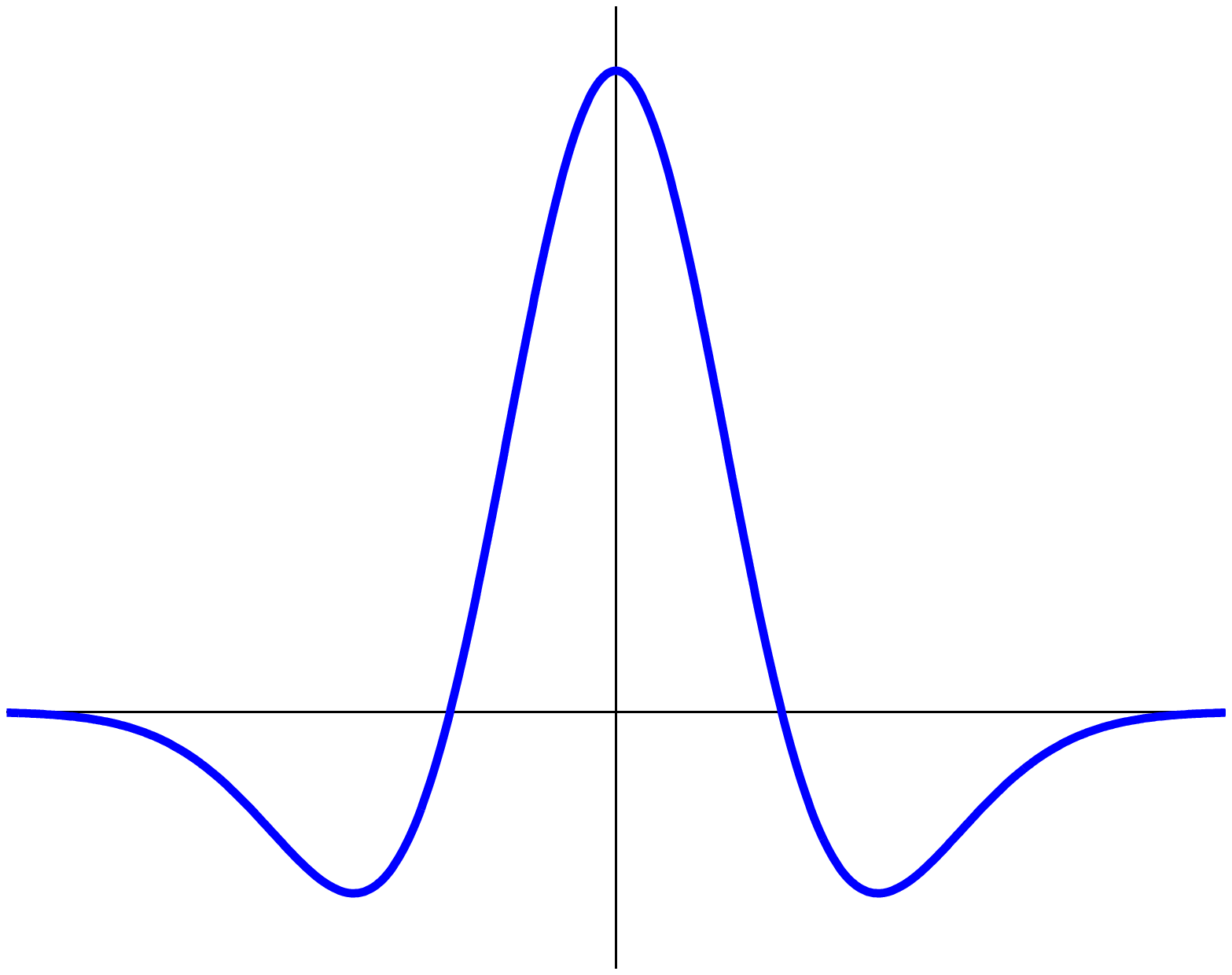}}
\put(3,-10){\includegraphics[height=2.25cm]{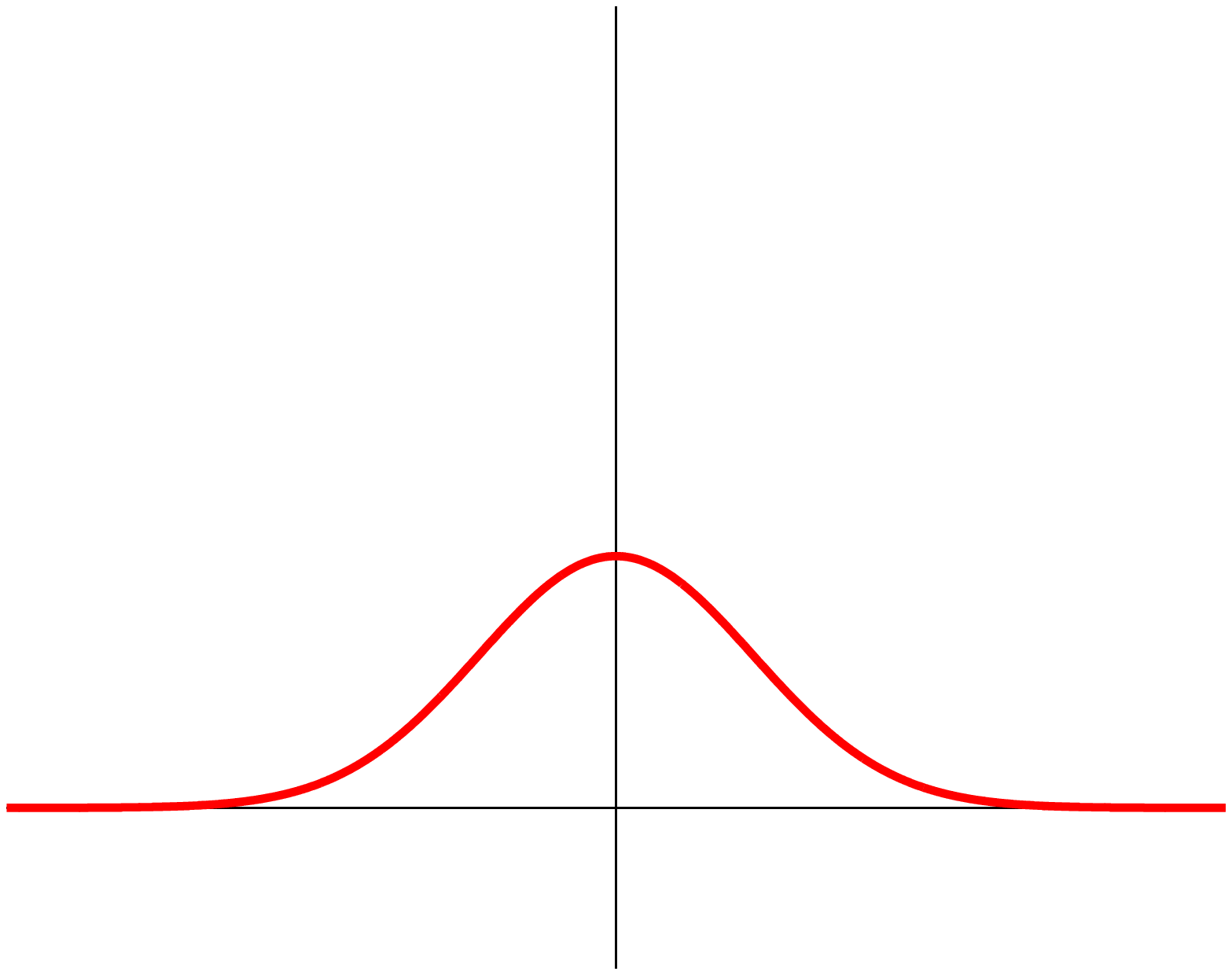}} 
\put(96,-10){\includegraphics[height=2.25cm]{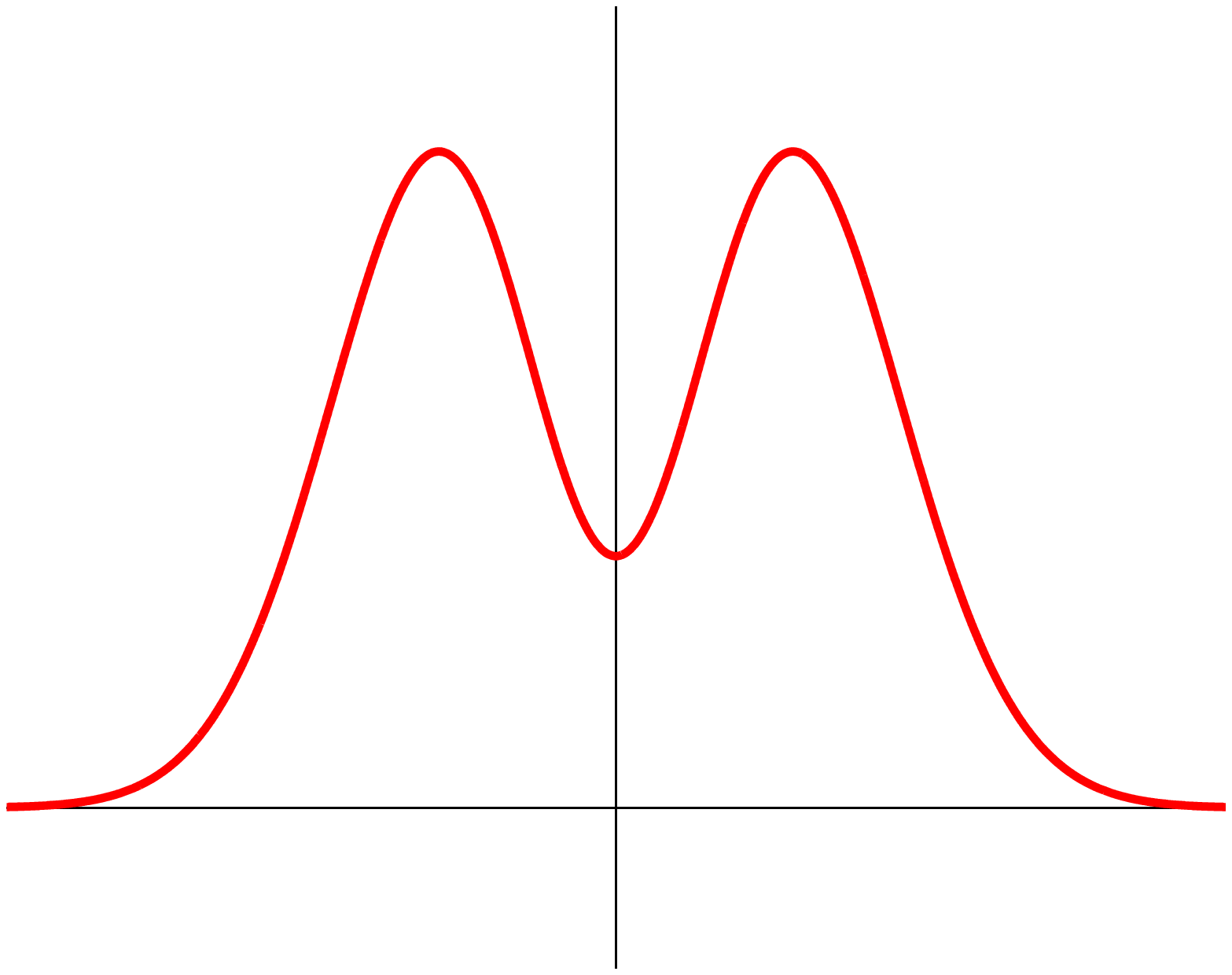}}
\put(46,152){\makebox(0,0){$\aaa$}} \put(140,152){\makebox(0,0){$\bbb$}}
\put(46,62){\makebox(0,0){$\AAA$}} \put(140,62){\makebox(0,0){$\BBB$}}
\put(215,-5){\includegraphics[height=5.5cm]{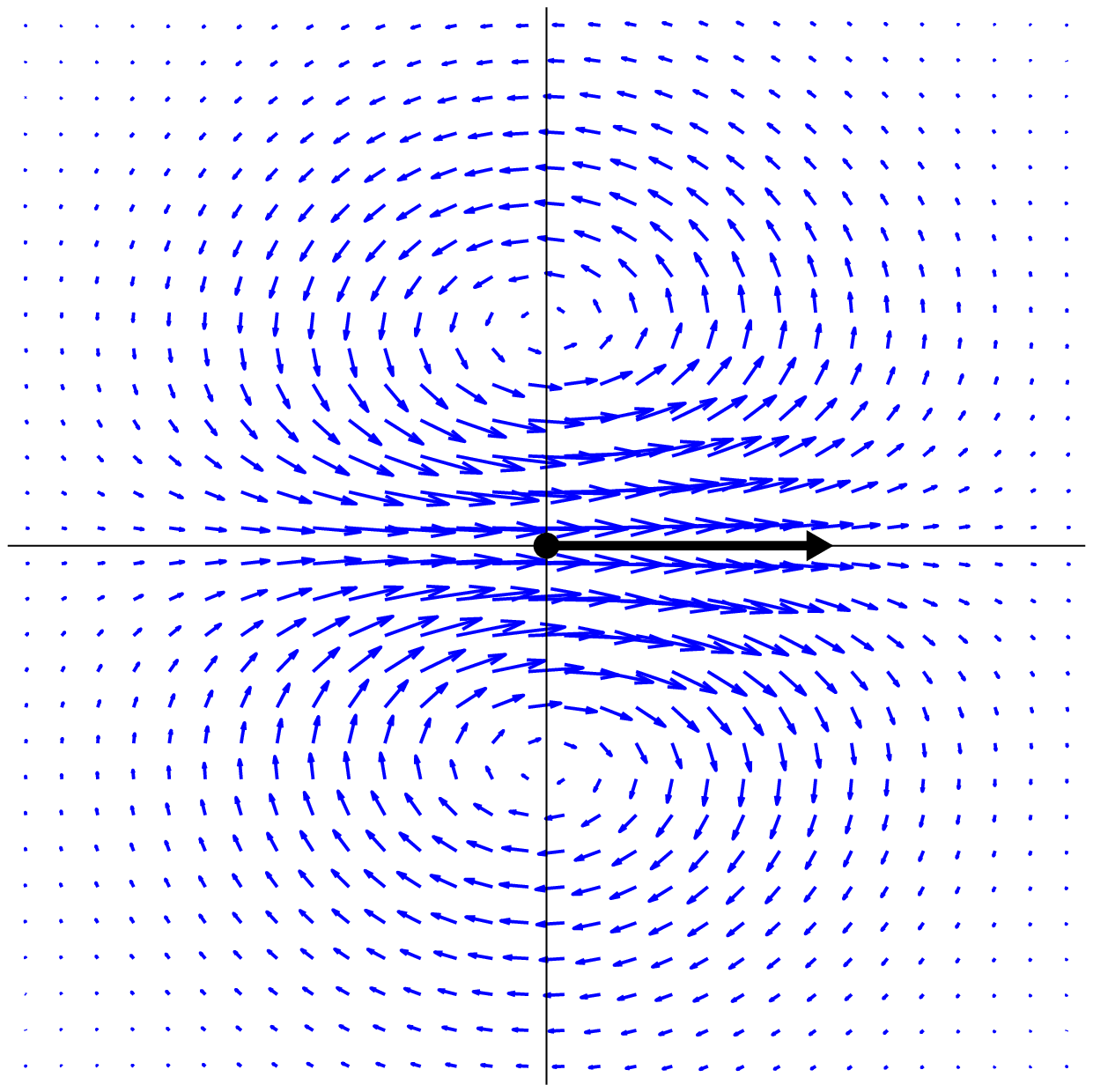}}
\put(295,157){\makebox(0,0){\small$x\mapsto\kk(x)e_1$}}
%\put(295,157){\makebox(0,0){\small$x\mapsto\kk(x)\alpha$, with~$\alpha=(1,0)$}}
%\put(239,122){\makebox(0,0){\Large$\Longrightarrow$}}
%\put(245,137){\makebox(0,0){\Huge$\Searrow$}}
%\put(0,160){\makebox(0,0){$\bullet$}} \put(370,160){\makebox(0,0){$\bullet$}}
%\put(0,0){\makebox(0,0){$\bullet$}} \put(370,0){\makebox(0,0){$\bullet$}}
\thicklines
%\put(220,140){\vector(1,0){40}} %
\end{picture}
\end{center}
%\caption{Graphs of~$\aaa$, $\bbb$, $\AAA$, $\BBB$ for Example~\ref{exGauss} 
%with~$a=1.5$, 
%$b=c=1$ in $d=2$ dimensions, 
%and the vector field~$x\mapsto\kk(x)\alpha$, with~$\alpha=(1,0)$ (also shown).
%Note that~$\AAA\geq0$ and~$\BBB\geq0$.}
\caption{Graphs of~$\aaa$, $\bbb$, $\AAA$, $\BBB$ for Example~\ref{exGauss} 
%with~$a=1.5$, 
%$b=c=1$ in $d=2$ dimensions, 
and the vector field~$x\mapsto\kk(x)\alpha$, with~$\alpha=e_1$ (shown).
%Note that~$\AAA\geq0$ and~$\BBB\geq0$.
}
\label{FigGauss}
\end{figure}
\end{example_n}
\begin{example_n}
\label{exGauss2}
We now modify the previous example and switch the roles of~$\aaa$ 
and~$\bbb$. In other words we want to find conditions on~$(a,b)\in\mathbb{R}^2$
such that the functions
\begin{equation}
\label{coe_cf}
\aaa(r) = (b-ar^2)\,e^{-cr^2}
\qquad
\mbox{and}
\qquad 
\bbb(r) = b\,e^{-cr^2},\quad r>0,
\end{equation}
%\begin{align*}
%\label{goodEx}
%\ccc(r)&= 
%\frac{\aaa(r)-\bbb(r)}{r^2}=-a\,e^{-cr^2}
%&
%&\mbox{and}
%&
%\bbb(r) &= b\,e^{-cr^2},
%&
%&\mbox{so that}
%&
%\aaa(r) &= (b-ar^2)\,e^{-cr^2},
%\end{align*}
with~$c>0$, define a positive definite~$\kk$ 
of the type~\eqref{invker}, i.e.~a TRI kernel.
This means we assume that~$\ccc(r)=-a\,e^{-cr^2}$,
hence both~$\ccc$ and~$\bbb$ are Gaussian, with the same variance. 
Calculations that are similar to those 
of the previous example lead to the following expressions
for~$\AAA$ and~$\BBB$:
\begin{align*}
\AAA(\varrho)
& =
\frac{\pi^{\mu+1}}{c^{\mu+1}}
\Big(b-\frac{a}{2c}+\frac{a\pi^2}{c^2}\varrho^2\Big)
\exp\!\Big(\!-\!\frac{\pi^2\varrho^2}{c}\Big),
\\
\BBB(\varrho)
& =
\frac{\pi^{\mu+1}}{c^{\mu+1}}
\Big(b-\frac{a}{2c}
%+
%\frac{a}{c^2}\pi^2\varrho^2
\Big)
\exp\!\Big(\!-\!\frac{\pi^2\varrho^2}{c}\Big).
\end{align*}
\par
The above are simultaneously nonnegative~$\varrho$, 
i.e.~$\kk$ is a TRI kernel,
if and only if $(a,b)$ is in 
$$
\textstyle
D_2
:=\big\{
(a,b)\in \mathbb{R}^2\,\big|\,
b\geq
\frac{a}{2c}
,
a\geq0
\big\}
$$
which is the wedge-shaped domain shown in Figure~\ref{wedge}(b). The vertical
boundary ($a=0$) again corresponds to scalar kernels, while the 
meaning of the other boundary (whose slope this time does \em not \em depend 
on the dimension~$d$ of the domain) will be explored later.
\par
The graphs of~the functions~$\aaa$, $\bbb$, $\AAA$ and $\BBB$ (again, symmetrized with 
respecto to~$r=0$) are shown 
in Figure~\ref{FigGauss2} 
for~$d=2$, $a=1.5$ and~$b=c=1$,
as well as the corresponding vector field~$x\mapsto\kk(x)\alpha$, $x\in\mathbb{R}^2$,
with~$\alpha=e_1\in\mathbb{R}^2$. 
For such~$\alpha$ the apparent ``sink'' and ``source'' of the vector field (where it is equal to zero)
are respectively located at points~$(\pm\sqrt{b/a},0)$.
\begin{figure}[t]
\begin{center}
\begin{picture}(370,170)
\setlength{\unitlength}{1pt}
\put(3,80){\includegraphics[height=2.25cm]{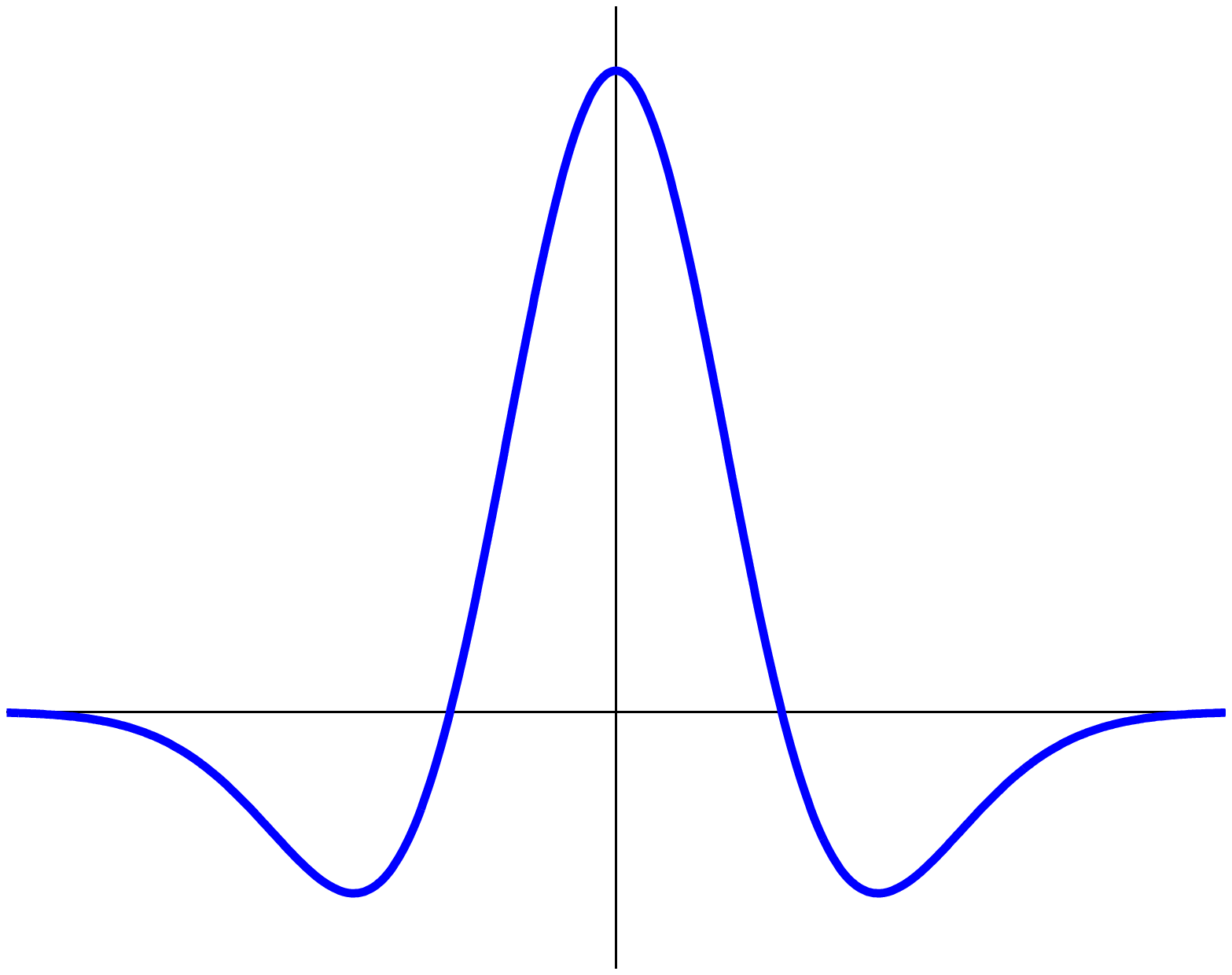}} 
\put(96,80){\includegraphics[height=2.25cm]{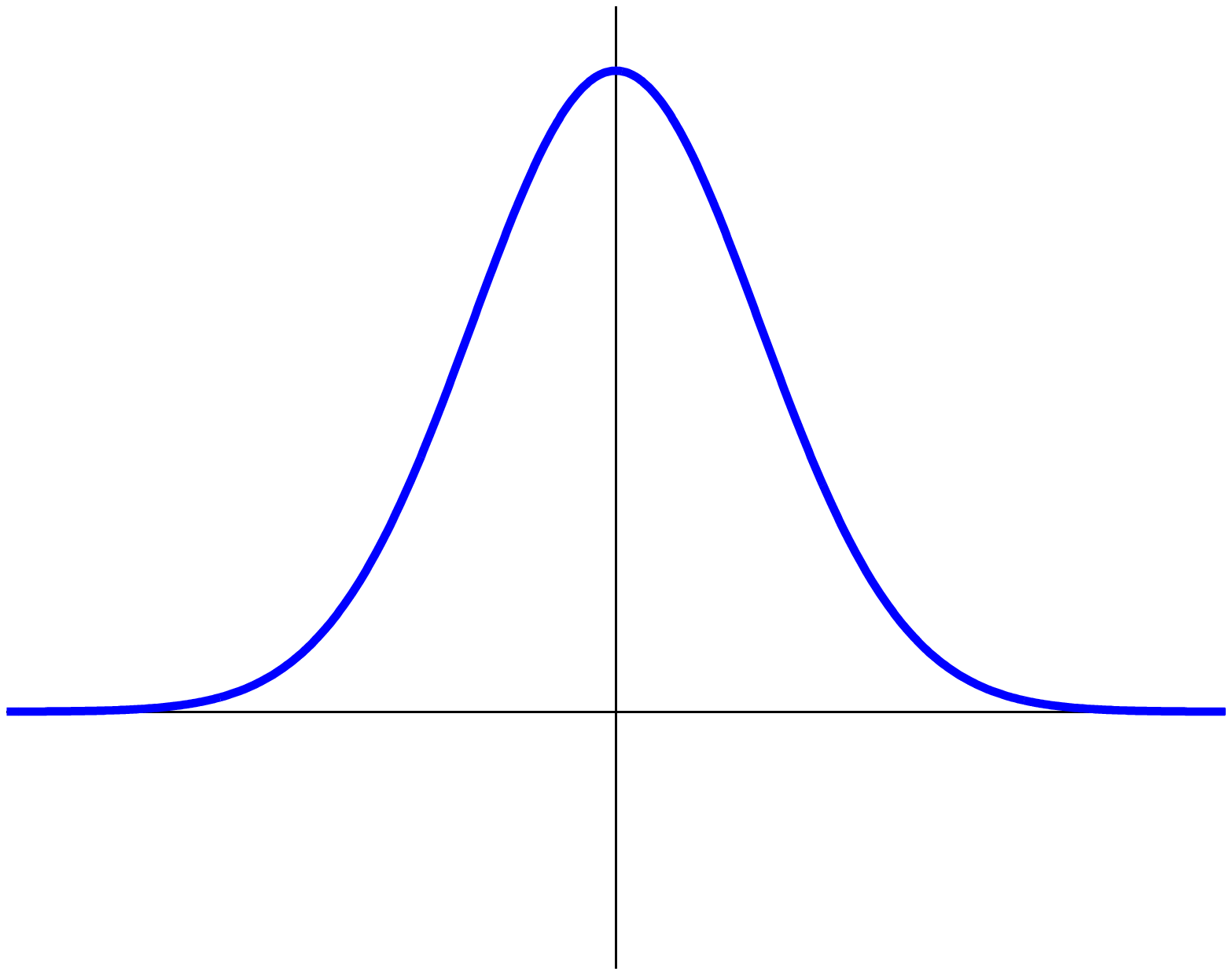}}
\put(3,-10){\includegraphics[height=2.25cm]{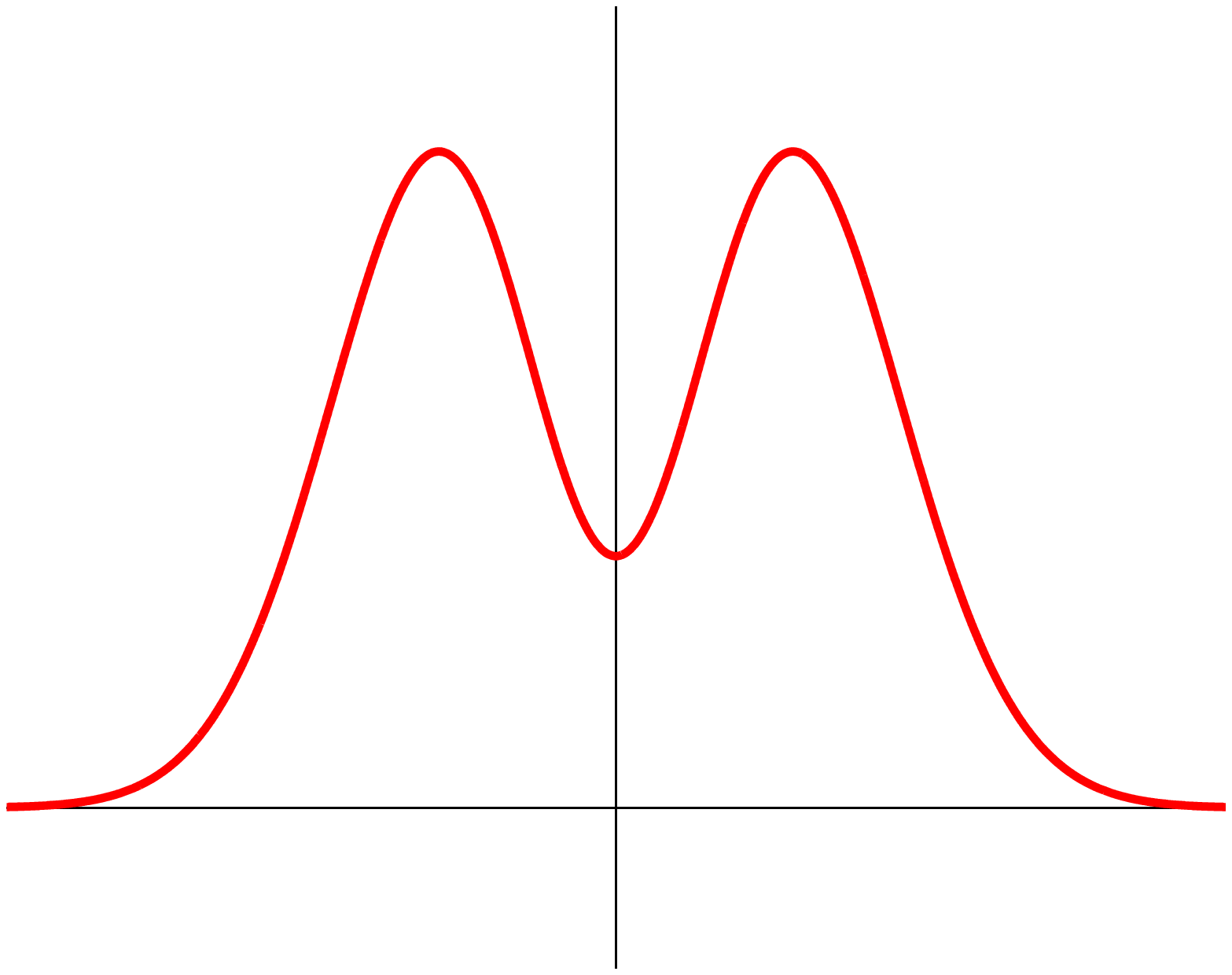}} 
\put(96,-10){\includegraphics[height=2.25cm]{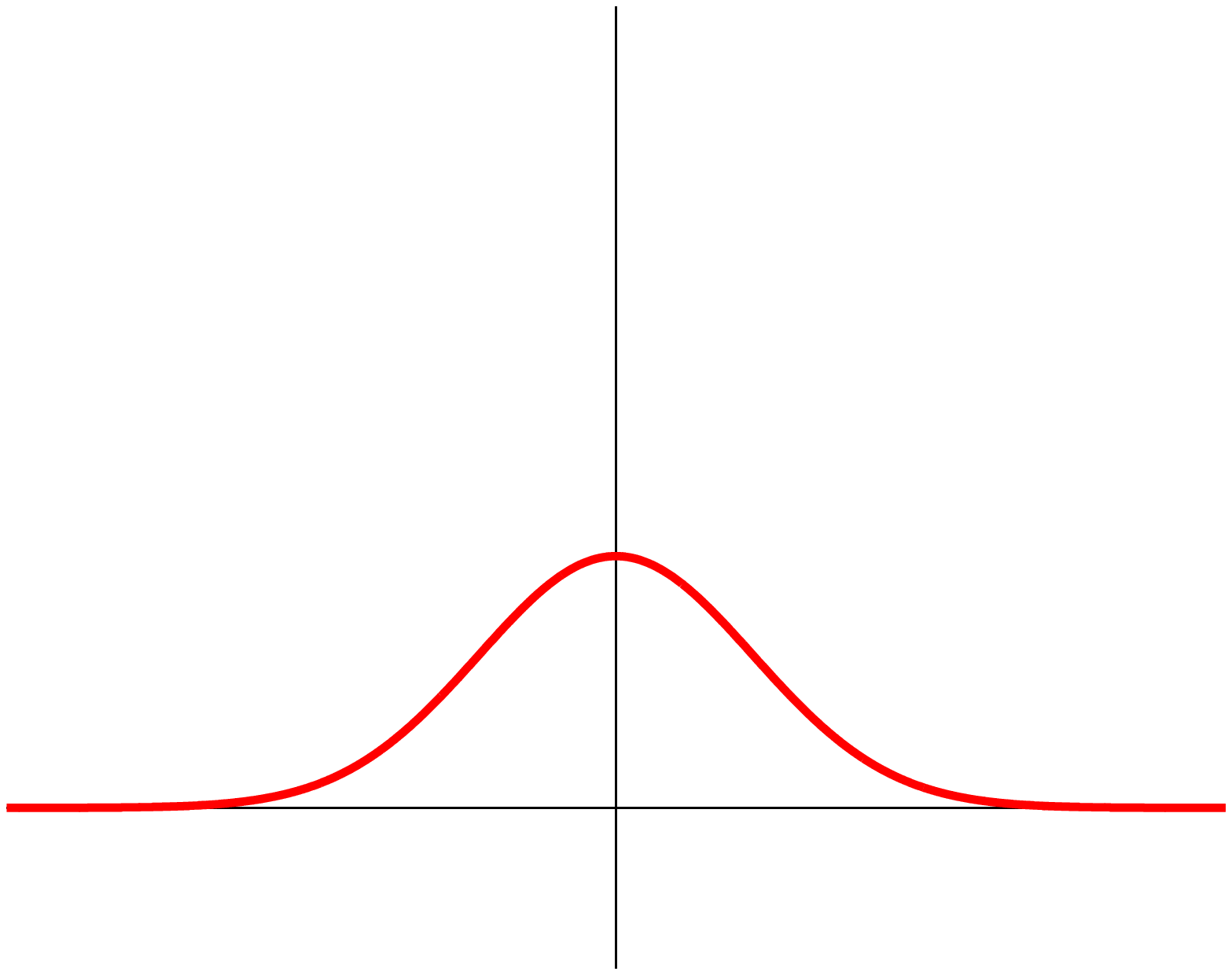}}
\put(46,152){\makebox(0,0){$\aaa$}} \put(140,152){\makebox(0,0){$\bbb$}}
\put(46,62){\makebox(0,0){$\AAA$}} \put(140,62){\makebox(0,0){$\BBB$}}
\put(215,-5){\includegraphics[height=5.5cm]{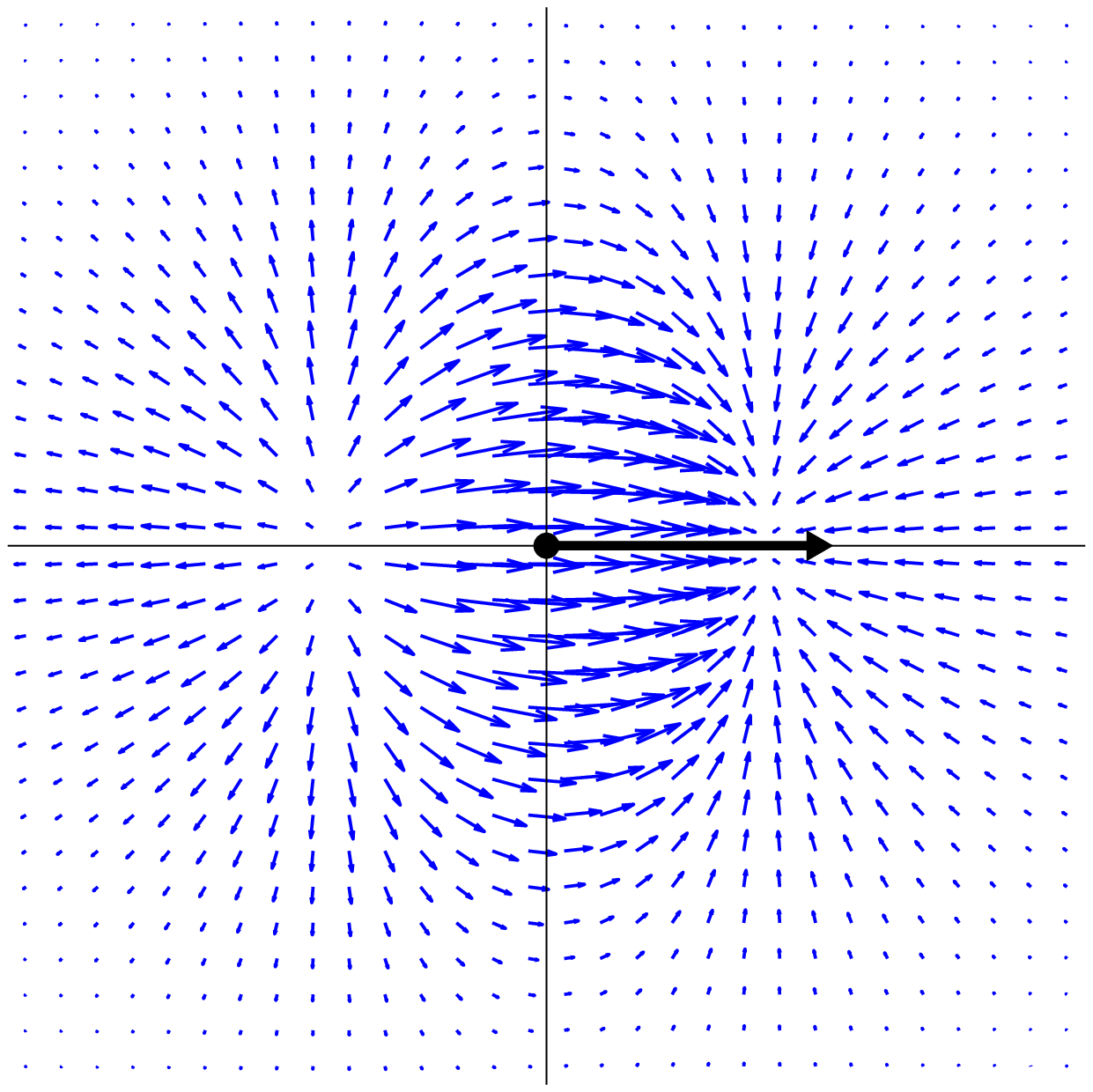}}
\put(295,157){\makebox(0,0){\small$x\mapsto\kk(x)e_1$}}
%\put(295,157){\makebox(0,0){\small$x\mapsto\kk(x)\alpha$, with~$\alpha=(1,0)$}}
%\put(239,122){\makebox(0,0){\Large$\Longrightarrow$}}
%\put(245,137){\makebox(0,0){\Huge$\Searrow$}}
%\put(0,160){\makebox(0,0){$\bullet$}} \put(370,160){\makebox(0,0){$\bullet$}}
%\put(0,0){\makebox(0,0){$\bullet$}} \put(370,0){\makebox(0,0){$\bullet$}}
\thicklines
%\put(220,140){\vector(1,0){40}} %
\end{picture}
\end{center}
%\caption{Graphs of~$\aaa$, $\bbb$, $\AAA$, $\BBB$ for Example~\ref{exGauss} 
%with~$a=1.5$, 
%$b=c=1$ in $d=2$ dimensions, 
%and the vector field~$x\mapsto\kk(x)\alpha$, with~$\alpha=(1,0)$ (also shown).
%Note that~$\AAA\geq0$ and~$\BBB\geq0$.}
\caption{Graphs of~$\aaa$, $\bbb$, $\AAA$, $\BBB$ for Example~\ref{exGauss2} 
%with~$a=1.5$, 
%$b=c=1$ in $d=2$ dimensions, 
and the vector field~$x\mapsto\kk(x)\alpha$, with~$\alpha=e_1$ (shown).
%Note that~$\AAA\geq0$ and~$\BBB\geq0$.
}
\label{FigGauss2}
\end{figure}%
\end{example_n}

\subsection{Inversion formulae}
\label{par_inv}
%In this section we give a constructive method for building positive definite translation- and rotation-invariant kernels, i.e.~of the type~\eqref{invker}. One can in fact start from nonnegative  functions~$\AAA$ and~$\BBB$ and find the corresponding functions~$\aaa$ and~$\bbb$  using the formulae provided by the following proposition.
Let~$M$ be the functional that maps
the coefficients~$(\aaa,\bbb)$ of a matrix-valued
function~$\kk\in L^1(\Rd,\Rdd)$ of the type~\eqref{invker}
to the coefficients~$(\AAA,\BBB)$ of its Fourier transform. That is,
\begin{equation}
\label{funcM}
M: L^1(\mathbb{R}^+,r^{d-1})
\times
L^1(\mathbb{R}^+,r^{d-1})
\longrightarrow
C_0(\mathbb{R}^+) \times
C_0(\mathbb{R}^+) : (\aaa,\bbb)\longmapsto(\AAA,\BBB),
\end{equation}
where~$(\AAA,\BBB)$ are expressed in terms of~$(\aaa,\bbb)$
precisely
by formulae~\eqref{Tpar} and~\eqref{Tper}.
Once again, the symbol~$C_0(\mathbb{R}^+)$ indicates
the set of real-valued continuous functions defined on~$\mathbb{R}^+$
that vanish at~$+\infty$.  
The map~$M$ is \em linear \em
%(but not bilinear in the variables~$\aaa$ and~$\bbb$)
and by the proposition that follows it is in fact an \em involution\em\/,
i.e.~$M^{-1}=M$, in the sense that is specified by the following proposition
(this generalizes the fact that~\eqref{cond_sc2} is an involution).
\begin{proposition} 
\label{invT}
Let~$\kk\in L^1(\Rd,\Rdd)$ be of the type~\eqref{invker}, and also 
assume~$\kkh\in L^1(\Rd,\Rdd)$. Then the 
%pair of 
functions~$(\aaa,\bbb)$ can be computed from~$(\AAA,\BBB)$
as follows:
\begin{subequations}
\begin{align}
\label{iTpar}
\aaa(r)
&
=
\frac{2\pi}{r^\mu}
\int_0^\infty
\varrho^{\mu+1}\,\AAA(\varrho)J_\mu(2\pi\varrho r)\,d\varrho
-
\frac{2\mu+1}{r^{\mu+1}}
\int_0^\infty
\varrho^{\mu}\big(\AAA(\varrho)-\BBB(\varrho)\big)J_{\mu+1}(2\pi\varrho r)\,d\varrho,
\\
\label{iTper}
\bbb(r)
&=
\frac{2\pi}{r^\mu}
\int_0^\infty
\varrho^{\mu+1}\,\BBB(\varrho)J_\mu(2\pi\varrho r)\,d\varrho
+
\frac{1}{r^{\mu+1}}
\int_0^\infty
\varrho^{\mu}\big(\AAA(\varrho)-\BBB(\varrho)\big)J_{\mu+1}(2\pi\varrho r)\,d\varrho;
\end{align}
\end{subequations}
that is,~$(\aaa,\bbb)$ are obtained from~$(\AAA,\BBB)$
by applying once again formulae~\eqref{Tpar} and~\eqref{Tper}.
\end{proposition}
\begin{proof}
The claim may be proven by direct computation, i.e.~by 
manipulating formulae~\eqref{Tpar}-\eqref{Tper} and
by inverting Hankel transforms (using formula~\eqref{inv_Htf2} 
in Appendix~\ref{appBessel}).
However, one can simply observe that
the Fourier transform, when restricted to 
functions that are symmetric in their argument,
i.e.~$\kk(x)=\kk(-x)$ for all~$x$ (and TRI kernels
belong to this class), is an involution. 
The map~$M$ computes the functions~$(\AAA,\BBB)$, 
which provide the
eigenvalues of~$\kkh$ in terms of those of~$\kk$, 
therefore it must be itself an involution. 
\end{proof}
The above proposition allows one to 
construct positive kernels by choosing nonnegative functions
$\AAA$ and $\BBB$ in~$L^1(\mathbb{R}^+,r^{d-1})$
and applying the inversion formulae~\eqref{iTpar} and~\eqref{iTper}.
This involves computing Hankel transforms, which may 
be done either by employing tables of 
transforms~\cite{ditkin,erdelyi_tables_v2,oberhettinger},
or numerically. However, in Section~\ref{constr} 
we shall illustrate a very simple, constructive method 
for building arbitrary TRI kernels from scalar kernels. %The proposition is proven in Appendix~\ref{appBochner} and 
%ensures that, as long as either~$\AAA$ or~$\BBB$ are positive 
%somewhere, the RKHS whose kernel is the corresponding~$\kk$
%is non-degenerate. Therefore all RKHS with a TRI kernel
%are non-degenerate (except for the trivial case~$\kk=0$).
\begin{remL1L2}
While it is well known that the Fourier transform is an 
isometry between~$L^2$ and itself, we have assumed that
the kernels are in~$L^1$ (and that, sometimes, so are their Fourier transforms)
to be able to use and manipulate the integral formulas in the definitions and proofs.
We should note that one could employ the usual argument that~$L^1$
is dense in~$L^2$~\cite{stein_weiss} and extend the results,
\em mutatis mutandis\em\/,
to square integrable kernels. 
However, for computational convenience, 
we will keep assuming that the kernels are in~$L^1$
when Fourier transforms are needed
throughout the rest of the paper.
\end{remL1L2}
\subsection{Divergence-free and Curl-free kernels}
\label{sec_divcurl}
In this section we explore some properties of 
the vector fields $x\mapsto \kk(x)\alpha$ generated by the kernels~$\kk$ of RKHS.
In particular we shall compute the functions~$\mathrm{div}(\kk(\cdot)\alpha)$
and~$\mathrm{curl}(\kk(\cdot)\alpha)$, and relate
them to the coefficients~$\AAA$ and~$\BBB$
of the Fourier transform of the kernel~$\kk$.
\begin{proposition}
Let~$\kk\in C^1(\Rd,\Rdd)$ be a~TRI 
kernel. It is the case that
$$
\displaystyle
\mathrm{div}\big(\kk(x)\alpha\big)
=
\frac{
%\langle\alpha,x\rangle
\alpha\cdot x}{\|x\|}
\Big(
(\dd-1)\frac{\aaa(\|x\|)-\bbb(\|x\|)}{\|x\|}+
\frac{d\aaa}{dr}(\|x\|)
\Big),
\qquad x\in\Rd\setminus\{0\}.
$$
\end{proposition}
\begin{proof} 
For~$x\not=0$ define~$\alpha^\parallel(x):=\mathrm{Pr}_x^\parallel\,\alpha$
and~$\alpha^\perp(x):=\mathrm{Pr}_x^\perp\,\alpha=\alpha-\alpha^\parallel(x)$,
so that~$\alpha=\alpha^\parallel(x)+\alpha^\perp(x)$.
It is the case that:
\begin{align}
\nonumber
\mathrm{div}& \big(\kk(x)\alpha\big)
=
\mathrm{div}\big(\aaa(\|x\|)\,\alpha^\parallel(x)\big)
+
\mathrm{div}\big(\bbb(\|x\|)\,\alpha^\perp(x)\big)
\\
\nonumber
&=
\nabla\big(\aaa(\|x\|)\big)\cdot\alpha^\parallel(x)
+
\aaa(\|x\|)\,\mathrm{div}\big(\alpha^\parallel(x)\big)
+
\nabla\big(\bbb(\|x\|)\big)\cdot\alpha^\perp(x)
+
\bbb(\|x\|)\,\mathrm{div}\big(\alpha^\perp(x)\big)
\\
&=
%\aaa_r
\frac{d\hpar}{dr}(\|x\|)
\,
\frac{x}{\|x\|}\cdot\alpha^\parallel(x)
+
\aaa(\|x\|)\,\mathrm{div}\big(\alpha^\parallel(x)\big)
+
\bbb(\|x\|)\,\mathrm{div}\big(\alpha^\perp(x)\big),
\label{aux1}
\end{align}
where we have used the fact that
$
\displaystyle
\frac{x}{\|x\|}\cdot \alpha^\perp(x)
=0$.
Also, 
$
\displaystyle
\frac{x}{\|x\|}\cdot \alpha^\parallel(x)
=\frac{x}{\|x\|}\cdot\alpha %\langle\alpha,x\rangle
$ and
\begin{align*}
\mathrm{div}\big(\alpha^\parallel(x)\big)
&=
\mathrm{div}\Big(\frac{\alpha\cdot x}{\|x\|^2}\,x\Big)
=
\frac{1}{\|x\|^2}\big(\nabla
(\alpha\cdot x)
%\langle\alpha,x\rangle
\big)\cdot x
+
(\alpha\cdot x)
%\langle\alpha,x\rangle
\Big(\nabla\frac{1}{\|x\|^2}\Big)\cdot x
+
\frac{
\alpha\cdot x
%\langle \alpha,x \rangle
}{\|x\|^2}\,\mathrm{div}x
\\
&=
\frac{
\alpha\cdot x
%\langle\alpha,x\rangle
}{\|x\|^2}
+
%\langle\alpha,x\rangle
(\alpha\cdot x)
\Big(\frac{-2}{\|x\|^3}
\frac{x}{\|x\|}\Big)
\cdot x+
\frac{
%\langle\alpha,x\rangle
\alpha\cdot x
}{\|x\|^2} \dd
=(\dd-1)\frac{
\alpha\cdot x
%\langle\alpha,x\rangle
}{\|x\|^2}.
\end{align*}
Moreover, $\mathrm{div}\big(\alpha^\perp(x)\big)
%=\mathrm{div}\big(\alpha-\frac{\langle \alpha,x \rangle}{\|x\|^2}x\big)
=-\mathrm{div}\big(\alpha^\parallel(x)\big)$.
Inserting these expressions into~\eqref{aux1} completes the proof.
\end{proof} 
\begin{corollary} 
\label{cor_div}
Let~$\kk\in C^1(\Rd,\Rdd)$ be a TRI kernel. 
The vector field~$x\mapsto \kk(x)\alpha$
is divergence-free %in~$\Rd$
if and only if
\begin{equation}
\label{divzero}
(\dd-1)\,
%\tilde{k}(r)
\frac{\hpar(r)-\hper(r)}{r} 
+  
%\frac{1}{r}
\frac{d\hpar}{dr}(r)=0, \qquad\mbox{for all } r>0.
\end{equation}
%where, as usual, $\frac{\hpar(r)-\hper(r)}{r}$.
\end{corollary}
We want to find a similar result for 
the \em curl \em of the vector field~$x\mapsto\kk(x)\alpha$.
Indicating with~$\Omega^k\Rd$ the space of differential $k$-forms
in~$\Rd$~\cite{Lee:1}, it is convenient
to identify
a vector field~$(v_1,\ldots,v_d)$ in~$\Rd$
with the 1-form $v=\sum v_i\,d_0x^i$, where~$d_0$
is the differential of a function. In fact indicating 
with~$d_k$ the exterior (Cartan) derivative
of a $k$-form
we may write the de~Rahm complex for~$\Rd$ as follows:
$$
\Omega^0\Rd
\;
%\stackrel{d_0}{\longrightarrow}
\stackrel{d_{0}}{\xrightarrow{\hspace*{1cm}}}
\;
\Omega^1\Rd
\;
%\stackrel{d_1}{\longrightarrow}
\stackrel{d_{1}}{\xrightarrow{\hspace*{1cm}}}
\;
\Omega^2\Rd
\;
%\stackrel{d_2}{\longrightarrow}
\stackrel{d_{2}}{\xrightarrow{\hspace*{1cm}}}
\;
\;
\cdots
\;
\;
%\stackrel{d_{d-2}}{\longrightarrow}
\stackrel{d_{d-2}}{\xrightarrow{\hspace*{1cm}}}
\;
\Omega^{d-1}\Rd
\;
%\stackrel{d_{d-1}}{\longrightarrow}
\stackrel{d_{d-1}}{\xrightarrow{\hspace*{1cm}}}
\;
\Omega^d\Rd.
$$
We define $\boxed{\mathrm{curl}:=d_1}$ (so that it 
corresponds to the classical ``curl'' when~$d=3$).
%The following result holds.
%In the case of landmarks in three dimensions we can also
%consider curl-free vector fields.
\begin{proposition}
Let~$\kk\in C^1(\Rd,\Rdd)$ be a~TRI 
kernel. %, with~$d\geq3$. 
It is the case that
%
%Consider a differentiable TRI
%kernel~$\kk:\Omega\rightarrow\Rdd$, with~$d\geq3$. We have that
\begin{equation}
\label{fprop_curl}
\mathrm{curl}\big(\kk(x)\alpha\big)
=\frac{\alpha \wedge x}{\|x\|}
\bigg(
\frac{\aaa(\|x\|)-\bbb(\|x\|)}{\|x\|}-
\frac{d\bbb}{dr}(\|x\|)
\bigg),
\qquad x\in\Rd\setminus\{0\},
\end{equation}
where $\wedge$ indicates the wedge product between 
differential forms.  
\end{proposition}
\begin{proof}
As in the previous proof, let~$\alpha^\parallel(x):=\mathrm{Pr}_x^\parallel\,\alpha$
and~$\alpha^\perp(x):=\mathrm{Pr}_x^\perp\,\alpha
=\alpha-\alpha^\parallel(x)$. We consider them
1-forms, i.e.~$\alpha^\parallel(x)=\sum_i\alpha^\parallel(x)_i\,d_0x^i$
and~$\alpha^\perp(x)=\sum_i\alpha^\perp(x)_i\,d_0x^i$.
It is the case that
%%%
%%%
%%%
\begin{align}
\nonumber
\mathrm{curl}\big(\kk(x)\alpha\big)
= d_1\big(\kk(x)\alpha\big)
%\nonumber
=\;\,
&d_0\big(\aaa(\|x\|)\big)
\wedge\alpha^\parallel(x)
+
\aaa(\|x\|)\,
d_1\!\big(\alpha^\parallel(x)\big)
\\
&+
d_0\big(\bbb(\|x\|)\big)
\wedge\alpha^\perp(x)
+
\bbb(\|x\|)\,
d_1\!\big(\alpha^\perp(x)\big).
\label{aux2}
\end{align}
The first term %on the right-hand side 
is zero 
because 
$\displaystyle d_0\big(\aaa(\|x\|)\big)=\frac{d\aaa}{dr}(\|x\|)\frac{x}{\|x\|}$
and
$x\wedge\alpha^\parallel(x)=0$.
We also have
%\begin{align*}
$$
d_1\big(\alpha^\parallel(x)\big)
%&
=
d_1
\Big(
\frac{\alpha\cdot x}{\|x\|^2}
\, x\Big)
=
d_0
\Big(
\frac{\alpha\cdot x}{\|x\|^2}
\Big)
\wedge x
+
\frac{\alpha\cdot x}{\|x\|^2}
\;
d_1x,
$$
%\end{align*}
where $d_1x=d_1\big(\sum_i x^id_0x^i\big)=
\sum_i d_0x^i\wedge d_0x^i=0$ and
%\begin{align}
$$
d_0
\Big(
\frac{\alpha\cdot x}{\|x\|^2}
\Big)
=
\frac{1}{\|x\|^2}
d_0(\alpha\cdot x)+
(\alpha\cdot x)
\;
d_0
%\Big(
\frac{1}{\|x\|^2}
%\Big)
=
\frac{\alpha}{\|x\|^2}
-2\,\frac{\alpha\cdot x}{\|x\|^4}\,x.
$$
Since~$x\wedge x=0$
we conclude that
$
\displaystyle 
d_1\big(\alpha^\parallel(x)\big)
=\frac{\alpha\wedge x}{\|x\|^2}$.
Moreover
$x \wedge \alpha^\perp(x) = x \wedge \alpha=-\alpha\wedge x$, whence
$$
\displaystyle
d_0\big(\bbb(\|x\|)\big)
\wedge \alpha^\perp(x)
=
\frac{d\bbb}{dr}(\|x\|)\,\frac{x}{\|x\|}
\wedge \alpha^\perp(x)
=
-
\frac{d\bbb}{dr}(\|x\|)\,
\frac{\alpha\wedge x}{\|x\|}.
$$
Finally, $d_1\big(\alpha^\perp(x)\big)
=-d_1\big(\alpha^\parallel(x)\big)$.
Insertion of such expressions into~\eqref{aux2} completes the proof.
\end{proof}
%Note that when~$d=3$
%we may simply substitute the wedge product~$\wedge$ 
%in~\eqref{fprop_curl} with the cross product~$\times$
%of~$\mathbb{R}^3$. 
%Also, in~$d=2$ dimensions we may compute the~2-dimensional
%curl of a vector field~$u=(u_1,u_2)$, i.e.~the scalar-valued
%function~$\mathrm{curl}\,u:=\frac{\partial u_2}{\partial x_1}
%-\frac{\partial u_1}{\partial x_2}$, as the third
%component of the curl of the 3-dimensional vector field~$(u_1,u_2,0)$.
%In the case of the 2-dimensional vector field $x\mapsto\kk(x)\alpha$ this 
%procedure yields
%$$
%\mathrm{curl}\big(\kk(x)\alpha\big)
%=
%\frac{\alpha \wedge x}{\|x\|}
%\|\alpha\|\sin(\angle\alpha x)
%\bigg(
%\frac{\aaa(\|x\|)-\bbb(\|x\|)}{\|x\|}-
%\frac{d\bbb}{dr}(\|x\|)
%\bigg),
%\qquad x\in \mathbb{R}^2 \setminus\{0\},
%$$
% because if we interpret~$\alpha$ and~$x$ as~3-dimensional 
% vectors whose third Cartesian coordinate is zero 
% then $(\alpha\times x)\cdot e_3=\|\alpha\|\|x\|
% \sin(\angle\alpha x)$.
%For this reason the following corollary holds for all dimensions~$d\geq2$.
\begin{corollary}
\label{cor_curl} 
Let~$\kk\in C^1(\Rd,\Rdd)$ be a~TRI 
kernel. %, with~$d\geq2$. 
The vector field $x\mapsto \kk(x)\alpha$
is curl-free %(irrotational) 
in~$\Rd$
if and only if
\begin{equation}
\label{curlzero}
%\tilde{k}(r)
\frac{\hpar(r)-\hper(r)}{r} 
-
%\frac{1}{r}
\frac{d\hper}{dr}(r)=0, \qquad\mbox{for all } r>0.
\end{equation}
%where, as usual, $\frac{\hpar(r)-\hper(r)}{r}$.
\end{corollary}
\par
Note that while condition~\eqref{divzero} depends on the dimension~$d$
of the space, condition~\eqref{curlzero} does not.
%\begin{proposition}
%\end{proposition}
The following fundamental theorem relates the incompressibility 
and irrotationality 
of vector fields of the 
type~$x\mapsto\kk(x)\alpha$ with the coefficients~$\AAA$
and $\BBB$ of the Fourier transform of the kernel~$\kk$.
\begin{theorem}
\label{fund_th}
Let~$\kk\in C^1(\Rd,\Rdd)\cap L^1(\Rd,\Rdd)$ be a TRI kernel
whose Fourier transform~$\kkh$ is also in~$L^1(\Rd,\Rdd)$.
Then the vector field~$x\mapsto\kk(x)\alpha$ 
is divergence-free for all~$\alpha\in\Rd$ 
if and only if~$\AAA=0$. 
On the other hand,~$x\mapsto\kk(x)\alpha$
is curl-free for all~$\alpha\in\Rd$ if and only if~$\BBB=0$.
\end{theorem}
\begin{proof}
It is the case that $\AAA(\varrho)=
%\frac{2\pi}{\rho^\mu}\AAAm(2\pi\varrho)
%\frac{2\pi}{\rho^\mu}
2\pi\AAAm(2\pi\varrho)/\varrho^\mu
$
and $\BBB(\varrho)=
%\frac{2\pi}{\rho^\mu}\BBBm(2\pi\varrho)
2\pi\AAAm(2\pi\varrho)/\varrho^\mu
$,
with~$\AAAm$ and~$\BBBm$ 
respectively given by~\eqref{AAAf} and~\eqref{BBBf},
and with~$\mu=\frac{d}{2}-1$. 
Inserting~\eqref{divzero} into~\eqref{AAAf} yields
\begin{align*}
\AAAm(\varrho)
& =  
\int_{0}^\infty r^{\mu+1}\aaa(r)J_\mu(\varrho r)\,dr
+
\frac{1}{\varrho}
\int_0^\infty r^{\mu+1}\frac{d\aaa}{dr}(r) J_{\mu+1}(\varrho r)\,dr
\end{align*}
(note that $2\mu+1=d-1$). We use integration by parts to compute
the second term on the right:
\begin{align}
\nonumber
\int_0^\infty r^{\mu+1}
&
\frac{d\aaa}{dr}(r) J_{\mu+1}(\varrho r)\,dr
=
\Big[
r^{\mu+1}\aaa(r)J_{\mu+1}(\varrho r)
\Big]_{r=0}^{\infty}
-
\int_0^\infty \aaa(r)\, \frac{d}{dr}\Big\{r^{\mu+1}J_{\mu+1}(\varrho r)\Big\}dr
\\
&=
\label{DerAux}
-\varrho\int_0^\infty r^{\mu+1}
\aaa(r) 
\Big\{\frac{\mu+1}{\varrho r}J_{\mu+1}(\varrho r)+J'_{\mu+1}(\varrho r)\Big\}dr
=
-\varrho\int_0^\infty r^{\mu+1}\aaa(r)J_\mu(\varrho r)\,dr,
\end{align}
where we have used the property~\eqref{Jrec3} of Bessel functions.
Therefore if~\eqref{divzero} holds then~$\AAA=0$.
Similarly,
inserting~\eqref{curlzero} into~\eqref{BBBf} gives
%$\BBB(\varrho)=\frac{2\pi}{\rho^\mu}B(2\pi\varrho)$, with~$B(\cdot)$
%given by~\eqref{BBBf}.
%$$
%B(\varrho):=
%\int_{0}^\infty r^{\mu+1}\aaa(r)J_\mu(\varrho r)\,dr
%+
%\frac{1}{\varrho}
%\int_0^\infty r^{\mu+2}\ccc(r) J_{\mu+1}(\varrho r)\,dr.
%$$
%Inserting~\eqref{curlzero} into such expression yields:
\begin{align*}
\BBBm(\varrho)
& =  
\int_{0}^\infty r^{\mu+1}\bbb(r)J_\mu(\varrho r)\,dr
+
\frac{1}{\varrho}
\int_0^\infty r^{\mu+1}\frac{d\bbb}{dr}(r) J_{\mu+1}(\varrho r)\,dr.
\end{align*}
and an identical integration by parts %(using~$\BBB$ instead of~$\AAA$) 
allows us to conclude that if~\eqref{curlzero} holds then~$\BBB=0$.
\par Assume now that~$\AAA=0$. In this case~$\aaa$ and~$\bbb$,
given by~\eqref{iTpar} and~\eqref{iTper}, are simply
%\begin{subequations}
\begin{align*}
%\label{iTpar1}
\aaa(r)
&
=
\frac{2\mu+1}{r^{\mu+1}}
\int_0^\infty
\varrho^{\mu}\BBB(\varrho)J_{\mu+1}(2\pi\varrho r)\,d\varrho,
\\
%\label{iTper1}
\bbb(r)
&=
\frac{2\pi}{r^\mu}
\int_0^\infty
\varrho^{\mu+1}\,\BBB(\varrho)J_\mu(2\pi\varrho r)\,d\varrho
-
\frac{1}{r^{\mu+1}}
\int_0^\infty
\varrho^{\mu}\BBB(\varrho)J_{\mu+1}(2\pi\varrho r)\,d\varrho,
\end{align*}
%\end{subequations}
therefore the auxiliary function~\eqref{def_ktilde} is given by
\begin{align*}
\tilde{k}(r)
&=
%\frac{\AAA(r)-\BBB(r)}{r^2}
%=
\frac{2\pi}{r^{\mu+2}}
\int_0^\infty
\varrho^{\mu+1}\BBB(\varrho)
\Big\{
2\frac{\mu+1}{2\pi\varrho r}J_{\mu+1}(2\pi\varrho r)
-J_\mu(2\pi\varrho r)
\Big\}
\,
d\varrho
%\\
%&
=
\frac{2\pi}{r^{\mu+2}}
\int_0^\infty
\varrho^{\mu+1}\BBB(\varrho)J_{\mu+2}(2\pi\varrho r)\,d\varrho,
\end{align*}
where we have used property~\eqref{Jrec1} in Appendix~\ref{appBessel}.
On the other hand,
\begin{align*}
&\frac{1}{r}
\frac{d\aaa}{dr}(r)
=
-(2\mu+1)\frac{\mu+1}{r^{\mu+3}}
\int_0^\infty
\varrho^{\mu}\BBB(\varrho)J_{\mu+1}(2\pi\varrho r)\,d\varrho
+
\frac{2\mu+1}{r^{\mu+2}}
\int_0^\infty
2\pi\varrho^{\mu+1}\BBB(\varrho)J'_{\mu+1}(2\pi\varrho r)\,d\varrho
\\
&=
-2\pi\frac{2\mu+1}{r^{\mu+2}}
\!\!
\int_0^\infty
\!\!\!\!
\varrho^{\mu+1}\BBB(\varrho)
\Big\{
\frac{\mu+1}{2\pi\varrho r}
J_{\mu+2}(2\pi\varrho r)
\!-\!
J'_{\mu+1}(2\pi\varrho r)
\Big\}
d\varrho
%\\
%&
=
-2\pi\frac{2\mu+1}{r^{\mu+2}}
\!\!\int_0^\infty
\!\!\!\!\varrho^{\mu+1}\BBB(\varrho)J_{\mu+2}(2\pi\varrho r)d\varrho,
\end{align*}
where we have used property~\eqref{Jrec4} in Appendix~\ref{appBessel}.
Since~$2\mu+1=d-1$ we have that~\eqref{divzero} holds and the vector field~$x\mapsto\kk(x)\alpha$
is divergence-free. Similarly, when~$\BBB=0$ it is the case that
\begin{align*}
\aaa(r)
&
=
\frac{2\pi}{r^\mu}
\int_0^\infty
\varrho^{\mu+1}\,\AAA(\varrho)J_\mu(2\pi\varrho r)\,d\varrho
-
\frac{2\mu+1}{r^{\mu+1}}
\int_0^\infty
\varrho^{\mu}\AAA(\varrho)J_{\mu+1}(2\pi\varrho r)\,d\varrho,
\\
\bbb(r)
&=
\frac{1}{r^{\mu+1}}
\int_0^\infty
\varrho^{\mu}\AAA(\varrho)J_{\mu+1}(2\pi\varrho r)\,d\varrho,
\end{align*}
so that
\begin{align*}
\tilde{k}(r)
&=
\frac{2\pi}{r^{\mu+2}}
\int_0^\infty
\varrho^{\mu+1}\AAA(\varrho)
\Big\{
J_\mu(2\pi\varrho r)
-2\frac{\mu+1}{2\pi\varrho r}J_{\mu+1}(2\pi\varrho r)
\Big\}
\,
d\varrho
%\\
%&
=
-\frac{2\pi}{r^{\mu+2}}
\int_0^\infty
\varrho^{\mu+1}\BBB(\varrho)J_{\mu+2}(2\pi\varrho r)\,d\varrho,
\end{align*}
by property~\eqref{Jrec1} in Appendix~\ref{appBessel}. On the other hand,
\begin{align*}
\frac{1}{r}
&\frac{d\bbb}{dr}(r)
=
-\frac{\mu+1}{r^{\mu+3}}
\int_0^\infty
\varrho^{\mu}\AAA(\varrho)J_{\mu+1}(2\pi\varrho r)\,d\varrho
+
\frac{1}{r^{\mu+2}}
\int_0^\infty
2\pi\varrho^{\mu+1}\AAA(\varrho)J'_{\mu+1}(2\pi\varrho r)\,d\varrho
\\
&=
\frac{2\pi}{r^{\mu+2}}
\int_0^\infty
\!\!\!\!
\varrho^{\mu+1}
\AAA(\varrho)
\Big\{
-\frac{\mu+1}{2\pi\varrho r}
J_{\mu+2}(2\pi\varrho r)
+
J'_{\mu+1}(2\pi\varrho r)
\Big\}
\,d\varrho
%\\
%&
=
-\frac{2\pi}{r^{\mu+2}}
\int_0^\infty
\!\!\!\!
\varrho^{\mu+1}\AAA(\varrho)J_{\mu+2}(2\pi\varrho r)\,d\varrho,
\end{align*}
again by~\eqref{Jrec4}.
Whence~\eqref{curlzero} holds and~$x\mapsto \kk(x)\alpha$ 
is curl-free. 
This completes the proof. 
\end{proof}
\begin{proposition}
Let~$H \hookrightarrow C^1(\Rd,\Rd)$
be a RKHS with a given TRI kernel~$\kk$. 
%Let~$V$ be a 1-admissible Hilbert space with 
%TRI kernel~$\kk:\Omega\rightarrow\Rdd$.
It is the case that $\mathrm{div}(\kk(\cdot)\alpha)=0$ for all~$\alpha\in\Rd$
if and only if~$\mathrm{div}\,u=0$ for all~$u\in H$. On the other hand, we have that
$\mathrm{curl}(\kk(\cdot)\alpha)=0$ for all~$\alpha\in\Rd$
if and only if~$\mathrm{curl}\,u=0$ for all~$u\in H$.
\end{proposition}
\begin{proof}
The ``if'' part is obvious. %, since~$\kk(\cdot)\alpha\in V$. 
Conversely, 
let~$H_0:=\mathrm{span}\{\kk(\cdot-x)\alpha\,|\,x,\alpha\in\Rd\}$.
Take any~$u\in H$ and a sequence~$\{u_n\}$ 
in~$H_0$ 
with~$u_n\rightarrow u$. For all indices~$i,j\in\{1,\ldots,d\}$,
by formula~\eqref{diff_ti}
we have that $e_i \cdot \partial_j u_n(x)=\langle
-\partial_j\kk(\cdot-x)e_i
,u_n\rangle_V$, 
which converges to
$\langle
-\partial_j\kk(\cdot-x)e_i
,u\rangle_V=e_i \cdot \partial_j u(x)$. 
That is, the sequence of functions~$\partial_ju_n$ converges to~$\partial_ju$ pointwise.
Therefore we have $\mathrm{div}\,u_n\rightarrow\mathrm{div}\,u$
and~$\mathrm{curl}\,u_n\rightarrow\mathrm{curl}\,u$ pointwise.
Since~$u_n\in
H_0$ for all~$n$, if~$\mathrm{div}(\kk(\cdot)\alpha)=0$ 
for all~$\alpha\in\Rd$
then $\mathrm{div}\,u_n=0$, whence~$\mathrm{div}\,u=0$.
A similar argument holds in the case~$\mathrm{curl}(\kk(\cdot)\alpha)=0$.
\end{proof}
The above results justify the following definition.
\begin{definition}
Let~$H\hookrightarrow C^1(\Rd\Rd)$ be a RKHS with a TRI kernel~$\kk$.
We call the space~$H$ and its kernel~$\kk$
\em divergence-free \em (or~\em curl-free\em\/)
if~$\mathrm{div}\, u=0$ (respectively, if $\mathrm{curl}\, u=0$) for all~$u\in H$.
%We call the kernel~$\kk$ \em curl-free \em (\em div-free \em) iid 
\end{definition}
\begin{table}[t]
\centering{
\begin{tabular}{c|c|c|}
\cline{2-3}
 & 
%\centering{
\hspace*{.7cm}
\bf Conditions on~$(\aaa,\bbb)$ 
\hspace*{.7cm}
 & 
\hspace*{.7cm}
\bf Conditions on~$\displaystyle(\AAA,\BBB)$
\hspace*{.7cm}
 \\
%\cline{2-3}
\hline
\multicolumn{1}{ |c| }{
%$\begin{array}{c}
%\mathrm{div}\, v =0
%\\
%\mbox{for all }v\in V
%\end{array}$
$\begin{array}{c}
\\
\mbox{$\kk$ is \bf div-free}
\\
\mbox{ }
\end{array}$
}
& 
$
\displaystyle
(d-1)\,
%\tilde{k}(r)
\frac{\hpar(r)-\hper(r)}{r} 
+
%\frac{1}{r}
\frac{d\hpar}{dr}(r)=0
$
& $\AAA=0$ \\
\hline
\multicolumn{1}{ |c| }{
$\begin{array}{c}
\\
\mbox{$\kk$ is \bf curl-free}
\\
\mbox{ }
\end{array}$
}
& 
$
\displaystyle
%\tilde{k}(r)
\frac{\hpar(r)-\hper(r)}{r} 
-
%\frac{1}{r}
\frac{d\hper}{dr}(r)\equiv0
$
& $\BBB=0$ \\
\hline
\end{tabular}
}
\caption{Conditions on the coefficients~$\kk$ and~$\kkh$ for
the kernel to be divergence-free or curl-free.}
\label{tabdiff}
\end{table}
\par
Table~\ref{tabdiff} summarizes the conditions on the coefficients~$(\aaa,\bbb)$ 
of~$\kk$
and on
the coefficients~$(\AAA,\BBB)$ 
of its Fourier transform
for a Reproducing Kernel Hilbert Space~$H$ with a TRI kernel to be either 
divergence-free or curl-free.  
We should note that if~$\kk$ is \em scalar\em\/, i.e.~of the 
type~$\kk(x)=k(\|x\|)\mathbb{I}_d$, then~$\aaa=\bbb$ and~$\AAA=\BBB$;
therefore by Corollaries~\ref{cor_div} and~\ref{cor_curl} the corresponding
RKHS can  be
\em neither \em divergence-free \em nor \em 
curl-free (unless we are dealing with the trivial case~$\kk=0$).   
\begin{paragraph}{Examples 1 and 2, revisited.} In light 
of the above results, the boundaries of the domains~$D_1$
and~$D_2$ of Examples~1 and~2 in this section 
have an interesting interpretation. 
We already noted that in both cases choosing~$a=0$
(vertical boundary of the domains) yields \em scalar \em kernels,
 i.e.~of the type $\kk(x)=k(\|x\|)\mathbb{I}_2$
with~$k(r)=b\exp(-cr^2)$ (Gaussian).
\par
In  Example~1, when~$b=(d-1)\frac{a}{2c}$ we have
$\ccc(r)=a \exp(-cr^2)$ and~$\aaa(r)=(d-1)\frac{a}{2c} \exp(-cr^2)$;
therefore~\eqref{divzero} holds and the kernel
is divergence-free. Similarly, in the case of
Example~2 when $b=\frac{a}{2c}$ we have
$\ccc(r)=-a \exp(-cr^2)$ and~$\bbb(r)=\frac{a}{2c} \exp(-cr^2)$;
whence~\eqref{curlzero} holds and the kernel
is curl-free.
Thus the boundaries of the two domains may be interpreted
as illustrated in Figure~\ref{wedgeint}.
\par
In the case~$d=2$ we have that~$D_1=D_2$.
Figure~\ref{FigTrans}
shows the vector field~$x\mapsto \kk(x)\alpha$, 
$\alpha=e_1$, for both Example~1 and Example~2,
with~$b=c=1$ and different choices of the parameter~$a$.
In particular, for~$a=0$ the point~$(a,b)$ is on the vertical part 
of the boundary of~$D_1$ and~$D_2$; this corresponds, in both cases,
to a scalar kernel. As we move along the line~$b=1$ towards
higher values of~$a$, we note that we transition to
a vector field~$x\mapsto \kk(x)\alpha$
that is divergence-free (Example~1) or curl-free (Example~2);
this occurs when~$(a,b)$ hits the slanted part of the boundaries 
of~$D_1$ and~$D_2$, i.e.~the line~$a=2bc$.
\end{paragraph}
\par
\begin{figure}
\begin{center}
\begin{tikzpicture}[scale=.75]
\fill [gray!20] (0,0) -- (3.333,2.5) -- (0,2.5);
\draw [->] (-.5,0) -- (5.5,0) node [below] {$a$};
\draw [->] (0,-.5) -- (0,3.25) node [left] {$b$};
\draw [->] (3,1) -- (2.05,1.48);
\draw [->] (-1,1.5) -- (-.05,1.5);
\draw (-.9,1.75) node [left] {\small scalar};
\draw (-.9,1.35) node [left] {\small kernels};
%\draw (3,.9) node [right] {\small$\displaystyle b=(d-1)\frac{a}{2c}$};
\draw (3,1.25) node [right] {\small$b=(d-1)\frac{a}{2c}$};
\draw (3,.75) node [right] {\small div-free kernels};
\draw (0.05,-.3) node [left] {$0$};
\draw [very thick] (0,0) -- (3.333,2.5);
\draw [very thick] (0,0) -- (0,2.5);
\draw (.45,1.5) node [right] {\large$D_1$};
\draw (2.25,-.75) node [right] {(a)};
\end{tikzpicture}
\hspace{.75cm}
\begin{tikzpicture}[scale=.75]
\fill [gray!20] (0,0) -- (5,2.5) -- (0,2.5);
\draw [->] (-.5,0) -- (5.5,0) node [below] {$a$};
\draw [->] (0,-.5) -- (0,3.25) node [left] {$b$};
\draw [->] (4,1) -- (3.05,1.48);
\draw [->] (-1,1.5) -- (-.05,1.5);
\draw (-.9,1.75) node [left] {\small scalar};
\draw (-.9,1.35) node [left] {\small kernels};
%\draw (4,.9) node [right] {\small$\displaystyle b=\frac{a}{2c}$};
\draw (4,1.25) node [right] {\small$b=\frac{a}{2c}$};
\draw (4,.75) node [right] {\small curl-free kernels};
\draw (0.05,-.3) node [left] {$0$};
\draw [very thick] (0,0) -- (5,2.5);
\draw [very thick] (0,0) -- (0,2.5);
\draw (.75,1.5) node [right] {\large$D_2$};
\draw (2.25,-.75) node [right] {(b)};
\end{tikzpicture}
\end{center}
\vspace*{-.5cm}
\caption{Interpretation of the boundaries of domains~$D_1$ and~$D_2$  
in Examples~\ref{exGauss} and~\ref{exGauss2}.}
\label{wedgeint}
\end{figure}
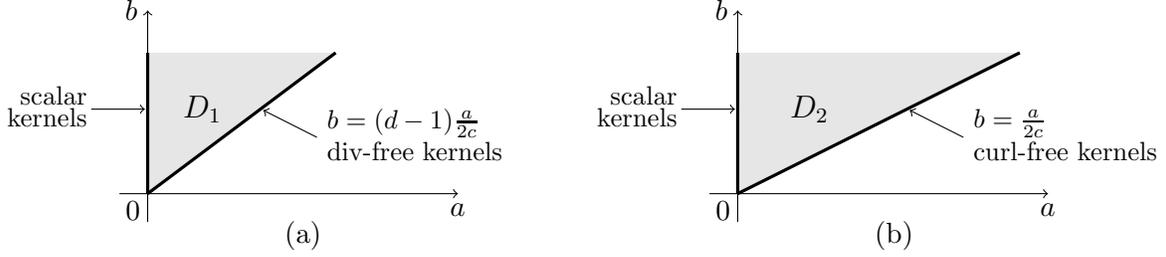%
\begin{figure}[t]
\begin{center}
\begin{picture}(460,255)
\setlength{\unitlength}{1pt}
\put(9.75,0){\includegraphics[height=3.8cm]{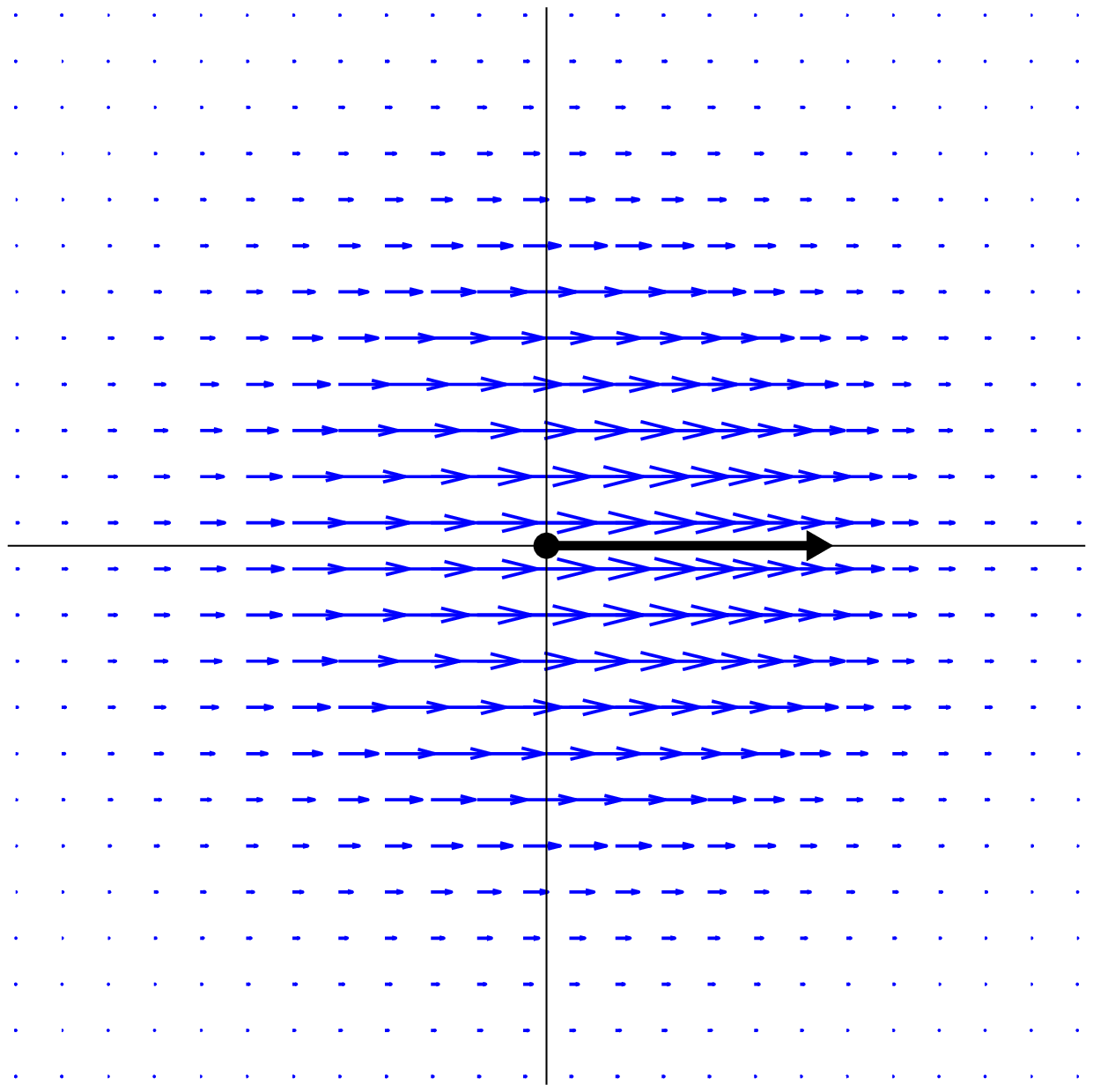}}
\put(124.75,0){\includegraphics[height=3.8cm]{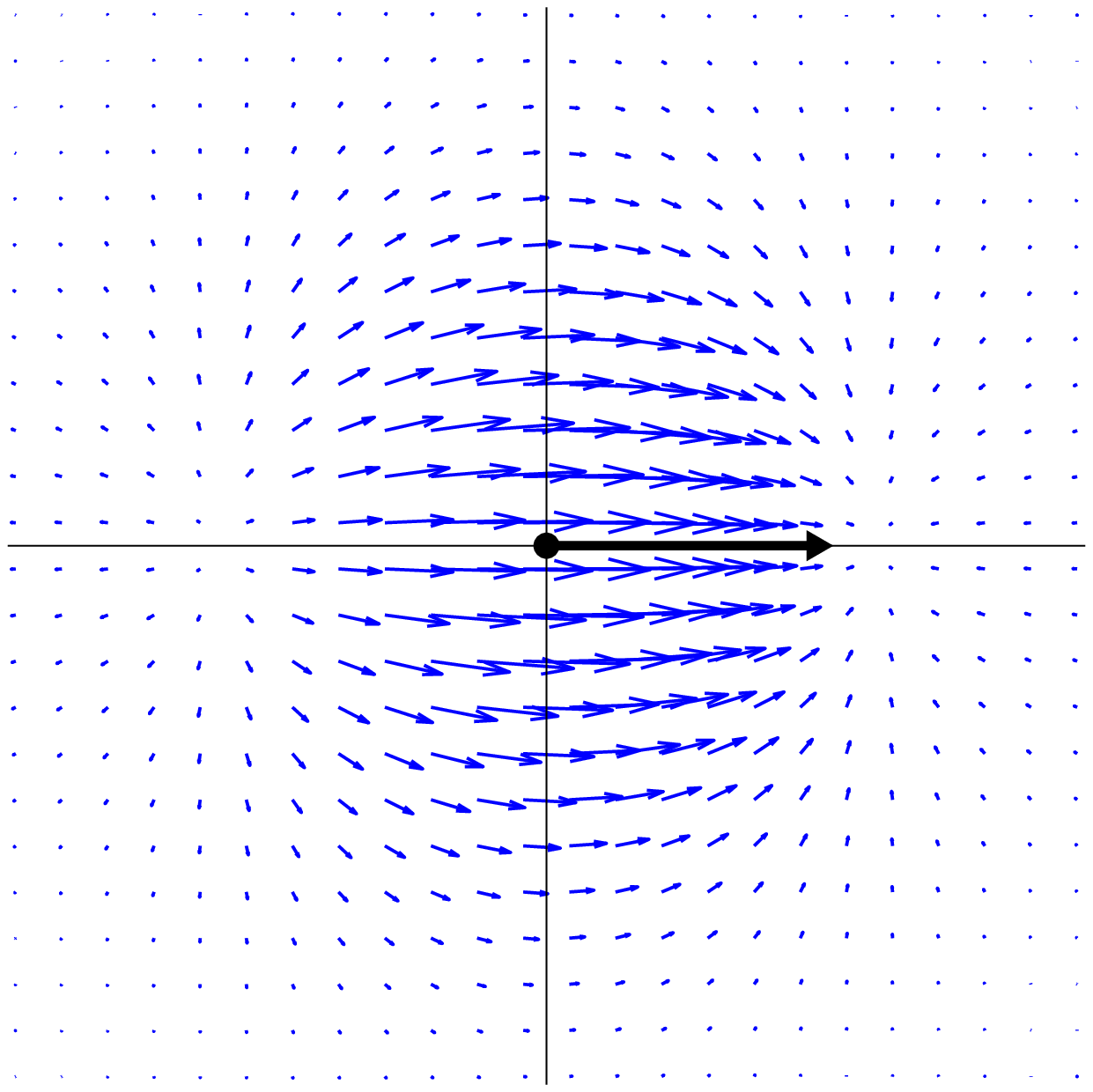}}
\put(239.75,0){\includegraphics[height=3.8cm]{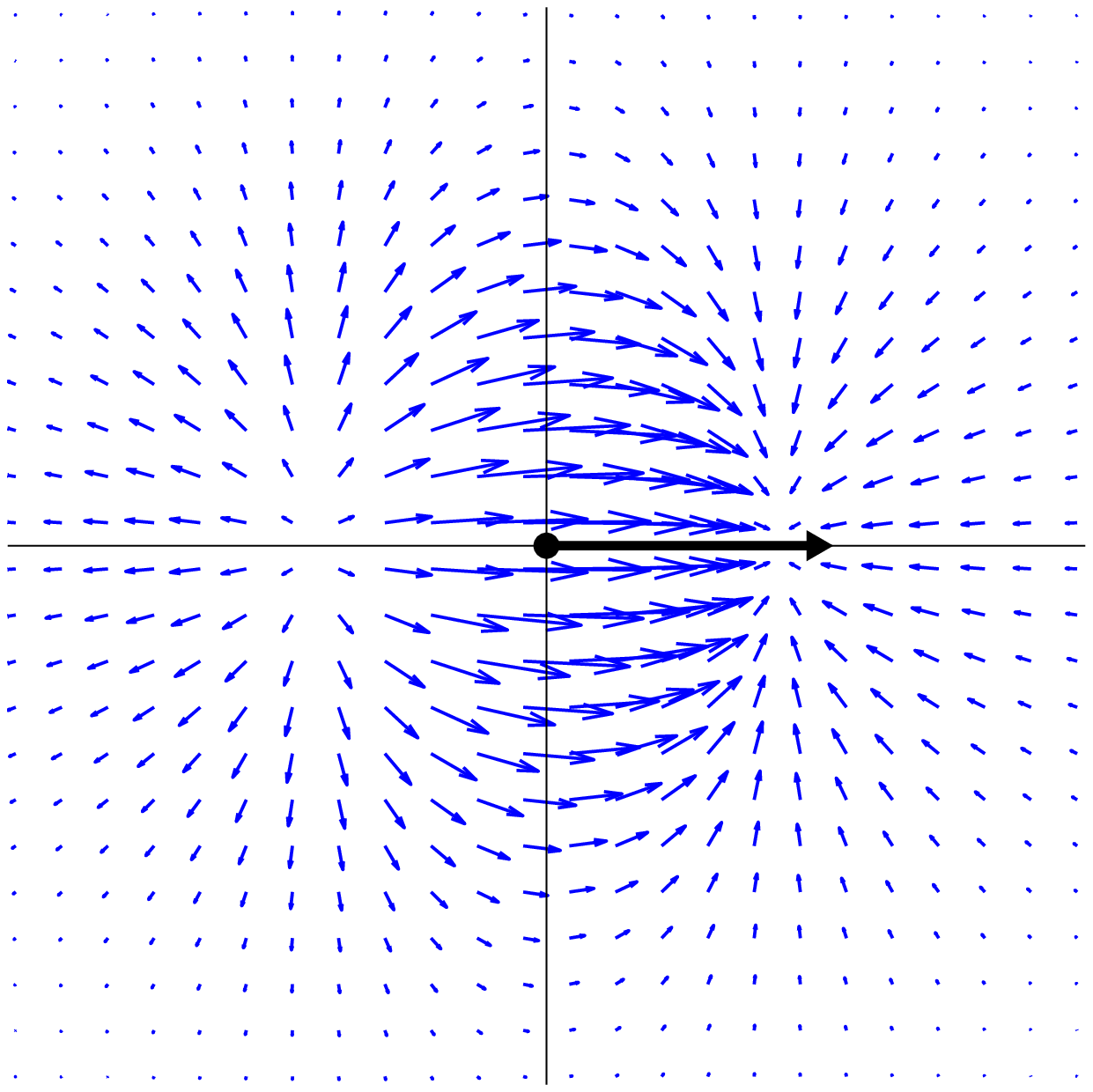}}
\put(354.75,0){\includegraphics[height=3.8cm]{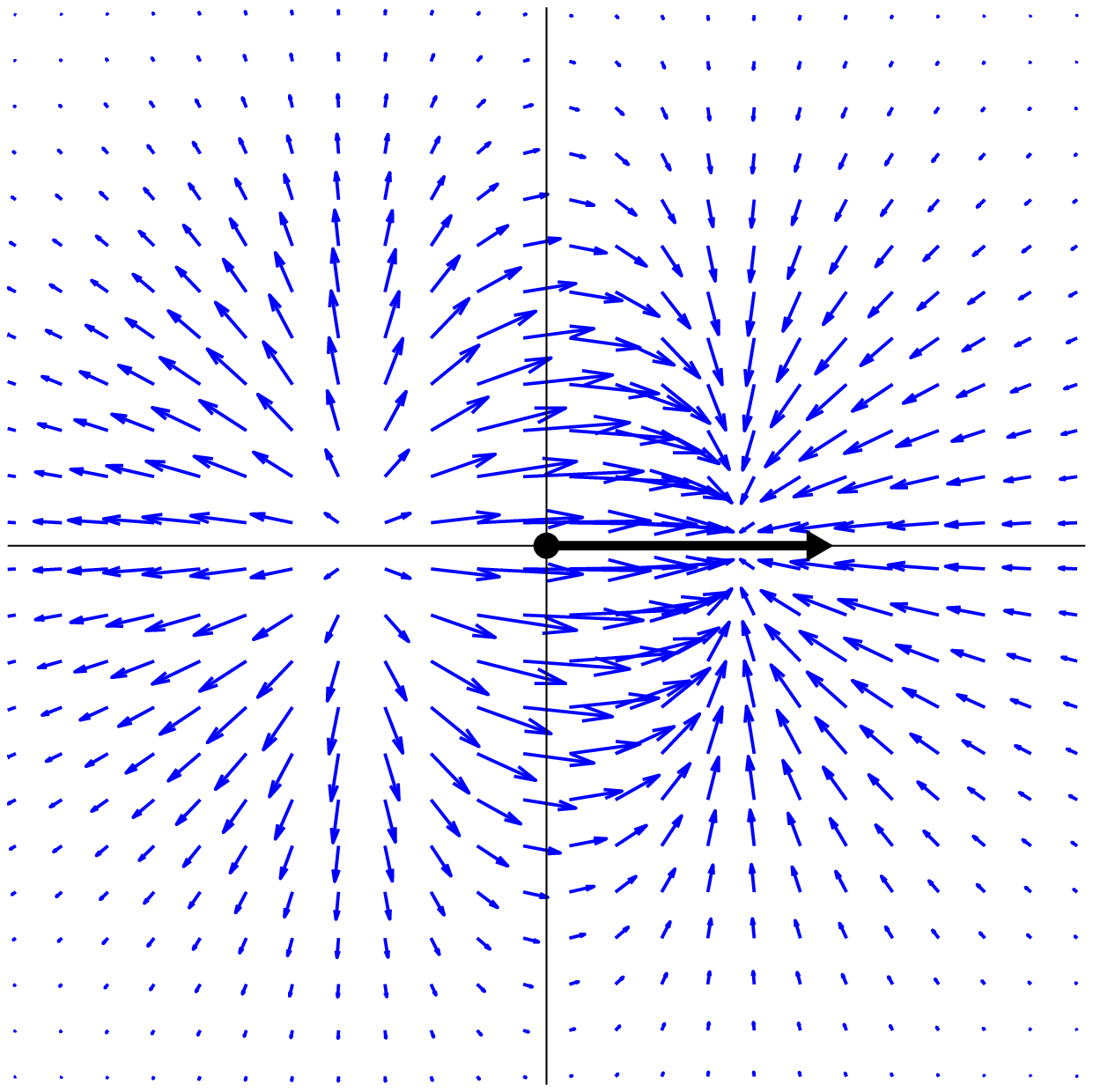}}
\put(64,-5){\makebox(0,0){\small  scalar}}
\put(409,-5){\makebox(0,0){\small curl-free}}
\put(9.75,125){\includegraphics[height=3.8cm]{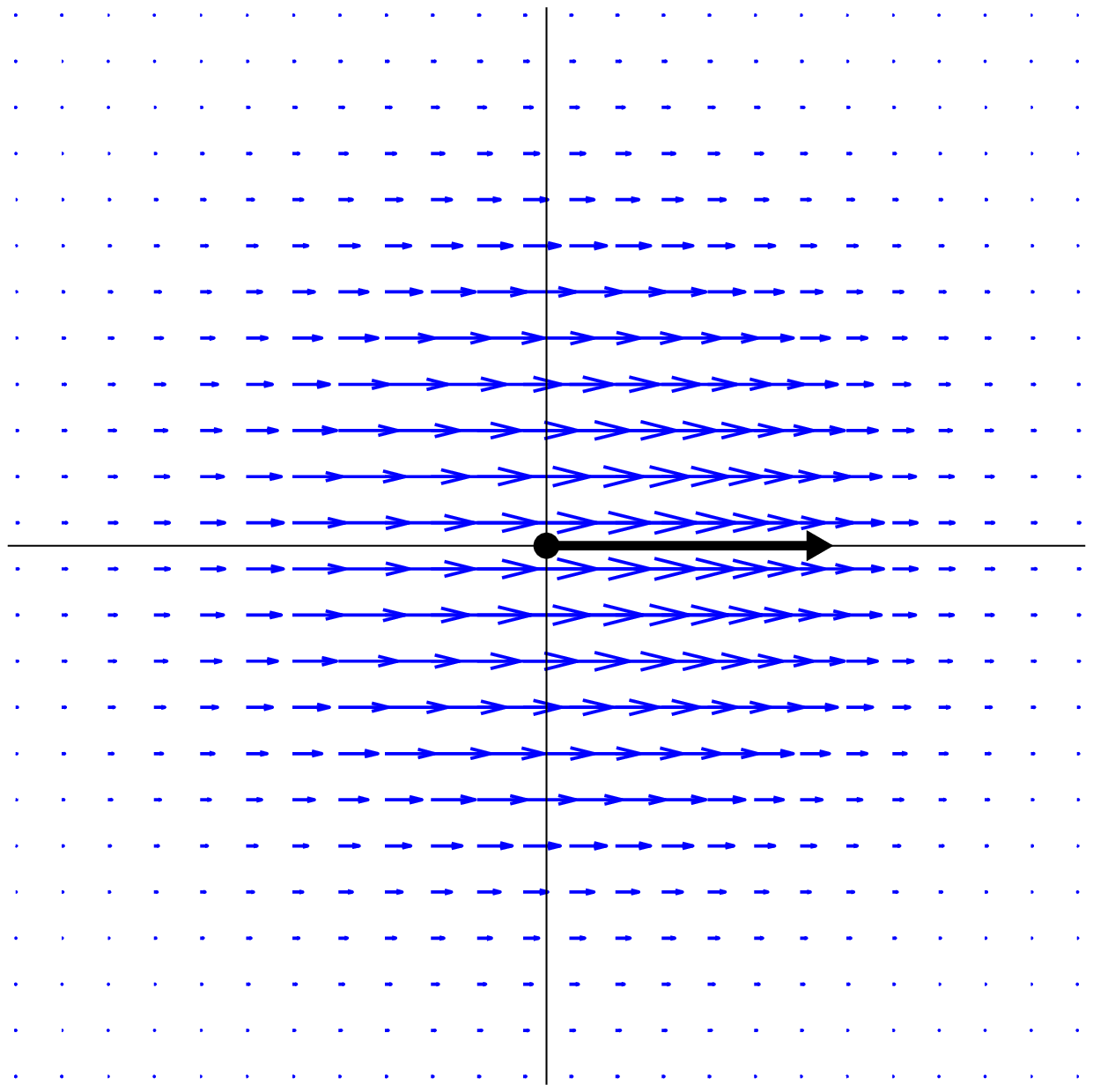}}
\put(124.75,125){\includegraphics[height=3.8cm]{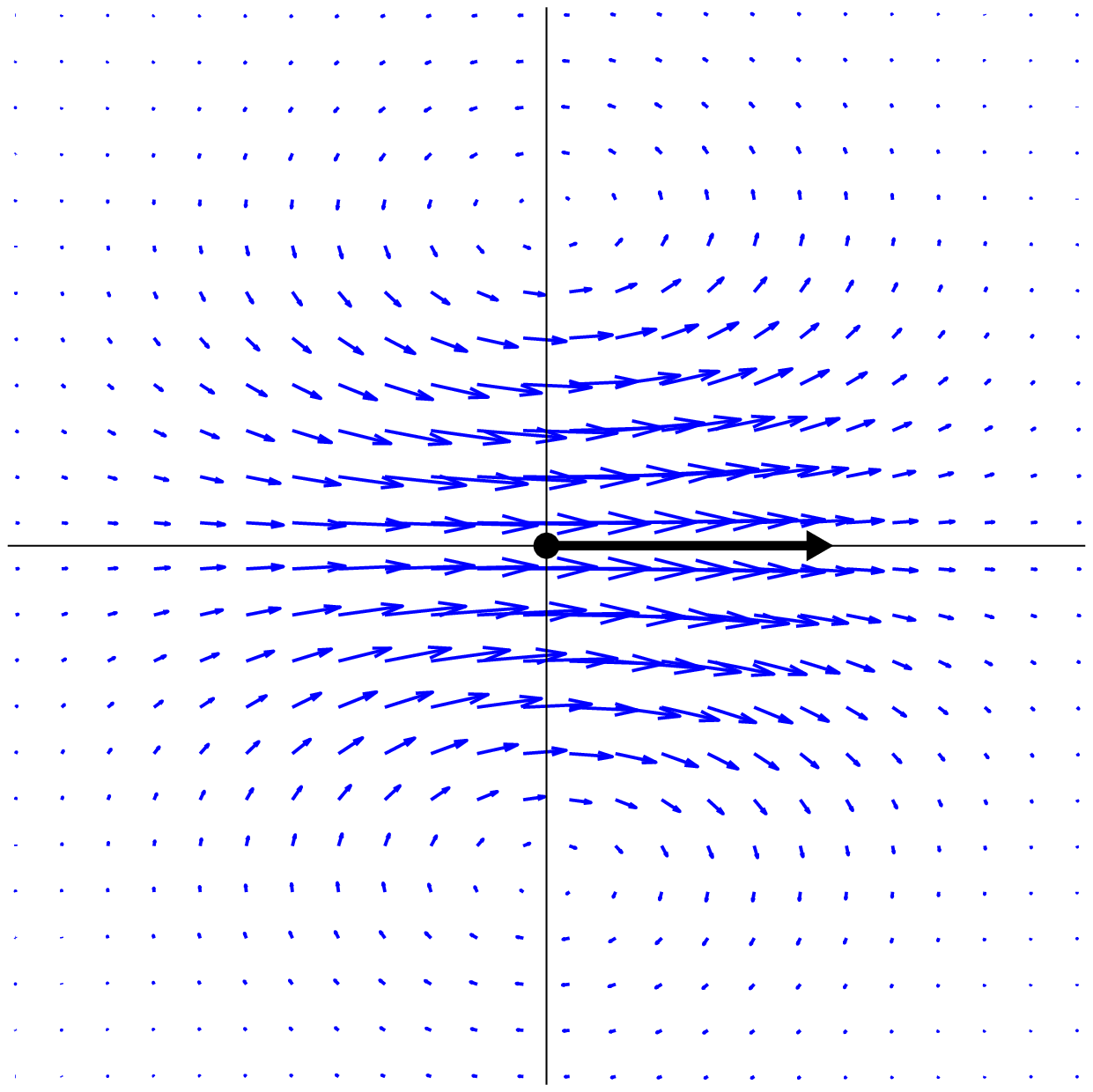}}
\put(239.75,125){\includegraphics[height=3.8cm]{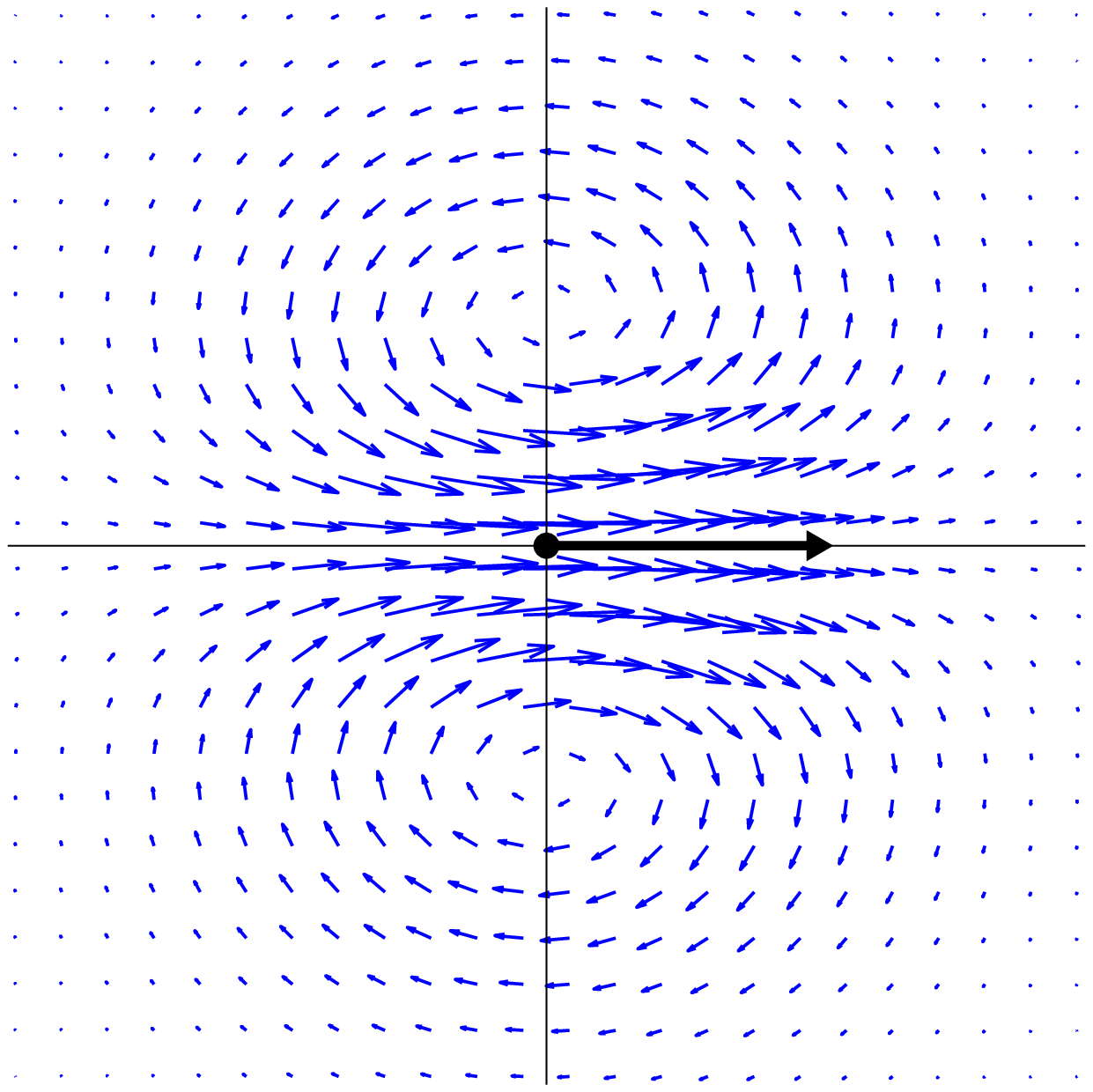}}
\put(354.75,125){\includegraphics[height=3.8cm]{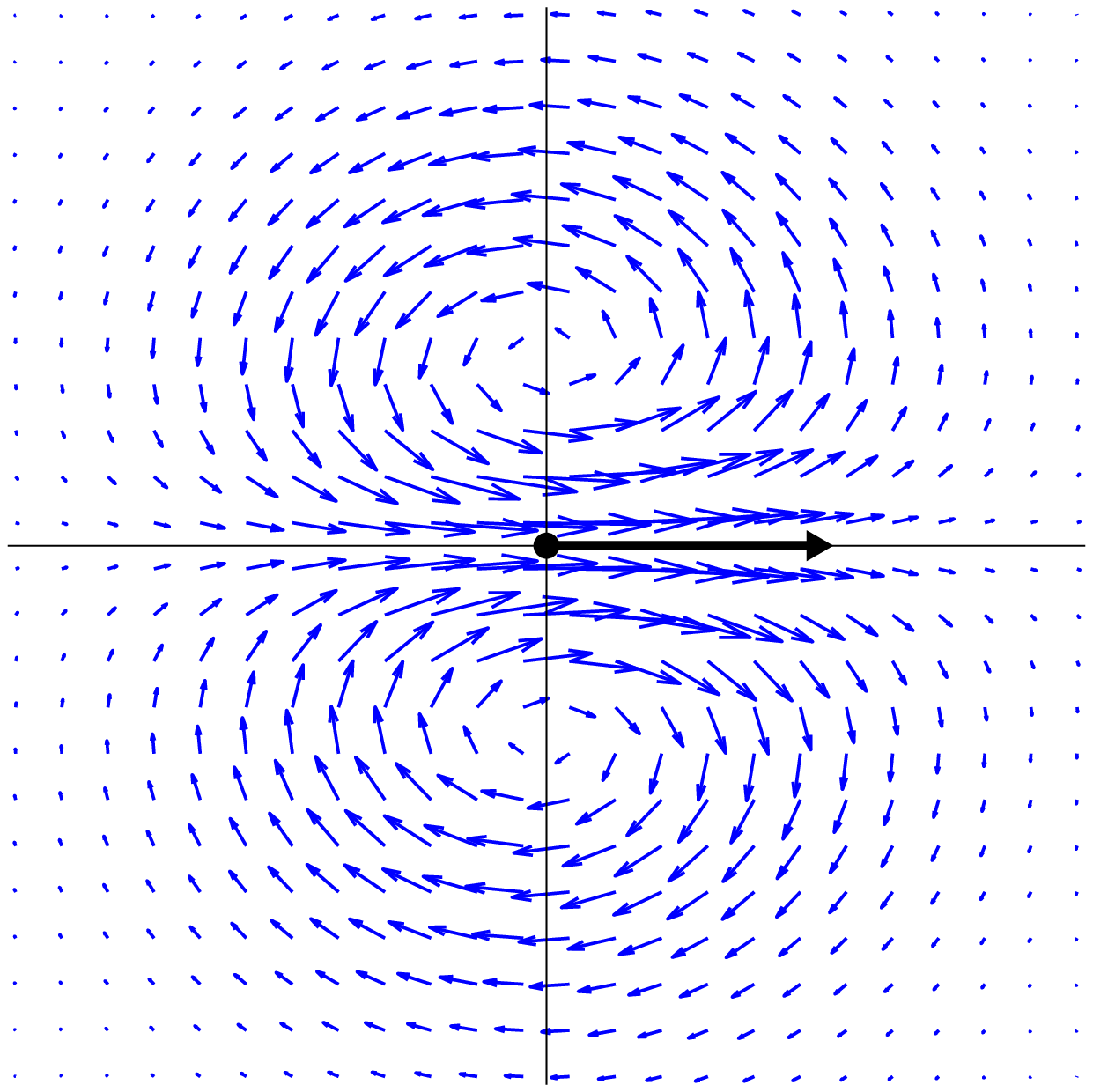}}
\put(64,120){\makebox(0,0){\small scalar}}
\put(409,120){\makebox(0,0){\small divergence-free}}
\put(64,243){\makebox(0,0){$a=0,\,b=1$}}
\put(179,243){\makebox(0,0){$a=2/3,\,b=1$}}
\put(294,243){\makebox(0,0){$a=4/3,\,b=1$}}
\put(409,243){\makebox(0,0){$a=2,\,b=1$}}
\put(-5.5,25){\rotatebox{90}{Example 2}}
\put(-5.5,150){\rotatebox{90}{ Example 1}}
\thinlines
\multiput(8,-11)(115,0){4}{\line(1,0){111}}
\multiput(8,-11)(115,0){4}{\line(0,1){121}}
\multiput(8,110)(115,0){4}{\line(1,0){111}}
\multiput(119,-11)(115,0){4}{\line(0,1){121}}
\multiput(8,114)(115,0){4}{\line(1,0){111}}
\multiput(8,114)(115,0){4}{\line(0,1){121}}
\multiput(8,235)(115,0){4}{\line(1,0){111}}
\multiput(119,114)(115,0){4}{\line(0,1){121}}
\end{picture}
\end{center}
\caption{ Transition from a scalar kernel to divergence-free (Example~1) 
and~curl-free (Example~2) kernels by changing the parameter~$a$ 
in the range~$[0,2]$, with $b=c=1$, in $d=2$ dimensions.
The vector fields~$x\mapsto\kk(x)\alpha$ are shown
for~$\alpha=e_1$, also shown in each graph.}
\label{FigTrans}
\end{figure}
\begin{helmdec}
Let~$\kk\in L^1(\Rd,\Rdd)$ be a TRI kernel
with coefficients~$(\aaa,\bbb)$, and let~$(\AAA,\BBB)$
be the coefficients of its Fourier transform, which we also
assume integrable. Using the 
functional~$M$ introduced in~\eqref{funcM}, we have~$(\aaa,\bbb)=M^{-1}(\AAA,\BBB)$. By the linearity of~$M^{-1}$ we also have
\begin{equation}
\label{dec_vf}
(\aaa,\bbb)
=
(\aaa_1,\bbb_1)
+
(\aaa_2,\bbb_2),
\quad
\mbox{where:}
\quad
\left\{
\begin{array}{l}
(\aaa_1,\bbb_1):=M^{-1}(\AAA,0)
\\
(\aaa_2,\bbb_2):=M^{-1}(0,\BBB)
\end{array}
\right.
%(\aaa_1,\bbb_1)=M^{-1}(\AAA,0)
%\quad
%\mbox{and}
%\quad
%(\aaa_2,\bbb_2)=M^{-1}(0,\BBB)
.
\end{equation}
With the above definitions, we let~$\kk=\kk_1+\kk_2$, where
\begin{equation}
\label{dec_vf2}
\kk_1(x)
:=
\aaa_1(x)\,\mathrm{Pr}^\parallel_x
+
\bbb_1(x)\,\mathrm{Pr}^\parallel_x
\qquad
\mbox{and}
\qquad
\kk_2(x)
:=
\aaa_2(x)\,\mathrm{Pr}^\parallel_x
+
\bbb_2(x)\,\mathrm{Pr}^\parallel_x;
\end{equation}
by Theorem~\ref{fund_th}
we have that for all $\alpha\in\Rd$ we may write
$\kk(x)\alpha=\kk_1(x)\alpha+\kk_2(x)\alpha$, $x\in\Rd$,
where \em 
the first term is a curl-free vector field \em and \em the second term is  instead
div-free\em\/.
In other words, 
for any integrable TRI kernel~$\kk$ and any~$\alpha\in\Rd$
this procedure allows one to perform
the \em Hodge decomposition\em~\cite{chorinmarsden,janich}
of the vector field~$x\mapsto \kk(x)\alpha$: the two terms,
curl-free and divergence-free, 
may be respectively computed 
from the coefficients~$\AAA$ and~$\BBB$ of the Fourier transform of~$\kk$
precisely %following the procedure 
%
%We may summarize this with the following critical remark:
%$$
%\boxed{
%\begin{array}{c}
%\mbox{\it Each of the two terms of the Hodge decomposition  (divergence-free
%+ curl-free)
%of}
%\\
%\mbox{\it the vector field $x\mapsto\kk(x)\alpha$ may be computed from the  
%coefficients~$\AAA$ and~$\BBB$  
%of $\kkh$.}
%\end{array}
%}
%$$
%The decomposition is performed % from the TRI kernel~$\kk$
%precisely by following the recipe 
%provided by~\eqref{dec_vf},
%i.e.~%first calculating~$(\AAA,\BBB)$ from~$(\aaa,\bbb)$, ad then 
by using the inversion formulae~\eqref{iTpar}
and~\eqref{iTper} twice, the first time with~$\BBB=0$
(to compute~$\kk_1$)
and then with~$\AAA=0$ (to compute~$\kk_2$);
note that since~$\kk$ vanishes at infinity there is no harmonic 
component in the Hodge decomposition. 
We shall call~$\kk_1$ and~$\kk_2$ the \em curl-free \em and 
\em divergence-free \em components of the kernel~$\kk$.
%Once again, this is done
%either employing tables or transforms, or numerically. 
\par
Incidentally, we note that if~$\kk\in L^1\cap L^2$
then for any~$\alpha\in\Rd$ the orthogonality
in~$L^2(\Rd,\Rd)$ of
the Hodge components~$\kk_1(\cdot)\alpha$
 and~$\kk_2(\cdot)\alpha$
 may be easily verified via Plancherel's theorem as follows:
%  
%and we consider its Hodge decomposition~\eqref{dec_vf2}, then 
%by Plancherel's theorem
$$
\big\langle\kk_1\alpha,\kk_2\alpha\big\rangle_{L^2}
=
\big\langle\kkh_1\alpha,\kkh_2\alpha\big\rangle_{L^2}
=
\int_{\Rd}
\big(
\AAA(\|\xi\|)
\mathrm{Pr}_\xi^\parallel\alpha
\big)
\cdot
\big(
\BBB(\|\xi\|)
\mathrm{Pr}_\xi^\perp\alpha
\big)
\,
d\xi=0.
$$
%for all~$\alpha\in\Rd$, so~the vector fields~$x\mapsto\kk_1(x)\alpha$ and~$x\mapsto\kk_2(x)\alpha$
%are orthogonal in~$L^2(\Rd,\Rd)$. 
The computation
of~$\kk_1$ and~$\kk_2$ from~$\kk$ via formulae~\eqref{dec_vf} and~\eqref{dec_vf2} may  
be viewed as an orthogonal
projection (in the~$L^2$ sense) of~$\kk$ onto the spaces of curl-free 
and devergence-free kernels, respectively.
\end{helmdec}
%%%
%%%
%%%
\begin{example_n} We now consider the Gaussian scalar kernel~$\kk(x)=k(\|x\|)\mathbb{I}_d$, $x\in\Rd$,
with~$k(r)=e^{-cr^2}$ for some $c>0$, and compute its Hodge decomposition. We saw earlier that its Fourier 
transform is
\begin{equation}
\label{hGauss}
\kkh(\xi)=h(\|\xi\|)\mathbb{I}_d,
\;\xi\in\Rd
\qquad
\mbox{with}
\qquad
h(\varrho)=\Big(\frac{\pi}{c}\Big)^{\mu+1}\exp\!\Big(\!-\!\frac{\pi^2\varrho^2}{c}\Big),
\qquad
\mbox{where }\mu=\frac{d}{2}-1.
\end{equation}
\par
We let~$\kk=\kk_1+\kk_2$, with~$\kkh_1(\xi)=h(\|\xi\|)\,\mathrm{Pr}^\parallel_\xi$
and~$\kkh_2(\xi)=h(\|\xi\|)\,\mathrm{Pr}^\perp_\xi$.
The coefficients~$(\aaa_1,\bbb_2)$ of~$\kk_1$ are
%i.e.~they are 
given by 
inversion formulae~\eqref{iTpar} and~\eqref{iTper} with~$\AAA=h$ and~$\BBB=0$, i.e.
\begin{subequations}
\begin{align}
\label{2iTpar}
\aaa_1(r)
&
=
\frac{2\pi}{r^\mu}
\int_0^\infty
\varrho^{\mu+1}\,h(\varrho)\,J_\mu(2\pi\varrho r)\,d\varrho
-
\frac{2\mu+1}{r^{\mu+1}}
\int_0^\infty
\varrho^{\mu}\,h(\varrho)\,J_{\mu+1}(2\pi\varrho r)\,d\varrho,
\\
\label{2iTper}
\bbb_1(r)
&=
\frac{1}{r^{\mu+1}}
\int_0^\infty
\varrho^{\mu}\,h(\varrho)\,J_{\mu+1}(2\pi\varrho r)\,d\varrho,
\end{align}
\end{subequations}
We compute the integral~\eqref{2iTper}
with~$h$ as in~\eqref{hGauss} 
using the Hankel transform~\eqref{HtfG2}, which yields
\begin{equation}
\label{d1kperp}
\bbb_1(r)=\frac{1}{2c^{\mu+1}r^{2\mu+2}}
\;\gamma\big(\mu+1,cr^2\big),
\qquad~r>0,
\end{equation}
where~$\gamma(\nu,x)=\int_0^x e^{-t}\,t^{\nu-1}dt$, $\Re\nu>0$,  is the  
lower incomplete gamma function~\cite[\S6.5.2]{abramowitz}. The other coefficients
of the kernels~$\kk_1$ and~$\kk_2$ are obtained from~\eqref{d1kperp}
simply as follows:
\begin{align}
\label{othks}
\aaa_1
&=
k-(2\mu+1)\bbb_1,
&
\aaa_2
&=
k-\aaa_1=(2\mu+1)\bbb_1,
&
\bbb_2
&=
k-\bbb_1;
\end{align}
the first one is derived from~\eqref{2iTpar}, while the other two
follow from~$k=\aaa_1+\aaa_2$ and~$k=\bbb_1+\bbb_2$ (which, in turn,
are a consequence of~$\kk=\kk_1+\kk_2$ and the fact that~$\kk$ is scalar). 
\par
When~$d=2$ (i.e.~$\mu=0$), since~$\gamma(1,x)=1-e^{-x}$
we simply have~$\bbb_1(r)=\frac{1}{2cr^2}(1-e^{-cr^2})$, and the other coefficients
are again derived from formulae~\eqref{othks} with~$\mu=0$. 
It is interesting to note that~$\kk_1$ and~$\kk_2$ are
\em not \em Gaussian like the 
curl-free and divergence-free kernels that we derived in 
Examples~\ref{exGauss2} and~\ref{exGauss},
respectively---in fact they
have much heavier ``tails'' (they decay
like~$1/r^2$ as~$r\rightarrow\infty$) than
the original scalar and Gaussian kernel~$\kk$;
%(which goes to zero much faster, like~$e^{-cr^2}$); 
however, these heavy tails cancel each other when summed. 
In the case~$c=1$ and~$d=2$
the Hodge decomposition of~$x\mapsto\kk(x)e_1$ into its curl-free 
and
divergence-free terms
is illustrated in Figure~\ref{FigHDec}; as observed above,
the two vector fields are orthogonal in the~$L^2$ sense.
\begin{figure}[t]
%\begin{center}
%\mbox{}\hfill
\noindent{}\begin{picture}(560,120)
\setlength{\unitlength}{1.20pt}
\put(10,0){\includegraphics[height=4.56cm]{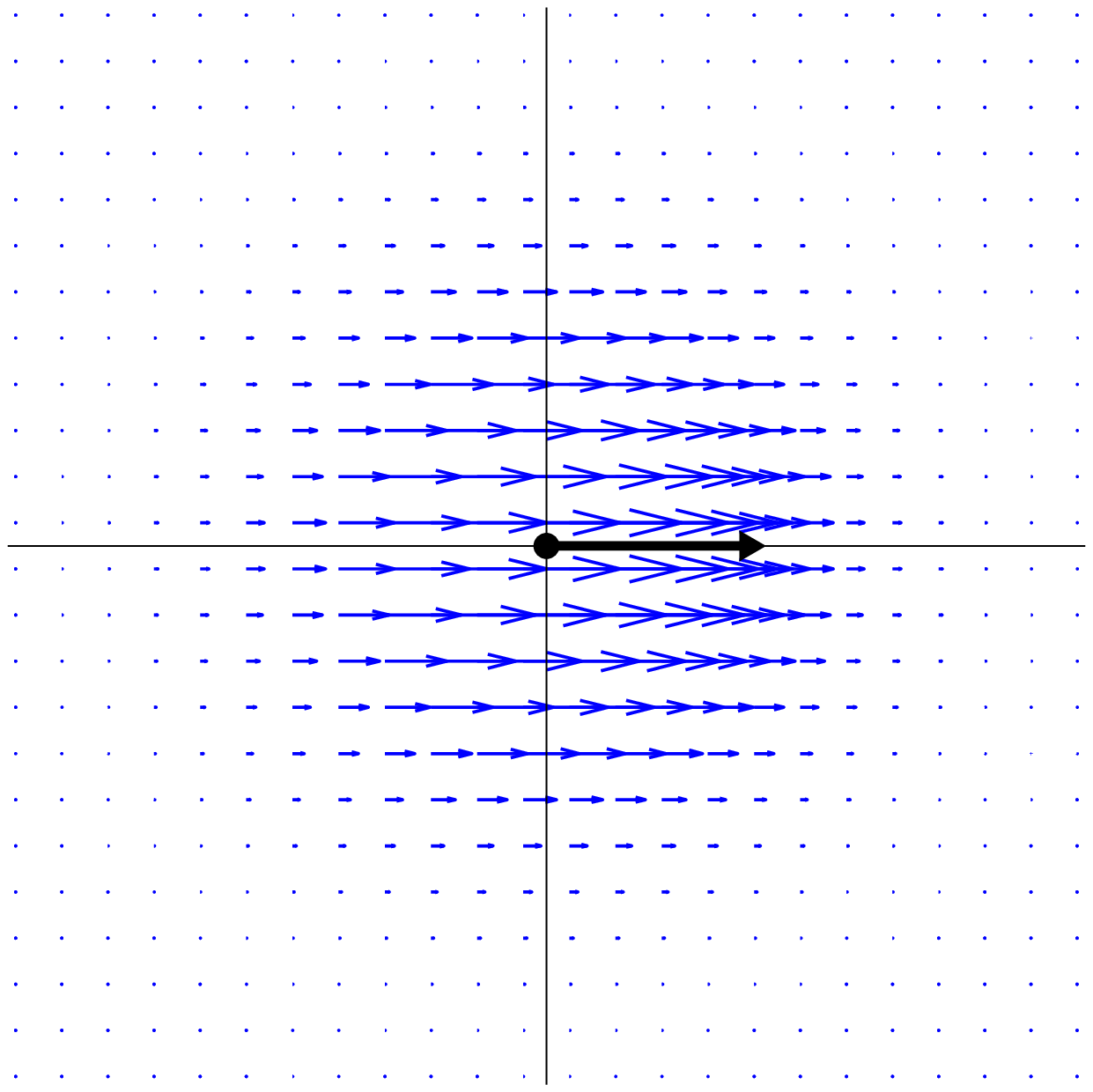}}
\put(155,0){\includegraphics[height=4.56cm]{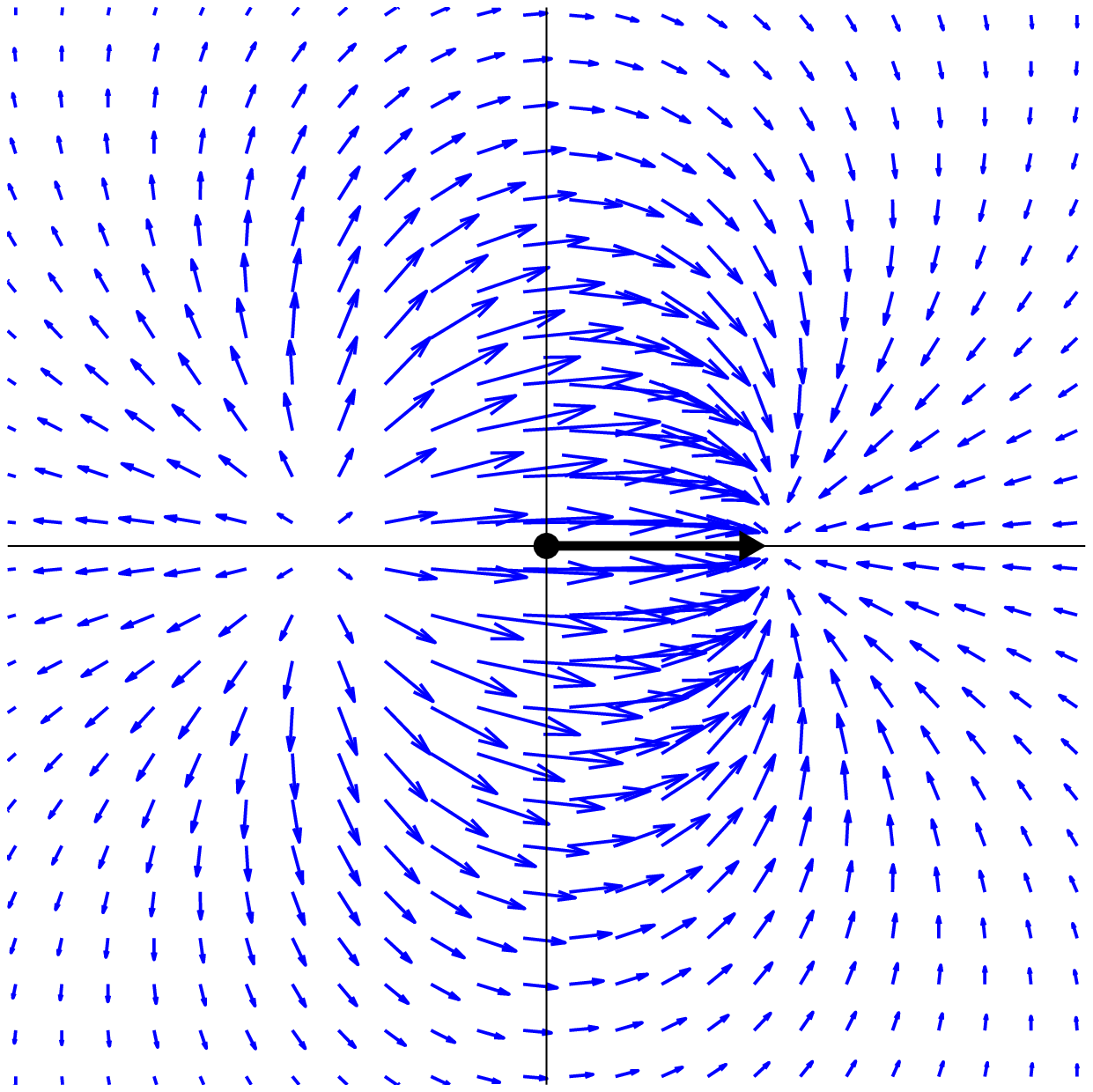}}
\put(300,0){\includegraphics[height=4.56cm]{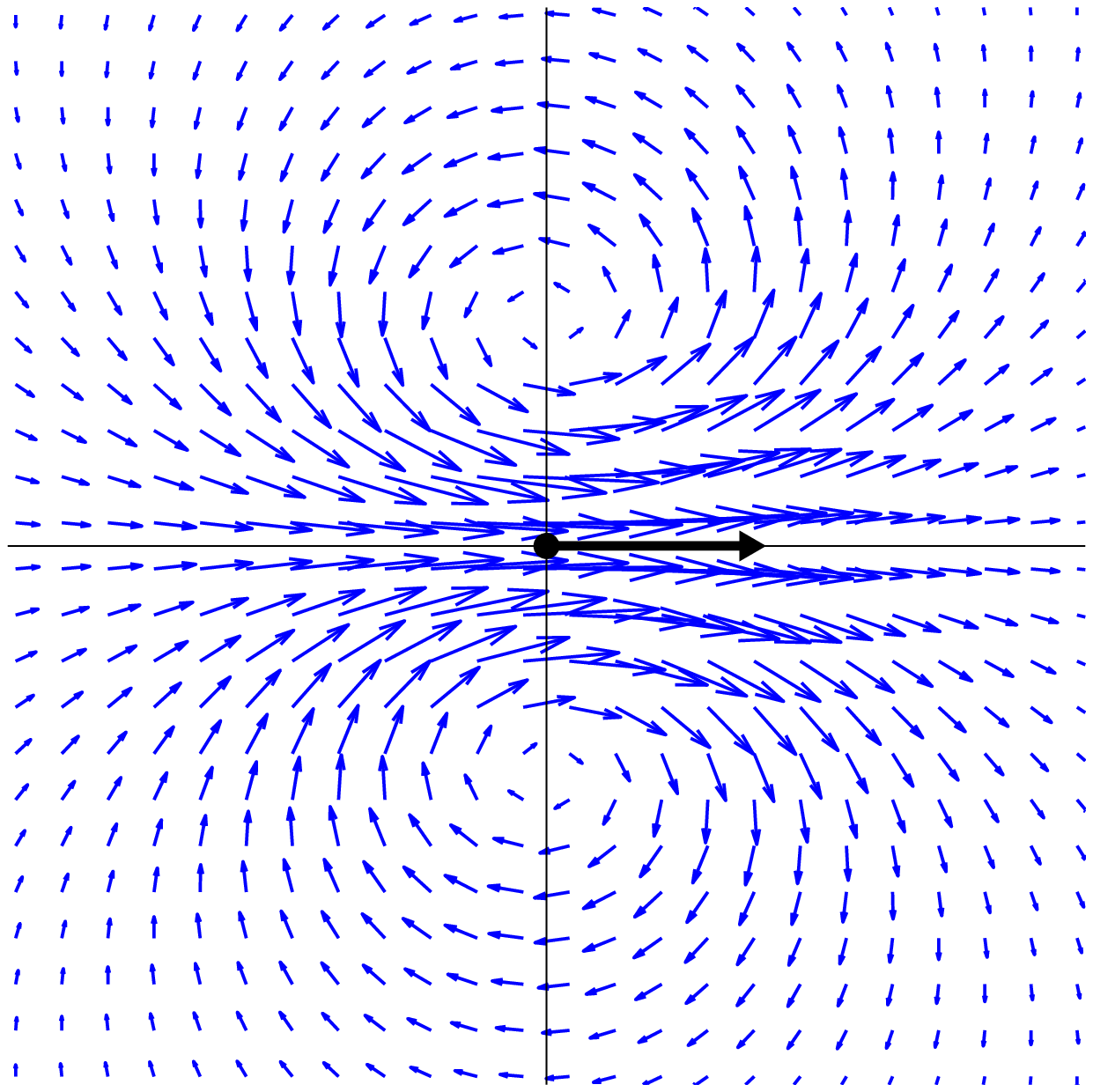}}
\put(65,-5){\makebox(0,0){\small  scalar}}
\put(210,-5){\makebox(0,0){\small curl-free}}
\put(355,-5){\makebox(0,0){\small divergence-free}}
\put(65,117){\makebox(0,0){\small $x\mapsto\kk(x) e_1$}}
\put(210,117){\makebox(0,0){\small $x\mapsto\kk_1(x) e_1$}}
\put(355,117){\makebox(0,0){\small $x\mapsto\kk_2(x) e_1$}}
\put(136,53){\makebox(0,0){\huge=}}
\put(281,53){\makebox(0,0){\huge+}}
\thinlines
\multiput(8,-11)(145,0){3}{\line(1,0){111}}
\multiput(8,-11)(145,0){3}{\line(0,1){121}}
\multiput(8,110)(145,0){3}{\line(1,0){111}}
\multiput(119,-11)(145,0){3}{\line(0,1){121}}
\end{picture}
%\end{center}
\hfill\mbox{}
\caption{Hodge decomposition of~$x\mapsto \kk(x)e_1$, when~$\kk$ is scalar and Gaussian.}
\label{FigHDec}
\end{figure}
\end{example_n}
The discussion above also suggests that a seemingly simple 
way to generate curl-free
(or divergence-free) kernels is 
to choose functions~$\AAA$ (respectively,~$\BBB$)
in~$L^1(\mathbb{R}^+,r^{d-1})$ and apply formulae~\eqref{iTpar}
and~\eqref{iTper} with $\BBB=0$ (respectively,~$\AAA=0$).
This is rather cumbersome because it involves 
the computation of Hankel transforms, analytically or numerically; 
in the next section
we shall introduce a simple technique for generating all~TRI kernels 
of interest while avoiding such tedious calculations.
\section{Construction of matrix-valued kernels from scalar kernels}
\label{constr}
In this section we illustrate a method for constructing 
matrix-valued kernels from scalar kernels---in fact, it will turn out that 
virtually all matrix-valued kernels of interest can be built this way. 
We remind the reader that a~$\cs$-admissible Hilbert space 
of $\Rd$-valued functions defined in~$\Omega$ is also a RKHS, and its kernel is such 
that~$K(\cdot,x)\alpha\in C_0^\cs(\Omega,\Rd)$ for all $x\in\Omega$  
and $\alpha\in \Rd$ (see Section~\ref{vv_rkhs}).
%%%
%%%
%%%
\subsection{Curl-free component}
\label{seccfc}
Given a $\cs$-admissible Hilbert space $H$ 
of \em scalar\em\/-valued differentiable functions defined on~$\Rd$, 
one can define the space~$V:=\{\nabla f\,|\,f\in H\}$.
Note that (i) if $\cs\geq 2$ all the elements of~$V$ are curl-free and (ii)~for all~$v\in V$
there is a unique $f\in H$ such that 
$v=\nabla f$ (in fact if~$v=\nabla f_1=\nabla f_2$ 
then~$f_1-f_2\in H$ is a constant that must be zero because 
$H\subset C_0^\cs(\Rd,\mathbb{R})$, i.e.~it vanishes at infinity).
The proposition that follows shows that 
the kernel of~$H$ induces a reproducing kernel
on~$V$.
%, and that in fact 
%essentially all RKHS of curl-free vector fields of interest can be obtained this way.
\begin{notation}
Given a \em scalar\em\/-valued 
differentiable function~$F:\Rd\times\Rd\rightarrow \mathbb{R}$
we indicate with~$\partial_{n,i}F$, $n\in\{1,2\}$, 
$i\in\{1,\ldots,d\}$  the derivative of~$F$ 
with respect to the $i^\mathrm{\,th}$ 
variable of the $n^\mathrm{th}$
set of variables; with~$\nabla_nF$, where $n\in\{1,2\}$, the gradient 
of~$F$ with respect to the $n^\mathrm{th}$
set of variables; and with~$\nabla_1\nabla_2^TF$ 
the matrix-valued
function whose $(i,j)^\mathrm{th}$ 
entry is~$\partial_{1,i}\partial_{2,j}F$, for~$1\leq i,j\leq d$.
Finally, for a vector-valued function~$\mathbf{ f}:\Rd\rightarrow\Rd$, we indicate 
with~$\nabla \mathbf{ f}$ the matrix
whose $(i,j)^\mathrm{th}$ 
entry is~$\partial_{i}\mathbf{ f}_j$, for~$1\leq i,j\leq d$.
\end{notation}
%%%
\begin{proposition}
\label{grad_H}
For fixed $r\geq1$, consider a Hilbert space 
$H\hookrightarrow C_0^\cs(\Rd,\mathbb{R})$ %Reproducing Kernel Hilbert Space
of scalar-valued functions,
with kernel $K_H:\Rd\times\Rd\rightarrow \mathbb{R}$.
Let~$V:=\{\nabla f\,|\,f\in H\}$
be endowed with the inner product
$\langle v_1,v_2  \rangle_V:=\langle f_1,f_2\rangle_H$,
%where $v_i=\nabla f_i$. 
where $v_1=\nabla f_1$ and $v_2=\nabla f_2$.
Then~$V\hookrightarrow C_0^{\cs-1}(\Rd,\Rd)$, i.e.~it 
is a $(\cs-1)$-admissible Hilbert space of $\Rd$-valued functions,
and its kernel is given by~$K_V=\nabla_1\nabla^T_2K_H$.
\end{proposition}
%%%
%%%
%%%
\begin{proof}
One can easily verify
that~$\langle\cdot,\cdot\rangle_V$ is an inner product,
that~$V$ is complete and continuously 
embedded in~$C_0^{\cs-1}(\Rd,\Rd)$.
So~$V$ is a RKHS, and for all~$x,\alpha\in\Rd$ and~$v=\nabla f\in V$, 
its kernel~$K_V$ must satisfy
\begin{align*}
\big\langle
&
K_V(\cdot,x)\alpha,
v
\big\rangle_V
=
\textstyle
\alpha\cdot v(x) 
= 
\alpha\cdot\nabla f(x) 
=
\textstyle 
\sum_{i=1}^d \alpha_i\, \partial_if(x)
=
\mbox{(by Theorem~\ref{Th_diff})}
\\
&=
%\stackrel{(\ast)}{=} 
\textstyle
\sum_{i=1}^d \alpha_i
\big\langle
\partial_{2,i}K_H(\cdot,x),f
\big\rangle_H
=
\big\langle
\alpha\cdot\nabla_2K_H(\cdot,x),f
\big\rangle_H
=
\big\langle
\nabla
\big(\alpha\cdot\nabla_2 K_H(\cdot,x)
\big),
v
\big\rangle_V,
\end{align*}
%where in~$(\ast)$ we have used Theorem~\ref{Th_diff}.
therefore $K_V(\cdot,x)\alpha=\nabla
\big(\nabla_2 K_H(\cdot,x)
\cdot\alpha
\big)=\big(\nabla_1\nabla_2^T K_H(\cdot,x)
\big)\alpha$,
where we have used the the property that $\nabla(\mathbf{ f}\cdot\alpha)=(\nabla \mathbf{ f})\alpha$
for any differentiable vector-valued function~$\mathbf{ f}:\Rd\rightarrow\Rd$.
%
%
% The latter can be computed as
%\begin{align*}
%\nabla
%&\big(\alpha\cdot\nabla_2 K_H(\cdot,x)
%\big)
%=
%\textstyle
%\nabla \big(
%\sum_{i=1}^d
%\alpha_i\,\partial_{2,i}K_H(\cdot,x)
%\big)
%=
%\sum_{i=1}^d
%\nabla \big(
%\partial_{2,i}K_H(\cdot,x)
%\big)\alpha_i
%\\
%&=
%\textstyle \sum_i
%\sum_{i=1}^d
%\left[
%\begin{array}{c}
%\sum_i
%\partial_{1,1}\partial_{2,i}K_H(\cdot,x) \alpha_i
%\\
%\sum_i
%\partial_{1,2}\partial_{2,i}K_H(\cdot,x) \alpha_i
%\\
%\vdots 
%\\
%\sum_i
%\partial_{1,d}\partial_{2,i}K_H(\cdot,x) \alpha_i
%\end{array}
%\right]
%=
%\left[
%\begin{array}{cccc}
%\partial_{1,1}\partial_{2,1}K_H(\cdot,x)
%&
%\partial_{1,1}\partial_{2,2}K_H(\cdot,x) 
%&
%\cdots
%&
%\partial_{1,1}\partial_{2,d}K_H(\cdot,x) 
%\\
%\partial_{1,2}\partial_{2,1}K_H(\cdot,x) 
%&
%\partial_{1,2}\partial_{2,2}K_H(\cdot,x) 
%&
%\cdots
%&
%\partial_{1,2}\partial_{2,2}K_H(\cdot,x) 
%\\
%\vdots 
%&
%\vdots 
%&
%\ddots 
%&
%\vdots 
%\\
%\partial_{1,d}\partial_{2,1}K_H(\cdot,x) 
%&
%\partial_{1,d}\partial_{2,2}K_H(\cdot,x) 
%&
%\cdots
%&
%\partial_{1,d}\partial_{2,d}K_H(\cdot,x) 
%\end{array}
%\right]\alpha.
%\end{align*}
By the arbitrariness of~$\alpha\in\Rd$ we conclude
that~$K_V(y,x)=\nabla_1\nabla^T_2K_H(y,x)$, for $y,x\in\Rd$. 
\end{proof}
\begin{remark}
Before proceeding to the case of translation- and rotation-invariant 
kernels %~\eqref{invker} 
 we note that
if the kernel~$K$ of a~$\cs$-admissible Hilbert space is translation-invariant,
i.e.~$K(x,y)=\kk(x-y)$, then~$\kk$ is differentiable~$2\cs$ times (see the remark
at the end of Section~\ref{gfik}).
\end{remark}
\begin{proposition} 
\label{HesskH}
Under the assumptions of Proposition~\ref{grad_H},
if $K_H(x,y)=\kk_H(x-y)$ for some function~$\kk_H:\Rd\rightarrow \mathbb{R}$,
then $K_V(x,y)=\kk_V(x-y)$, with $\kk_V=-\mathrm{Hess}\,\kk_H$ 
(i.e.~$\kk_V^{ij}=-\partial_i\partial_j \kk_H$). %, $x\in\Rd$.
\end{proposition}
\begin{proof} 
We have that~$\partial_{2,j}K_H(x,y)=-\partial_j\kk_H(x-y)$ 
and~$\partial_{1,i}\partial_{2,j}K_H(x,y)=-\partial_i\partial_j\kk_H(x-y)$,
for all $x,y\in\Rd$.
The result follows by applying Proposition~\ref{grad_H}.
\end{proof}
\begin{corollary}
\label{corCF}
Under the assumptions of Propositions~\ref{grad_H} with~$\cs\geq2$ and~\ref{HesskH}, 
for any~$\alpha\in\Rd$
the vector field~$x\mapsto\kk_V(x)\alpha$ is curl-free.
\end{corollary}
\begin{proof} Using the fact that~$(\nabla \mathbf{ f})\alpha=\nabla(\mathbf{ f}\cdot\alpha)$
for any differentiable vector-valued function~$\mathbf{ f}:\Rd\rightarrow\Rd$, we have
$\kk_V(x)\alpha=(-\mathrm{ Hess}\,\kk_H(x))\alpha=-(\nabla\nabla^T\kk_H(x))\alpha=-\nabla(\nabla\kk_H(x)\cdot\alpha)$, and we conclude immediately.
\end{proof}
\begin{proposition}
\label{prkV1}
Under the assumptions of Propositions~\ref{grad_H} 
and~\ref{HesskH},
if $\kk_H(x)=k_H(\|x\|)$ for some $k_H:\mathbb{R}^+\rightarrow \mathbb{R}$
then~the kertnel of~$V$ is TRI, i.e.~$\kk_V(x)=k_V^\parallel(\|x\|)\,\mathrm{Pr}_x^\parallel
+k_V^\perp(\|x\|)\,\mathrm{Pr}_x^\perp$ 
with
\begin{align}
\label{kV1}
%\mbox{with}\qquad
\aaa_V(r)
&=
-\frac{d^2\kscalar}{dr^2}(r)
&
&
\mbox{and }
&
\bbb_V(r)
&=
-\frac{1}{r}\frac{d \kscalar }{dr}(r), \quad r>0.
\end{align}
\end{proposition}
\begin{proof}
First note that we have 
$\displaystyle \partial_{j} \kk_H(x)=\kscalar'(\|x\|)
\, x_j/\|x\|
%\frac{x_j}{\|x\|}
$ and 
$$
%\textstyle
-\partial_{i}\partial_{j} \kk_H(x)
=
-\kscalar''(\|x\|)\frac{x_ix_j}{\|x\|^2}
-\kscalar'(\|x\|)\frac{\partial}{\partial x^i}\frac{x_j}{\|x\|}
=
-\kscalar''(\|x\|)\frac{x_ix_j}{\|x\|^2}
-\kscalar'(\|x\|)
%\big(\frac{\delta_{ij}}{\|x\|}-\frac{x_ix_j}{\|x\|^3}\big),
\Big(\frac{\delta_{ij}}{\|x\|}-\frac{x_ix_j}{\|x\|^3}\Big),
$$
where~$\delta_{ij}$ is Kronecker's symbol. 
Therefore by Proposition~\ref{HesskH}
%Since 
%$K_V(x,y)=\nabla_1\nabla_2^T K_H(x,y)$ 
it must be the case that
\begin{align*}
\kk_V(x)
&=
-\kscalar''(\|x\|)\frac{xx^T}{\|x\|}-\frac{\kscalar'(\|x\|)}{\|x\|}
\Big(\mathbb{I}_d-\frac{xx^T}{\|x\|^2}\Big)
=
-\kscalar''(\|x\|)
%\frac{xx^T}{\|x\|}
\,\mathrm{Pr}^\parallel_x
-
\frac{\kscalar'(\|x\|)}{\|x\|}
\,\mathrm{Pr}^\perp_x.
 \qedhere
\end{align*}
%whence equations~\eqref{kV1} hold. AAAA
\end{proof}
\begin{corollary} 
\label{corkCF}
Under the assumptions of Proposition~\ref{prkV1},
the kernel~$\kk_V$ is curl-free. 
\end{corollary}
When~$\cs\geq2$, the above follows from Corollary~\ref{corCF}. However, one can verify 
that~$\aaa_V$ and~$\bbb_V$, given by~\eqref{kV1}, satisfy equation~\eqref{curlzero},
and that the latter operation only involves \em two \em derivatives of~$k_H$
(and~$k_H$ is twice differentiable when~$s=1$, by Theorem~\ref{Th_diff} and the remark
at the end of Section~\ref{gfik}). Therefore Corollary~\ref{corkCF} also holds when~$s=1$.
The curl-free property may be seen in the Fourier domain as follows.
%\begin{proposition}
%\label{prftkv}
%Under the assumptions of Propositions~\ref{grad_H} 
%and~\ref{HesskH}, if~$\kk_H$ is integrable then
%\begin{equation}
%\label{ftkv}
%\widehat{\kk}_V(\xi)
%=
%(
%2\pi\|\xi\|
%)^2
%\,
%\widehat{\kk}_H(\xi)\,\mathrm{Pr}_\xi^\parallel.
%\end{equation}
%\end{proposition}
%\begin{proof} For all~$j,\ell\in\{1,\ldots,d\}$
%it is the case that
%\begin{align*}
%\widehat{\kk}^{j\ell}_V(\xi)
%=
%-\int_{\Rd}\partial_j\partial_\ell\kk_H(x)\,e^{-2\pi i \xi\cdot x}dx
%=
%-(2\pi i \xi_j)(2\pi i \xi_\ell)\int_{\Rd}\kk_H(x)\,e^{-2\pi i \xi\cdot x}dx
%=(2\pi)^2\xi_j\xi_\ell\,\widehat{\kk}_H(\xi),
%\end{align*}
%from which~\eqref{ftkv} follows immediately.
%\end{proof}
\begin{proposition}
\label{prkV1f}
Under the assumptions of Proposition~\ref{prkV1}, if~$\kk_H$ is integrable then
$\widehat{\kk}_H(\xi)=h_H(\|\xi\|)$
and 
$\widehat{\kk}_V(\xi)=\AAA_V(\|\xi\|)\,\mathrm{Pr}_\xi^\parallel
+\BBB_V(\|\xi\|)\,\mathrm{Pr}_\xi^\perp$, 
where~$\BBB_V=0$ and
\begin{equation}
\label{hparkh}
%\displaystyle
\AAA_V(\varrho)
=
(2\pi\varrho)^2\,h_H(\varrho),
\quad\mbox{with }\quad
h_H(\varrho)=
\frac{2\pi}{\varrho^{\mu}}
\int_0^\infty
\!\!\!
r^{\mu+1}\,\kscalar (r)\,J_\mu(2\pi\varrho r)\,dr
\qquad
(\mbox{as usual, } {\textstyle\mu=\frac{d}{2}-1}).
\end{equation} 
%where~$
%\displaystyle
%\mu=\frac{d}{2}-1$.
\end{proposition}
\begin{proof} 
Under the sole assumption of translation-invariance, for all~$j,\ell\in\{1,\ldots,d\}$
it is the case that
\begin{align*}
\widehat{\kk}^{j\ell}_V(\xi)
=
-\int_{\Rd}\partial_j\partial_\ell\kk_H(x)\,e^{-2\pi i \xi\cdot x}dx
=
-(2\pi i \xi_j)(2\pi i \xi_\ell)\int_{\Rd}\kk_H(x)\,e^{-2\pi i \xi\cdot x}dx
=(2\pi)^2\xi_j\xi_\ell\,\widehat{\kk}_H(\xi),
\end{align*}
from which we have $\widehat{\kk}_V(\xi)
=
(
2\pi\|\xi\|
)^2
\,
\widehat{\kk}_H(\xi)\,\mathrm{Pr}_\xi^\parallel$.
When~$\kk_H$ is also rotation-invariant,
by Proposition~\ref{FT} 
we must have
$\widehat{\kk}_H(\xi)=h_H(\|\xi\|)$ with 
$h_H(\varrho)$ as in~\eqref{hparkh}, 
and we conclude immediately.
%This concludes the proof.
\end{proof}
\begin{figure}
\hspace{.5cm}
\begin{tikzpicture}[scale=.75]
\draw (-10,3) node [right] {\small\it Space~$H$:};
\draw (-10,0) node [right] {\small\it Space~$V$:};
\draw (-4.9,3) node [right] {$\kk_H(x)=k_H(\|x\|)$};
\draw (5.6,3) node [right] {$\kkh_H(x)=h_H(\|x\|)$};
\draw (-6.45,0) node [right] {$\kk_V: (k^\parallel_{V},k^\perp_{V})
=\big(-k''_H,-\frac{1}{r}k'_H\big)$};
\draw (4,0) node [right] {$\kkh_{V}: (h^\parallel_{V},h^\perp_{V})
=\big((2\pi\varrho)^2h_H,0\big)$};
%
% horizontal arrows
\draw [|->] (1.2,3) -- (3.5,3);
\draw [|->] (1.2,0) -- (3.5,0);
\draw (2.3,3.35) node {\small$\mathcal{F}$};
\draw (2.3,.4) node {\small$\mathcal{F}$};
% vertical arrows
\draw [|->] (-3,2.3) -- (-3,.6);
\draw [|->] (7.5,2.3) -- (7.5,.6);
\draw (-3,1.45) node [right] {\small$-\mathrm{Hess}$};
\draw (7.5,1.45) node [right] {\small$(-\mathrm{Hess})^\wedge$};
\end{tikzpicture}
\caption{Commutative diagram illustrating the construction, from 
the scalar-valued kernel~$\kk_H:\Rd\rightarrow \mathbb{R}$,
of the matrix-valued, curl-free kernel~$\kk_V:\Rd\rightarrow\Rdd$,
with interpretation in the Fourier domain.}
\label{FigCurlF}
\end{figure}
The construction of the curl-free kernel~$\kk_V$ %:\Rd\rightarrow\Rdd$
from the scalar-valued kernel~$\kk_H$ %:\Rd\rightarrow\mathbb{R}$
is summarized  in Figure~\ref{FigCurlF}. We have used
the compact notation~$(\aaa_V,\bbb_V)$ to represent
the coefficients of~$\kk_V$ and~$(\AAA_V,\BBB_V)$ to represent
those of its Fourier transform~$\kkh_V$. The dual operator of the negative Hessian
in the Fourier domain is the map~$h_H\mapsto((2\pi\varrho)^2h_H,0)$.
Under the assumptions of Proposition~\ref{prkV1f}
we may actually compute 
the corresponding ``initial'' kernel~$\kk_H$
in terms of the non-zero 
coefficient~$\AAA_V$
of~$\kkh_V$; in fact, using 
the inversion formula~\eqref{invhank}
and the fact that~$h_H(\varrho)=(2\pi\varrho)^{-2}\AAA_V(\varrho)$ yields:
%$$
%\begin{equation}
%\label{inv_Htf}
%\int_0^\infty
%\Big\{
%\int_0^\infty
%f(r) \, J_\mu(2\pi\varrho r) \,r\,dr
%\Big\} J_\mu(2\pi t\varrho)\, \varrho\,d\varrho=\frac{f(t)}{(2\pi)^2},
%\quad t>0.
%$$
%\end{equation}
%In fact, a simple manipulation of~\eqref{hparkh} using the above
%equation yields the following result.
%\begin{corollary}
%Under the assumptions of Proposition~\ref{prkV1f},
%it is the case that
%the function~$k_H$ may be expressed in terms of the
%coefficient~$\AAA_V$ as follows:
\begin{equation}
\label{kHhV}
\kk_H(x)=k_H(\|x\|),\;\; x\in\Rd,
\qquad
\mbox{with}
\qquad
k_H(r)
=
\frac{1}{2\pi r^\mu}
\int_0^\infty
\varrho^{\mu-1}\,\AAA_V(\varrho)\, J_\mu(2\pi r \varrho) \,d\varrho,
\;\;
r>0.
\end{equation}
%\end{corollary}
\par
We conclude that 
any~integrable, TRI, curl-free kernel~$\kk_V$  (i.e.~with nonnegative~$\AAA_V$,
and~\mbox{$\BBB_V=0$}) of a~$(\cs-1)$-admissible Hilbert space~$V$
may be derived with the procedure described in Propositions~\ref{grad_H} and~\ref{prkV1} 
(ultimately,
by applying formulae~\eqref{kV1}) from the kernel~\mbox{$\kk_H(x)=k_H(\|x\|)$}
of a $\cs$-admissible Hilbert space of scalar-valued functions, where~$k_H$
is given by formula~\eqref{kHhV}. This proves the
generality of the above method for constructing RKHS of curl-free vector fields
with TRI kernels.
\begin{paragraph}{Example~\ref{exGauss2}, revisited again.} 
Consider a RKHS of scalar valued functions defined in~$\Rd$,
denoted by~$H$,
with kernel~$\kk_H(x)=k_H(\|x\|)$, where
$k_H(r)=\frac{b}{2c}e^{-cr^2}$. As long as~$b>0$ and~$c>0$
it is the case that the Hilbert space is~$\cs$-admissible, for all~$\cs$. Formulae~\eqref{kV1}
yield
\begin{equation}
\label{cfGauss}
\aaa_V(r)=be^{-cr^2}
\qquad
\mbox{and}
\qquad
\bbb_V(r)=(b-2bcr^2)e^{-cr^2},
\end{equation} 
which are precisely the equations of the coefficients
of the kernel in Example~\ref{exGauss2} with~$a=2bc$;
that is, along the diagonal line of the domain~$D_2$,
which corresponds to curl-free kernels. See Figure~\ref{wedgeint}(b). 
\end{paragraph}
\begin{example_n}[Bessel-type curl-free kernels] 
Consider the 
scalar-valued $\kk_H(x)=k_H(\|x\|)$,
with %~$k_H$ given by 
\begin{equation}
\label{kscbess}
k_H(r) = C_0\,
\Big(\frac{r}{\sigma}\Big)^{\ce-\frac{\dd}{2}}\,
K_{\ce-\frac{\dd}{2}}\!\Big(\frac{r}{\sigma}\Big), \qquad r\geq 0,
\end{equation}
where~$K_\nu$ is a modified Bessel function of order~$\nu$.
When % the constant is given by
$C_0=C(\sigma,d,\ce):=(2^{\ce+\frac{\dd}{2}-1}\pi^{\frac{\dd}{2}}\,\Gamma(\ce)
\,\sigma^\dd)^{-1}$, as in~\eqref{k_bess}, it is 
the Green's function of~$L=(1-\sigma^2\Delta)^\ell$. 
The corresponding
Hilbert space~$H$ is $\cs$-admissible, i.e.~$H\hookrightarrow C_0^\cs(\Rd,\mathbb{R})$,
if~$\ce\geq\cs+d/2$. 
To apply our procedure, we must have at least~$s\geq1$. %We need the following result:
%We remind the reader that~$K_\nu$ is a modified Bessel function of order~$\nu$.
By Lemma~\ref{lembess} in Appendix~\ref{appBessel}, 
applying~\eqref{kV1} to the kernel~\eqref{kscbess} yields
\begin{align}
\label{cfBess}
\aaa_V(r)
&
=
\frac{C_0}{\sigma^2}
\Big(\frac{r}{\sigma}\Big)^{\nu-1}
\Big\{
(2\nu-1)\,
K_{\nu-1}
\Big(\frac{r}{\sigma}\Big)
-
\frac{r}{\sigma}\,
K_\nu
\Big(\frac{r}{\sigma}\Big)
\Big\}
&
&\mbox{and}
&
\bbb_V(r)
&=
\frac{C_0}{\sigma^2}
\Big(\frac{r}{\sigma}\Big)^{\nu-1}K_{\nu-1}\Big(\frac{r}{\sigma}\Big),
\end{align}
where~$\nu=\ce-\frac{d}{2}$. From the asymptotic expansion for large arguments 
of~$K_\nu$, given by~\eqref{KnuLarge} in Appendix~\ref{appBessel},
one sees that the kernel with coefficients~\eqref{cfBess}
has heavier tails than the ``Gaussian'' case~\eqref{cfGauss}.
\end{example_n}
\subsection{Divergence-free component}
We have seen how one can create the matrix-valued kernel
a curl-free RKHS by taking the ``double gradient'' of a scalar-valued kernel.
We will now illustrate how a the kernel of a divergence-free RKHS 
may be created by taking the ``double curl'' of a matrix-valued scalar kernel
(i.e.~with equal diagonal entries). For simplicity 
we shall limit ourselves to the case~$d=3$, which is the most relevant in applications.
\begin{lemma}
\label{Lem_inj}
A 1-admissible Hilbert space~$W\hookrightarrow C_0^1(\Rt,\Rt)$
is divergence-free if and only if the linear map 
\mbox{$\mathrm{curl}:W\rightarrow C_0(\Rt,\Rt)$}
is injective.
\end{lemma} 
\begin{proof}
First assume that~$W$ is divergence-free. If~$\mathrm{curl}\,w=0$ then
$w=\nabla f$ for some scalar function~$f$; therefore
$\Delta f=\mathrm{div}\,w=0$ because~$W$ is divergence-free,
and in fact~$f$ must be zero because it 
is harmonic and it vanishes at infinity.
Vice versa, assume that curl is injective on~$W$. Take an arbitrary~$w\in W$
and let $w=w_1+w_2$ be its Hodge decomposition
(i.e.~$\mathrm{curl}\, w_1=0$ and $\mathrm{div}\, w_2=0$). We have
$\mathrm{curl}\, w=\mathrm{curl}\, w_2$, so by the injectivity assumption
$w=w_2$, whence~$w_1=0$. Therefore $\mathrm{div}\, w=\mathrm{div}\, w_1=0$.
\end{proof} 
\par
Let~$\cs\geq 1$ and consider a~$\cs$-admissible Hilbert 
space~$W \hookrightarrow C_0^\cs(\Rt,\Rt)$, and the 
space of vector fields $V:=\{\mathrm{curl}\, w\,|\,w\in W\}$. 
By Lemma~\eqref{Lem_inj}
the ``primitive''~$w\in W$ of~$v=\mathrm{curl}\,w\in V$
is unique if and only if~$W$ is divergence-free.
With this condition,
%for the uniqueness of the ``primitive'' of~$v=\mathrm{curl}\, w$, 
we can
certainly follow a path that is similar to the one described in the previous section
and endow~$V$ with a reproducing kernel induced by
the kernel of~$W$.  
%The proposition that follows makes this precise.
\begin{notation}
Given a differentiable matrix-valued function~$G:\Rt\times\Rt\rightarrow\Rtt$
we indicate with $\mathrm{curl}^\mathrm{C}_n G$
the matrix-valued funcion whose~$i^\mathrm{\,th}$ column (with $i=1,2,3$)
is the curl of the~$i^\mathrm{\,th}$ column of~$G$; 
the curl is computed with respect to the~$n^\mathrm{th}$ 
set of variables of~$G$ (with $n=1$ or $2$). 
The definition of $\mathrm{curl}^\mathrm{R}_n\, G$
is obtained by substituting the word ``column'' with the word ``row''.
If the matrix-valued functions of only three variables
$\mathbf{ g}:\Rt\rightarrow\Rtt$ is differentiable,
%When~$G$ depends on only one set of variables, i.e.~$G:\Rt\rightarrow\Rtt$,
$\mathrm{curl}^\mathrm{C} \mathbf{ g}$ (respectively, $\mathrm{curl}^\mathrm{R} \mathbf{ g}$)
simply indicates the matrix whose~$i^\mathrm{\,th}$
column (row) is the curl of the~$i^\mathrm{\,th}$ column (row) of~$\mathbf{ g}$,
for $i=1,2,3$.
\end{notation}
%%%
\begin{proposition}
\label{curl_W}
For fixed $\cs\geq1$, consider a Hilbert 
space $W\hookrightarrow C_0^\cs(\Rt,\Rt)$ 
of divergence-free vector fields,
with kernel $K_W:\Rt\times\Rt\rightarrow \Rt$.
Let~$V:=\{\mathrm{curl}\, w\,|\,w\in W\}$
be endowed with the inner product
$\langle v_1,v_2  \rangle_V:=\langle w_1,w_2\rangle_W$,
where $v_1=\mathrm{curl}\, w_1$ and $v_2=\mathrm{curl}\, w_2$.
Then~$V\hookrightarrow C_0^{\cs-1}(\Rt,\Rt)$, i.e.~it is a 
$(\cs-1)$-admissible Hilbert 
space of $\Rt$-valued functions, and its
kernel is given 
by~$K_V=\mathrm{curl}_1^\mathrm{C}\,
\mathrm{curl}_2^\mathrm{R}\, K_W$.
\end{proposition}
%%%
\begin{proof}
We indicate with~$K_W^{i,\cdot}$ and~$K_W^{\cdot,j}$  the $i^{\mathrm{th}}$ row
and the~$j^{\mathrm{th}}$ column of~$K$, respectively.
One can easily verify
that~$\langle\cdot,\cdot\rangle_V$ is an inner product, $V$ is complete, 
and~$V \hookrightarrow C_0^{\cs-1}(\Rt,\Rt)$;
therefore $V$ is a RKHS.
For arbitrary~$w\in W$ let~$v=\mathrm{curl}\, w$;
for all~$x,\alpha\in\Rt$ the 
kernel~$K_V$ must be such that
\begin{align*}
&\big\langle
K_V(\cdot,x)\alpha,
v
\big\rangle_V
=
\alpha\cdot v(x) 
= 
\alpha\cdot\mathrm{curl}\, w(x) 
\\
&=
\alpha_1\big(\partial_2w_3(x)-\partial_3w_2(x)\big)
+
\alpha_2\big(\partial_3w_1(x)-\partial_1w_3(x)\big)
+
\alpha_3\big(\partial_1w_2(x)-\partial_2w_1(x)\big)
=\mbox{(by Theorem~\ref{Th_diff})}
\\
&=
\big\langle
\alpha_1\big(\partial_{2,2}K_W^{\cdot,3}(\cdot,x)-\partial_{2,3}K_W^{\cdot,2}(\cdot,x)\big)
+
\alpha_2\big(\partial_{2,3}K_W^{\cdot,1}(\cdot,x)-\partial_{2,1}K_W^{\cdot,3}(\cdot,x)\big)
\\
&
\hspace{6.115cm}+
\alpha_3\big(\partial_{2,1}K_W^{\cdot,2}(\cdot,x)-\partial_{2,2}K_W^{\cdot,1}(\cdot,x)\big)
,w
\big\rangle_W
\\
&=
%\Bigg\langle
%\left[
%\begin{array}{c}
%\alpha\cdot\mathrm{curl}_2\,K_{1,\cdot}(\cdot,x)
%\\
%\alpha\cdot\mathrm{curl}_2\,K_{2,\cdot}(\cdot,x)
%\\
%\alpha\cdot\mathrm{curl}_2\,K_{3,\cdot}(\cdot,x)
%\end{array}
%\right]
%,w
%\Bigg\rangle_W
%=
\big\langle
\big(
\mathrm{curl}_2^\mathrm{R}\,K_W(\cdot,x)
\big)\alpha,w
\big\rangle_W
=
\big\langle
\mathrm{curl}\big[\big(
\mathrm{curl}_2^\mathrm{R}\,K_W(\cdot,x)
\big)\alpha\big],v
\big\rangle_V
%\\
%&
=
%\stackrel{(\ast)}{=}
%\Bigg\langle
%\mathrm{curl}\left[
%\begin{array}{c}
%\sum_j c_{1j}(\cdot)\alpha_j
%\\
%\sum_j c_{2j}(\cdot)\alpha_j
%\\
%\sum_j c_{3j}(\cdot)\alpha_j
%\end{array}
%\right]
%,v
%\Bigg\rangle_V
%=
%\Bigg\langle
%\sum_{j=1}^3
%\mathrm{curl}\left[
%\begin{array}{c}
%c_{1j}(\cdot)
%\\
%c_{2j}(\cdot)
%\\
%c_{3j}(\cdot)
%\end{array}
%\right]\alpha_j
%,v
%\Bigg\rangle_V
%=
\big\langle
\big[
\mathrm{curl}^\mathrm{C}_1\,
\mathrm{curl}^\mathrm{R}_2\,
K_W(\cdot,x)
\big]
\alpha
,v
\big\rangle_V,
\end{align*}
where in the last step we have used the 
fact that~$\mathrm{curl}(\mathbf{ g}\alpha)
=
(\mathrm{curl}^\mathrm{C}\mathbf{ g})\alpha$
for any differentiable matrix-valued function~$\mathbf{ g}:\Rt\rightarrow\Rtt$
and any $\alpha\in\Rt$. This concludes the proof.
\end{proof}
In the translation-invariant case we have the following immediate consequence.
\begin{proposition}
\label{curl2}
Under the assumptions of Proposition~\ref{curl_W},
if $K_W(x,y)=\kk_W(x-y)$ for some function~$\kk_W:\Rt\rightarrow \mathbb{R}^{3\times 3}$,
then $K_V(x,y)=\kk_V(x-y)$ 
with $\kk_V=-\mathrm{curl}^\mathrm{C}\mathrm{curl}^\mathrm{R}\, \kk_W.$
\end{proposition}
\begin{corollary}
\label{corcurl2}
Under the assumptions of Propositions~\ref{curl_W} with~$r\geq2$ and~\ref{curl2}, 
for any choice of~$\alpha\in\Rd$
the vector field~$x\mapsto\kk_V(x)\alpha$ is divergence-free.
\end{corollary}
\begin{proof} Since~$(\mathrm{curl}^\mathrm{C} \mathbf{ g})\alpha
=\mathrm{curl}(\mathbf{ g}\alpha)$
for any differentiable matrix-valued function~\mbox{$\mathbf{ g}:\Rt\rightarrow\Rtt$}, we have
$\kk_V(x)\alpha=(-\mathrm{curl}^\mathrm{C} \mathrm{curl}^\mathrm{R}\kk_W(x))\alpha
%=-(\nabla\nabla^T\kk_H(x))\alpha
=-\mathrm{curl}[(\mathrm{curl}^\mathrm{R}\kk_W(x))\alpha]$, and we conclude immediately.
\end{proof}
%\begin{remark}
\begin{lemma}
\label{rotrotk}
For any twice differentiable matrix-valued
function~$\mathbf{ g}=[g_{ij}]_{1\leq i,j \leq 3}:
\mathbb{R}^3\rightarrow \mathbb{R}^{3\times3}$ we have:
\par\vspace*{-.7cm}
\begin{align*}
&\mathrm{curl}^\mathrm{C}\mathrm{curl}^\mathrm{R}\,\mathbf{ g}
=
\mathrm{curl}^\mathrm{R}\mathrm{curl}^\mathrm{C}\,\mathbf{ g} =
\\
&
\!
\small
\left[
\!\!
\begin{array}{ccc}
\partial_2^2g_{33} 
-
\partial_2\partial_3(g_{23}+g_{32})
+
\partial_3^2g_{22} 
&
\!\!
-
\partial_1\partial_2g_{33}
\!+\!
\partial_1\partial_3g_{23}
\!+\!
\partial_2\partial_3g_{31}
\!-\!
\partial_3^2 g_{21}
&
\partial_1\partial_2 g_{32}
\!-\!
\partial_1\partial_3 g_{22}
\!-\!
\partial_2^2 g_{31}
\!+\!
\partial_2\partial_3 g_{21}
\\
-
\partial_1\partial_2g_{33}
\!+\!
\partial_1\partial_3g_{32}
\!+\!
\partial_2\partial_3g_{13}
\!-\!
\partial_3^2 g_{12}
&
\partial_1^2g_{33} 
-
\partial_1\partial_3(g_{13}+g_{31})
+
\partial_3^2g_{11} 
&
-
\partial_1^2 g_{32}
\!+\!
\partial_1\partial_2 g_{31}
\!+\!
\partial_1\partial_3 g_{12}
\!-\!
\partial_2\partial_3 g_{11}
\\
\partial_1\partial_2 g_{23}
\!-\!
\partial_1\partial_3 g_{22}
\!-\!
\partial_2^2 g_{13}
\!+\!
\partial_2\partial_3 g_{12}
&
-
\partial_1^2 g_{23}
\!+\!
\partial_1\partial_2 g_{13}
\!+\!
\partial_1\partial_3 g_{21}
\!-\!
\partial_2\partial_3 g_{11}
&
\partial_1^2g_{22} 
-
\partial_1\partial_2(g_{12}+g_{21})
+
\partial_2^2g_{11} 
\end{array}
\!\!
\right]
\!.
\end{align*}
\end{lemma}
%\noindent{The} 
\par The 
above lemma may be proven by direct computation.
We note that~$\mathrm{curl}^\mathrm{R}$ and $\mathrm{curl}^\mathrm{C}$ 
\em commute\em\/, 
and if~$\mathbf{ g}$ is symmetric 
(i.e.~$g_{ij}=g_{ji}$ for~$1\leq i,j \leq 3$) then so is the matrix-valued function
$\mathrm{curl}^\mathrm{C}\mathrm{curl}^\mathrm{R}\,\mathbf{ g}$. 
\par
\begin{remark} Before proceeding further we should %stop for a moment, take a 
%breath and 
note that the above procedure and results require the 
$\cs$-admissible space of vector fields~$W$ 
(that we use to build the new space~$V$)
to be divergence-free. 
This sounds like a severe limitation because, as we said at the beginning
of the section, our goal here is to formulate a procedure
that \em generates \em RKHS of divergence-free vector fields.
Having to start from such a space does not seem very useful.  
The next results, that refer to the TRI case, solve this problem.
In order to interpret the results in the Fourier domain,
we shall assume that the kernels are integrable.
\end{remark}
\begin{proposition}
\label{prSameCC}
Let~$\kk_{1}$, $\kk_{2}:\Rt\rightarrow\Rtt$ be the TRI, integrable kernels of two
$\cs$-admissible Hilbert
spaces, with $\cs\geq 2$.
If their Fourier transforms have the 
the same divergence-free coefficient~$\BBB$, i.e.
$$
\kkh_{1}(\xi)
=
\AAA_{1}(\|\xi\|)
\,\mathrm{Pr}_\xi^\parallel
+
\BBB_{1}(\|\xi\|)
\,\mathrm{Pr}_\xi^\perp
\quad
\mbox{and}
\quad
\kkh_{2}(\xi)
=
\AAA_{2}(\|\xi\|)
\,\mathrm{Pr}_\xi^\parallel
+
\BBB_{2}(\|\xi\|)
\,\mathrm{Pr}_\xi^\perp,
\quad
\mbox{with }
\BBB_{1}=\BBB_{2}:=\BBB,
$$
then~$-\mathrm{curl}^\mathrm{C}\mathrm{curl}^\mathrm{R} \kk_{1}=
-\mathrm{curl}^\mathrm{C}\mathrm{curl}^\mathrm{R} \kk_2=:\kk$,
and~$\kk$ is a divergence-free kernel.
\end{proposition}
\begin{proof}
We first observe that if the Fourier transform a kernel~$\kk$ has coefficient~$\BBB=0$
then, for all~$\alpha\in\Rt$, it is the case that~$\mathrm{curl}(\kk\alpha)=0$;
but~$\mathrm{curl}(\kk\alpha)=(\mathrm{curl}^\mathrm{C}\kk)\alpha$, therefore
$\mathrm{curl}^\mathrm{C}\kk=0$ and~$\mathrm{curl}^\mathrm{C}
\mathrm{curl}^\mathrm{R}\kk=0$ (by the commutativity of~$\mathrm{curl}^\mathrm{C}$
and~$\mathrm{curl}^\mathrm{R}$, see Lemma~\ref{rotrotk}). Now, 
by the linearity of the operator~$M$ (and its inverse), introduced by equation~\eqref{funcM},
for~$i=1,2$ we have
\begin{align}
\kk_i(x)
&=\aaa_{i}(\|x\|)
\mathrm{Pr}_\xi^\parallel
+
\bbb_{i}(\|x\|)
\mathrm{Pr}_x^\perp,
\quad\mbox{with }
\nonumber
\\
\label{2comps}
(\aaa_i,\bbb_i)&=M^{-1}(\AAA_i,\BBB)=M^{-1}(\AAA_i,0)+M^{-1}(0,\BBB);
\end{align}
note that~$\BBB_1=\BBB_2=\BBB$.
We now apply~$\mathrm{curl}^\mathrm{C}
\mathrm{curl}^\mathrm{R}$ to both~$\kk_1$
and~$\kk_2$ and, by the first part of the proof,
we have that the component that corresponds to the 
first term on the right-hand side of~\eqref{2comps} vanishes for both~$\kk_1$ and~$\kk_2$;
what survives is the second component, which is the same for both kernels. 
Therefore~$\mathrm{curl}^\mathrm{C}\mathrm{curl}^\mathrm{R} \kk_{1}=
\mathrm{curl}^\mathrm{C}\mathrm{curl}^\mathrm{R} \kk_2$. The resulting kernel is divergence-free
because, for~$i=1,2$, we have
$(\mathrm{curl}^\mathrm{C}\mathrm{curl}^\mathrm{R} \kk_{i})\alpha=
%(\mathrm{curl}^\mathrm{R}\mathrm{curl}^\mathrm{R} \kk_i)\alpha=
\mathrm{curl}[(\mathrm{curl}^\mathrm{R} \kk_i)\alpha]$, for all~$\alpha\in\Rd$.
\end{proof}
Consequently,  
if the coefficients of the Fourier transform~$\kkh$ of a kernel~$\kk$ are~$(\AAA,\BBB)$,
then the coefficients of the Fourier transform of~$\mathrm{curl}^\mathrm{C}
\mathrm{curl}^\mathrm{R}\kk$ 
must be of the type~$(0,h_\ast)$, where~$h_\ast$ only depends on~$\BBB$.
Going back to our original problem of inducing a RKHS
structure on~$V:=\{\mathrm{curl}\, w\,|\,w\in W\}$ from the one on~$W$ (Propositions~\ref{curl_W} 
and~\ref{curl2}),
this suggests an alternative way of computing the kernel~$\kk_V$.
\par
Namely, we build a \em new \em auxiliary  
Hilbert space~$W_0\hookrightarrow C_0^\cs(\Rt,\Rt)$,
with a \em scalar \em kernel, 
i.e.~of the form~$\kk_{W_0}(x)=k_{W_0}(\|x\|)\mathbb{I}_3$, $x\in\Rt$.
%
%~$\kk_{W_0}$ is a ``scalar version'' of~$\kk_W$, as follows. 
Assuming~$\kk_W$ is integrable and divergence-free, 
the coefficients of its Fourier transform~$\kkh_W$ are~$(\AAA_W,\BBB_W)=(0,h)$, 
for some nonegative function~$h:\mathbb{R}^+\rightarrow \mathbb{R}$.
We define the scalar kernel~$\kk_{W_0}$ via its Fourier transform, as 
\begin{equation}
\label{khw0}
\kkh_{W_0}(\xi):=h_{W_0}(\|\xi\|)\mathbb{I}_3,
\;\;\xi\in\Rt,\quad \mbox{ with}\quad h_{W_0}:=h; 
\end{equation}
in other words, the coefficients of~$\kk_{W_0}$ and~$\kkh_{W_0}$ 
are 
%as functions of the types~\eqref{invker}
%and~\eqref{khat_gen}, their coefficients are 
respectively~$\aaa_{W_0}=\bbb_{W_0}=k_{W_0}$ 
and~$\AAA_{W_0}=\BBB_{W_0}=h_{W_0}$, where we have set~$h_{W_0}=h$,
i.e.~equal to the non-zero coefficient of the Fourier transform of~$\kk_W$.
Equivalently, and perhaps more interestingly, 
$\kk_{W_0}$ is \em the only scalar kernel whose divergence-free
component is precisely\em\/~$\kk_W$.
The next proposition
relates the kernels~$\kk_W$ and~$\kk_{W_0}$
explicitly in the spatial domain (as opposed to the Fourier domain); this is summarized by the first and third rows of the diagram of Figure~\ref{FigScW2}.
\begin{figure}
\hspace{.15cm}
\begin{tikzpicture}[scale=.75]
\draw (-13,3) node [right] {\small\it Space~$W$:};
\draw (-13,0) node [right] {\small\it Space~$V$:};
\draw (-13,-3) node [right] {\small\it Space~$W_0$:};
%
%\draw (-7,2.7) node {\small\it Spatial};
%\draw (-7,2.3) node {\small\it domain:};
%\draw (-7,0.2) node {\small\it Fourier};
%\draw (-7,-0.2) node {\small\it domain:};
%
\draw (-7.8,3) node [right] {$\kk_W: (\aaa_W,\bbb_W)$};
\draw (3,3) node [right] {$\kkh_W: (\AAA_W,\BBB_W)=(0,h)$};
\draw (-9.25,0) node [right] {$\kk_{V}: (k^\parallel_{V},k^\perp_{V})
=\big(
\textstyle
-\frac{2}{r}k'_{W_0}\,,\,-\frac{1}{r}k'_{W_0}\!\!-\!k''_{W_0}
\big)$};
%\draw (-1.7,0.2) node {\small\it given by};
%\draw (-1.7,-0.3) node {\small\it Prop.~\ref{p_ccW0}};
\draw (3,0) node [right] {$\kkh_{V}: (h^\parallel_{V},h^\perp_{V})
=\big(0,(2\pi\varrho)^2h\big)$};
%\draw (-10,-3) node [right] {$\kk_{W_0}: (\aaa_{W_0},\bbb_{W_0})=
%\big(\frac{1}{2}\aaa_W+\bbb_W,\frac{1}{2}\aaa_W+\bbb_W\big)$};
\draw (-9.5,-3) node [right] {$\kk_{W_0}(x)=k_{W_0}(\|x\|)\mathbb{I}_3$,
\, $k_{W_0}=\frac{1}{2}\aaa_W+\bbb_W$};
%\draw (-1.5,-2.8) node {\small\it given by};
%\draw (-1.5,-3.25) node {\small\it Prop.~\ref{prOpS}};
\draw (3,-3) node [right] {$\kkh_{W_0}: (\AAA_{W_0},\BBB_{W_0})=(h,h)$};
%
%
%\draw (3,2.5) node [right] {$\kk_{V}: (k^\parallel_{W_0},k^\perp_{W_0})=$};
%
% horizontal arrows
\draw [|->] (0.4,3) -- (2.7,3);
\draw [|->] (0.4,0) -- (2.7,0);
\draw [|->] (0.4,-3) -- (2.7,-3);
%
%\draw (1.95,2.85) node {\small$-\mathrm{curl}^C\mathrm{curl}^R$};
%\draw (1.95,.4) node {\small$(-\mathrm{curl}^C\mathrm{curl}^R)^\wedge$};
\draw (1.5,3.35) node {\small$\mathcal{F}$};
\draw (1.5,.4) node {\small$\mathcal{F}$};
\draw (1.5,-2.65) node {\small$\mathcal{F}$};
%
% vertical arrows
\draw [|->] (-5.8,2.3) -- (-5.8,.6);
\draw [|->] (5.95,2.3) -- (5.95,.6);
\draw [|->] (-5.8,-2.4) -- (-5.8,-.7);
\draw [|->] (5.95,-2.4) -- (5.95,-.7);
\draw (-5.8,1.45) node [right] {\small$-\mathrm{curl}^\mathrm{C}\mathrm{curl}^\mathrm{R}$};
\draw (5.95,1.45) node [right] {\small$(-\mathrm{curl}^\mathrm{C}\mathrm{curl}^\mathrm{R})^\wedge$};
\draw (-5.8,-1.55) node [right] {\small$-\mathrm{curl}^\mathrm{C}\mathrm{curl}^\mathrm{R}$};
\draw (5.95,-1.55) node [right] {\small$(-\mathrm{curl}^\mathrm{C}\mathrm{curl}^\mathrm{R})^\wedge$};
\end{tikzpicture}
\caption{Construction of the divergence-free 
kernel~$\kk_{V}$ %=k_{W_0}(\|x\|)\mathbb{I}_3$,
from the divergence-free kernel~$\kk_{W}$, and (equivalently)
from the auxiliary scalar 
kernel~$\kk_{W_0}$, %(x)=k_{W_0}(\|x\|)\mathbb{I}_3$, $x\in\Rt$,
with interpretation in the Fourier domain.}
\label{FigScW2}
\end{figure}
%The first and third rows of Figure~\ref{FigScW2} illustrate this (in both domains).
\begin{proposition}
\label{prOpS}
Under the assumptions of Proposition~\ref{curl_W}, 
if the kernel of~$W$ is TRI and integrable, 
with Fourier transform~$\kkh_{W}=h(\|\xi\|)\,\mathrm{Pr}_\xi^\perp$,
$\xi\in\Rd$, then the kernel~$\kk_{W_0}$, defined in the Fourier domain 
by~\eqref{khw0}, has the form~$\kk_{W_0}=k_{W_0}(\|x\|)\mathbb{I}_3$, $x\in\Rt$, with
%
%the map~$S$
%is given by:
%$
%(\aaa_W,\bbb_W)
%\mapsto
%(k_{W_0},k_{W_0}),$
%\mbox{with}
%with 
$$
k_{W_0}=\frac{1}{2}\,\aaa_W+\bbb_W.
$$
\end{proposition}
\begin{proof} We have that~$
\displaystyle 
k_{W_0}(r)=\frac{2\pi}{r^{\mu}}
\int_0^\infty\varrho^{\mu+1}\,h(\varrho)\,J_{\mu}(2\pi\varrho r)d\varrho$,
with $\displaystyle \mu=\frac{d}{2}-1=\frac{1}{2}$ (since~$d=3$). In addition,
since~$(\AAA_W,\BBB_W)=(0,h)$,
by~Proposition~\ref{invT}
it is the case that
\begin{align*}
\aaa_W(r)
&
%\textstyle
=
\frac{2\mu+1}{r^{\mu+1}}
\int_0^\infty
\varrho^{\mu}\,h(\varrho)\,J_{\mu+1}(2\pi\varrho r)\,d\varrho
\qquad \mbox{(where~$2\mu+1=d-1=2$)}
\\
\mbox{and}\quad
\bbb_W(r)
&
%\textstyle
=
\frac{2\pi}{r^\mu}
\int_0^\infty
\varrho^{\mu+1}\,h(\varrho)\,J_\mu(2\pi\varrho r)\,d\varrho
-
\frac{1}{r^{\mu+1}}
\int_0^\infty
\varrho^{\mu}\,h(\varrho)\,J_{\mu+1}(2\pi\varrho r)\,d\varrho
\;=\; k_{W_0}(r)-\frac{1}{2} \aaa_W(r).\;\;\qedhere
\end{align*}
%by Proposition~\ref{invT}, since~$(\AAA_W,\BBB_W)=(0,h)$. 
%An elementary manipulation yields the final result. 
\end{proof}
%%%
By Proposition~\ref{prSameCC}, since
$\kk_W$ and~$\kk_{W_0}$ have the same coefficient~$\BBB$ in the Fourier domain, 
it must be the case that $\mathrm{curl}^\mathrm{C}\mathrm{curl}^\mathrm{R} \kk_{W}=
\mathrm{curl}^\mathrm{C}\mathrm{curl}^\mathrm{R} \kk_{W_0}$;
the right-hand side is simpler to compute because~$\kk_{W_0}$
is characterized by only one coefficient (i.e.~the function~$k_{W_0}$) in the spatial domain. 
This is done next.
%%%
\begin{proposition}
\label{p_ccW0}
If~$W_0\hookrightarrow C_0^\cs(\Rt,\Rt)$, with~$\cs\geq 1$, has scalar 
kernel~$\kk_{W_0}(x)=k_{W_0}(\|x\|)\mathbb{I}_3$, $x\in\Rt$, 
and we define~$\kk_V:=-\mathrm{curl}^\mathrm{C}\mathrm{curl}^\mathrm{R} \kk_{W_0}$, 
then~$\kk_V(x)=k_V^\parallel(\|x\|)\mathrm{Pr}_x^\parallel
+k_V^\perp(\|x\|)\mathrm{Pr}_x^\perp$ 
with
\begin{align}
\label{kV2}
\aaa_V(r)
&=
-\frac{2}{r}\frac{dk_{W_0}}{dr}(r)
&
&\mbox{and}
&
\bbb_V(r)
&=
-\frac{1}{r}\frac{d k_{W_0} }{dr}(r)
-\frac{d^2 k_{W_0}}{dr^2}(r),
\quad
r>0.
\end{align}
\end{proposition}
\begin{proof}
%Assume now that the kernel of~$W$ is TRI and scalar, i.e.~$K_W(x,y)
%=\kvector(\|x-y\|)\mathbb{I}_3$
%for some function $k_W:\mathbb{R}^+\rightarrow \mathbb{R}$.
We first compute the curl ``by rows'' of~$\kk_{W_0}$:
$$
\mathrm{curl}^\mathrm{R}\, \kk_{W_0}(x)
=\frac{\kvector'(\|x\|)}{\|x\|}
\left[
\begin{array}{ccc}
0 & x_3 & -x_2 
\\
-x_3 & 0 & x_1 
\\
x_2 & -x_1 & 0 
\end{array}
\right]
=f(x)\;
\Big[
\,v_1(x)
\,\Big|\,
v_2(x)
\,\Big|\,
v_3(x)
\, \Big],
$$
%where~$f(x,y):=\frac{\kvector'(\|x-y\|)}{\|x-y\|}$
where~$f(x):=k_{W_0}'(\|x\|)/\|x\|$,
 and the~$v_i$'s are the columns of the above matrix. Now we have
%$\mathrm{curl}_1 (fv_1)=\nabla_1 f\times v_1 +f\, \mathrm{curl}_1\, v_1$,
%where
%$$
\begin{align*}
\nabla f(x)
&=\Big(
\frac{k_{W_0}''(\|x\|)}{\|x\|^2}
-
\frac{k_{W_0}'(\|x\|)}{\|x\|^3}
\Big)
x
\quad
\mathrm{and}
\quad
\mathrm{curl}\, v_i(x) = 2e_i, \mbox{ for } i=1,2,3,%\quad\mbox{ so that}
\end{align*}
%$$
so that 
\begin{align*}
\displaystyle
%\\
\mathrm{curl} (fv_1)
&=\nabla f\times v_1 +f\, \mathrm{curl}\, v_1
%\\
%&
=
\Big(
\frac{k_{W_0}''(\|x\|)}{\|x\|^2}
-
\frac{k_{W_0}'(\|x\|)}{\|x\|^3}
\Big)
\left[
\begin{array}{c}
x_2^2+x_3^2
\\
-x_1x_2
\\
-x_1x_3
\end{array}
\right]+2\,\frac{k_{W_0}'(\|x\|)}{\|x\|} e_1.
\end{align*}
Similar computations hold for $\mathrm{curl} (fv_2)$
and
$\mathrm{curl} (fv_3)$, which lead to
\begin{align*}
\kk_V(x)
=-
\mathrm{curl}^\mathrm{C}\,
\mathrm{curl}^\mathrm{R}\,
\kk_{W_0}
&=
-
\Big(
\frac{k_{W_0}''(\|x\|)}{\|x\|^2}
-
\frac{k_{W_0}'(\|x\|)}{\|x\|^3}
\Big)
\Big(
\|x\|^2\mathbb{I}_3
- 
%\frac{xx^T}{\|x\|^2}
xx^T
\Big)
-2\,\frac{ k_{W_0}'(\|x\|)}{\|x\|}\mathbb{I}_3
\\
&=
-2\,\frac{k_{W_0}'(\|x\|)}{\|x\|} \mathrm{Pr}^\parallel_{x}
-\Big(
k_{W_0}''(\|x\|)+\frac{k_{W_0}'(\|x\|)}{\|x\|}
\Big)
 \mathrm{Pr}^\perp_{x}.
 \qedhere
\end{align*}
\end{proof}
\begin{corollary}
\label{corkDF}
Under the assumptions of Proposition~\ref{p_ccW0}, the elements of~$V$ are divergence-free.
\end{corollary}
%\begin{proof}
%It is the case that 
%$\displaystyle \frac{2}{r}\big(\aaa_{W_0}(r)-\bbb_{W_0}(r)\big)+\frac{d\aaa_{W_0}}{dr}(r)=0$,
%thus Corollary~\ref{cor_div} applies.
%\end{proof}
When~$\cs\geq2$, we already knew that the above was true from Corollary~\ref{corcurl2}. However (similarly to what happens in the construction of curl-free kernel, described in Section~\ref{seccfc}) one can easily verify 
that~$\aaa_V$ and~$\bbb_V$, given by~\eqref{kV2}, satisfy equation~\eqref{divzero} for~$d=3$,
and that once again the latter operation only involves \em two \em derivatives of~$k_{W_0}$,
which is twice differentiable when~$r=1$. Therefore Corollary~\ref{corkDF} also holds when~$s=1$.
The kernel~$\kk_V$ can be characterized in the Fourier domain as follows.
\begin{lemma}
\label{rotrotk2}
If~$g:\Rt\rightarrow \mathbb{R}$
is a twice-differentiable scalar-valued function and we define~$\mathbf{ g}=g\mathbb{I}_3$,
then it is the case 
that~$\mathrm{curl}^\mathrm{C}\mathrm{curl}^\mathrm{R} 
\mathbf{ g}= (\Delta \mathbb{I}_3-\mathrm{Hess})g.$ 
\end{lemma}
\begin{proof}
%It follows directly from 
By Lemma~\ref{rotrotk}, % with~$g_{ij}=g\delta_{ij}$, 
$
%\displaystyle
\mathrm{curl}^\mathrm{C}\mathrm{curl}^\mathrm{R}\mathbf{ g}
=
\!\left[\!
\begin{array}{ccc}
(\partial_2^2\!+\!\partial_3^2)g & -\partial_1\partial_2g & -\partial_1\partial_3g \\
-\partial_1\partial_2g & (\partial_1^2\!+\!\partial_3^2)g & -\partial_2\partial_3g \\
-\partial_1\partial_3g & -\partial_2\partial_3g & (\partial_1^2\!+\!\partial_2^2)g
\end{array}\!\right]
=\big(\Delta \mathbb{I}_3-\mathrm{Hess}\big)g
$.
%with~$g_{ij}=g \delta_{ij}$
%(here~$\delta_{ij}$ is Kronecker's delta).
\end{proof}
%The curl-free property may be seen in the Fourier domain as follows.
\begin{proposition} 
\label{prCCF}
Under the assumptions of Proposition~\ref{p_ccW0},
if~$\kk_{W_0}$ is integrable and~$\kkh_{W_0}(\xi)=h(\|\xi\|)\mathbb{I}_3$, $\xi\in\Rt$,
then $\kkh_V(\xi)=\AAA_V(\|\xi\|)\,\mathrm{Pr}_\xi^\parallel+
\BBB_V(\|\xi\|)\,\mathrm{Pr}_\xi^\perp$, with
$\AAA_V=0$ and $\BBB_V(\varrho)=(2\pi\varrho)^2h(\varrho)$.
%$$
%h^\perp_V(\varrho)=(2\pi\varrho)^2h(\varrho)
%=
%\frac{(2\pi)^3}{\varrho^{\mu-2}}
%\int_0^\infty
%r^{\mu+1}k_{W_0}(r)J_{\mu}(2\pi\varrho r)dr,\quad \varrho>0
%\qquad(\mathrm{where}\; \mu=\frac{1}{2})
%$$
%where $\mu=\frac{1}{2}$.
\end{proposition} 
\begin{proof}
We have~$\kk_{W_0}(x)=g(x)\mathbb{I}_3$ with~$g(x):=k_{W_0}(\|x\|)$, $x\in\Rt$;
by Lemma~\ref{rotrotk2}, 
$\kk_V=-\mathrm{curl}^\mathrm{C}\mathrm{curl}^\mathrm{R}\,\kk_{W_0}=(\mathrm{Hess}-\Delta \mathbb{I}_3)g$.
%we have that
%\begin{align*}
%\kk_V(x)&=-\mathrm{curl}^\mathrm{C}\mathrm{curl}^\mathrm{R}\,\kk_{W_0}(x)
%=
%-\!\left[\!
%\begin{array}{ccc}
%(\partial_2^2\!+\!\partial_3^2)F(x) & -\partial_1\partial_2F(x) & -\partial_1\partial_3F(x) \\
%-\partial_1\partial_2F(x) & (\partial_1^2\!+\!\partial_3^2)F(x) & -\partial_2\partial_3F(x) \\
%-\partial_1\partial_3F(x) & -\partial_2\partial_3F(x) & (\partial_1^2\!+\!\partial_2^2)F(x)
%\end{array}\!\right]
%=\big(\mathrm{Hess}-\Delta \mathbb{I}_3\big)F(x).
%\end{align*}
Since~$%\displaystyle
\int_{\Rd}\partial_j\partial_\ell g(x)\, e^{-2\pi i\xi\cdot x}dx
%= (2\pi i \xi_j)(2\pi i \xi_\ell) \int_{\Rd} F(x)\, e^{-2\pi i\xi\cdot x}dx
=
-(2\pi)^2\xi_j\xi_\ell\,\widehat{g}(\xi)$ and $\widehat{g}(\xi)=h(\|\xi\|)$,
we may finally compute 
\begin{align*}
\widehat{\kk}_V(\xi)&
=
%(2\pi)^2
%\Bigg(
%\left[\begin{array}{ccc}
%\xi_1^2+\xi_2^2+\xi_3^2 & 0 & 0 \\
%0 & \xi_1^2+\xi_2^2+\xi_3^2 & 0 \\
%0 & 0 & \xi_1^2+\xi_2^2+\xi_3^2
%\end{array}
%\right]
%-
%\left[\begin{array}{ccc}
%\xi_1^2 & \xi_1\xi_2 & \xi_1\xi_3 \\
%\xi_1\xi_2 & \xi_2^2 & \xi_2\xi_3 \\
%\xi_1\xi_3 & \xi_2\xi_3 & \xi_3^2
%\end{array}
%\right]
%\Bigg)\,
%h(\|\xi\|)
%\\
%&
%=
(2\pi)^2\big(\|\xi\|^2\mathbb{I}_3-\xi\xi^T\big)\,h(\|\xi\|)
=(2\pi\|\xi\|)^2\,h(\|\xi\|)\,\mathrm{Pr}_\xi^\perp.
\qedhere
\end{align*}
%In conclusion, $\widehat{\kk}_V(\xi)
%=
%\AAA_V(\|\xi\|)\, \mathrm{Pr}_\xi^\parallel
%+
%\BBB_V(\|\xi\|)\, \mathrm{Pr}_\xi^\perp
%$
%with~$\AAA_V=0$ and~$\BBB_V(\varrho)=(2\pi\varrho)^2h(\varrho)$.
\end{proof}
The situation is completely 
illustrated by the commutative diagram in Figure~\ref{FigScW2}.
In summary, given an arbitrary $\cs$-admissible Hilbert space~$W$ 
of divergence-free vector fields with kernel~$\kk_W$
(i.e.~given the corresponding coefficient in the frequency domain~$\BBB_W=h$), 
we can always build
another space~$W_0$
with a \em scalar \em kernel~$\kk_{W_0}$ such that~$\mathrm{curl}^\mathrm{C}\mathrm{curl}^\mathrm{R} \kk_{W}=
\mathrm{curl}^\mathrm{C}\mathrm{curl}^\mathrm{R} \kk_{W_0}$. 
Then by Proposition~\ref{curl2} %and~\ref{prSameCC}, 
the 
kernel of the space~$V:=\{\mathrm{curl}\, w\,|\,w\in W\}$
may simply be computed as~$\kk_V=-\mathrm{curl}^\mathrm{C}\mathrm{curl}^\mathrm{R} \kk_{W_0}$, i.e.~by
equations~\eqref{kV2}.
\par
In fact, given the arbitrariness of~$\BBB_{W}=h$, we may start from any
Hilbert space~$W_0\hookrightarrow C_0^\cs(\Rt,\Rt)$
with a scalar kernel, compute the new coefficients %of~$\kk_V$
via~\eqref{kV2} and these 
will correspond to the kernel of $V:=\{\mathrm{curl}\, w\,|\,w\in W\}$ for \em some \em 
$\cs$-admissible space~$W$ of divergence-free vector fields, whose kernel~$\kk_W$ 
\em we do not need to
make explicit\em\/. This provides us with a practical procedure for
building kernels of divergence-free vector fields. 
Since~$h(\varrho)=(2\pi\varrho)^{-2}\AAA_V(\varrho)$,
we finally note that
by formula~\eqref{invhank} we have
%
%in fact, using using 
%the inversion formula~\eqref{invhank}
%and the fact that~$h_H(\varrho)=(2\pi\varrho)^{-2}\AAA_V(\varrho)$ yields:
%$$
%\begin{equation}
%\label{inv_Htf}
%\int_0^\infty
%\Big\{
%\int_0^\infty
%f(r) \, J_\mu(2\pi\varrho r) \,r\,dr
%\Big\} J_\mu(2\pi t\varrho)\, \varrho\,d\varrho=\frac{f(t)}{(2\pi)^2},
%\quad t>0.
%$$
%\end{equation}
%In fact, a simple manipulation of~\eqref{hparkh} using the above
%equation yields the following result.
%\begin{corollary}
%Under the assumptions of Proposition~\ref{prkV1f},
%it is the case that
%the function~$k_H$ may be expressed in terms of the
%coefficient~$\AAA_V$ as follows:
\begin{equation}
\label{kHhV2}
\kk_{W_0}(x)=k_{W_0}(\|x\|)\mathbb{I}_3,\;\; x\in\Rt,
\qquad
\mbox{with}
\qquad
k_{W_0}(r)
=
\frac{1}{2\pi r^\mu}
\int_0^\infty
\varrho^{\mu-1}\,\AAA_V(\varrho)\, J_\mu(2\pi r \varrho) \,d\varrho,
\;\;
r>0
\end{equation}
(where~$\mu=\frac{d}{2}-1=\frac{1}{2}$). 
Therefore
the procedure 
is general, in the sense that 
any integrable, TRI, divergence-free kernel~$\kk_V$
of a~$(\cs-1)$-admissible Hilbert space~$V$
may be derived~by applying formulae~\eqref{kV2}
to the a scalar kernel~$\kk_{W_0}$ of a certain Hilbert space~$W_0\hookrightarrow C_0^{\cs}(\Rt,\Rt)$,
in fact
given by formula~\eqref{kHhV2}.
\begin{paragraph}{Generalization to arbitrary dimensions.}
Thanks to the interpretation of the differential 
operator~$-\mathrm{curl}^\mathrm{C}\mathrm{curl}^\mathrm{R}$
in the Fourier domain provided by Proposition~\ref{prCCF}, 
we may actually extend the above procedure,
and formulae~\eqref{kV2}, 
to~$\Rdd$-valued TRI kernels, with arbitrary~$d$. 
%The following result makes this precise.
\begin{proposition}
Let~$\kk_{W_0}\hookrightarrow C_0^\cs(\Rd,\Rd)$ be integrable 
and~of the form~$\kk_{W_0}(x)=k_{W_0}(\|x\|)\mathbb{I}_d$,
with Fourier transform~$\kkh_{W_0}(\xi)=h(\|\xi\|)\mathbb{I}_d$.
If we define~$\kk_V$ via its Fourier transform 
as~$\kkh_V(\xi)=\BBB_V(\|\xi\|)\mathrm{Pr}_\xi^\perp$
with~$\BBB_V(\varrho)=(2\pi\varrho)^2h(\varrho)$, $\varrho>0$,
then
the coefficients of~$\kk_V$ are given by
\begin{align}
\label{kV3}
\aaa_V(r)
&=
-\frac{d-1}{r}\frac{dk_{W_0}}{dr}(r)
&
&\mbox{and}
&
\bbb_V(r)
&=
-\frac{d-2}{r}\frac{d k_{W_0} }{dr}(r)
-\frac{d^2 k_{W_0}}{dr^2}(r),
\quad
r>0.
\end{align}
\end{proposition} 
\begin{proof} By~\eqref{iTpar} with~$\AAA=0$ we 
have~$\displaystyle\aaa_V(r)=\frac{2\mu+1}{r^{\mu+1}}
(2\pi)^2\int_0^\infty
\varrho^{\mu+2}h(\varrho)
J_{\mu+1}(2\pi\varrho r)\,d\varrho$. 
We note that
\begin{align*}
\frac{2\pi}{\varrho^{\mu+1}}
\int_0^\infty
r^{\mu+2}
\Big(
\frac{1}{r}
\frac{dk_{W_0}}{dr}(r)
\Big)J_{\mu+1}(2\pi\varrho r)\,d\varrho
\stackrel{(\ast)}{=}
-\frac{(2\pi)^2}{\varrho^\mu}
\int_0^\infty
r^{\mu+1}
k_{W_0}(r)\,
J_{\mu}(2\pi\varrho r)\,d\varrho
=-2\pi h(\varrho);
\end{align*}
in step~$(\ast)$ one 
uses an argument 
that is in all similar to the computation in~\eqref{DerAux}. Since 
the transformation~\eqref{cond_sc2}, with~$\mu+1$ 
instead of~$\mu$, is an involution,
it must be the case that
$$
\frac{1}{r}\frac{dk_{W_0}}{dr}(r)
=
\frac{2\pi}{r^{\mu+1}}
\int_0^\infty
\!\!
\varrho^{\mu+2}
\,
\big(\!-\!2\pi h(\varrho)\big)\,J_{\mu+1}(2\pi\varrho r)\, d\varrho,
$$
which allows us to complete the computation of~$\aaa_V$
(note that~$2\mu+1=d-1$). 
To calculate~$\bbb_V$ we simply use
formula~\eqref{divzero}, which is valid 
for divergence-free kernels, % like~$\kk_V$,
rewritten as
$
%\displaystyle
\bbb_V=\aaa_V+\frac{r}{d-1}
%\frac{d\aaa_V}{dr}
\,d\aaa_V/dr$.
%We have:
%\begin{align*}
%\bbb_v(r)
%&=
%ccc
%\qedhere
%\end{align*}
\end{proof}
\noindent By considerations that are in all similar to those above, 
we conclude that 
all divergence-free kernels, in arbitrary dimensions, may be obtained 
from scalar ones via formulae~\eqref{kV3}, which generalize~\eqref{kV2}.
\end{paragraph}
%%%
\begin{paragraph}{Example~\ref{exGauss}, revisited again.} 
Let~$W_0\hookrightarrow C_0^\cs(\Rd,\Rd)$ be a RKHS 
and~$\kk_{W_0}(x)=k_{W_0}(\|x\|)\mathbb{I}_d$, where
$k_{W_0}(r)=\frac{b}{2c(d-1)}e^{-cr^2}$, with~$b>0$,~$c>0$;
$W_0$ is~$\cs$-admissible, for all~$\cs$. Formulae~\eqref{kV3}
yield
\begin{equation}
\label{dfGauss}
\aaa_V(r)=be^{-cr^2}
\qquad
\mbox{and}
\qquad
\bbb_V(r)=\Big(b-\frac{2bc}{d-1}\,r^2\Big)e^{-cr^2},
\end{equation}
which are precisely the equations of the coefficients
of the kernel in Example~\ref{exGauss} with~$a=2bc/(d-1)$,
%That is, 
which is %the equation of 
the line on the boundary of~$D_1$ that
corresponds to divergence-free kernels. See Figure~\ref{wedgeint}(a). 
\end{paragraph}
\begin{example_n}[Bessel-type divergence-free kernels] 
Consider the 
scalar kernel $\kk_{W_0}(x)=k_{W_0}(\|x\|)\mathbb{I}_d$,
with $k_{W_0}$ given by the right-hand side of~\eqref{kscbess};
if~$\ce>\cs+d/2$ then~$W_0\hookrightarrow C_0^\cs(\Rd,\Rd)$.
Using Lemma~\ref{lembess}, if we apply formulae~\eqref{kV3} to such 
kernel~$k_{W_0}$
we get:
\begin{equation}
%\begin{align}
\label{dfBess}
\aaa_V(r)
%&
%=
%\frac{C_0}{\sigma^2}
%\Big(\frac{r}{\sigma}\Big)^{\nu-1}
%\Big\{
%(2\nu-1)\,
%K_{\nu-1}
%\Big(\frac{r}{\sigma}\Big)
%-
%\frac{r}{\sigma}\,
%K_\nu
%\Big(\frac{r}{\sigma}\Big)
%\Big\}
%&
%&\mbox{and}
%&
%\bbb_V(r)
%&
=
\frac{C_0}{\sigma^2}(d-1)
\Big(\frac{r}{\sigma}\Big)^{\nu-1}
\!K_{\nu-1}\Big(\frac{r}{\sigma}\Big),
%&
%&
\,\mbox{ and }\,\,
%&
\bbb_V(r)
%&
=
\frac{C_0}{\sigma^2}
\Big(\frac{r}{\sigma}\Big)^{\nu-1}
\Big\{
(2\nu+d-3)\,
K_{\nu-1}
\Big(\frac{r}{\sigma}\Big)
-
\frac{r}{\sigma}\,
K_\nu
\Big(\frac{r}{\sigma}\Big)
\Big\},
\end{equation}
where~$\nu=\ce-\frac{d}{2}$. Once again, the kernel with 
coefficients~\eqref{dfBess}
has heavier tails than the Gaussian case~\eqref{dfGauss}.
%Note that since $\lim$ the tails of the corresponding kernels~$\kk_V$
%are heavier than those of Example~\eqref{}.
\end{example_n}
\section{Application: matching of landmark points}
\label{secApp}
To illustrate an application of the tools that we developed thus far,
in this section we shall study the problem of matching feature points, or
``landmarks''~\cite{glaunes:1,joshi,MMM1,Younes10}. 
Let~$\Omega$ be an open subset of~$\Rm$.
The set of~$N$~\em labeled 
landmark points  in~$\Omega$ \em is defined as: %as follows:
\begin{equation}
\mathcal{L}^N(\Omega)
:=\Big\{
(P_1,\ldots,P_N)\,:\,P_a\in\Omega,\, P_a\not=P_b \mbox{ for }a\not=b
\Big\},
\end{equation}
which is in fact a manifold of dimension~$n=Nm$.
The generic element of~$\mathcal{L}^N(\Omega)$
is called \em landmark set\em\/. 
For any pair of elements of~$\mathcal{L}^N(\Omega)$ we will look for 
a time-dependent velocity field of ``minimal energy'' (to be defined)
that deforms the first landmark set  into the other; we shall consider  
velocity fields in RKHS of the type described in this paper. This will
result in the formulation of a Riemannian distance in~$\mathcal{L}^N(\Omega)$.
%and the dynamics of landmark 
%points~\cite{glaunes:1,joshi,MMM1,Younes10},
%where the vector fields are of the type described previously. 
%The purpose of this section is to illustrate an application of 
%our machinery, thus we shall not consider the most general setting. 
Before  proceeding, we first review some some known results. 
\par
%[\ldots]
%\par
%\begin{remark}
%Note that in the definition above we 
%could replace~$C_0^1(\Omega ,\Rd )$
%with the larger~$C_0(\Omega ,\Rd )$;
%this would still ensure that admissible Hilbert spaces have reproducing kernels. 
%However, we impose differentiability of the vector fields so that the the differential 
%equations that we introduce below have solutions.
%Also, from now on we assume that~$\Omega=\Rd$.
%\end{remark}
%\par\vspace{-.4cm}
\subsection{Ordinary differential equations and groups of diffeomorphisms}
\label{rODE}
We briefly summarize some facts on ordinary differential equations
where the time-dependent velocity fields take their values in RKHS of the type described
in this paper, and on the deformations that such dynamical systems generate. 
The interested reader is referred to~\cite{Younes10} for further 
details and results.
\par
%Suppose~$\Omega$ is an open subset of~$\Rd$
%and~$v:[0,T]\times\Omega\rightarrow \Rd$
\par
%For a fixed integer~$r\geq 1$, 
Consider a 
$1$-admissible Hilbert space~$V\hookrightarrow C_0^1(\Omega,\Rd)$,
where~$\Omega$ is an open subset of~$\Rm$;
for now we shall not introduce further assumptions 
of invariance.
%on the kernel (e.g.~it may be not translation-invariant).
%for some integer~$r\geq 0$. 
We indicate with $\mathcal{X}_V^p:=L^p([0,1],V)$
the space of~$V$-valued 
time-dependent vector fields $t\mapsto v(t,\cdot)\in V$ %with the norm
such that
$$
\|v\|_{\mathcal{X}_V^p}:=
\Big(
\int_0^1\|v(t,\cdot)\|_V^p\,dt
\Big)%
%^{1/p}
^{\!\frac{1}{p}}<\infty.
$$
In particular,~$\mathcal{X}_V^2$ 
is a subset of~$\mathcal{X}_V^1$
and is Hilbert itself with the inner product
$
\langle u,v\rangle_{\mathcal{X}_V^2}
:=\int_0^1\langle u,v\rangle_{V}dt
$.
\par
If~$V\hookrightarrow C_0^1(\Omega,\Rd)$
then it is the case~\cite[\S8.2]{Younes10} that for any~$v\in \mathcal{X}_V^1$, $t_0\in[0,1]$ and~$x_0\in\Omega$
%For~$x\in\Omega$ we consider 
the initial value problem
$
\dot{z}=v(t,z), $
%\qquad
$z(t_0)=x_0
$
has a unique solution 
of the type~$z(t)=\psi(t,t_0,x)$, for~$t\in[0,1]$.
We define~$\varphi_{st}^v(x):=\psi(t,s,x)$
and call it
the \em flow \em associated to~$v$; it has the
group property $\varphi_{st}^v=\varphi_{rt}^v\circ\varphi_{sr}^v$
for all $r,s,t\in[0,1]$ and~$v\in \mathcal{X}_V^1$.
In fact, under the assumption that~$V\hookrightarrow C_0^1(\Omega,\Rd)$, 
%it is the case 
we have
that for all~$v\in \mathcal{X}_{V}^1$ and $s,t\in[0,1]$ the map
$\varphi_{st}^v:\Omega\rightarrow\Omega$ is a \em diffeomorphism \em of~$\Omega$.
\par 
If~$V\hookrightarrow C_0^1(\Omega,\Rd)$, 
we denote with~$\mathcal{G}_V:=\{ \varphi_{01}^v : v\in \mathcal{X}_V^1 \}$
the set of diffeomorphisms provided by flows associated to 
elements of~$v\in\mathcal{X}_V^1$
at time~1. 
For any $\psi_1,\psi_2\in\mathcal{G}_V$ we define
%One can define a \em distance \em on $\mathcal{G}_V$ as follows:
$$
d_V(\psi_1,\psi_2)
:=
\inf_{v\in \mathcal{X}_V^1}
\Big\{
\|v\|_{\mathcal{X}_V^1}
\,
:
%\big|
\,
\psi_2=\psi_1\circ\varphi_{01}^v
\Big\};
$$
it was proven by A.~Trouv\'e~\cite[Theorem~8.15]{Younes10} that the function~$d_V$ 
is a \em distance \em on~$\mathcal{G}_V$, and that~$(\mathcal{G}_V,d_V)$
is a complete metric space. For any $\psi_1,\psi_2\in\mathcal{G}_V$,
their distance may in fact be computed as follows: 
$
d_V(\psi_1,\psi_2)
=
\inf_{v\in \mathcal{X}_V^2}
\{
\|v\|_{\mathcal{X}_V^2}:\psi_2=\psi_1\circ\varphi_{01}^v
\};
$
that is, the search of~$v$ may be restricted 
to the set~$\mathcal{X}_V^2\subset \mathcal{X}_V^1$. 
The infimum attained in~$\mathcal{X}_V^2$: 
more precisely, 
for all \mbox{$\psi_1,\psi_2\in\mathcal{G}_V$} there exists a
time-dependent vector field~$v\in\mathcal{X}_V^2$
such that~$d_V(\psi_1,\psi_2)=\|v\|_{\mathcal{X}_V^1}=\|v\|_{\mathcal{X}_V^2}$,
and it is such that~$\|v(t,\cdot)\|_V$
is constant with respect to~$t\in[0,1]$. See~\cite[Chapter~8]{Younes10} for details.
\par
\subsection{Interpolation of vector fields}
Let~$V$ be a 
non-degenerate 
RKHS with kernel~$K:\Omega\times\Omega\rightarrow\Rdd$, where~$\Omega$ is
an open %and connected 
subset of~$\Rm$;
for now we shall introduce no further assumption of translation- and rotation-invariance, or differentiability, of the kernel.
We consider the following problem. 
\begin{IP}[IP$_1$]
For any
landmark set~$S=(x_1,\ldots,x_N)\in \mathcal{L}^N(\Omega)$ 
and any $N$-tuple
of vectors~$(\beta_1,\ldots,\beta_N)\in(\Rd)^N$, find
a vector field~$u\in V$
such that
%\begin{equation}
%\label{eIP1}
$$
\left\{
\begin{array}{l}
u(x_a)=\beta_a \mbox{ for }1\leq a\leq N, \mbox{ and}
\\
\|u\|_{V} \mbox{ is minimal}.
\end{array}
\right.
$$
%\end{equation}
\end{IP}
\par
For this purpose, for a fixed~$S=(x_1,\ldots,x_N)\in \mathcal{L}^N(\Omega)$
we define the linear subspace of~$V$:
$$
V_S:=
\big\{
u\in V
\,\big|
\,
u(x_a)=0\mbox{ for all }a=1,\ldots,N
\big\}.
$$
Its orthogonal complement with respect to
the inner product~$\langle\cdot,\cdot\rangle_V$
is given by the following result.
\begin{lemma} If~$S=(x_1,\ldots,x_N)\in \mathcal{L}^N(\Omega)$,
then~$V_S^\perp=\mathrm{span}\big\{
K(\cdot,x_a)\alpha\,\big|\,a=1,\ldots,N,\,\alpha\in\Rd
\big\}$. 
\end{lemma}
\begin{proof}
We have~$u\in V_S$ if and only 
if~$\sum_{a=1}^N \alpha_a\cdot 
u(x_a)=\langle\sum_{a=1}^N K(\cdot,x_a)\alpha_a,u \rangle_V=0$
for all $\alpha_1,\ldots,\alpha_N\in\Rd$.
So~$V_S=\mathrm{span}\{K(\cdot,x_a)\alpha\,
\big|\,a=1,\ldots,N,\,\alpha\in\Rd\}^{\perp}$.
But $\{K(\cdot,x_a)\alpha\,
\big|\,a=1,\ldots,N,\,\alpha\in\Rd\}$ has finite dimension
(and  whence it is closed), so the claim follows immediately.
\end{proof}
\begin{proposition} 
\label{prIP1}
Fix~$S=(x_1,\ldots,x_N)\in \mathcal{L}^N(\Omega)$.
For any~$u\in V$, there exists a unique~$u_\ast\in V_S^\perp$
such that~$u_\ast(x_i)=u(x_i)$ for all~$i=1,\ldots,N$; it is given by the
orthogonal projection of~$u$ onto~$V_S^\perp$.
%Let~$V$ be a RKHS.
%If a solution~$u$ to~$\mathrm{IP}_1$ exists, 
%then~$u\in V_S^\perp$.
%If~$u\in V_S^\perp$ is a solution to~$\mathrm{IP}_1$ 
%restricted to~$V_S^\perp$,
%then it is a solution to~$\mathrm{IP}_1$ on all of~$V$.
\end{proposition}
\begin{proof}
Let~$u_\ast$  be the
orthogonal projection of~$u$ onto~$V^\perp_S$;
then~$u-u_\ast\in V_S$, i.e.~$u(x_a)=u_\ast(x_a)$
for~$a=1,\ldots,N$. 
On the other hand, if~$u_\ast\in V^\perp_S$ is such
that~$u(x_a)=u_\ast(x_a)$ for~$a=1,\ldots,N$, then it is the case
that~$u-u_\ast\in V_S$, i.e.~$u_\ast$ is the orthogonal projection of~$u$
onto~$V^\perp_S$. 
%
%If~$u$ is a solution to~$\mathrm{IP}_1$, let~$u^\ast$ be its 
%orthogonal projection onto~$V^\perp_S$;
%then~$u-u^\ast\in V_S$, i.e.~$u(x_a)=u^\ast(x_a)$
%for~$a=1,\ldots,N$. By the projection
%we have~$\|u^\ast\|_V\leq\|u\|_V$, but also~$\|u\|_V\leq\|u^\ast\|_V$ by assumption;
%whence~$\|u\|_V=\|u^\ast\|_V$ and in fact~$u=u^\ast$, which completes the first part.
%Now, if~$u$ is a solution to~$\mathrm{IP}_1$ restricted to~$V_S^\perp$, let~$\hat{u}\in V$
%be such that~$\hat{u}(x_a)=\beta_a$ for~$a=1,\ldots,N$.
%We have~$\hat{u}-u\in V_S$, whence~$\hat{u}-u\perp u$. Therefore
%$\|\hat{u}\|_V^2=\|\hat{u}-u\|_V^2+\|u\|_V^2\geq\|u\|_V^2$,
%and we conclude that~$u$ is a solution to~$\mathrm{IP}_1$ on all of~$V$.
\end{proof}
\begin{corollary}
\label{cIP1}
Let~$V$ be a RKHS.
If a solution~$u$ to~$\mathrm{IP}_1$ exists, 
then~$u\in V_S^\perp$.
If~$u\in V_S^\perp$ is a solution to~$\mathrm{IP}_1$ 
restricted to~$V_S^\perp$,
then it is a solution to~$\mathrm{IP}_1$ on all of~$V$.
\end{corollary}
Therefore if the solution to~$\mathrm{IP}_1$ exists then
it must be in~$V_S^\perp$, i.e.~of the form
\begin{equation}
\label{sIP1}
%\textstyle
u(\cdot)=\sum_{b=1}^NK(\cdot,x_b)\alpha_b,
\end{equation}
for some choice of vectors~$\alpha_1,\ldots,\alpha_N\in\Rd$.
The squared norm of such vector fields may be written as
$\|u\|^2_V=\sum_{a,b=1}^N\alpha_a\cdot K(x_a,x_b)\alpha_b$.
Inserting~$x_1,\ldots,x_N$ in the above expression 
and imposing the conditions~$u(x_a)=\beta_a$, $a=1,\ldots,N$, we see that we are 
looking for vectors~$\alpha_1,\ldots,\alpha_N$ such that
\begin{equation}
\label{intcond}
%\textstyle
\beta_a = \sum_{b=1}^NK(x_a,x_b)\alpha_b,\qquad
a=1,\ldots,N.
\end{equation}
\par
Introducing now the vectors in~$\mathbb{R}^{Nd}$ and the matrix 
in~$\mathbb{R}^{Nd\times Nd}$
\begin{equation}
\label{VecNot}
\boldsymbol{\alpha}:=
\left[
\!\!
\begin{array}{c}
\alpha_1
\\
\vdots
\\
\alpha_N
\end{array}
\!\!
\right],
\quad
\boldsymbol{\beta}:=
\left[
\!\!
\begin{array}{c}
\beta_1
\\
\vdots
\\
\beta_N
\end{array}
\!\!
\right],
\quad
\mathbf{ K}(\boldsymbol{x})
:=
\left[\!
\begin{array}{cccc}
K(x_1,x_1) & K(x_1,x_2) & \cdots & K(x_1,x_N) \\
K(x_2,x_1) & K(x_2,x_2) & \cdots & K(x_2,x_N) \\
\vdots & \vdots & \ddots & \vdots \\
K(x_N,x_1) & K(x_N,x_2) & \cdots & K(x_N,x_N)
\end{array}
\!
\right],
\end{equation}
we may rewrite~\eqref{intcond} simply as $\boldsymbol{\beta}
=\mathbf{ K}(\boldsymbol{x})
\boldsymbol{\alpha}$, while~$\|u\|_V^2=\boldsymbol{\alpha}
\cdot\mathbf{ K}(\boldsymbol{x})
\boldsymbol{\alpha}$.
Note that the matrix $\mathbf{ K}(\boldsymbol{x})
$ is invertible (in fact 
if~$\mathbf{ K}(\boldsymbol{x})
\boldsymbol{\alpha}=0$ for some~$\boldsymbol{\alpha}
\in \mathbb{R}^{Nd}$
then~$\boldsymbol{\alpha}\cdot \mathbf{ K}(\boldsymbol{x})
\boldsymbol{\alpha}=0$; 
but $\boldsymbol{\alpha}\cdot 
\mathbf{ K}(\boldsymbol{x})
\boldsymbol{\alpha}$ is the left-hand side of~\eqref{eq:Gpos},
so if~$V$ is non-degenerate we must have~$\boldsymbol{\alpha}=0$).
Therefore the only solution 
to~$\boldsymbol{\beta}
=\mathbf{ K}(\boldsymbol{x})
\boldsymbol{\alpha}$ 
is~$\boldsymbol{\alpha}=\mathbf{ K}(\boldsymbol{x})^{-1}\boldsymbol{\beta}$. 
In conclusion, there exists a unique solution to~$\mathrm{IP}_1$
when this is restricted to~$V_S^\perp$, which is the vector field~\eqref{sIP1} 
with~$\boldsymbol{\alpha}=
\mathbf{ K}(\boldsymbol{x})^{-1}\boldsymbol{\beta}$. 
By Corollary~\ref{cIP1} this is the only solution 
to~$\mathrm{IP}_1$ (on all of~$V$).
\subsection{Landmark matching via diffemorphisms, with geometric intepretation}
We will now formulate and provide a method for solving 
the more difficult problem of
matching landmark sets with flows of diffeomorphisms. 
For~$\Omega\subseteq\Rd$ open and connected,
we introduce the further assumption that the non-degenerate
RKHS is~1-admissible, i.e.~$V\hookrightarrow C_0^1(\Omega,\Rd)$.
\begin{IP}[IP$_2$]
%Fix a Hilbert space~$V\hookrightarrow C_0^1(\Rd,\Rd)$, and  
For any
two landmark sets~$I=(x_1,\ldots,x_N)$ 
and $J=(y_1,\ldots,y_N)$
in~$\mathcal{L}^N(\Omega)$, find a time-dependent 
vector field~$v\in \mathcal{X}_V^2$
such that
$$
\left\{
\begin{array}{l}
\varphi_{01}^v(x_a)=y_a \mbox{ for }1\leq a\leq N, \mbox{ and}
\\
\|v\|_{\mathcal{X}_V^2} \mbox{ is minimal}.
\end{array}
\right.
$$
\end{IP}
For~$v\in \mathcal{X}_V^2$, $S=(x_1,\ldots,x_N)\in\mathcal{L}^N(\Omega)$
and~$s,t\in[0,1]$,
we let~$\varphi_{st}^v(S):=(\varphi_{st}^v(x_1),\ldots,\varphi_{st}^v(x_N))$,
%note that 
where~$\varphi_{st}^v$ is the flow associated to~$v$.
Since~$\varphi_{st}^v$ is a diffeomorphism we have 
$\varphi_{st}^v(x_a)\not=\varphi_{st}^v(x_b)$ for~$a\not=b$, 
whence~$\varphi_{st}^v(S)\in\mathcal{L}^N(\Omega)$.
We also introduce the subset of~$\mathcal{X}_V^2$: 
$$
\mathcal{X}_V^2(S):=
\big\{
v\in \mathcal{X}_V^2
\,\big|\,
v(t,\cdot)\in V^\perp_{\varphi_{0t}^v(S)},
\mbox{ for all }t\in[0,1]
\big\}.
$$
\begin{theorem}
Assume that~$V$ is a 1-admissible Hilbert space, and 
let~$S=(x_1,\ldots,x_N)\in\mathcal{L}^N(\Omega)$. 
For any~$v\in \mathcal{X}_V^2$ there exists a vector field
$v_\ast\in \mathcal{X}_V^2(S)$ such that 
%for all~$t\in[0,1]$,
%$\|v_\ast(t,\cdot)\|_V\leq\|v(t,\cdot)\|_V$
$\varphi_{0t}^{v_\ast}(x_a)=\varphi_{0t}^{v}(x_a)$
for all~$t\in[0,1]$
and $a=1\ldots,N$,
and~$\|v_\ast(t,\cdot)\|_V\leq\|v(t,\cdot)\|_V$
for all~$t\in[0,1]$.
\end{theorem}
\begin{proof}
Fix $v\in \mathcal{X}_V^2$,
and for all~$t\in[0,1]$ let~$v_\ast(t,\cdot)$ be the projection
of~$v(t,\cdot)\in V$ onto~$V_{\varphi_{0t}^v(S)}^\perp$; 
by Corollary~\ref{cIP1}
it is in fact the vector field of minimal norm~$\|\cdot\|_V$ such 
that, for all $a=1,\ldots,N$, $v_\ast(t,\varphi_{0t}^v(x_a))
=v(t,\varphi_{0t}^v(x_a))$. 
Therefore,
by the uniqueness of solutions to ordinary 
differential equations (discussed in~\S\ref{rODE}), 
it must be the case that~$\varphi_{0t}^{v_\ast}(x_a)
=\varphi_{0t}^{v}(x_a)$,
for all~$t\in[0,1]$ and~$a=1,\ldots,N$. 
This also implies that
$V_{\varphi_{0t}^{v_\ast}(S)}
=V_{\varphi_{0t}^{v}(S)}$, 
therefore~$v_\ast(t,\cdot)\in V_{\varphi_{0t}^{v_\ast}(S)}$
for all~$t\in[0,1]$, i.e.~$v_\ast\in \mathcal{X}_V^2(S)$.
Also, by construction we have that
~$\|v_\ast(t,\cdot)\|_V\leq\|v(t,\cdot)\|_V$
for all~$t\in[0,1]$.
\end{proof}
\par
Whence the solutions of the problem~$\mathrm{IP}_2$ must 
be searched among the vector fields of the form
\begin{equation}
\label{sIP2}
v(t,\cdot)
=
%\sum_{b=1}^NK\big(\,\cdot\,,\varphi_{0t}^v(x_b)\big)\alpha_b(t),
\sum_{b=1}^NK\big(\,\cdot\,,x_b(t)\big)\alpha_b(t),
\qquad t\in[0,1],
\end{equation}
for some set of functions~$\alpha_a:[0,1]\rightarrow \Rd$, 
$a=1,\ldots,N$,
%Introducing for simplicity the notation
where we have introduced for simplicity the notation
$x_a(t):=\varphi_{0t}^v(x_a)$, $a=1,\ldots,N$
(the above equation is precisely expression~\eqref{eqone} 
in the introduction).
Note therefore that expression~\eqref{sIP2} for~$v$ 
is implicit, in that
the trajectories~$x_a$, $a=1,\ldots,N$, also depend on~$v$. 
%Such trajectories verify the following integral equations:
%\begin{equation}
%\label{sysint}
%x_a(t)=x_a+\int_0^t
%\sum_{b=1}^NK\big(x_a(s),x_b(s)\big)\alpha_b(s)\,ds,\qquad
%t\in[0,1],\quad
%a=1,\ldots,N.
%\end{equation}
So we can parametrize the search space~$\mathcal{X}_V^2(S)$
of the solution to~$\mathrm{IP}_2$ 
by the 
functions~$\alpha_a$, $a=1,\ldots,N$. %, 
%$a=1,\ldots,N$.
%, 
%$a=1,\ldots,N$. 
%The next result states how their norm can be controlled.
Similarly to earlier, for all~$t\in[0,1]$ we introduce the vectors 
in~$\mathbb{R}^{Nd}$ and matrix  in~$\mathbb{R}^{Nd\times Nd}$: 
\begin{equation}
\label{VecNot2}
\boldsymbol{\alpha}(t)
\!
:=
\!
\left[
\!\!
\begin{array}{c}
\alpha_1(t)
\\
\vdots
\\
\alpha_N(t)
\end{array}
\!\!
\right]\!,
\;\;
\boldsymbol{x}(t)
\!:=
\!
\left[
\!\!\!
\begin{array}{c}
x_1(t)
\\
\vdots
\\
x_N(t)
\end{array}
\!\!\!
\right]\!,
\;\;
\mathbf{ K}\big(\boldsymbol{x}(t)\big)
\!
:=
\!
\left[\!
\begin{array}{ccc}
%K(x_1(t),x_1(t)) & K(x_1(t),x_2(t)) & \cdots & K(x_1(t),x_N(t)) \\
%K(x_2(t),x_1(t)) & K(x_2(t),x_2(t)) & \cdots & K(x_2(t),x_N(t)) \\
%\vdots & \vdots & \ddots & \vdots \\
%K(x_N(t),x_1(t)) & K(x_N(t),x_2(t)) & \cdots & K(x_N(t),x_N(t))
K\big(x_1(t),x_1(t)\big) &  \cdots & K\big(x_1(t),x_N(t)\big) \\
K\big(x_2(t),x_1(t)\big) &  \cdots & K\big(x_2(t),x_N(t)\big) \\
\vdots & \ddots & \vdots \\
K\big(x_N(t),x_1(t)\big) &  \cdots & K\big(x_N(t),x_N(t)\big)
\end{array}
\!
\right]\!.
\end{equation}
\par
Note that since we assumed that~$V\hookrightarrow C_0^1(\Omega,\Rd)$ 
is non-degenerate, the 
matrix~$\mathbf{ K}(\boldsymbol{x}(t))$ is invertible for all~$t\in[0,1]$.
Note also that the norm of the vector fields of the type~\eqref{sIP2}
may be written as $\|v(t,\cdot)\|_V=
\boldsymbol{\alpha}(t)
\cdot \mathbf{ K}(\boldsymbol{x}(t))
\boldsymbol{\alpha}(t)$, by the reproducing property of~$K$.
Therefore the solutions to~$\mathrm{IP}_2$
must be searched in the space~$\mathcal{X}^2_V(S)$,
whose elements have norm that may also be parametrized 
by the functions~$\alpha_a$, $a=1,\ldots,N$, as follows:
%may be written, at each~$t\in[0,1]$
%and the norm~$\|\cdot\|_V$ in such space
%may be written as a function of the parameters
%$\alpha_a\in L^2([0,1],\mathbb{R}^{Nd})$, $a=1,\ldots,N$, as follows:
\begin{equation}
\label{nrg}
%E[\boldsymbol{\alpha}]
%:=
\|v\|^2_{\mathcal{X}_V^2}
=
%\int_0^1\|v(t,\cdot)\|^2_{V}\,dt
%=
\int_0^1
\sum_{a,b=1}^N
\alpha_a(t)\cdot K\big(x_a(t),x_b(t)\big)\alpha_b(t)\,dt
=\int_0^1
\boldsymbol{\alpha}(t)
\cdot \mathbf{ K}(\boldsymbol{x}(t))
\boldsymbol{\alpha}(t)
\,dt,
\end{equation}
where, we remind the reader,~$x_a(t):=\varphi_{0t}^v(x_a)$,
for~$t\in[0,1]$ and~$a=1,\ldots,N$. Therefore, by definition, it must be the case 
that $\dot{x}_a(t)=v(t,x_a(t))$;  by expression~\eqref{sIP2},
we must have
$$
\dot{x}_a(t)
=
\sum_{b=1}^NK\big(x_a(t),x_b(t)\big)\alpha_b(t),\qquad
t\in[0,1],\quad a=1,\ldots,N,
$$
or, with the compact notation introduced in~\eqref{VecNot2}, 
$\dot{\boldsymbol{x}}(t)=\mathbf{ K}(\boldsymbol{x}(t))\boldsymbol{\alpha}(t)$, 
for~$t\in[0,1]$.
We conclude that the squared norm~\eqref{nrg} 
of vector fields in~$\mathcal{X}_V^2$ of the type~\eqref{sIP2} 
may be rewritten as:
\begin{equation}
\label{nrg2}
\|v\|^2_{\mathcal{X}_V^2}
=
\int_0^1 \dot{\boldsymbol{x}}(t)\cdot \mathbf{ K}(\boldsymbol{x}(t))^{-1}\dot{\boldsymbol{x}}(t)\,dt,
\end{equation}
where~$x_a:[0,1]\rightarrow \mathcal{L}^N(\Omega)$,
and~$a=1,\ldots,N$, are the  trajectories  determined by the
vector field~$v$.
\par
We note that the right-hand side of~\eqref{nrg2}
may be interpreted in differential-geometric  terms~\cite{jost,lee:2} 
as the \em energy of a curve \em 
$t\mapsto \boldsymbol{x}(t)$
on the $Nd$-dimensional manifold~$\mathcal{L}^N(\Omega)$ 
with respect
to a Riemannian metric
(remember that~$\Omega\subseteq\Rd$ in this section).  
In the coordinates of the 
landmark points, which we indicate here generically 
as~$(q^1,\ldots,q^N)\in \mathcal{L}^N(\Omega)$ 
with~$q^a=(q^{a1},\ldots,q^{ad})\in\Rd$ for~$a=1,\ldots,N$, 
the corresponding metric tensor
%may be 
is written as 
the 
%$\mathbb{R}^{Nm\times Nm}$
$Nm\times Nm$
matrix-valued function:
\begin{equation}
\label{gtensor}
%\mathbb{R}^{Nm}\rightarrow
%GL(\mathbb{R}^{Nm})
%\,:\,
%(z_1,\ldots,z_N)\mapsto 
g(q^1,\ldots,q^N):=
\left[\!
\begin{array}{cccc}
K(q^1,q^1) & K(q^1,q^2) & \cdots & K(q^1,q^N) \\
K(q^2,q^1) & K(q^2,q^2) & \cdots & K(q^2,q^N) \\
\vdots & \vdots & \ddots & \vdots \\
K(q^N,q^1) & K(q^N,q^2) & \cdots & K(q^N,q^N)
\end{array}
\!
\right]^{-1}.
\end{equation}
We now introduce notation from 
%analytical 
mechanics~\cite{arnold:1},
namely~$q^a(t):=x_a(t)\in \mathbb{R}^{Nd}$ and~$p_a(t):=\alpha_a(t)\in \mathbb{R}^{Nd}$,
with $t\in[0,1]$ and $a=1,\ldots,N$, 
for landmark \em positions \em and \em momenta\em\/, respectively;
also, we shall write $q(t):= \boldsymbol{x}(t)$ and~$p(t):= \boldsymbol{\alpha}(t)$, 
$t\in[0,1]$.
Whence~$\dot{\boldsymbol{x}}(t)=
\mathbf{ K}(\boldsymbol{x}(t))\boldsymbol{\alpha}(t)$
translates into  %the equations
\begin{equation}
\label{sharp}
\dot{q}^a(t) = \sum_{b=1}^N
K\big(q^a(t),q^b(t)\big)\,p_b(t),
\qquad
t\in[0,1],\quad a=1,\ldots,N,
\end{equation}
which is nothing but the ``sharp'' isomorphism of differential geometry
(raising of the indices) with
respect to the metric tensor~\eqref{gtensor}, that maps cotangent 
vectors into a tangent vectors (this justifies the
upper location of the index in~$q^a$, $a=1,\ldots,N$).  From the 
energy of the curve~\eqref{nrg2} we deduce immediately
the form of the Lagrangian function~\cite{arnold:1,jost}
for landmarks, namely
$
L(q,\dot{q})=\frac{1}{2}\dot{q}\cdot g(q)\dot{q},
$
and the Hamiltonian function which we may write as  follows:
\begin{equation}
\label{Hamiltonian}
H(p,q)
=
\frac{1}{2}\,p\cdot \big(g(q)\big)^{-1}p
=
\frac{1}{2}\sum_{a,b=1}^N
%\Big\langle
p_a\cdot K(q^a,q^b)p_b
%\Big\rangle_{\mathbb{R}^d}
=
\frac{1}{2}
\sum_{a,b=1}^N
\sum_{i,j=1}^d
K^{ij}(q^a,q^b)p_{ai}p_{bj},
\end{equation}
where we have explicitly written the components of the 
momenta~$p_a=(p_{a1},\ldots,p_{ad})$, 
$a=1,\ldots,N$,
and we have indicated with~$K^{ij}$ the~$(i,j)^\mathrm{th}$
element of the matrix-valued kernel~$K$.
We shall use~\eqref{Hamiltonian} 
in the next section to study the dynamics induced by the landmark
matching problem.
\begin{discussion} In light of the above,
the landmark trajectories resulting
from vector fields~$v\in \mathcal{X}_V^2$
that solve~$\mathrm{IP}_2$ will be 
length-minimizing geodesics in~$\mathcal{L}^N(\Omega)$,
with respect to the metric tensor~\eqref{gtensor},
that connect the landmark sets~$I=(x_1,\ldots,x_N)$
and~$J=(y_1,\ldots,y_N)$. Once such length-minimizing
geodesics $q:[0,1]\rightarrow\mathcal{L}^N(\Omega)$
are known, the momenta are computed via the inversion
of the equations~\eqref{sharp} (i.e.~by
$p(t)=\mathbf{ K}(q(t))^{-1}\dot{q}(t)$, $t\in[0,1]$),
while the corresponding~$v\in \mathcal{X}_V^2$
is obtained via formula~\eqref{sIP2}.
%, with~$x_a=q^a$ and~$\alpha_a=p_a$,
%$a=1,\ldots,N$.
We also have, by the results summarized in~\S\ref{rODE},
that the resulting geodesic distance between the sets
$I=(x_1,\ldots,x_N)$
and~$J=(y_1,\ldots,y_N)$ in~$\mathcal{L}^N(\Omega)$
is  equal to~$d_V(\varphi_{01}^v,\mathrm{id})$,
i.e.~the distance in the group of diffeomorphisms~$\mathcal{G}_V$
between the~$\varphi_{01}^v$ and the identity diffeomorphism. 
\par
Finally, we note that the shape of the geodesics in~$\mathcal{L}^N(\Omega)$
must clearly depend on the initial choice of the Hilbert 
space~$V\hookrightarrow C_0^1(\Omega,\Rd)$, i.e.~on its kernel;
some examples will be seen later.
We refer the reader to~\cite{MMM1} for a study of the
geometry (in particular, of the sectional curvature) of~$\mathcal{L}^N(\Rd)$
equipped with the metric~\eqref{gtensor}
in the case of scalar kernels, i.e.~of the form~$K(x,y)=k(\|x-y\|)\mathbb{I}_d$.
\end{discussion}

\subsection{Hamiltonian dynamics}
In the previous section we derived the Hamiltonian
induced by the matching problem for landmark sets.
We shall now develop Hamilton's equations, i.e.~the geodesic
equations in~$\mathcal{L}^N(\Omega)$ with respect to the metric
tensor~\eqref{gtensor}. Also, we will consider the particular case of~TRI kernels.
\begin{proposition}
Let~$V\hookrightarrow C_0^1(\Omega,\Rd)$, with~$\Omega\subseteq\Rd$
open and connected, be a RKHS with kernel~$K$.
Hamilton's equations for the manifold~$\mathcal{L}^N(\Rd)$ 
equipped with the metric tensor~\eqref{gtensor} are given by:
\begin{subequations}
\begin{align}
\label{gham_1}
\dot{q}^{a} & = \sum_{b=1}^N
K(q^a,q^b)\,p_{b} \\
\label{gham_2}
\dot{p}_{ai} &= - \sum_{b=1}^N
\sum_{j,k=1}^{\dd}
\partial_{1,i}K^{jk}(q^a,q^b)\,
p_{aj}
p_{bk} 
\end{align}
\end{subequations}
for $a=1,\ldots,N$ and $i=1,\ldots, \dd$.
\end{proposition}
\noindent{As} in Section~\ref{constr},
 $\partial_{n,i}K$ in~\eqref{gham_2}
indicates the partial derivative of~$(x_1,\ldots,x_\dd,y_1,\ldots,y_\dd)\mapsto K(x,y)$
with respect to the~$i$-th component ($i=1,\ldots,\dd$) of the~$n$-th 
variable ($n=1,2$).
\begin{proof}
Computing~$\dot{q}^{ai}=\frac{\partial H}{\partial p_{ai}}$ yields the first equation immediately.
Also:
\begin{align*}
\frac{\partial}{\partial q^{ai}}
\,
K^{jk}
&
\big(q^{b1},\ldots,q^{bD},q^{c1},\ldots,q^{cD}\big)
=
\sum_{\ell=1}^D
\Big(
\partial_{1,\ell} K^{jk}
%\frac{\partial K^{jk}}{\partial x^\ell}
%(q^b,q^c)
\,
\frac{\partial q^{b\ell}}{\partial q^{ai}}
+
\partial_{2,\ell} K^{jk}
%\frac{\partial K^{jk}}{\partial y^\ell}
%(q^b,q^c)
\,
\frac{\partial q^{c\ell}}{\partial q^{ai}}
\Big)
\\
&=
\sum_{\ell=1}^D
\Big(
\partial_{1,\ell} K^{jk}
%\frac{\partial K^{jk}}{\partial x^\ell}
%(q^b,q^c)
\,
\delta_a^b\delta_i^\ell
+
\partial_{2,\ell} K^{jk}
%\frac{\partial K^{jk}}{\partial y^\ell}
%(q^b,q^c)
\,
\delta_a^c\delta_i^\ell
\Big)
=
\partial_{1,i} K^{jk}
%\frac{\partial K^{jk}}{\partial x^i}
(q^b,q^c)
\,
\delta_a^b
+
\partial_{2,i} K^{jk}
%\frac{\partial K^{jk}}{\partial y^i}
(q^b,q^c)
\,
\delta_a^c.
\end{align*}
Since $K^{jk}(x,y)=K^{kj}(y,x)$ it is also 
the case that~$\partial_{2,i}^{jk}K(x,y)=\partial_{1,i}^{kj}K(y,x)$, whence
\begin{align*}
&\dot{p}_{ai}
=-
\frac{\partial H}{\partial q^{ck}}
=
-\frac{1}{2}
\sum_{b,c=1}^N\sum_{j,k=1}^D
\Big(
\partial_{1,i} K^{jk}
%\frac{\partial K^{jk}}{\partial x^i}
(q^b,q^c)
\,
\delta_a^b
+
\partial_{1,i} K^{kj}
%\frac{\partial K^{jk}}{\partial y^i}
(q^c,q^b)
\,
\delta_a^c
\Big)
p_{bj}
p_{ck}
\\
&\;\;=
-
\frac{1}{2}
\!
\sum_{j,k=1}^D
\!\!
\Big(
\sum_{c=1}^N
\partial_{1,i}K^{jk}(q^a,q^c)p_{aj}p_{ck}
+
\sum_{b=1}^N
\partial_{1,i}K^{kj}(q^a,q^b)p_{bj}p_{ak}
\Big)
%\\&
=
-\!
\sum_{j,k=1}^D
\sum_{c=1}^N
\partial_{1,i}K^{jk}(q^a,q^c)p_{aj}p_{ck}
\end{align*}
where, in the last step, we have simply relabeled the indices. This concludes the proof.
\end{proof}
We now assume that~$\Omega=\Rd$ and that the kernel of~$V$
is translation-invariant, i.e.~$K(x,y)=\kk(x-y)$, for
some function~$\kk:\Rd\rightarrow\Rdd$.
In this case Hamilton's equations obviously become:
\begin{subequations}
\begin{align}
\label{ham_1}
\dot{q}^a & = \sum_{b=1}^N
\kk(q^a-q^b)\,p_b \\
\label{ham_2}
\dot{p}_{ai} &= 
- \sum_{b=1}^N
%\Big\langle 
p_a\cdot
\frac{\partial \kk}{\partial x^i}(q^a-q^b)\,p_b 
%\Big\rangle_{\Rd }
\end{align}
\end{subequations}
with $a=1,\ldots,N$ and $i=1,\ldots, \dd$.
In the translation- \em and \em rotation-invariant case (TRI kernels)
the equations take a form that is determined by differentiating
the general expression~\eqref{invker}.
\begin{proposition}
\label{parderk}
Consider a Hilbert space~$V\hookrightarrow C_0^1(\Rd,\Rd)$
with a TRI kernel. Then
% partial derivatives of the matrix~$\kk$ are given by
\begin{equation}
\label{partialK}
%\begin{aligned}
%\!\!\!\!
\frac{\partial \kk}{\partial x^i}(x)
=
\frac{x^i}{\|x\|}
\Big[
\frac{d\aaa}{dr}
(\|x\|)\,
\prparx
+
\frac{d\bbb}{dr}
(\|x\|)\,
\prperx
\Big]
%\frac{xx^T}{\|x\|^2}
%\\ 
%&+
%\frac{\aaa(\|x\|)-\bbb(\|x\|)}{\|x\|}
\!+\!
\|x\|\,\tilde{k}(\|x\|)
\Big[
\frac{e_ix^T+xe_i^T}{\|x\|}
-2\frac{x^i}{\|x\|}
\prparx
\Big],
%\end{aligned}
\end{equation}
for $i=1,\ldots,\dd$, where $\tilde{k}$ is given by~\eqref{def_ktilde}.
%$\{e_i\}_{i=1,\ldots,\dd}$ is the standard basis of $\Rd $.
\end{proposition}
\begin{proof}
By differentiating the generic matrix element~\eqref{kelement} with respect to~$x^i$ we get
\begin{align*}
\frac{\partial\kk^{jk}}{\partial x^i}&(x)
=
\frac{x^jx^k}{\|x\|^2}
\frac{\partial}{\partial x^i}
\big[\hpar(\|x\|)-\hper(\|x\|)\big]
+
\big[\hpar(\|x\|)-\hper(\|x\|)\big]
\frac{\partial}{\partial x^i}
\frac{x^jx^k}{\|x\|^2}
+
\frac{\partial}{\partial x^i}
\hper(\|x\|)\,\delta^{jk}
\\
=&
\Big[\frac{d\hpar}{dr}(\|x\|)-\frac{d\hper}{dr}(\|x\|)\Big]
\frac{x^ix^jx^k}{\|x\|^3}
+
%\big[\hpar(\|x\|)-\hper(\|x\|)\big]
\|x\|^2\tilde{k}(\|x\|)
\Big(
\frac{\delta^j_ix^k+\delta^k_ix^j}{\|x\|^2}-2\frac{x^ix^jx^k}{\|x\|^4}
\Big)
%\\
%&
+
\frac{d\hper}{dr}(\|x\|)\,\frac{x^i}{\|x\|}\,\delta^{jk},
\end{align*}
where~$\delta_i^j$ is Kronecker's delta.
We note that we have $\big(e_ix^T\big)^{jk}=\delta_i^jx^k$ and
$\big(xe_i^T\big)^{jk}=x^j\delta_i^k$, so 
we can rewrite the above in matrix form as follows:
%\begin{align*}
$$
\frac{\partial\kk}{\partial x^i} (x)
= 
%& 
\frac{x^i}{\|x\|}
\Big[\frac{d\hpar}{dr}(\|x\|)
\!
-
\!
\frac{d\hper}{dr}(\|x\|)\Big]
\frac{xx^T}{\|x\|^2}
+
\!
%\frac{\hpar(\|x\|)-\hper(\|x\|)}{\|x\|}
\|x\|\,\tilde{k}(\|x\|)
\Big(
\frac{e_ix^T+xe_i^T}{\|x\|}-\frac{2x^i}{\|x\|}\frac{xx^T}{\|x\|^2}
\Big)
%\\
%&
\!
+
\!
\frac{d\hper}{dr}(\|x\|)\,\frac{x^i}{\|x\|}\,\mathbb{I}_\dd ;
$$
%\end{align*}
the proof is completed by using the definitions of the projection operators~\eqref{proj_op}.
\end{proof}
%
%
%
%\input{extracorollary}
%%% 
\begin{corollary}
Under the assumptions of Proposition~\ref{parderk}, 
if~$V$ is also divergence-free, i.e.~if condition~\eqref{divzero}
holds, % for all~$x\in \Rd $,  
the partial derivatives~\eqref{partialK} become
\begin{equation}
\label{pK_2}
\begin{aligned}
\frac{\partial \kk}{\partial x^i}(x)
=&
\frac{x^i}{\|x\|}
\Big[
\frac{\dd+1}{\dd-1}
\frac{d\aaa}{dr}
(\|x\|)\,
\prparx
+
\frac{d\bbb}{dr}
(\|x\|)\,
\prperx
\Big]
%\frac{xx^T}{\|x\|^2}
%\\ 
%&
-\frac{1}{\dd-1}
\frac{d\hpar}{dr}(\|x\|)\,
%\frac{\aaa(\|x\|)-\bbb(\|x\|)}{\|x\|}
\frac{e_ix^T+xe_i^T}{\|x\|}.
\end{aligned}
\end{equation}
\end{corollary}
%%% 
\begin{corollary}
Under the assumptions of Proposition~\ref{parderk}, 
if~$V$ is also curl-free, i.e.~if condition~\eqref{curlzero}
holds, % for all~$x\in \Rd $,  
the partial derivatives~\eqref{partialK} become
\begin{equation}
\label{pK_3}
\frac{\partial \kk}{\partial x^i}(x)
=
\frac{x^i}{\|x\|}
\Big[
\Big(
\frac{d\aaa}{dr}
(\|x\|)
-2
\frac{d\bbb}{dr}
(\|x\|)
\Big)
\prparx
+
\frac{d\bbb}{dr}
(\|x\|)\,
\prperx
\Big]
%\frac{xx^T}{\|x\|^2}
%\\ 
%&
+\frac{d\hper}{dr}(\|x\|)\,
%\frac{\aaa(\|x\|)-\bbb(\|x\|)}{\|x\|}
\frac{e_ix^T+xe_i^T}{\|x\|}.
\end{equation}
\end{corollary}
For any given choice of differentiable coefficients~$(\aaa,\bbb)$,
the expression~\eqref{partialK} %for the derivative of the matrix-valued kernel~$\kk$ 
can be inserted into~\eqref{ham_2},
and Hamilton's equations may be solved numerically
to yield the landmark trajectories. Once these are known,
the time dependent velocity field~\eqref{sIP2}
can be computed (with~$x_a=q^a$ and~$\alpha_a=p_a$), and 
numerical integration of the differential 
equation~$\partial_t\varphi_{0t}^v(t,x)=v(t,\varphi^v_{0t}(x))$,
with initial condition~\mbox{$\varphi^v_{00}(x)=x$},
yields the diffeomorphisms~$\varphi_{01}^v$.
This is shown in the next section
for different choices of the TRI kernel~$\kk$.%
\subsection{Numerical results}
\begin{figure}
\begin{center}
\begin{picture}(450,130)
\put(0,0){\includegraphics[height=5 cm]{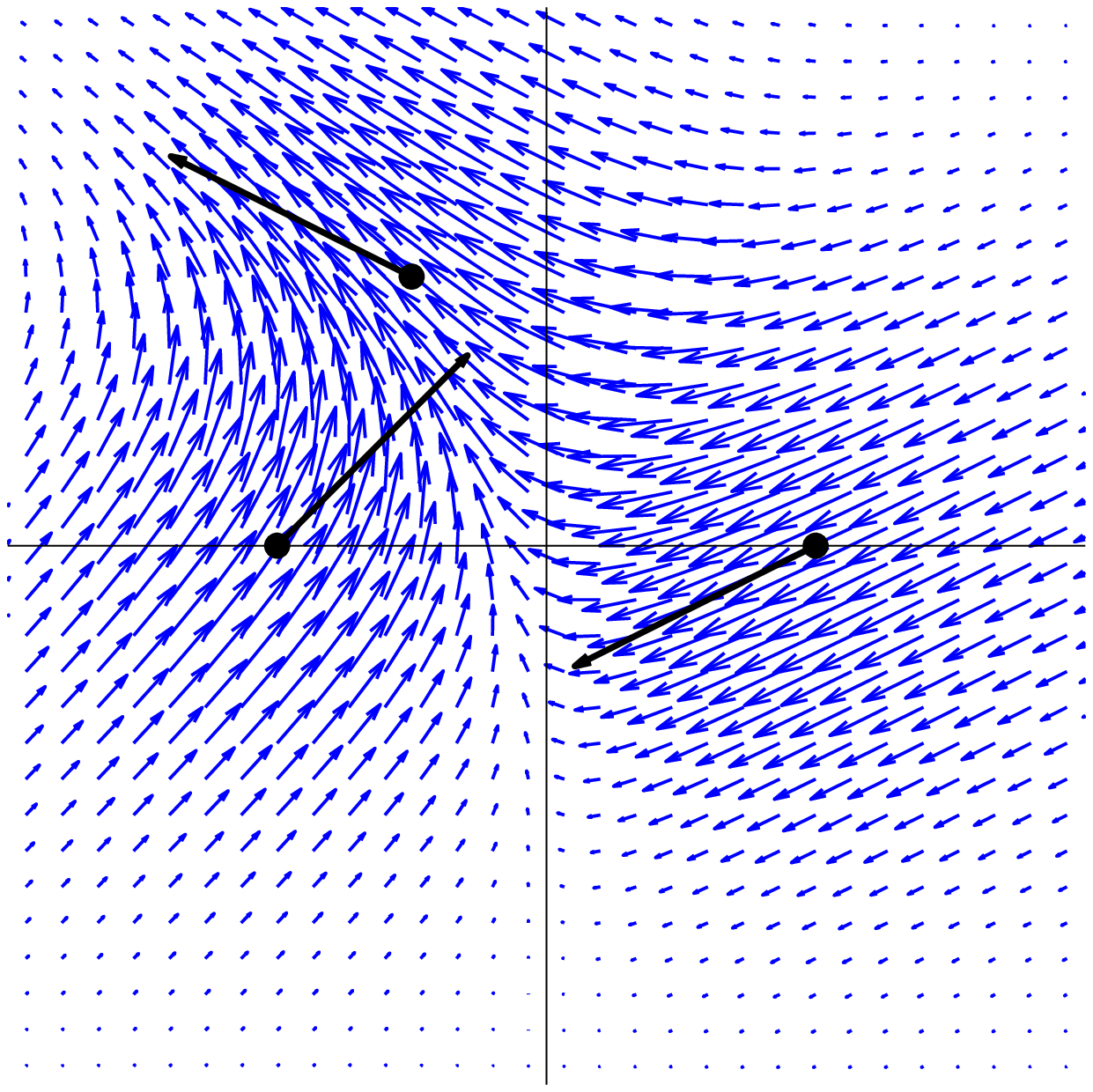}}
\put(150,0){\includegraphics[height=5 cm]{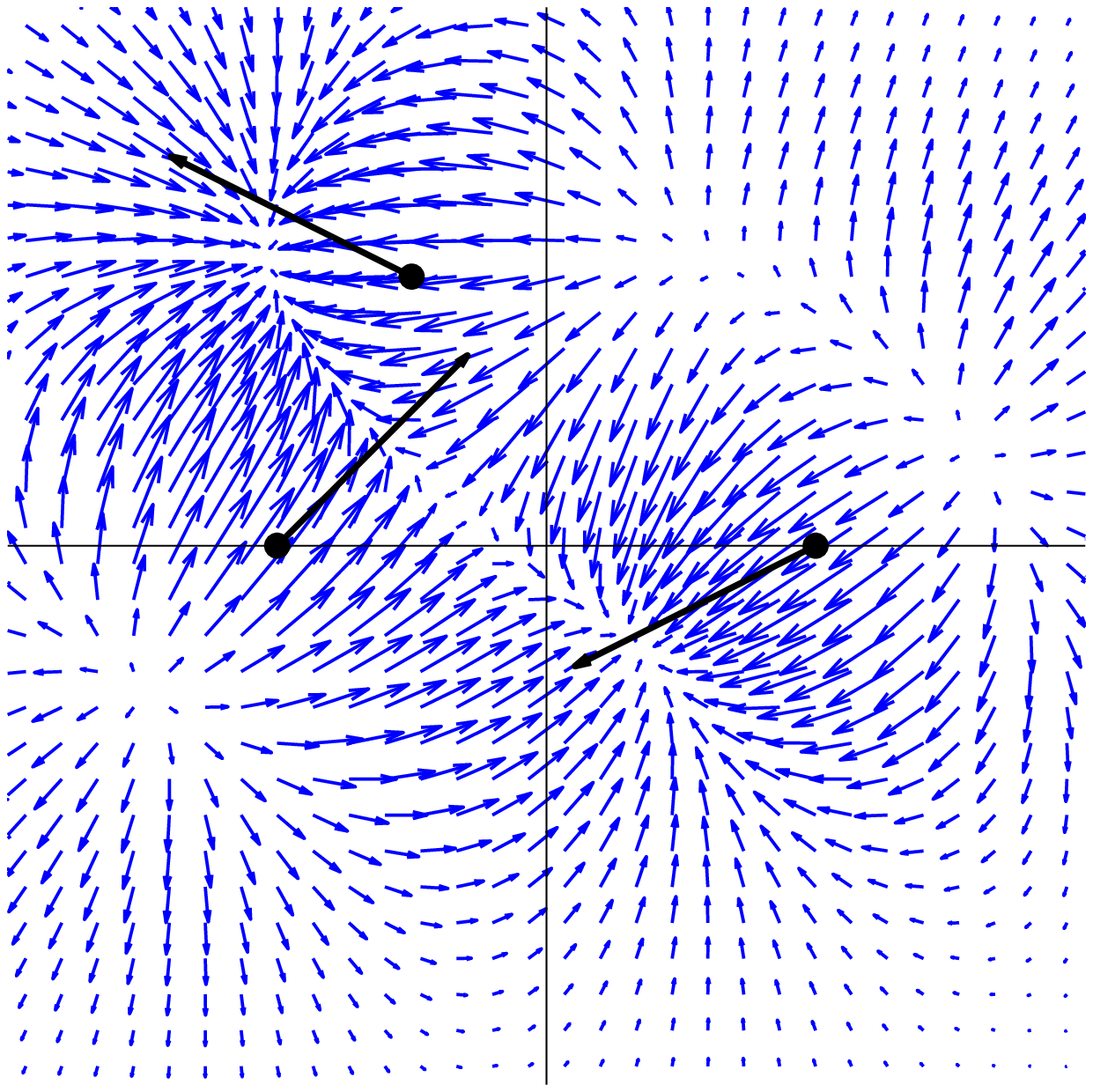}}
\put(300,0){\includegraphics[height=5 cm]{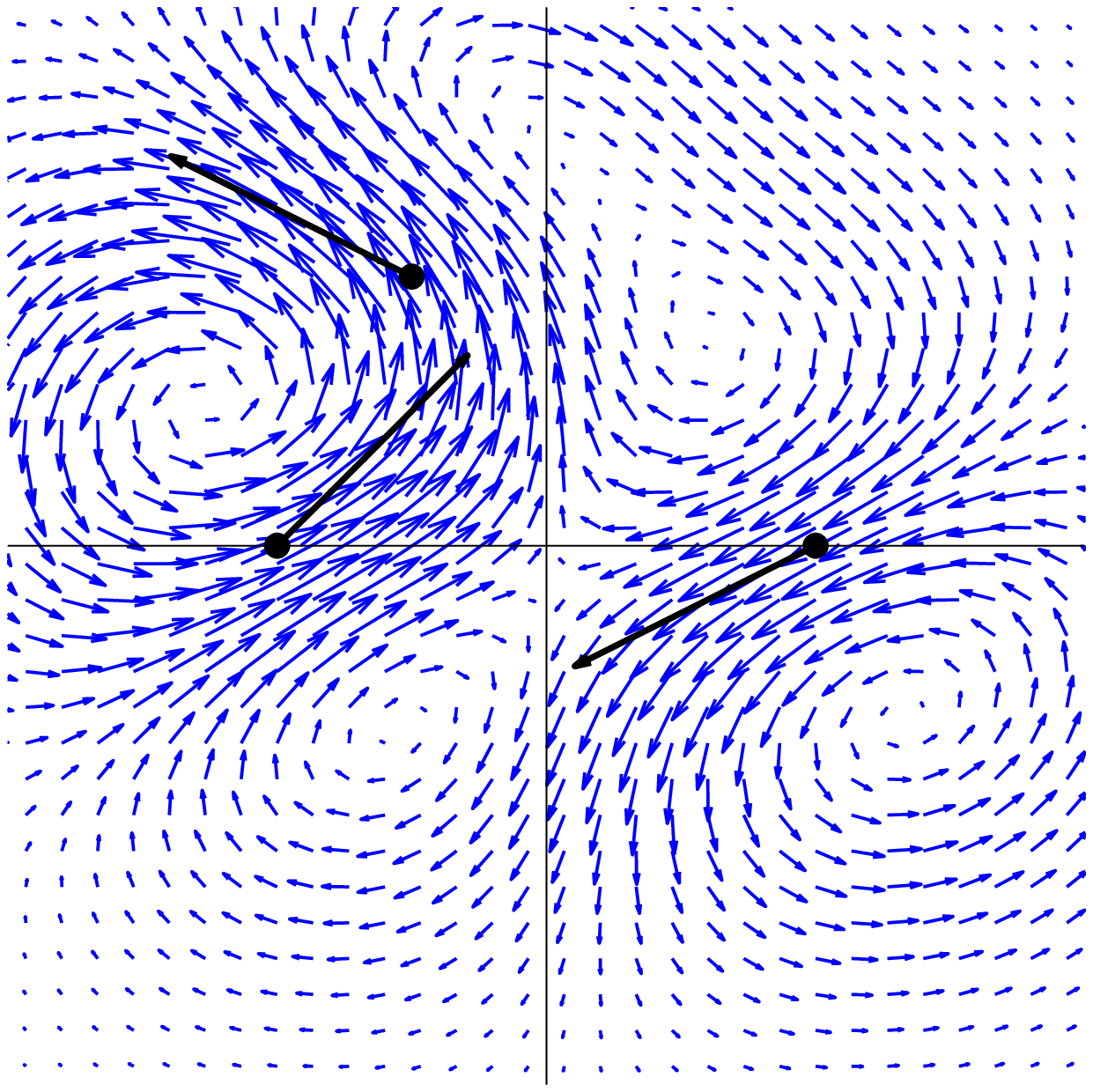}}
\put(70,-10){\makebox(0,0){\small scalar kernel}}
\put(220,-10){\makebox(0,0){\small curl-free kernel}}
\put(370,-10){\makebox(0,0){\small divergence-free kernel}}
%\put(233,5){\makebox(0,0){\small$\sum_{a=1}^3\textcolor{brick}{K}(x-q^a)p_a$}}
%\put(162,85){\makebox(0,0){\small$T_m\mathcal{M}$}}
%\put(95,20){\makebox(0,0){\small $J_1$}}
%\put(154,58){\makebox(0,0)[r]{\small $m$}}
%\put(230,50){\makebox(0,0)[l]{$\mathrm{Log}_m(J_1)=\,???$}}
\end{picture}
\end{center}
\caption{
%\mbox{%
Plots of the vector field~$%\displaystyle
v(x)=\sum_{a=1}^3\!\kk(x-q^a)\alpha_a$ for three different choices of the kernel,  
with the same landmarks~$q^a$ and momenta~$\alpha_a$, $a=1,2,3$.}
%exponential map for scalar, curl-free, and divergence-free Gaussian-type kernels.
%}
\label{Fig3fields}
\end{figure}
Figure~\ref{Fig3fields}
shows the vector field~$u(x)=\sum_{a=1}^3\kk(x-q^a)\alpha_a$,
which is of the type~\eqref{sIP2}, with landmarks 
\mbox{$q^1=(-1,0)$},
$q^2=(-1/2,1)$, 
$q^3=(1,0)$
in~$\mathcal{L}^3(\mathbb{R}^2)$,
and momenta
$\alpha_1=(1/\sqrt{2},1/\sqrt{2})$,
$\alpha_2=(-2/\sqrt{5},1/\sqrt{5})$, 
$\alpha_3=(-2/\sqrt{5},-1/\sqrt{5})$
(so in this case~$\|\alpha_a\|_{\mathbb{R}^2}=1$ for $a=1,2,3$), 
but for three different choices of
the TRI kernel~$\kk$. The first one is the Gaussian
scalar kernel~$\kk(x)=k(\|x\|)\mathbb{I}_2$
where~$k(r)=e^{-cr^2}$, with~$c=1$,
the second one is the kernel~\eqref{coe_cf}
from Example~\ref{exGauss2}, 
and the third one is the kernel~\eqref{coe_df}
from Example~\ref{exGauss};
in the last two cases, $a=2$, $b=1$, $c=1$ and $d=2$,
so that the two kernels are, respectively, curl-free and divegence-free. It 
should be apparent 
that given the completely different nature of the kernels
and the corresponding vector fields,
so should be %that 
the landmark trajectories obtained by integrating
Hamilton's equation~\eqref{ham_1} and~\eqref{ham_2}
%should be radically different too for such kernel choices, 
as well as %so should 
%be 
the corresponding diffeomorphisms of the ambient space. %~$\varphi_{01}^v$.
Here we shall give examples of such 
trajectories and diffeomorphisms for the above choices of the kernel.
\par
\begin{figure}[t]
\begin{center}
\begin{picture}(500,320)
\put(0,0){\includegraphics[height=5.7 cm]{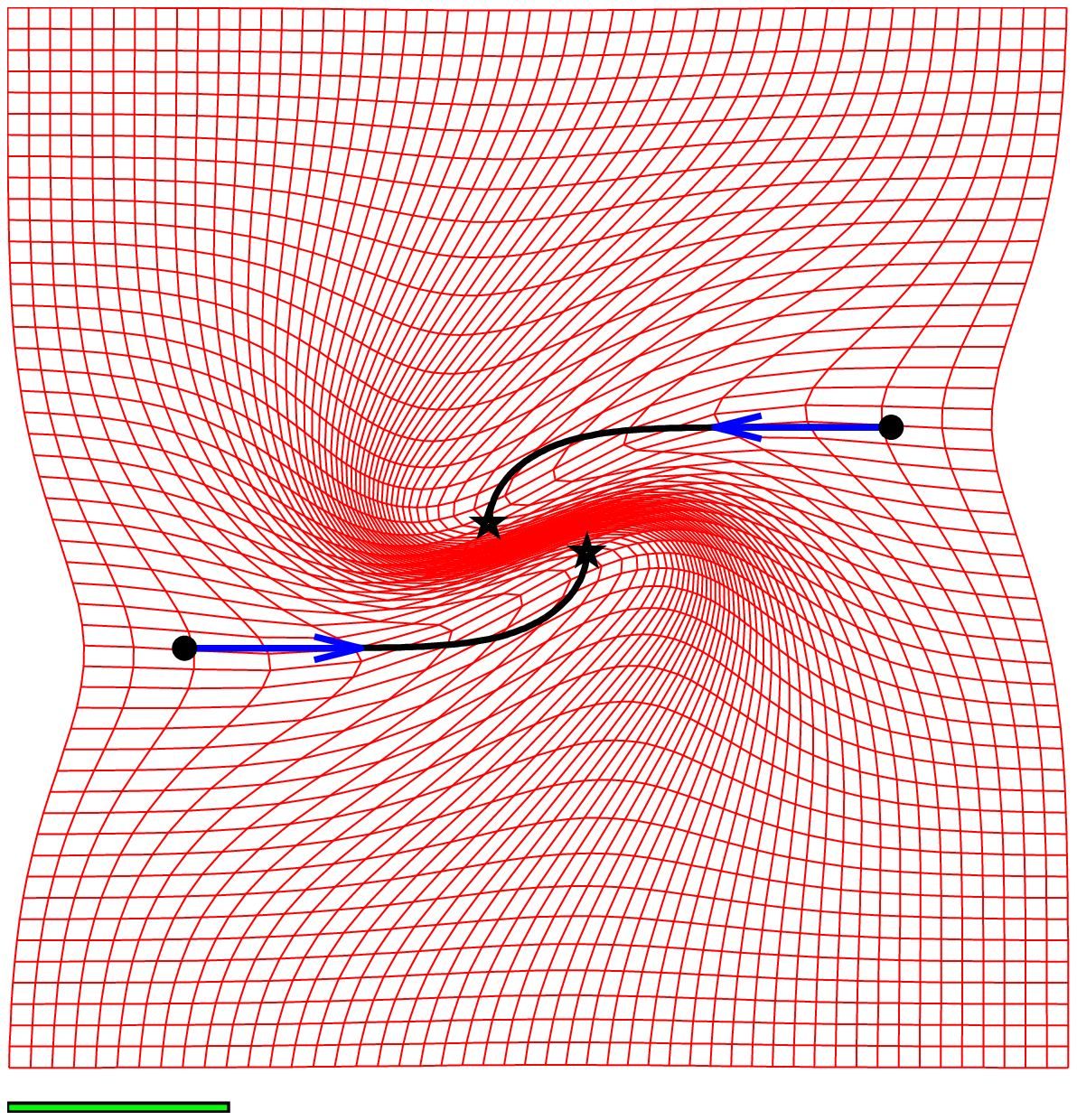}}
\put(170,0){\includegraphics[height=5.7 cm]{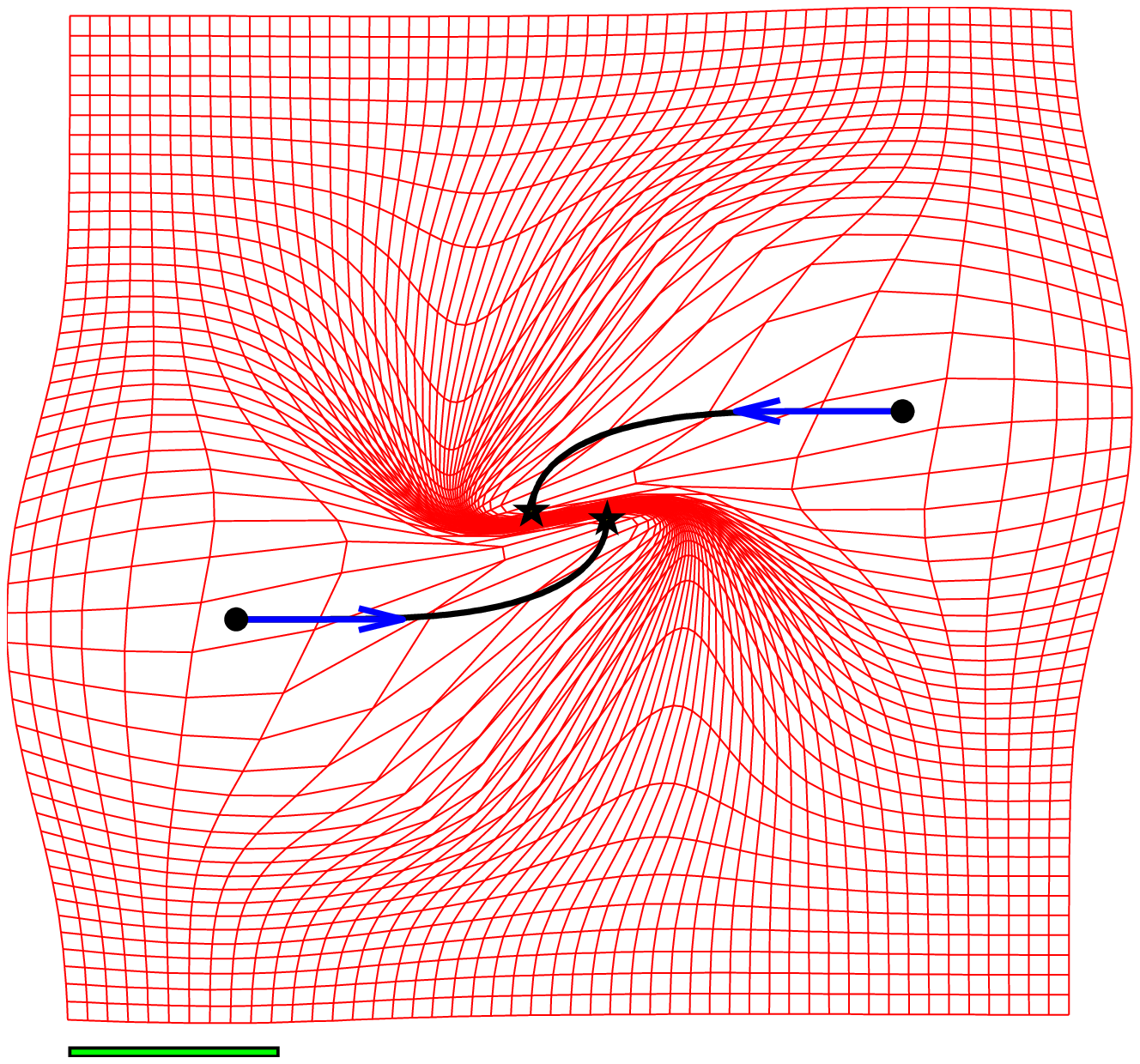}}
\put(340,0){\includegraphics[height=5.7 cm]{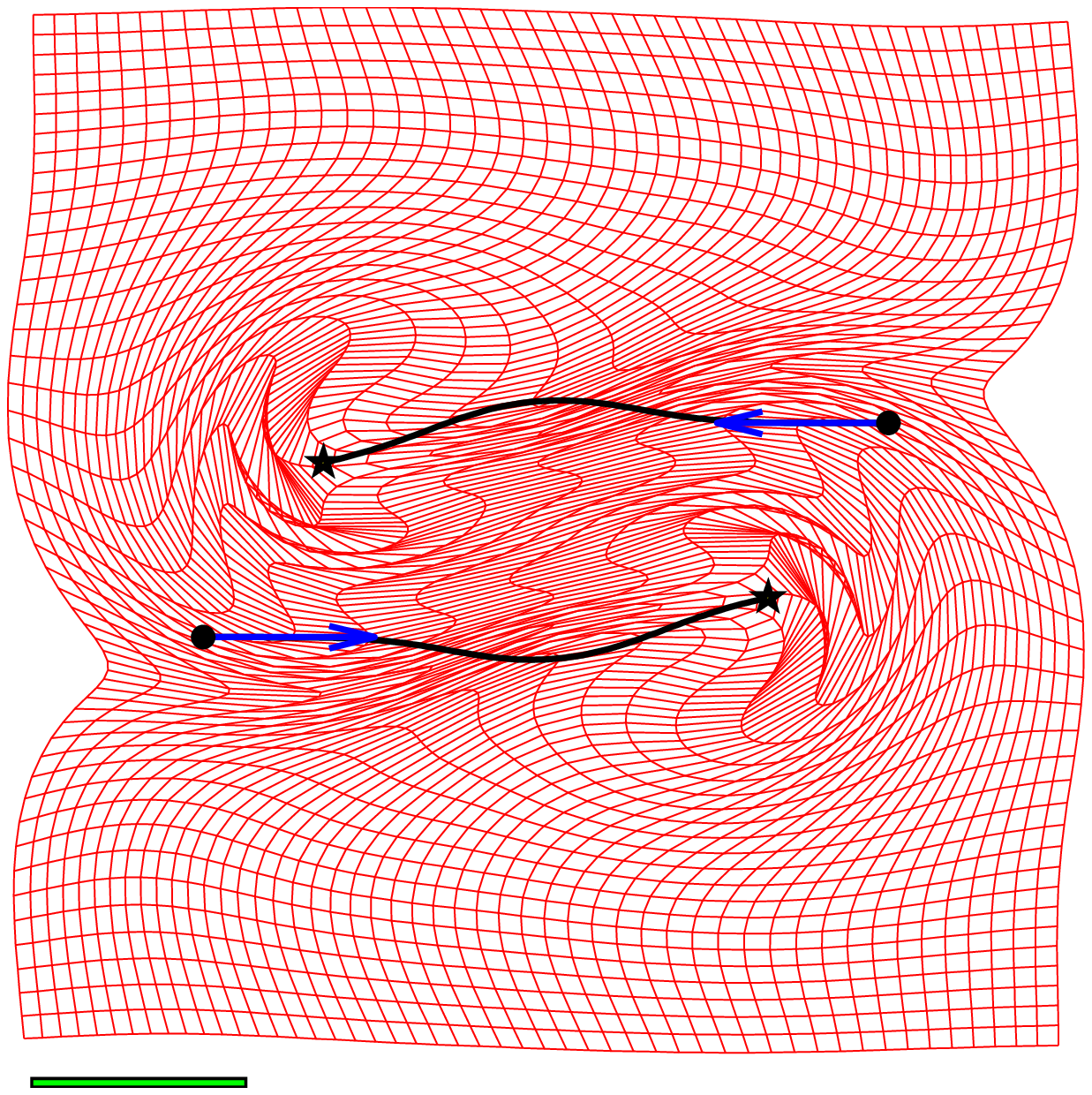}}
\put(80,-10){\makebox(0,0){\small scalar kernel}}
\put(250,-10){\makebox(0,0){\small curl-free kernel}}
\put(420,-10){\makebox(0,0){\small divergence-free kernel}}
\put(-10,80){\makebox(0,0){\small (B)}}
\put(0,170){\includegraphics[height=5.7 cm]{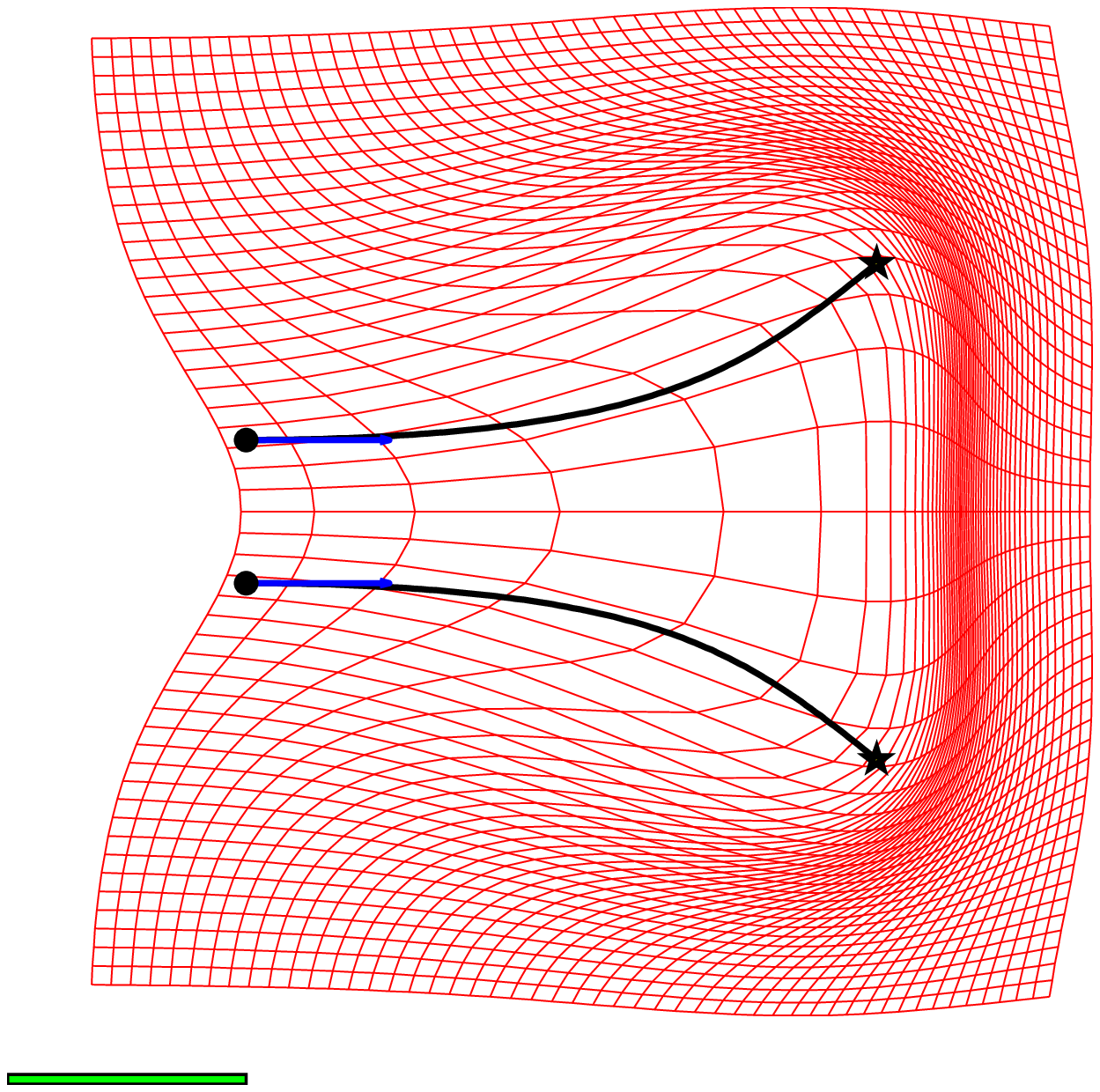}}
\put(170,170){\includegraphics[height=5.7 cm]{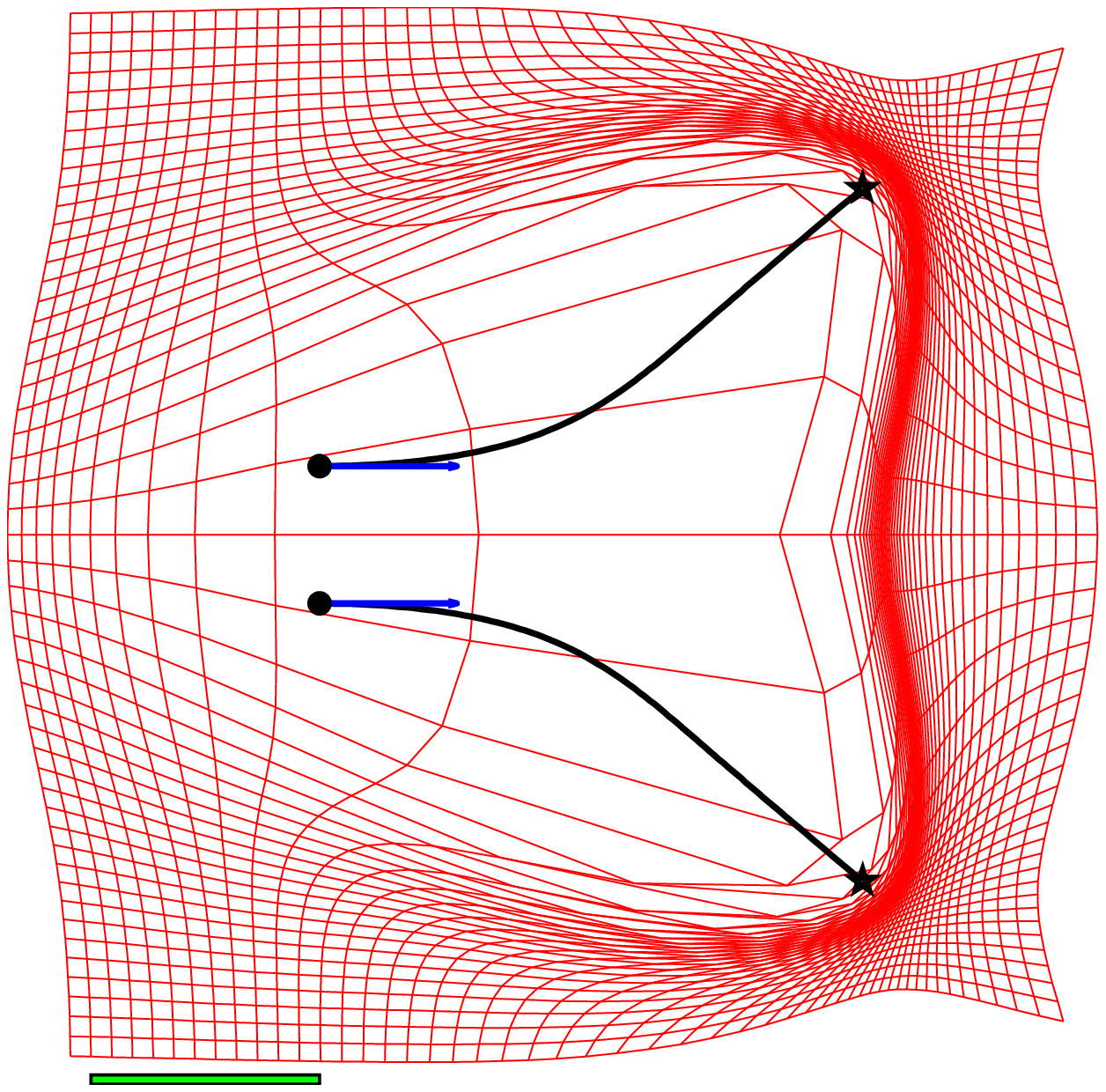}}
\put(340,170){\includegraphics[height=5.7 cm]{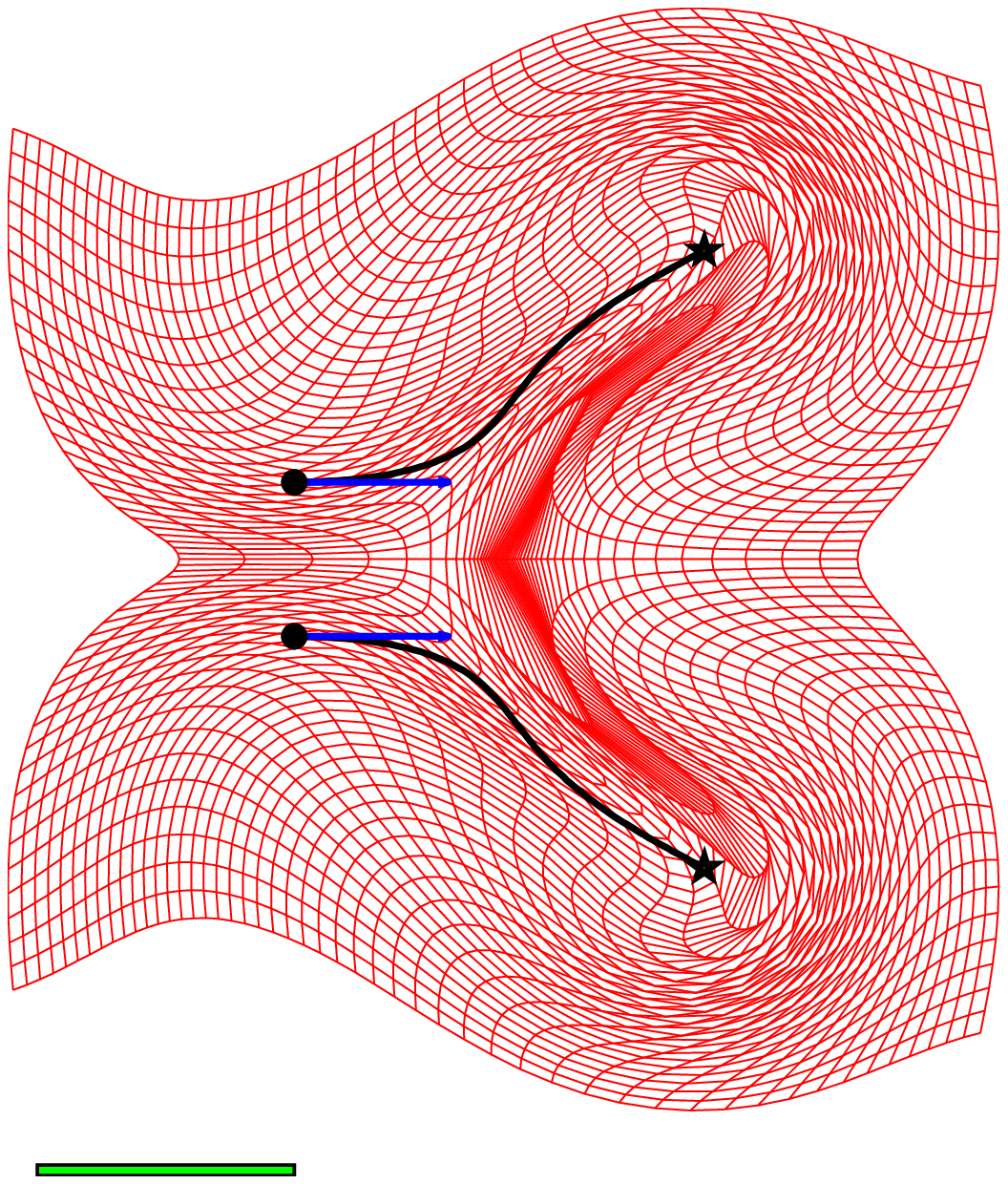}}
\put(-10,250){\makebox(0,0){\small (A)}}
%\put(233,5){\makebox(0,0){\small$\sum_{a=1}^3\textcolor{brick}{K}(x-q^a)p_a$}}
%\put(162,85){\makebox(0,0){\small$T_m\mathcal{M}$}}
%\put(95,20){\makebox(0,0){\small $J_1$}}
%\put(154,58){\makebox(0,0)[r]{\small $m$}}
%\put(230,50){\makebox(0,0)[l]{$\mathrm{Log}_m(J_1)=\,???$}}
\end{picture}
\end{center}
\caption{
%\mbox{%
Landmark trajectories in~$\mathcal{L}^2(\mathbb{R}^2)$ for three different choices of the 
kernel (scalar, curl-free, and divergence-free); bullets~($\bullet$) and stars~($\star$)
are the initial and final configurations, respectively, and the arrows are the initial momenta. The grid 
represents the diffeomorphism~$\varphi_{01}^v$ of the ambient space.
The bar at the bottom left of each graph represents the ``width''~$1/\sqrt{c}$
of each kernel.
}
%exponential map for scalar, curl-free, and divergence-free Gaussian-type kernels.
%}
\label{Fig3diff}
\end{figure}
\par
{\bf Shooting landmarks.} Figure~\ref{Fig3diff} 
shows the landmark trajectories that result from 
shooting landmarks, i.e.~from solving Hamilton's
equations~\eqref{ham_1} and~\eqref{ham_2} as an initial value problem, 
namely with specified initial landmark sets
in~$\mathcal{L}^2(\mathbb{R}^2)$ and initial momenta;
the grid in each graph represent the 
corresponding diffeomorphism~$\varphi_{01}^v$,
which is obtained
from the landmark trajectories by 
integrating numerically 
the differential 
equation~$\partial_t\varphi_{0t}^v(t,x)=v(t,\varphi^v_{0t}(x))$
with initial condition~\mbox{$\varphi^v_{00}(x)=x$},
where~$v$ is given by equation~\eqref{sIP2}.
This is done
for the above three choices of Gaussian kernels (scalar, curl-free, and
divergence-free), with parameters~$c=16$, $b=1/(2c)$, $a=2bc$ and obviously~$d=2$
since we are working in two dimensions. 
Specifically, row~(A) shows 
the numerical solution of the above problem
when the initial landmark set is~$(q^1,q^2)=((0,0),(0,0.15))$
and the momenta are parallel and both 
oriented to the right, namely~$(p_1,p_2)=((15,0),(15,0))$. 
The geodesics in the scalar case had already been
examined in~\cite{MMM1}; we note that their behavior in the curl-free case is qualitatively similar.
By Liouville's Theorem~\cite{arnold:1}
a time-dependent divergence-free vector field~$v$ %is such that~$\mathrm{div}\, v\equiv0$,
generates a flow~$\varphi_{0t}^v$ that preserves volume; such preservation
shows clearly for the diffeomorphism 
in the last figure on the right, which is distinctly different from the previous two
in the shape of the landmark geodesics as well.
On the other hand, row~(B)
shows the geodesic curves resulting from
shooting landmarks in parallel but opposite directions: that is,
with initial configuration~$(q^1,q^2)=((-0.4,-0.125),(0.4,0.125))$ 
and initial momenta~$(p_1,p_2)=((20,0),(-20,0))$.
Once again, we notice that volume is preserved in the diffeomorphism
induced by the divergence-free kernel;
the path followed by the landmarks
in this case is also characteristically different, in that in order to 
avoid compression %of the ``fluid'' 
the two landmarks do not spiral but tend to steer away from each other.
%{\bf Matching of 2 points - opposite} Figure~... shows the landmark trajectories that result from matching

{\bf Exponential map.} Finally, Figure~\ref{FigExp} shows
exponential maps $\exp_q\!\big(t\dot{q}(0)\big)$~\cite{jost} in~$\mathcal{L}^2(\mathbb{R}^2)$, 
for the above choices of the Gaussian 
kernel (scalar, curl-free, and divergence-free);
with the same choice of parameters as in the landmark shooting examples.
In all three cases, the \em initial \em landmark positions are fixed 
at~$q=(q^1,q^2)=((0,-0.125),(0,0.125))$,
whereas the \em initial \em momenta are of the type
$p=(p_1,p_2)=(50(\cos\theta,\sin\theta),50(\cos\theta,-\sin\theta))$
with~$\theta\in[-\frac{\pi}{2},\frac{\pi}{2}]$
(the first of Hamilton's equations~\eqref{ham_1} relates the momenta~$p$ and the 
landmark velocities),
and the paramenter~$t$ is sampled in~[0,1].
We observe that the exponential maps for the Gaussian scalar and curl-free kernels
are qualitatively similar, in that they have one ``folding'': this 
illustrates the existence of {\it conjugate points\/}, i.e.~landmark sets %configurations
in~$\mathcal{L}^2(\mathbb{R}^2)$ 
that are connected to the initial configuration~$q=(q^1,q^2)$ 
by distinct geodesic paths. The situation becomes
considerably more complex in the divergence-free kernel case, where we observe
\em multiple \em foldings: %in fact 
the exponential map keeps adding extra folds as~$t$ 
increases, and the existence of several distinct trajectories connecting 
landmark sets to the initial configuration is apparent. 
\begin{figure}[t]
\begin{center}
\begin{picture}(500,240)
\setlength{\unitlength}{1pt}
\put(0,-10){\includegraphics[height=8.5cm]{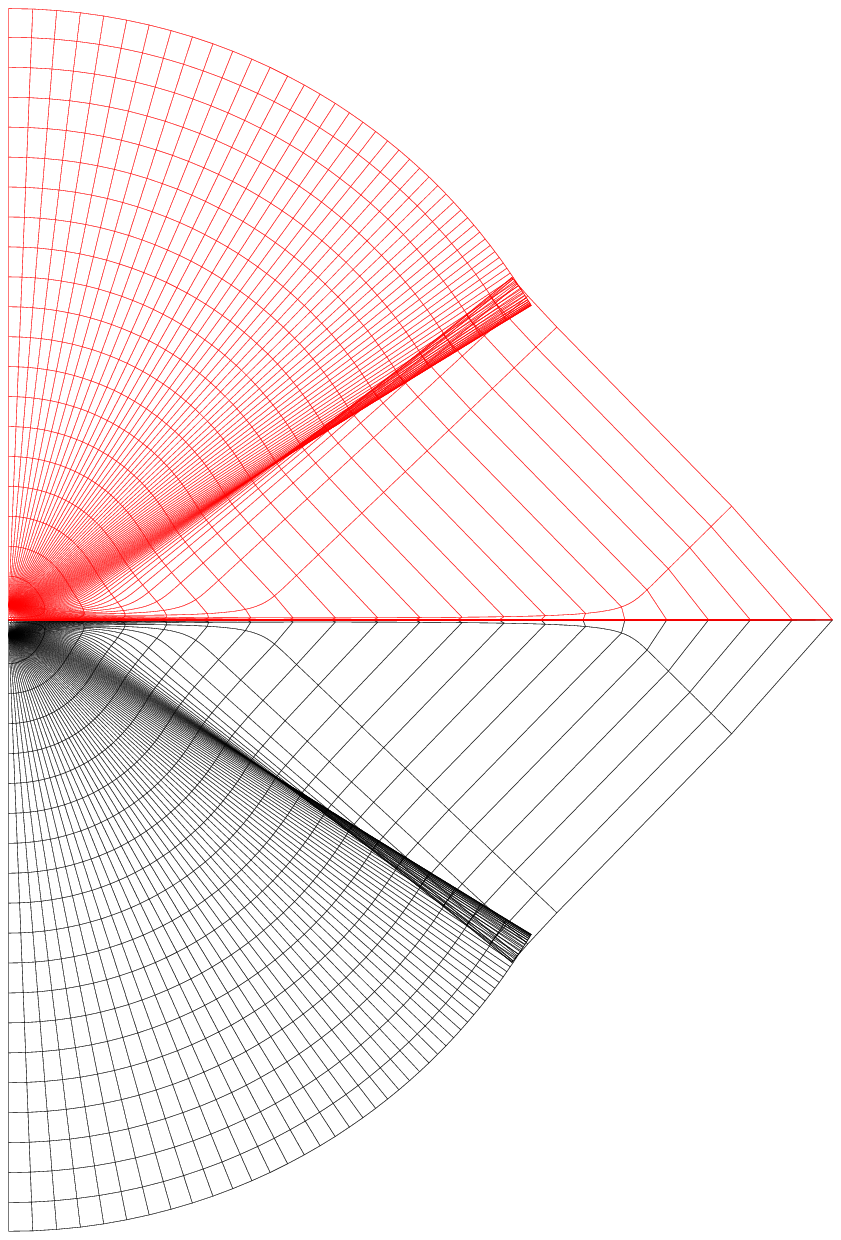}}
\put(170,-10){\includegraphics[height=8.5cm]{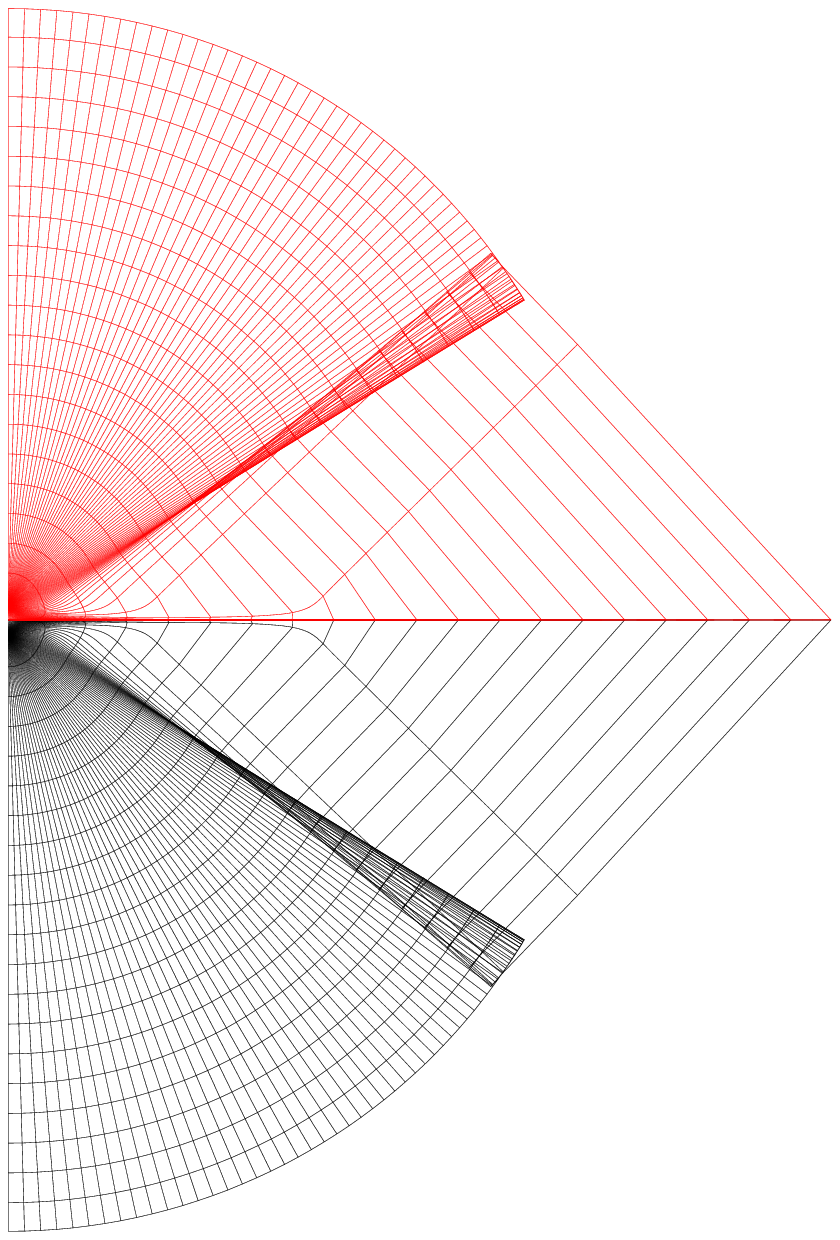}}
\put(340,-10){\includegraphics[height=8.5cm]{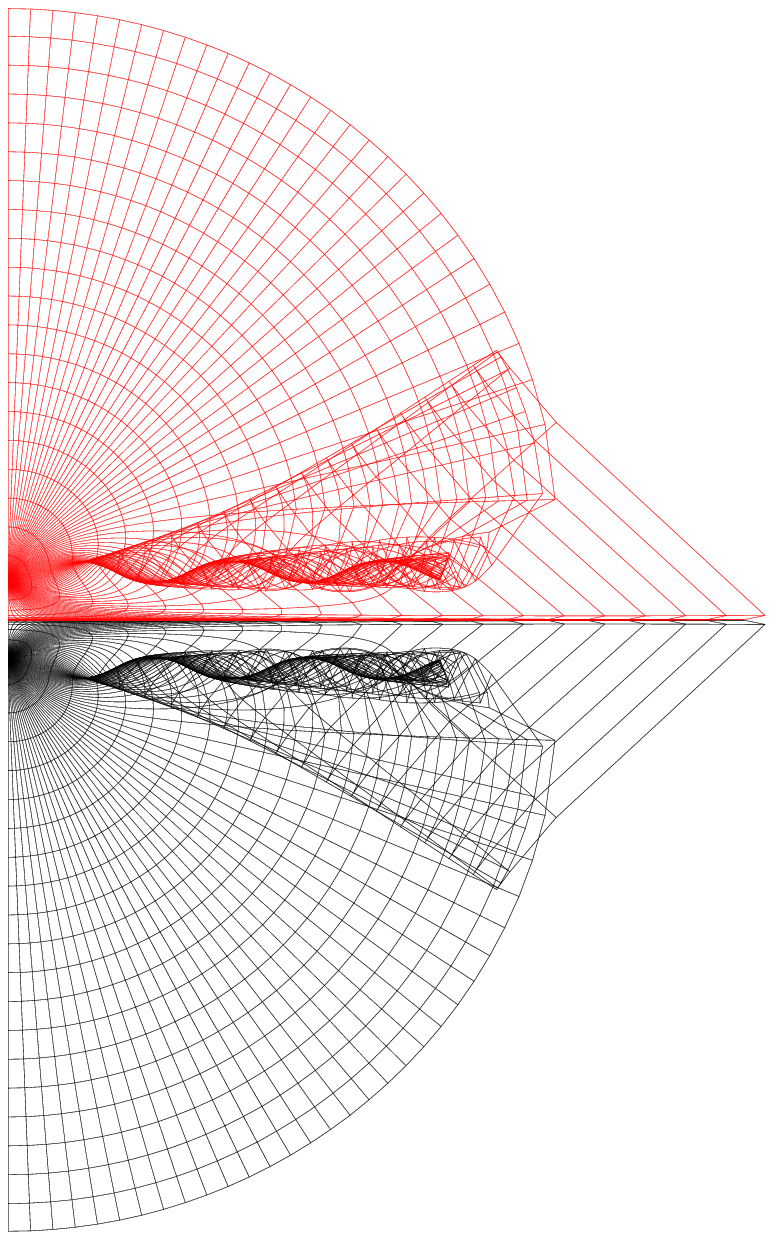}}
\put(-5,103.5){\makebox(0,0){\small$q^1$}}
\put(-5,116.0){\makebox(0,0){\small$q^2$}}
\put(0,106.5){\makebox(0,0){\tiny$\bullet$}}
\put(0,113.0){\makebox(0,0){\tiny$\bullet$}}
\put(95,10){\makebox(0,0){\bf scalar}}
\put(95,0){\makebox(0,0){\bf kernel}}
\put(265,10){\makebox(0,0){\small\bf curl-free}}
\put(265,0){\makebox(0,0){\bf kernel}}
\put(455,10){\makebox(0,0){\bf divergence-free}}
\put(455,0){\makebox(0,0){\bf kernel}}
\put(-12,145){\rotatebox{90}{\small increasing $t$}}
\put(-1.6,140){\vector(0,1){65}} %
\put(-12,30){\rotatebox{90}{\small increasing $t$}}
\put(-1.6,85){\vector(0,-1){65}} %
\put(35,229){\vector(-4,1){0}} %
\put(75,209){\vector(1,-1){0}} %
\qbezier[180](35,229)(59,223)(75,209)
\put(78,242){\makebox(0,0){\small varying}}
\put(78,233){\makebox(0,0){\small initial momenta}}
\put(35,180){\colorbox{white}{\small$p_2$}}
\put(35,38){\colorbox{white}{\small$p_1$}}
%\put(95,0){\makebox(0,0){\small kernel}}
\thicklines
\put(0,113){\vector(1,2){52}} %
\put(0,106.5){\vector(1,-2){52}} %
\end{picture}
\end{center}
\caption{
Exponential maps~$\exp_q\!\big(t\dot{q}(0)\big)$ in~$\mathcal{L}^2(\mathbb{R}^2)$ for scalar, 
curl-free, and divergence-free Gaussian-type kernels;
the initial landmark configuration~$q=(q^1,q^2)$ %for all three maps 
is shown
in the first graph. The black and red grids illustrate the geodesic 
trajectories initiating from~$q^1$ and~$q^2$
respectively, for different choices of the initial momenta 
($p_1$ and~$p_2$ are chosen symmetrically with respect to the~$x$-axis),
and sampled in~$t\in[0,1]$.
}
\label{FigExp}
\end{figure}\par
\begin{figure}
\begin{center}
\begin{picture}(425,210)
\setlength{\unitlength}{1pt}
\put(0,70){\includegraphics[height=5cm]{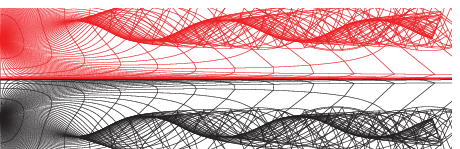}}
\put(-30,140){\makebox(0,0){\small (A)}}
\put(-30,15){\makebox(0,0){\small (B)}}
\put(0,-10){\includegraphics[height=1.88cm]{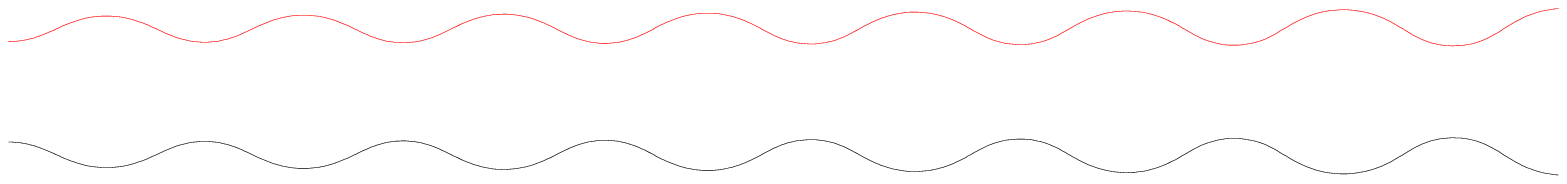}}
\put(-6,0){\makebox(0,0){\small$q^1$}}
\put(-6,30){\makebox(0,0){\small$q^2$}}
\put(0,30){\makebox(0,0){\tiny$\bullet$}}
\put(0,0){\makebox(0,0){\tiny$\bullet$}}
\put(0,100){\makebox(0,0){\tiny$\bullet$}}
\put(0,180){\makebox(0,0){\tiny$\bullet$}}
\put(-6,100){\makebox(0,0){\small$q^1$}}
\put(-6,180){\makebox(0,0){\small$q^2$}}
\end{picture}
\end{center}
\caption{
(A) Detail of the exponential map and (B)
a landmark trajectory in the divergence-free case.
}
\label{FigExp2}
\end{figure}
Figure~\ref{FigExp2} shows a detail of the same exponential map
in the region where most foldings occur,
and a single trajectory of landmarks in the divergence-free case 
(with initial horizontal momenta,~$\theta=0$), at different spatial scales;
note that the latter trajectory exhibits an apparent ``periodic'' behavior. Our interpretation
of this phenomenon, that only 
appears in the divergence-free case, is that for two landmarks to move in the same direction 
in an incompressible fluid
they have to ``swim'' in such fashion so that the fluid gets out of the way as they 
proceed along their paths.
In any case, this is certainly a qualitative 
behavior that needs to be investigated more and suggests that
the differential geometry of the manifold of landmark points~$\mathcal{L}^N(\Rd)$
with the metric~\eqref{gtensor},
with matrix-valued kernels, should also be studied further.
\section{Conclusions}
\label{secCon}
In this paper we have provided a thorough and systematic study of 
matrix-valued definite positive kernels
that induce translation- and rotation-invariant metrics in the corresponding 
Reproducing Kernel Hilbert Spaces.
We have conducted a novel 
analysis of such kernels 
in the Fourier domain: here the classification
of spaces of curl-free and divergence free vector
fields is very natural, in that each corresponds to
one of the two coefficients of the Fourier transform 
of the kernel being equal to zero. This allows
one to perform the Hodge decomposition in RKHS
via the computation of Hankel transforms (which 
may be done analytically or numerically).
We have also devised a method 
for building curl-free and divergence-free 
kernels via an appropriate 
differentiation of scalar kernels;
we have also proven that this constructive
procedure is general, in that any TRI kernel
(more precisely, the terms of its Hodge
decomposition)
may be obtained this way. 
Such method may thus be used to
design a large class of positive kernels.
We have given a  summary 
of the application of the LDDMM approach 
to the interpolation of vector fields and the matching
of landmark points, together with the 
differential-geometric interpretation and the 
derivation of the corresponding Hamiltonian. 
Finally, we have given numerical examples of 
the dynamics of landmark points and the corresponding
flow of diffeomorphisms for different choices of
TRI kernels.
\par
Future work should focus on the design 
of kernels that are appropriate for different applications,
especially in computational anatomy and related fields.
For example, the shape of geodesics should be 
studied in relation to the choice of the kernel; in particular,
note that all the kernels that we have considered in the examples
have a length constant (e.g.~the parameter~$\sigma>0$ for Bessel-type kernels), 
so that the induced dynamics in the 
space of landmarks (and, for that matter, 
in the group of diffeomorphisms itself) is not scale invariant; furthermore, the regularity
of the kernels at zero and the heaviness of their tails are also
crucial in determining both the qualitative dynamics and the  
complexity of numerical implementations.
Last,  but not least, %it is apparent from 
the visualizetion of the exponential map 
in the previous section 
(especially in the case of divergence-free TRI kernels) reveals
the presence of several conjugate points, i.e.~landmark
configurations that are connected by distinct geodesics;
%an understanding of 
the study of the dynamics of~$\mathcal{L}^N(\Rd)$ 
induced by the action of diffeomorphism groups
should be furthered, and 
the investigation of the 
differential geometry of the space of landmarks
%induced by the action of diffeomorphism groups
that was started in~\cite{MMM1} in the case of scalar kernels,
which is useful for the ultimate goal of performing statistics
on such manifold,
should be %at least in part
extended to the case of generic TRI kernels.

%as discussed in~\cite{MMM1} and related papers,
%this would provide useful information

%as this would be useful 
%allow one to infer about 
%the uniqueness of geodesics between shapes, the existence of conjugate points, and the well-
%posedness of the problem of computing the intrinsic (or KarcherÕs) mean of a database of shapes

\section{Acknowledgements}
This work was supported by the Agence Nationale de la Recherche of France 
(ANR grant ``KaraMetria'' ANR-09-BLAN-0332). The authors wish
to thank David Mumford for his inspirational vision, 
invaluable advice, and continuous encouragement and support.
\appendix
\par\vspace{.8cm}
%\clearpage
\noindent{\Large{\bf Appendices}}
%\section{Proof of Proposition~\ref{scalarK}}
%\section{Proof of Theorem~\ref{exist}}
%\label{appExist}
%\section{Proof of Theorem~\ref{bochner2}}
%\section{Recurrence relations of Bessel functions% of the first kind
%~\cite[\S9.1.7]{abramowitz} } 
\section{Relevant properties of Hankel transforms and Bessel functions}
\label{appBessel}
%{$\bullet$ 
{\bf $\bullet$ Fourier-Bessel 
integral}~\cite[Chapter~2]{ditkin}~\cite[Chapter~1]{oberhettinger}{\bf:}
this is also known as Hankel's integral formula, or the inversion formula 
for Hankel transforms:
\begin{equation}
\label{inv_Htf2}
\int_0^\infty
\Big\{
\int_0^\infty
f(r) \, J_\mu(2\pi\varrho r) \,r\,dr
\Big\} J_\mu(2\pi t\varrho)\, \varrho\,d\varrho=\frac{f(t)}{(2\pi)^2},
\quad t>0.
%$$
\end{equation}
%\noindent{$\bullet$ 
\noindent{\bf $\bullet$ Recurrence relations for Bessel functions of the 
first kind}~\cite[\S9.1.7]{abramowitz}{\bf:}
\begin{align}
\label{Jrec1}
J_{\nu-1}(z)+J_{\nu+1}(z) & = \frac{2\nu}{z}J_\nu(z)
\\
\label{Jrec3}
J_{\nu-1}(z)-\frac{\nu}{z}J_{\nu}(z) & = J'_\nu(z)
\\
\label{Jrec4}
-J_{\nu+1}(z)+\frac{\nu}{z}J_{\nu}(z) & = J'_\nu(z)
\end{align}
\noindent{\bf $\bullet$ A result on modified Bessel functions of the second kind}{\bf:}
\begin{lemma} 
\label{lembess}
For fixed~$\sigma>0$ and arbitrary~$\nu$ 
the function~$\displaystyle f_\nu(r):=\Big(\frac{r}{\sigma}\Big)^\nu
K_\nu\Big(\frac{r}{\sigma}\Big)$, $r>0$,
is such that:
\begin{align}
\label{auxfnu}
f'_\nu(r)
&=
-\frac{1}{\sigma}
\Big(\frac{r}{\sigma}\Big)^\nu
K_{\nu-1}\Big(\frac{r}{\sigma}\Big),
&
f''_\nu(r)
&=
\frac{1}{\sigma^2}f_\nu(r)+\frac{2\nu-1}{r}f'_\nu(r)\,.
\end{align}
\end{lemma}
\begin{proof}
We have
$
f'_\nu(r)
=
\frac{\nu}{\sigma}
\big(\frac{r}{\sigma}\big)^{\nu-1}
K_{\nu}\big(\frac{r}{\sigma}\big)
+\frac{1}{\sigma}
\big(\frac{r}{\sigma}\big)^{\nu}
K'_{\nu}\big(\frac{r}{\sigma}\big)
$; 
using the 
property~$K'_\nu(z)=-K_{\nu-1}(z)-\frac{\nu}{z}K_\nu(z)$ 
of modified Bessel functions~\cite[\S9.6.26]{abramowitz}
yields the first of~\eqref{auxfnu}. Differentiating the latter
gives
$$
f''_\nu(r)=
\frac{d}{dr}
\Big\{
-\frac{1}{\sigma}\Big(\frac{r}{\sigma}\Big)
\cdot
\Big(\frac{r}{\sigma}\Big)^{\nu-1}K_{\nu-1}
(\frac{r}{\sigma}\Big)
\Big\}
\stackrel{(\ast)}{=}
-\frac{1}{\sigma^2}\Big(\frac{r}{\sigma}\Big)^{\nu-1}
K_{\nu-1}
\Big(\frac{r}{\sigma}\Big)
+
\frac{1}{\sigma^2}
\Big(\frac{r}{\sigma}\Big)^{\nu}
K_{\nu-2}
\Big(\frac{r}{\sigma}\Big),
$$
where in~$(\ast)$ we have used the first of equations~\eqref{auxfnu}
with~$\nu-1$ instead of~$\nu$. At this point we use
the property $K_{\nu-2}(z)=K_{\nu}(z)-2\frac{\nu-1}{z}K_{\nu-1}(z)$ of modified Bessel 
functions~\cite[\S9.6.26]{abramowitz},
and elementary manipulations lead to the second of equations~\eqref{auxfnu}.
\end{proof}
\noindent{\bf $\bullet$ Asymptotic expansion for large 
arguments of the modified Bessel function~$K_\nu$}~\cite[\S9.7.2]{abramowitz}{\bf:}
\begin{equation}
\label{KnuLarge}
K_{\nu}(z)
\sim
\sqrt{\frac{\pi}{2z}}
e^{-z}
\Big\{
1
+
\frac{\mu-1}{8z}
+
\frac{(\mu-1)(\mu-9)}{2!(8z)^2}
+
\frac{(\mu-1)(\mu-9)(\mu-25)}{3!(8z)^3}
+
\ldots
\Big\}
\end{equation}
valid for fixed~$\nu$ and $|\arg z|<\frac{3}{2}\pi$; in the above,~$\mu=4\nu^2$.
\section{Bochner's Theorem and consequences}
\label{appBochner}
We shall reformulate the generalized version of Bochner's theorem as follows:
\begin{theorem}[Bochner]
\label{bochner3}
Consider a matrix-valued function~$\kk\in L^1(\Rm,\Rdd)$ whose 
Fourier transform~$\widehat{\kk}$ is also~$L^1$, and 
define~$K(x,y):=\kk(x-y)$,
 $x,y\in \Rm$. The following facts are equivalent:
%\par\vspace*{-.25cm}
\begin{packed_enum}
\item $K$ is a positive definite function;
\item for all~$f\in C^\infty_c(\Rm,\mathbb{C}^d)$ it is the case that 
$\displaystyle\int_\Rm f^\ast(\xi)\, \widehat{\mathbf{ k}}(\xi)\, f(\xi)\, d\xi\geq 0$;
\item for all~$\xi\in\Rm$, it is the case that\
 $\widehat{\mathbf{ k}}(\xi)$ is a positive hermitian matrix.
\end{packed_enum}
\end{theorem}
\begin{proof}
We have that~$\mathbf{ k}(x)=\int_\Rm
\widehat{\mathbf{ k}}(\xi)e^{2\pi i\langle x,\xi\rangle}d\xi
$, $x\in\Rm$
so we first note that for fixed~$N\in \mathbb{N}$,
${\boldsymbol x} =(x_1,\ldots,x_N)\in (\Rm)^N$ and ${\boldsymbol\alpha}=(\alpha_1,\ldots,\alpha_N)\in(\Rd)^N$
we may write
\begin{equation}
\label{rewrite}
\sum_{a,b=1}^N \alpha_a\cdot \mathbf{ k}(x_a-x_b)\alpha_b
%& 
= 
\int_\Rm
\fxa^\ast(\xi)\,\kkh(\xi)\fxa(\xi)\,d\xi,
\end{equation}
where~$\fxa(\xi):=\sum_{a=1}^N e^{-2\pi i x_a\cdot\xi}\alpha_a$. We now proceed with the proof.
\par
(i.$\,\Rightarrow\,$ii.) 
Fix $f\in C^\infty_c(\Rm,\mathbb{C}^d)$ and~$\varepsilon>0$.
By the continuity of integration, there exists~$A>0$ such that %the hypercube 
$C_A:=[-A,A]^m$ verifies~$\int_{\Rm\setminus C_A}\|\kkh(\xi)\|d\xi<\varepsilon$.
We can choose~$A$ large enough such that~$\mathrm{supp} f\subseteq C_A$.
Consider now the \em periodic extension \em of~$f$ in~$\Rm$ and its
Fourier series expansion:%, i.e.~the function
$$
f_{\mathrm{per}}(\xi):=
\sum_{p\in\mathbb{Z}^m}
f(\xi-Ap)
=
\sum_{p\in\mathbb{Z}^m}
c_p e^{-2\pi i\frac{p}{2A}\cdot\xi},\;\xi\in\Rm,
\quad
\mathrm{with}
\;\;
c_p:=\int_\Rm f(\xi)\, 
e^{2\pi i \frac{p}{2A}\cdot\xi}\,d\xi,\;p\in \mathbb{Z}^m.
$$
The above Fourier series converges uniformly,
whence 
there exist~$M\in\mathbb{N}$, points~\mbox{$p_1,\ldots,p_M\in\mathbb{Z}^m$}
such that $\big|f(\xi)-\sum_{a=1}^N
\exp\big({-2\pi i \frac{p_a}{2A}\cdot\xi}\big)\alpha_a\big|<\varepsilon$ for all
$\xi\in C_A$,
where we have set~$\alpha_a:=c_{p_a}$, for $a=1,\ldots,M$. In other words, 
there exist ${\boldsymbol x}=(x_1,\ldots,x_M)$ 
and~${\boldsymbol \alpha}=(\alpha_1,\ldots,\alpha_M)$ such that 
%\mbox{
$$\textstyle%\|f\|_{C_A,\infty}:=
\sup_{\xi\in C_A}\big\{\big|f(\xi)-\sum_{a=1}^N
\exp\big({-2\pi i  x_a\cdot\xi}\big)\alpha_a\big|\big\}<\varepsilon.
$$
Since~$\|f-\fxa\|_{C_A,\infty}<\varepsilon$ we %also 
have 
%that 
$\|f_{\mathrm{per}}-\fxa\|_{\infty}<\varepsilon$,
whence
%and therefore 
$\|\fxa\|_\infty<\|f_\mathrm{per}\|_\infty+\varepsilon=\|f\|_\infty+\varepsilon$.
Now,
\begin{align*}
\bigg|
\int_\Rm
&
\fxa^\ast(\xi)\,
\kkh(\xi)\fxa(\xi)\,d\xi
-
\int_\Rm
f^\ast(\xi)\,\kkh(\xi)f(\xi)\,d\xi\,
\bigg|
\\
&
\leq
\underbrace{\bigg|
\int_{\Rm\setminus C_A}
\fxa^\ast(\xi)\,\kkh(\xi)\fxa(\xi)\,d\xi
\,
\bigg|}_{I_1}
+\underbrace{\bigg|
\int_{C_A}
\fxa^\ast(\xi)\,\kkh(\xi)\fxa(\xi)\,d\xi
-
\int_{C_A}
f^\ast(\xi)\,\kkh(\xi)f(\xi)\,d\xi\,
\bigg|}_{I_2},
\end{align*}
%with
\begin{align*}
\mbox{with }\;\; 
I_1&\leq
\int_{\Rm\setminus C_A}
\big\|\kkh(\xi)\big\|
\|\fxa(\xi)\|^2d\xi
<
\big(\|f\|_\infty+\varepsilon
\big)^2
\int_{\Rm\setminus C_A}
\big\|\kkh(\xi)\big\|
d\xi
<
\big(\|f\|_\infty+\varepsilon
\big)^2\varepsilon
\\
\mbox{and }\;\; 
I_2&\leq 
\bigg|
\int_{C_A}
(\fxa(\xi)-f(\xi))^\ast\,\kkh(\xi)\fxa(\xi)\,d\xi
\,\bigg|
+
\bigg|
\int_{C_A}
f^\ast(\xi)\,\kkh(\xi)(\fxa(\xi)-f(\xi))\,d\xi\,
\bigg|
\\
&\leq
\int_{C_A}\|\fxa(\xi)-f(\xi)\| 
\big\|\kkh(\xi)\big\|
\|\fxa(\xi)\|
\,
d \xi
+
\int_{C_A}
\|f(\xi)\|
\big\|\kkh(\xi)\big\|
\|\fxa(\xi)-f(\xi)\|
\,
d \xi
\\
&<
\varepsilon
\|\kkh\|_{L^1}
\big(\|f\|_\infty+\varepsilon\big)
+
\|f\|_\infty
\|\kkh\|_{L^1}\,
\varepsilon,
\end{align*}
where~$\|\kkh\|_{L_1}=\int_\Rm\|\kkh(\xi)\|d\xi$. Therefore $I_1+I_2 \leq 
\big(\|f\|_\infty+\varepsilon
\big)^2\varepsilon
+
\|\kkh\|_{L^1}
\big(\|f\|_\infty+\varepsilon
\big)\varepsilon
+
\|\kkh\|_{L^1}
\|f\|_\infty
\,\varepsilon$,
which concludes this part of the proof.
\par
(ii.$\,\Rightarrow\,$iii.) 
The matrix~$\kkh(\xi)$ is self-adjoint for all~$\xi$ 
by the arbitrariness of~$f\in C^\infty_c(\Rm,\mathbb{C}^d)$.
Assume now that for a~$\xi_0\in\Rm$ the matrix~$\kkh(\xi_0)$
has an eigenvalue~$\lambda_0<0$, with eigenvector~\mbox{$\gamma_0\in\mathbb{C}^m$}.
By the continuity of~$\kkh$ there exists a neighborhood~$\Omega_0$ of~$\xi_0$
such that $\gamma_0^\ast\kkh(\xi)\gamma_0<0$ 
in $\Omega_0$; let~$f(\xi):=\psi(\xi)\gamma_0$, $\xi\in\Rm$, where 
$\psi\in C_c^\infty(\Rm,\mathbb{R})$, with $\mathrm{supp}\,\psi\subseteq\Omega_0$. Then
$\int_\Rm f^\ast\kkh f\, d\xi=\int_{\Omega_0}\psi^2\gamma_0^\ast\kkh\gamma_0\,d\xi<0$,
which is a contradiction.
\par
(iii.$\,\Rightarrow\,$i.) If~$\kkh(\xi)$ is self-adjoint with nonnegative 
eigenvalues for all~$\xi\in\Rm$ then the right-hand side of~\eqref{rewrite}
is nonnegative for all~${\boldsymbol x}$
and~${\boldsymbol\alpha}$; whence the function~$K(x,y)$ is positive definite.
\end{proof}
We now proceed with the proof of Proposition~\ref{pos_h}, that states
that translation- and rotation-invariant kernels~$\kk$ of the type~\eqref{invker}
are \em strictly \em positive definite as long as at least one of 
the two coefficients~$\AAA,\BBB:
\mathbb{R}^+\rightarrow \mathbb{R}$ 
of the Fourier transform of~$\kk$ is positive somewhere.
\begin{proof}[Proof of Prop.~\ref{pos_h}]
In the TRI kernel setting~$m=d$.
Note that for any 
${\boldsymbol x} =(x_1,\ldots,x_N)\in (\Rd)^N$ and~${\boldsymbol\alpha}=(\alpha_1,\ldots,\alpha_N)\in(\Rd)^N$ (where at least one of the~$\alpha_a$'s is nonzero),
and for any~$\xi\in\Rd$, the vector~$\fxa(\xi)\in \mathbb{C}^m
%\mathbb
%:=\sum_{a=1}^N e^{-2\pi i\langle x_a,\xi\rangle}\alpha_a
$ introduced in the previous proof
is \em neither \em parallel \em nor \em perpendicular to~$\xi$.
\par
First assume that there exists~$r_0\geq0$ such that $\AAA(r_0)>0$ or
$\BBB(r_0)>0$. Choose~$\xi_0\in\Rd$ such that~$\|\xi_0\|=r_0$,
and arbitrary ${\boldsymbol x}=(x_1,\ldots,x_N)$ and~${\boldsymbol\alpha}=(\alpha_1,\ldots,\alpha_N)$. 
We shall write:
$$
\fxa(\xi_0)=u^\parallel_0+u^\perp_0, 
\qquad
\mathrm{with}\;\; u_0^\parallel:=\mathrm{Pr}_{\xi_0}^\parallel \fxa(\xi_0)
%\langle\fxa,\frac{\xi_0}{\|\xi_0\|}\rangle\frac{\xi_0}{\|\xi_0\|}
\;\;
\mathrm{and}\;\; u_0^\perp:=\mathrm{Pr}_{\xi_0}^\perp \fxa(\xi_0);
$$ 
by the remark above, we have that~$u_0^\parallel\not=0$ and~$u_0^\perp\not=0$. 
It is immediate to verify that 
\begin{align*}
\fxa^\ast(\xi_0)&
\mathrm{Pr}^\parallel_{\xi_0}
\fxa(\xi_0)=\|u_0^\parallel\|^2>0
&
&\mathrm{and}
&
\fxa^\ast(\xi_0)&
\mathrm{Pr}^\perp_{\xi_0}
\fxa(\xi_0)=\|u_0^\perp\|^2>0,
\end{align*}
while $\fxa^\ast(\xi)
\mathrm{Pr}^\parallel_{\xi}
\fxa(\xi)\geq0$ and
$\fxa^\ast(\xi)
\mathrm{Pr}^\perp_{\xi}
\,\fxa(\xi)\geq0$ for generic~$\xi\in\Rd$.
In conclusion
$$
\int_\Rd
\fxa^\ast(\xi)\,\kkh(\xi)\fxa(\xi)\,d\xi
=
\int_\Rd
\Big\{
h^\parallel(\|\xi\|)
\fxa^\ast(\xi)
\mathrm{Pr}^\parallel_{\xi}
\fxa(\xi)
+
h^\perp(\|\xi\|)
\fxa^\ast(\xi)
\mathrm{Pr}^\perp_{\xi}
\fxa(\xi)
\Big\}d\xi>0,
$$
and by~\eqref{rewrite} the kernel is \em strictly \em definite positive.
Vice versa, by~\eqref{rewrite} if~$\kk$ is strictly definite positive then 
the above strict inequality holds for all  
${\boldsymbol x} =(x_1,\ldots,x_N)\in (\Rd)^N$ 
and~\mbox{${\boldsymbol\alpha}=(\alpha_1,\ldots,\alpha_N)\in(\Rd)^N$} where at least one of the~$\alpha_a$'s is nonzero.
Therefore either~$\AAA$ or~$\BBB$ must be strictly positive 
somewhere in~$(0,\infty)$. This completes the proof.
\end{proof}
\section{Proof of Proposition~\ref{FT}}
\label{AppA}
Note that since~$f$ is symmetric so must
be its Fourier transform, i.e.~$\widehat{f}(\xi)=\widehat{f}(-\xi)$.
We now prove
that~it is actually \em radially \em symmetric. 
Using polar coordinates in~$\Rd$~\cite{folland} we can write:
\begin{equation}
\label{ffxi}
\widehat{f}(\xi)
=\widehat{f}(-\xi)
:=
\int_{\Rd}
f(x)
e^{2\pi i  x\cdot\xi}dx
=
\int_0^\infty
r^{\dd-1}g(r)
\int_{\mathbb{S}^{\dd-1}}
e^{2\pi i r s\cdot\xi}
\,ds\,dr,
\qquad
\xi\in \Rd,
\end{equation}
where we use the symbol~$\mathbb{S}^{n}$ for the unit $n$-sphere. Fix~$\xi$, and
let~$\eta$
be such that~$\|\eta\|=\|\xi\|$; in other words $\eta=R\xi$
for some $R\in O(\Rd)$. We have that
$
\widehat{f}(\eta)
=
\int_0^\infty
r^{\dd-1}g(r)
\int_{\mathbb{S}^{\dd-1}}
e^{2\pi i r s\cdot\eta}
ds\,dr$, 
but
$$
\int_{\mathbb{S}^{\dd-1}}
e^{2\pi i r s\cdot\eta}
ds
=
\int_{\mathbb{S}^{\dd-1}}
e^{2\pi i r s\cdot R\xi }
ds
=
\int_{\mathbb{S}^{\dd-1}}
e^{2\pi i r (R^T\!s)\cdot\xi }
ds
=
\int_{\mathbb{S}^{\dd-1}}
e^{2\pi i r\,s'\cdot\xi}
ds',
$$
since $|\mathrm{det}\, R|=1$. Whence $\widehat{f}(\eta)=\widehat{f}(\xi)$,  
i.e.~$\widehat{f}(\xi)=G(\|\xi\|)$,
for some function~$G: \mathbb{R}^+\rightarrow \mathbb{R}$.
Without loss of generality in~\eqref{ffxi} we can choose $\xi$ to be parallel
to the standard basis vector~$e_1$. Letting
$\varrho=\|\xi\|$ 
and $\cos\varphi= s\cdot e_1 $
we can write~$\xi\cdot s=\varrho\cos\varphi$
(note that~$\|s\|=1$ on~$\mathbb{S}^{\dd-1}$). This yields:
\begin{align}
\nonumber
\int_{\mathbb{S}^{\dd-1}}
&
e^{2\pi ir\, s\cdot \xi}\,
ds
=
\int_{\mathbb{S}^{\dd-1}}
e^{2\pi i\varrho r\cos\varphi }\,
ds
=
\sigma(\mathbb{S}^{\dd-2})
\int_0^\pi
e^{2\pi i\varrho r\cos\varphi }
(\sin\varphi)^{\dd-2}
\;d\varphi
\\
&
\stackrel{(\ast)}{=}
\sigma(\mathbb{S}^{\dd-2})
%\frac{2\pi^\frac{n-1}{2}}{\Gamma(\frac{n-1}{2})}
\int_{-1}^1
e^{2\pi i\varrho ru}
(1-u^2)^{\frac{\dd-3}{2}}\,du
=
2\sigma(\mathbb{S}^{\dd-2})
\int_{0}^1
(1-u^2)^{\frac{\dd-3}{2}}
\cos(2\pi \varrho ru)
\,du,
\label{auxffxi}
\end{align}
where $\sigma(\mathbb{S}^n)
=
2\pi^\frac{n+1}{2}/\Gamma(\frac{n+1}{2})
$ is the surface area of the unit $n$-sphere;
in~($\ast$) we have performed the substitution $u=\cos\varphi$.
We now use the integral representation of Bessel functions% of the first kind 
~\cite[\S9.1.20]{abramowitz}
$$
J_\nu(z)
=
\frac{2(\frac{1}{2}z)^{\nu}}{\sqrt{\pi}\,\Gamma(\nu+\frac{1}{2})}
%\Big(\frac{z}{2}\Big)^\nu
\int_{0}^1
%e^{iz\xi}
(1-u^2)^{\nu-\frac{1}{2}}\cos(zu)
\,du\,,
\qquad
\mbox{ valid for }
{\textstyle\Re\nu>-\frac{1}{2}},
$$
%with $\Re\nu>-\frac{1}{2}$, 
to get
$$
\int_{0}^1
(1-u^2)^{\frac{\dd-3}{2}}
\cos(2\pi \varrho ru)
\,du
=
\frac{\sqrt{\pi}\,\Gamma(\frac{\dd-1}{2})}{2(\pi\varrho r)^{\frac{\dd}{2}-1}}\;
J_{\frac{\dd}{2}-1}(2\pi\varrho r);
$$
inserting this expression into~\eqref{auxffxi} and the latter into~\eqref{ffxi} 
finally completes the proof.

%\nocite{janich}
%\nocite{bracewell}
%\nocite{evans}
%\nocite{papoulis}
%\nocite{sneddon}
%\nocite{vladimirov}
%\clearpage
%\bibliography{/Users/mariomicheli/Dropbox/WORK/micheli}
\bibliography{micheli}
\bibliographystyle{abbrv}

\end{document}